\documentclass{imsart}

\usepackage{amsmath, amsfonts, amssymb, amsthm, graphicx, color,caption}
\usepackage{epsfig}
\usepackage{fancyhdr}
\usepackage{framed}
\usepackage{floatrow}
\usepackage{mdframed}
\usepackage{url}
\usepackage{mathabx}
\usepackage{cfr-lm}
\usepackage{color}

\usepackage{imakeidx}
\makeindex

\RequirePackage[numbers]{natbib}
\RequirePackage[colorlinks,citecolor=blue,urlcolor=blue]{hyperref}

\theoremstyle{plain}
\newtheorem{lem}{Lemma}
\newtheorem*{lem*}{Lemma}
\newtheorem{thm}{Theorem}
\newtheorem{prop}{Proposition}

\newtheorem*{prop*}{Proposition}
\newtheorem{Def}{Definition}
\newtheorem{cor}{Corollary}
\newtheorem*{cor*}{Corollary}
\newtheorem{con}{Construction}

\theoremstyle{definition}
\newtheorem{ex}{Example}
\newtheorem{Rem}{Remark}
\newtheorem{obs}{Observation}

\def \v {\mathrm{v}}
\def \h {\mathrm{h}}
\def \cS {\mathcal{S}}
\def \cG {\mathcal{G}}
\def \cL {\mathcal{L}}
\def \cT {\mathcal{T}}
\def \cR {\mathcal{R}}
\def\cN{{\mathcal{N}}}
\def \cE {\mathcal{E}}
\def \cB {\mathcal{B}}
\def \cBT {\mathcal{BT}}
\def \cBL {\mathcal{BL}}
\def \cV {\mathcal{V}}
\def \cM {\mathcal{M}}

\def \cD {\mathcal{D}}
\def \cH {\mathcal{H}}

\def\cQ{{\mathcal{Q}}}
\def \E {\mathcal{E}^{\rm ex}}
\def \L {\mathcal{L}_{\rm plane}}
\def \mBL {\widetilde{\mathcal{BL}}_{\rm plane}}
\def \T {\mathcal{T}_{\rm plane}}
\def \S {\mathcal{S}_t}
\def \mS {\widetilde{\mathcal{S}}_t}
\def \D {\mathcal{D}_t}
\def \BT {\mathcal{BT}_{\rm plane}}
\def \BL {\mathcal{BL}_{\rm plane}}

\def \E {{\sf E}}
\def \Var {{\sf Var}}

\def\be{\begin{equation}}
\def\ee{\end{equation}}

\startlocaldefs
\endlocaldefs

\begin{document}

\begin{frontmatter}

\title{Random Self-Similar Trees:\\ A mathematical theory of Horton laws\thanksref{t1}}

\thankstext{t1}{This is an original survey paper}
\runtitle{Random Self-Similar Trees}


\author{\fnms{Yevgeniy} \snm{Kovchegov}\thanksref{t2}\corref{}\ead[label=e1]{kovchegy@math.oregonstate.edu}}
\thankstext{t2}{The work is supported by FAPESP award 2018/07826-5 and by NSF award DMS-1412557.}
\address{Department of Mathematics, Oregon State University \\ 2000 SW Campus Way, Corvallis, OR 97331-4605\\ \printead{e1}}
\and
\author{\fnms{Ilya} \snm{Zaliapin}\thanksref{t3}\corref{}\ead[label=e2]{zal@unr.edu}}
\thankstext{t3}{The work is supported by NSF award EAR-1723033.}
\address{Department of Mathematics and Statistics, University of Nevada Reno, \\ 1664 North Virginia st., Reno, NV 89557-0084\\ \printead{e2}}

\runauthor{Y. Kovchegov and I. Zaliapin}

\begin{abstract}
The Horton laws originated in hydrology with a 1945 paper by Robert E. Horton, and for a long 
time remained a purely empirical finding.
Ubiquitous in hierarchical branching systems, 
the Horton laws have been rediscovered in many disciplines
ranging from geomorphology to genetics to computer science.
Attempts to build a mathematical foundation behind the Horton laws
during the 1990s revealed their close connection to 
the operation of pruning -- erasing a tree from the leaves down to the root.
This survey synthesizes recent results on invariances and self-similarities 
of tree measures under 
various forms of pruning.
We argue that pruning is an indispensable instrument for describing 
branching structures and 
representing a variety of coalescent and annihilation dynamics.
The Horton laws appear as a characteristic imprint of self-similarity,
which settles some questions prompted by geophysical data.
\end{abstract}




\end{frontmatter}

\tableofcontents


\section{Introduction}
Invariance of the Galton-Watson tree measures with respect to pruning (erasure) 
that begins at the leaves and progresses down to the tree root has been recognized 
since the late 1980s.
Both continuous \cite{Neveu86} and discrete \cite{BWW00} versions of prunings have been 
studied. 
The prune-invariance of the trees naturally translates to the 
symmetries of the respective Harris paths \cite{Harris}. 
The richness of such a connection is supported by the well-studied embeddings of 
the Galton-Watson trees in the excursions of random walks and Brownian motions 
(e.g., \cite{NP2,LeGall93,Pitman}).
This provides a point of departure for this survey of 
recent results on {\it prune-invariance},
and more restrictive {\it self-similarity}, of tree measures and related
stochastic processes on the real line. 
While the critical Galton-Watson tree and its Harris path (which is known to be a 
random walk) serve
as an important example, the results extend to trees with more complicated structure
and non-Markovian Harris paths. 
The main attention is paid to a discrete {\it Horton pruning} for finite trees
(Sects.~\ref{sec:defs_notations}-\ref{sec:Kingman}),
yet we also consider infinite and real trees, and general forms of pruning
(Sects.~\ref{sec:pruning}-\ref{sec:infiniteT}).
Looking at random trees through a prism of {\it self-similarity}
offers a concise parameterization of the respective measures via
their {\it Tokunaga sequences} (Sect.~\ref{TSS}), and
uncovers a variety of structures and symmetries 
(e.g., Thms.~\ref{thm:HLSST},\ref{HBPmain2},\ref{thm:mainGBM},\ref{Mthm},\ref{thm:main}).
The surveyed results suggest that particular forms of pruning may underline
the evolution of familiar dynamical systems, allowing their efficient
analytical treatment (Sects.~\ref{sec:Kingman},\ref{sec:annihilation}). 
The surveyed results also pose new questions related to random self-similar trees.

\medskip
We begin by summarizing the key empirical observations that provided an impetus 
for the topic 
(Sect.~\ref{sec:constraint}) and discussing the structure and main results of this
survey (Sect.~\ref{sec:structure}).
Here, we keep the references to a minimum, and indicate survey sections where 
one can find future information. 

\subsection{Early empirical evidence}
\label{sec:constraint}
The theory of random self-similar trees originated in the studies 
of river networks, which supplied the key empirical observations reviewed below. 
\index{random self-similar trees}

\medskip
\noindent{\bf Horton-Strahler orders (Sects.~\ref{sec:HS}, \ref{Rem:altHS}).}
Informally, the aim of orders is to quantify the importance of vertices and edges 
in the tree hierarchy.
It is natural to agree that the orders of a vertex and its parental edge are the same. 
Hence, we are only concerned with ordering vertices. 
In a perfect binary tree (where all leaves are located at the same depth, 
i.e., at the same  distance from the root) one can assign orders inversely 
proportional to the vertex depth; see Fig.~\ref{fig:HS_def}(a). 
In other words, we start with order $1$ at the leaves and increase the order by unity
with every step towards the root.

A celebrated ordering scheme that generalizes this idea to an arbitrary tree
(not necessarily binary) has been originally developed by Robert E. Horton
\cite{Horton45}, and later redesigned by Arthur N. Strahler \cite{Strahler}
to its present form. 
It assigns integer orders to tree vertices and edges, beginning with 
order $1$ at the leaves and increasing the order by unity every time 
a pair of edges of the same order meets at a vertex; see Fig.~\ref{fig:HS_def}(b). 
A sequence of adjacent vertices/edges with the same order is called a 
{\it branch}. 

An example of Horton-Strahler ordering is shown in Fig.~\ref{fig:Beaver}(a) for 
a small river network in the south-central US.
Here, the orders serve as a good proxy for (a logarithm of) various physical 
characteristics of river channels:
channel length, the area of the contributing basin, etc.
The Horton-Strahler orders (a.k.a. Strahler numbers) provide an efficient ranking
of the tree branches and have proven essential 
in numerous fields 
(see Sect.~\ref{sec:appl}).
As an example, the highest-order channel in a river basin 
commonly coincides with the basin's namesake river (e.g., Amazon river is the highest-order 
channel of the Amazon basin).
One may find it quite impressive that such an identification can be done 
using purely combinatorial properties of the basin.
Further examples of Horton-Strahler ordering are shown in 
Figs.~\ref{fig:terminology},\ref{fig:HS_example},\ref{fig:Tok_example1}.

\begin{figure}[t] 
\centering\includegraphics[width=0.9\textwidth]{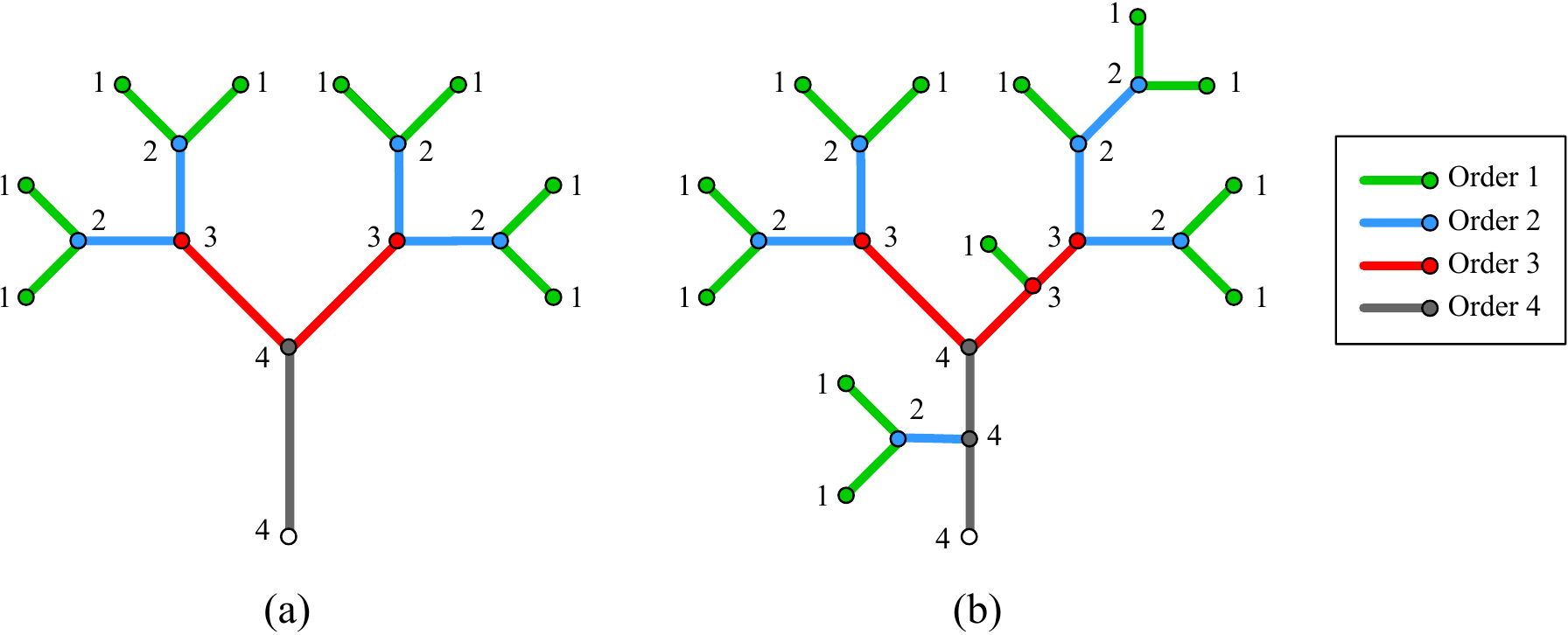}
\caption[Definition of Horton-Strahler orders]
{Horton-Strahler orders in a binary tree.
Different colors correspond to different orders of vertices and edges, 
as indicated in legend. 
(a) Perfect binary tree -- orders are inversely proportional to vertex/edge depth. 
(b) General binary tree -- orders are assigned according to the Horton-Strahler rule. }
\label{fig:HS_def}
\end{figure}

\medskip
\noindent {\bf Horton laws and Horton exponents (Sect.~\ref{sec:HLSST}).} 
A geometric decay of the number of branches of increasing Horton-Strahler orders was first 
described by Robert E. Horton \cite{Horton45} in a study of river stream networks.
Since then, the Horton law and its ramifications have proven indispensable in hydrology and have been reported 
in multiple other areas; see Sect.~\ref{sec:appl} for details and references. 

The Horton law for branch numbers
states that the numbers $N_K$ of channels (branches) of order $K$ in a large basin 
decay geometrically with the order:
\be\label{eq:Horton_emp}
\frac{N_K}{N_{K+1}}=R \quad \Leftrightarrow \quad N_K~\propto ~ R^{-K}
\ee
for some {\it Horton exponent} $R>1$.
Figure~\ref{fig:Beaver_a}(a) illustrates the Horton law for branch numbers 
in the Beaver creek network of 
Fig.~\ref{fig:Beaver}(a).
In this basin, we find $R \approx 4.55$.

The Horton laws are also found for multiple other river statistics (basin area, 
basin magnitude, channel length, etc.), with different Horton exponents. 
Figure~\ref{fig:Beaver_a}(b) illustrates the Horton laws for the average magnitude (the number of leaves) $M_K$
in a subbasin of order $K$, and the average number $L_K$ of edges in a channel 
of order $K$ in the Beaver creek network of Fig.~\ref{fig:Beaver}(a).
The respective Horton exponents here are $R_M \approx 4.55$ (for magnitude) and 
$R_L \approx 2.275$ (for edge number).

\medskip
\noindent{\bf Horton pruning and its generalizations (Sects.~\ref{pruning}, \ref{sec:pruning}).}
The Horton-Strahler orders are naturally connected to the Horton pruning operation, which 
erases the leaves of a tree together with the adjacent edges, and removes the degree-$2$ 
vertices that might result from such erasure.
Figure~\ref{fig:Beaver} illustrates a consecutive application of the Horton 
pruning to the Beaver creek network.
The channels (branches) of order $K$ are being erased at 
the $K$-th iteration of the Horton pruning. 
The mathematical theory of Horton laws concerns the tree measures 
that are invariant with respect to the Horton pruning. 
We also introduce a generalized dynamical pruning that allows one
to erase a metric tree from the leaves down to the root in different ways,
both continuous (metric) and discrete (combinatorial),
and consider the respective prune-invariance.
 
\begin{figure}[!t] 
\centering\includegraphics[width=\textwidth]{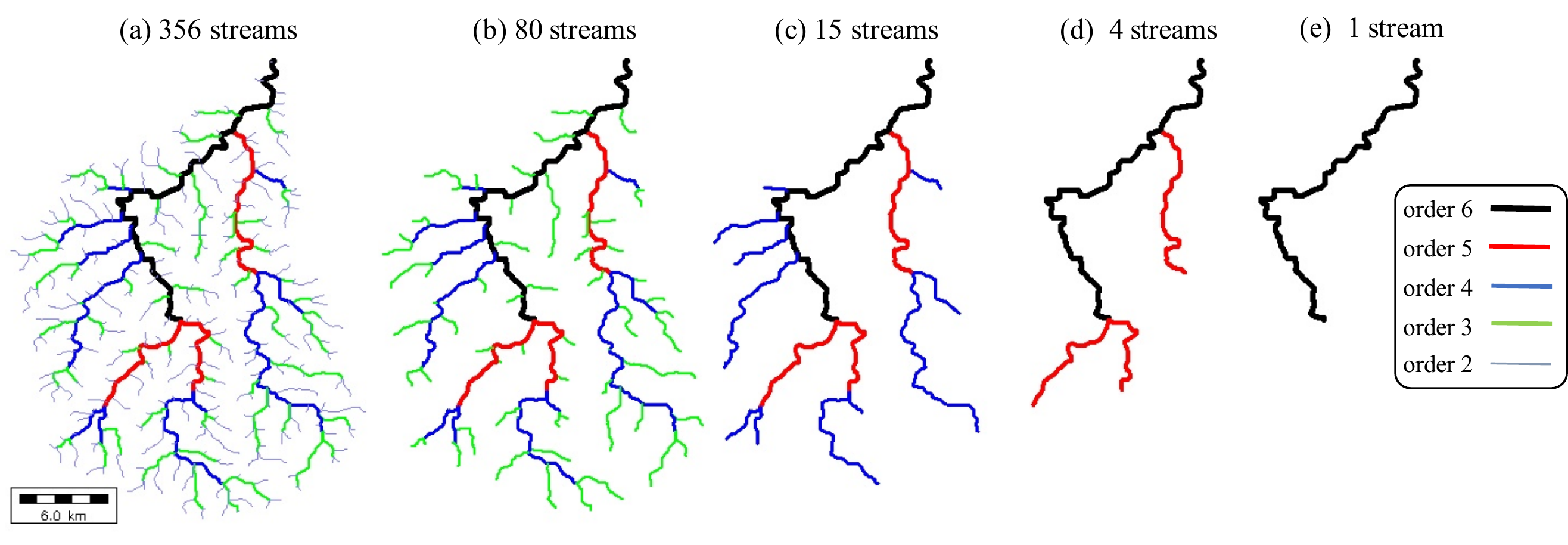}
	\caption{Stream network of Beaver creek, Floyd County, KY.
	(a) Streams (branches) of orders $K=2,\dots,6$ are shown by different colors 
	(see legend on the right).
	Streams of order $1$ are not shown for visual convenience. 
	(b)--(e) Consecutive Horton prunings of the river network; 
	uses the same color code for branch orders as panel (a). 
	The channel extraction is done using {\sf RiverTools} software (\url{http://rivix.com}).
}
\label{fig:Beaver}
\end{figure}

\medskip
\noindent{\bf Tokunaga model (Sects.~\ref{sec:Tok},~\ref{sec:HBPmartingale},~\ref{sec:combHBP}).}
A notable observation inherited from the study of river networks is the Tokunaga law \cite{Tok78}.
It complements the Horton law by describing the mergers of branches 
of distinct orders. 
Informally, the Tokunaga law suggests that the average number $\bar{N}_{i,j}$, 
$i<j$, of branches of order $i$
that merge with a branch of order $j$ in a given basin is an exponential function of the 
order difference, $\ln(\bar{N}_{i,j})~\propto~j-i$.
The Tokunaga model is surprisingly powerful in approximating the observed river 
networks \cite{ZZF13} and predicting the values of multiple Horton exponents. 
Figure~\ref{fig:Beaver_a} shows how a one-parametric critical Tokunaga model $S^{\rm Tok}$ of Sect.~\ref{sec:Tok}
fits the average values of three branching statistics in the Beaver creek network. 

In this work, we show the fundamental importance of the Toeplitz constraint  
$\bar{N}_{i,j}= f(j-i)$.
We also provide a theoretical justification for the classical 
version of the Tokunaga law, which corresponds to a particular choice $\ln f(x)~\propto~x$.

\subsection{Survey structure}
\label{sec:structure}
Our primary goal is to survey the recent developments in the theory of 
random self-similar trees; yet a number of results, models, and approaches presented here are original. 
These novel results are motivated by the need to connect the dots and bridge the gaps
when presenting a unified theory from the perspective of
Horton pruning and its generalizations.
We highlight some of these original contributions below in a list of survey topics. 

\medskip
The survey begins with the main definitions and notations in Sect.~\ref{sec:defs_notations}. 
This includes the definitions of finite rooted trees and tree spaces,
and a brief overview of real trees. 
Next, {\it Horton pruning} and {\it Horton-Strahler orders} are introduced.

\medskip
Section~\ref{TSS} defines the main types of invariances for tree measures 
sought-after in this survey. 
This includes a strong, distributional, {\it Horton self-similarity} and a weaker 
{\it mean Horton self-similarity}. 
Importantly, we justify the requirement of {\it coordination}, which, together
with prune-invariance, constitutes the {\it self-similarity}
studied in this work.
Every Horton self-similar tree (either mean or distributional) is associated
with a sequence of nonnegative {\it Tokunaga coefficients} $\{T_k\}_{k\ge 1}$, which are theoretical
analogs of the empirical averages $\bar N_{i,i+k}$.
The {\it Tokunaga self-similar} trees are a two-parameter sub-family 
of the mean Horton self-similar trees, with $T_k = ac^{k-1}$.

\medskip
The {\it Horton law} for tree measures is formally defined in Sect.~\ref{sec:HLSST} in terms 
of the random counts $N_k[T]$ of branches of order $k$ in a random tree $T$.
We introduce two versions of the {\it strong Horton law}, where one is 
convergence in probability and the other is convergence of expectation ratios. 
The main result of the section (Thm.~\ref{thm:HLSST}) establishes that the 
mean Horton self-similarity implies the strong Horton law in expectation ratios,
and expresses the Horton exponent $R$ via the Tokunaga sequence $\{T_k\}$.
Subsequently, we survey computations of the entropy rate for trees that 
satisfy the strong Horton law, as a function of the Horton exponent $R$,
and for the Tokunaga self-similar trees, as a function of the 
Tokunaga parameters $(a,c)$.
This emphasizes a special role played by the critical Tokunaga self-similar trees
with $a=c-1$, and a special point $(a,c)=(1,2)$ that describes (but is not limited to) 
the critical binary Galton-Watson tree.
The section concludes with a brief discussion of the applications 
of Horton-Strahler orders and Horton laws in natural and computer sciences.

\medskip
Section~\ref{cgw} discusses the Horton law and 
Tokunaga self-similarity for the combinatorial critical binary Galton-Watson tree. 
The proofs of the strong Horton law for branch numbers (Cor.~\ref{cor:ratio4j1})  
and the Central Limit Theorem for branch numbers (Cor.~\ref{cor:CLT_Nj}) are novel, 
and emphasize the power of the pruning approach.  
We also find here the length and height of 
the critical binary Galton-Watson tree with i.i.d. 
exponential edge lengths
that is called the {\it exponential critical binary Galton-Watson tree}. 

\medskip
Section~\ref{HBP} introduces a multi-type {\it Hierarchical Branching Process (HBP)}, which
is the main model of this work.
The process trajectories are described by time oriented trees; this induces 
a probability measure on the space of planar binary trees with edge lengths.
The HBP can generate trees with an arbitrary sequence of Tokunaga coefficients $\{T_k\}$.
The combinatorial part of these trees is always mean Horton self-similar;
the measures are also (distributionally) Horton self-similar under mild 
conditions (Thm.~\ref{HBPmain}). 
A hydrodynamic limit is established (Thm.~\ref{thm:hydro}) 
that describes the averaged branch dynamics as a 
deterministic system of ordinary differential equations (ODEs). 
This system of ODEs is used to detect a phase transition that 
separates fading and explosive behavior of the average process progeny
(Thm.~\ref{wfss}).
A subclass of {\it critical Tokunaga processes} (Def.~\ref{def:TokProcess}) that 
happens at the phase transition boundary and corresponds to $T_k=(c-1)c^{k-1}$
reproduces many of the symmetries seen in the exponential critical 
binary Galton-Watson tree, including independence of edge lengths. 
The exponential critical binary Galton-Watson tree is 
a special case of the critical Tokunaga process with $c=1$. 

\begin{figure}[!t] 
\centering\includegraphics[width=.49\textwidth]{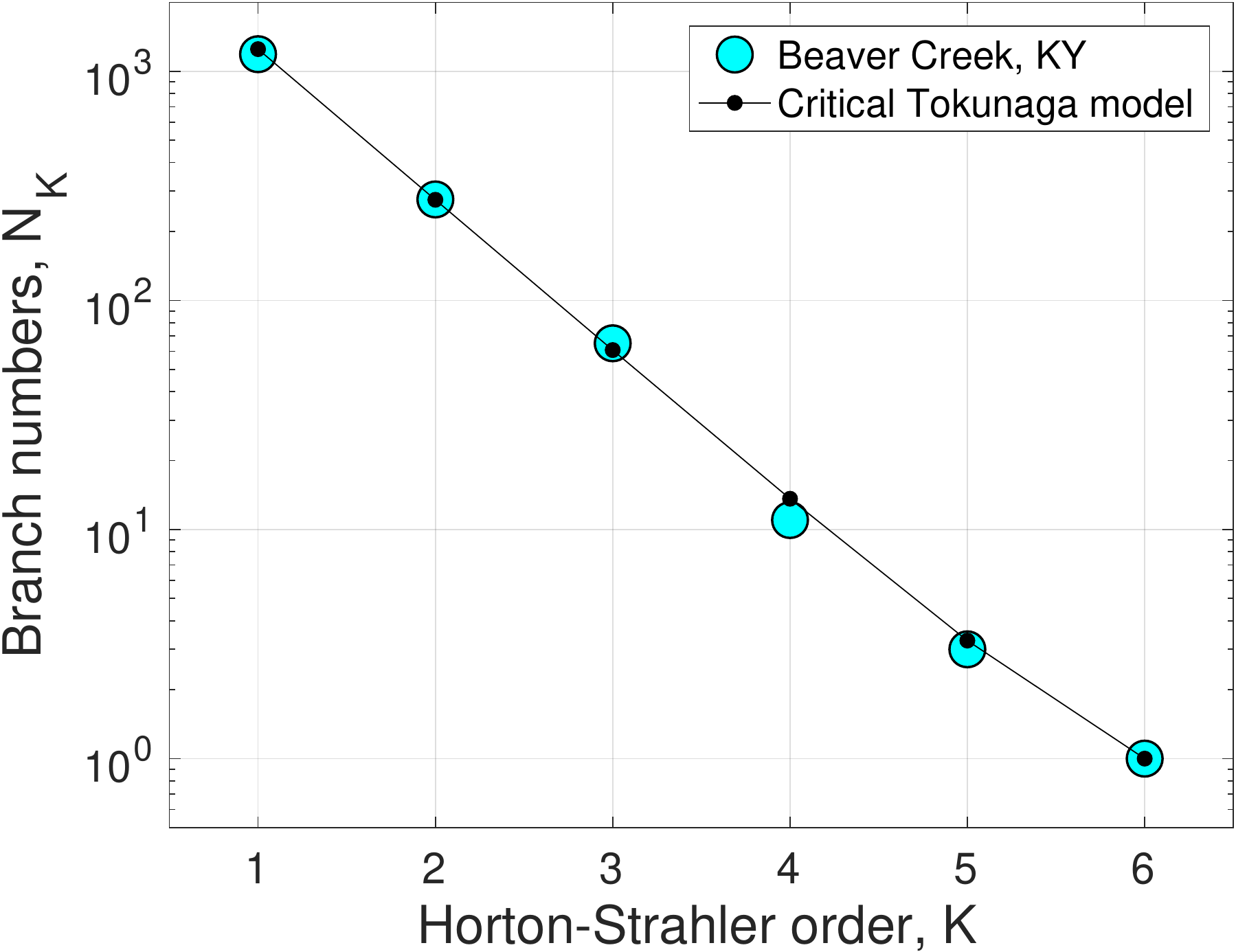}
\centering\includegraphics[width=.49\textwidth]{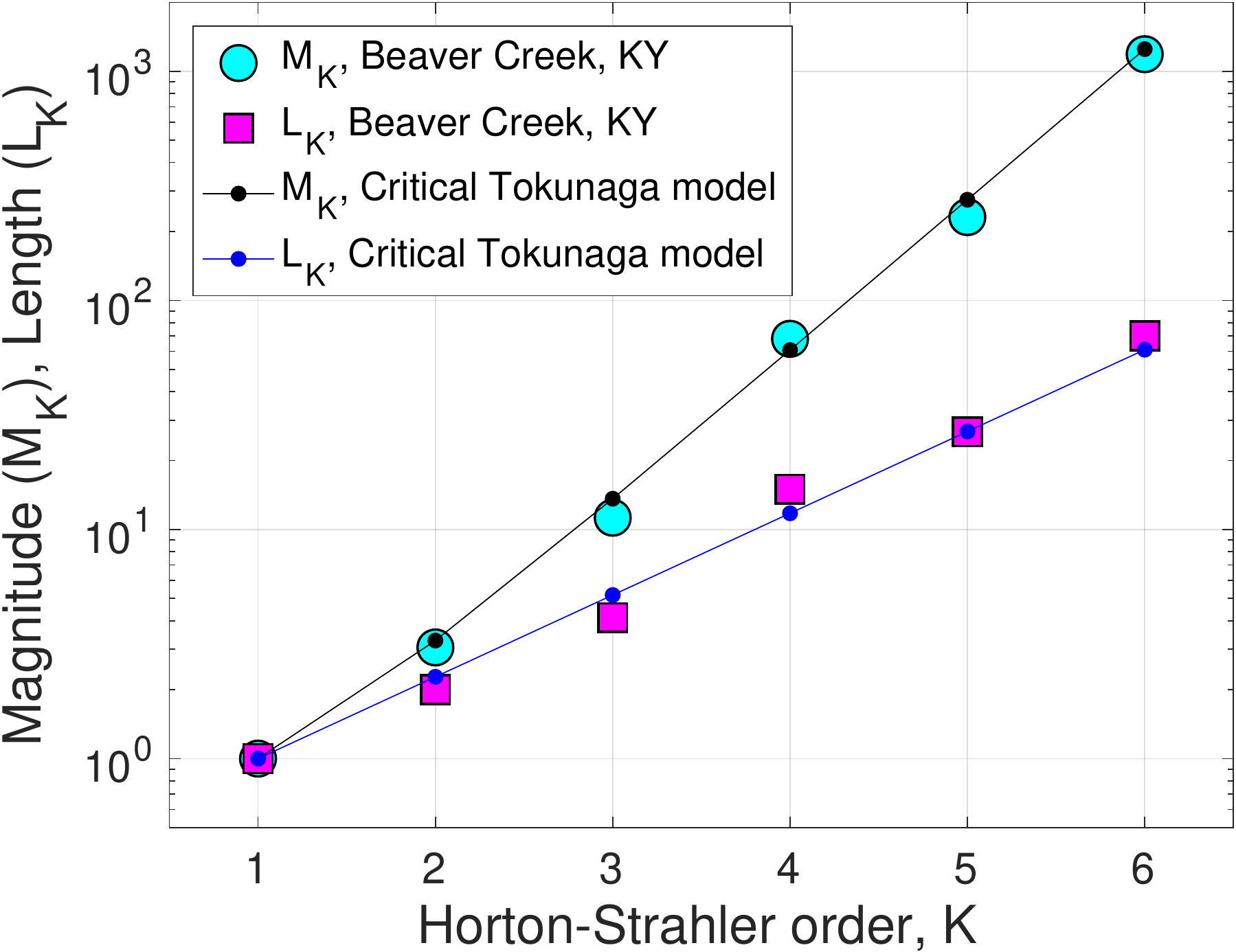}
	\caption{Horton laws in the Beaver creek network of Fig.~\ref{fig:Beaver}. 
	(a) Number $N_K$ of streams (branches) of order $K$.
	(b) Average magnitude (number of leaves) $M_K$ in a subtree of order $K$.
	Average number $L_K$ of edges in a channel (branch) of order $K$. 
	Large circles and rectangles correspond to the Beaver creek data. 
	Small dots and lines correspond to the critical Tokunaga process $S^{\rm Tok}(t;c,\gamma)$
	of Sect.~\ref{sec:Tok} with $c=2.275$, $R=2c=4.55$ 
	($\gamma$ is arbitrary, as it corresponds to metric tree properties not used in this analysis).
}
\label{fig:Beaver_a}
\end{figure}

The results in Sect. \ref{sec:HBPmartingale} are original. 
We introduce a Markov tree-valued process that generates the critical Tokunaga trees. 
We find a two-dimensional martingale with respect to the filtration of this Markov tree process 
and use Doob's Martingale Convergence Theorem for establishing the  
strong Horton law for the branch numbers (Thm.~\ref{thm:DoobMartingaleSHL}, Cor.~\ref{cor:DoobMartingaleSHL}).

The {\it Geometric Branching Process} that describes the combinatorial part of a Horton self-similar HBP is examined in Sect.~\ref{sec:combHBP}.
We show, in particular, that invariance of this process with respect to the unit time shift
is equivalent to a one-dimensional version, $a=c-1$, 
of the  Tokunaga constraint $T_k=ac^{k-1}$ (Thm.~\ref{thm:mainGBM}).
This provides an independent justification for studying the critical Tokunaga process.
We show that the complete non-empty descendant subtrees in a 
combinatorial critical Tokunaga tree have the same distribution, and two non-overlapping
trees are independent if and only if the process is critical binary 
Galton-Watson (Cor.~\ref{cor1}).
Moreover, the empirical frequencies of edge/vertex orders in a 
large random critical Tokunaga tree approximate the order distribution in the 
respective space of trees (Props.~\ref{prop:Tok1},~\ref{prop:Tok2}). 
This property is convenient for applied statistical analysis, where one might only 
be able to examine a handful of (large) trees.

\medskip
Section~\ref{LST} extends the Horton self-similarity results to time series
via tree representation of continuous functions, a construction that goes 
back to Menger \cite{Menger}, Kronrod \cite{Kronrod} and the celebrated 
Kolmogorov-Arnold representation theorem \cite{Arnold1959, Vit04}.
The {\it level set tree} for a continuous function is defined following the well known 
pseudo-metric approach \eqref{eqn:tree_dist} \cite{AldousI,AldousIII,LeGall93,NP,DLG02,Pitman}. 
We emphasize the connection of this construction with the Rising Sun Lemma (Lem.~\ref{sun}) of 
F.\,Riesz \cite{Riesz}.
Proposition~\ref{ts_prune} reveals equivalence between the Horton pruning and transition
to the local extrema of a function.
This allows us to interpret the Horton self-similarity for level set trees
as the existence of a time series whose distribution is
invariant under transition to local extrema; see \eqref{eqn:d_inv}.  
An example of such an extreme-invariant process is given by the symmetric exponential random walk
of Sect.~\ref{sec:erw}.

The results in Sect. \ref{sec:cgw} are novel; they refer to  
the level set tree $T$ of a positive excursion of a symmetric homogeneous 
random walk $\{X_k\}_{k\in\mathbb{Z}}$ on $\mathbb{R}$.
The main result of this section (Thm.~\ref{thm:excursion-tree-main}) 
shows that the combinatorial shape of $T$ is distributed as the critical 
binary Galton-Watson tree, for any choice of the transition kernel for $\{X_k\}$. 
We also show (Lem.~\ref{lem:shapeVSexcessvalue})
that $T$ has identically distributed edge lengths if and only if 
the transition kernel of $\{X_k\}$ is the probability density function of the Laplace distribution.
The results of this section complement Thm.~\ref{Pit7_3}, a classical result 
on Galton-Watson representation of the level set tree
of an exponential excursion, that can be found in \cite[Lemma 7.3]{Pitman} and \cite{LeGall93,NP}. 

Section~\ref{sec:white_Kingman} demonstrates a close connection between the level set tree of 
a sequence of i.i.d. random variables (discrete white noise) and the tree of the Kingman's
coalescent process. 
The two trees are separated by a single Horton pruning (Thm.~\ref{main2}). 

Section~\ref{sec:morse} expands the level set tree construction to 
a Morse function defined on a multidimensional compact differentiable manifold.
The key results from the Morse theory \cite{MilnorEtAl,Nicolaescu,Carmo} are used to 
describe the tree structure (Cor.~\ref{cor:KovMilnor}, Lem.~\ref{lem:KovMorse}).

\medskip
Section~\ref{sec:Kingman} establishes a weak form of Horton law for a 
tree representation of Kingman's coalescent process (Thm.~\ref{Mthm}). 
The proof is based on a Smoluchowski-type system of Smoluchowski-Horton ODEs \eqref{Aeta}
that describes evolution of the number of branches 
of a given Horton-Strahler order in a tree that represents Kingman's $N$-coalescent, 
in a hydrodynamic limit. 
Section~\ref{sec:hydro} uses T.~Kurtz's weak convergence results for density
dependent population processes (Appendix \ref{sec:kurtz})
to give a new, shorter than the original \cite{KZ17ahp}, derivation of the hydrodynamic limit.
We present two alternative, more concise, versions of the Smoluchowski-Horton ODEs in
\eqref{eqn:ODEg} and \eqref{eqn:ODEh}, and
use them to find a close numerical approximation to the Horton exponent 
in the Kingman's coalescent: $R = 3.0438279\dots$. 
This exponent also applies to the level set tree of a discrete white noise, via 
the equivalence of Thm.~\ref{main2} in Sect.~\ref{sec:white_Kingman}.

\medskip
Section~\ref{sec:pruning} introduces the {\it generalized dynamical pruning} \eqref{eqn:GenDynamPruning}.
This operation erases consecutively larger parts of a tree $T$, starting from the leaves and going down towards the root, 
according to a monotone nondecreasing pruning function $\varphi$ along the tree. 
The generalized dynamical pruning encompasses a number of discrete and continuous pruning operations, 
notably including the tree erasure of Jacques Neveu \cite{Neveu86} 
(Sect.~\ref{ex:height}) and Horton pruning (Sect.~\ref{ex:H}).
Important for our discussion, it generically includes erasures that do not
satisfy the semigroup property (Sects.~\ref{ex:L}, \ref{ex:numL}).
Theorem~\ref{thm:main} establishes prune invariance (Def.~\ref{def:pi2}) of the 
exponential critical binary Galton-Watson tree with respect to a generalized dynamical pruning 
with an arbitrary admissible pruning function $\varphi$.
The scaling exponents (Def.~\ref{def:pi2}(ii)) that describe such pruning 
for the function $\varphi$ being
the tree length, tree height, or Horton-Starhler order are found in Thm.~\ref{pdelta}. 

\medskip
As an illuminating application of the generalized dynamical pruning, 
Sect.~\ref{sec:annihilation} examines the continuum 1-D ballistic annihilation model 
$A+A \rightarrow \zeroslash$ for a constant initial particle density and
initial velocity that alternates between the values of $\pm 1$. 
The model dynamics creates coalescing shock waves, similar to those that appear in Hamilton-Jacobi
equations \cite{BK07}, that have tree structure. 
We show (Cor.~\ref{cor:V} of Thm.~\ref{thm:SWT}) that the shock tree is isometric 
to the level set tree of 
the initial potential (integral of velocity), and 
the model evolution is equivalent to a generalized dynamical pruning of 
the shock tree, with the pruning function equal to the total tree length 
(Thm.~\ref{thm:pruning}). 
This equivalence allows us to construct a complete probabilistic description of the 
annihilation dynamics for the initial velocity that alternates between the values of $\pm 1$ 
at the epochs of a constant rate Poisson point process 
(Thms.~\ref{thm:annihilation}, \ref{thm:rand_mass}, \ref{thm:mass}).
A real tree representation of the continuum ballistic annihilation is presented 
in Sect.~\ref{sec:Rtree}.

\medskip
Section~\ref{sec:infiniteT} is novel. 
Here we construct an infinite level set tree, built from leaves down,
 for a time series $\{X_k\}_{k\in\mathbb{Z}}$.
This gives a fresh perspective on multiple earlier results; 
e.g., those concerning the level set trees of random walks (Sect. \ref{sec:erw}), 
the generalized dynamical pruning (Sect. \ref{sec:PI}), or
the evolution of an infinite exponential potential in the 
continuum annihilation model (Sect. \ref{sec:rand_mass}).
For instance, the infinite-tree version of prune-invariance for the exponential
Galton-Watson tree (Thm.~\ref{thm:infmain}) can be established in a much simpler way than
its finite counterpart (Thm.~\ref{thm:main}).
Although this natural perspective has always influenced 
our research, this is the first time it is presented in explicit form.

\medskip
The survey concludes with a short list of open problems (Sect.~\ref{open}).

\medskip
Many concepts used in this survey are overlapping with the recent expositions 
on random trees, branching and coalescent processes by Aldous \cite{AldousI,AldousIII,Aldous}, 
Berestycki \cite{Berestycki}, Bertoin \cite{Bertoin}, Drmota \cite{Drmota_book}, 
Duquesne and LeGall \cite{DLG02},
Evans \cite{Evans2005}, Le Gall \cite{LeGall_book}, Lyons and Peres \cite{LP2017},
and Pitman \cite{Pitman}.
We expect that the perspectives displayed in the present survey will with time connect 
and intertwine with better established topics in the theory of random trees.

\section{Definitions and notations}\label{sec:defs_notations}

\subsection{Spaces of finite rooted trees}
A connected acyclic graph is called a {\it tree}\index{tree}.
Consider the space $\cT$ of finite unlabeled rooted reduced trees with no planar embedding.
The (combinatorial) distance between a pair of tree vertices is the number of edges in
a shortest path between them.
A tree is called {\it rooted}\index{tree!rooted} if one of its vertices, denoted by $\rho$, 
is selected as the tree root.
The existence of root imposes a parent-offspring relation between each pair of adjacent vertices: 
the one closest to the root is called the {\it parent}, 
and the other the {\it offspring}\index{vertex!parental}\index{vertex!offspring}. 
The space $\cT$ includes the {\it empty tree}\index{tree!empty} 
$\phi$ comprised of a root vertex and no edges.
The absence of planar embedding\index{tree!planar embedding} in this context is the absence of order 
among the offspring of the same parent.  
The tree root is the only vertex that does not have a parent.
We write $\#T$ for the number of non-root vertices, equal to the number of edges, in a tree $T$\index{tree!size}.
Hence, a finite tree $T=\rho\cup\{v_i,e_i\}_{1\le i \le \#T}$ is comprised of the root $\rho$
and a collection of non-root vertices $v_i$, each of which is connected to its 
unique parent ${\sf parent}(v_i)$ by the parental edge $e_i$, $1\le i \le \#T$.
Unless indicated otherwise, the vertices are indexed in order of depth-first search, starting from the root.
A tree is called {\it reduced}\index{tree!reduced} if it has no vertices of degree $2$, with the root as the only possible exception.
 
The space of trees from $\cT$ with positive edge lengths is denoted by $\cL$.
The trees in $\cL$, also known as {\it weighted tree} \cite{Pitman, LP2017}, 
can be considered metric spaces\index{tree!weighted}\index{tree!with edge lengths}.
Specifically, the trees from $\cL$ are isometric to one-dimensional
connected sets comprised of a finite number of line segments that can share
end points.
The distance along tree paths is defined according to the Lebesgue measure
on the edges. 
Each such tree can be embedded into $\mathbb{R}^2$ without creating additional 
edge intersections (see Fig.~\ref{fig:embed}).
Such a two-dimensional pictorial representation serves as the best intuitive model
for the trees discussed in this work. 

\begin{figure}[t] 
\centering\includegraphics[width=\textwidth]{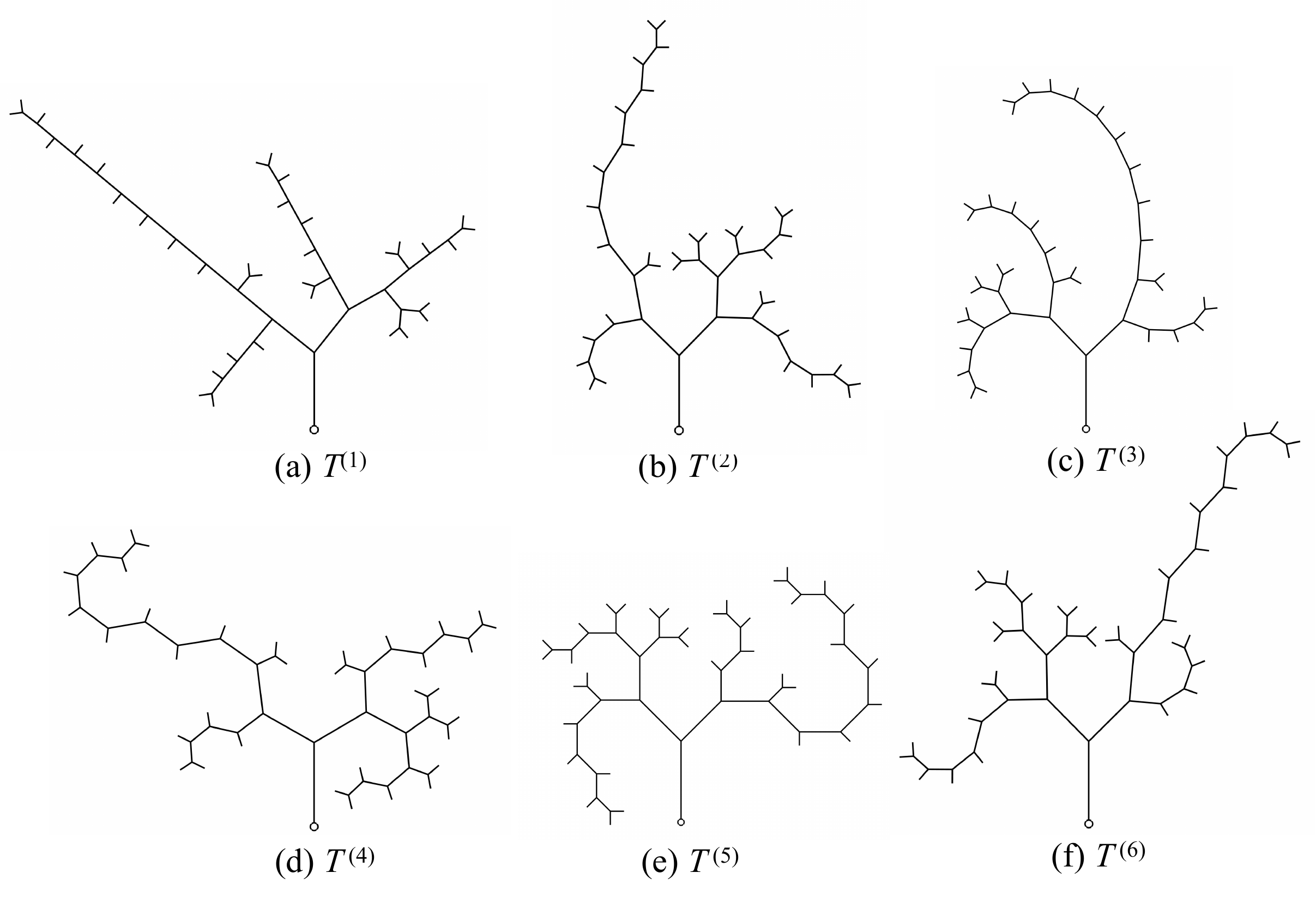}
\caption[Planar embedding: an illustration]
{Examples of alternative planar embeddings $T^{(i)}\in\L$, $i=1,\dots,6$ of the same tree $T\in\cL$,
so that $\textsc{shape}(T^{(i)})=T$.
Different panels correspond to different (random) ordering of offspring of the same parent,
and to different drawing styles. }
\label{fig:embed}
\end{figure}

We write $\T$ and $\L$ for the spaces of trees from $\cT$ and $\cL$ with planar embedding,
respectively.
Any tree from $\cT$ or $\cL$ can be embedded in a plane by selecting an order for the offsprings of the same parent. 
Choosing different embeddings for the same tree $T\in\cT$ (or $\cL$) leads, in general, to different trees from $\T$ (or $\cL_{\rm plane}$). 
Figure~\ref{fig:embed} illustrates alternative planar embeddings of a tree $T\in\cL$.
Planar embedding (offspring order) should not be confused with drawing style,
related to how edges are represented in a plane. 
Each panel in Fig.~\ref{fig:embed} uses a separate drawing style.

Sometimes we focus on the combinatorial tree ${\textsc{shape}(T)}$\index{tree!combinatorial}, 
which retains the combinatorial structure of $T \in \cL$ (or $\L$) while omitting its edge lengths and embedding. 
Similarly,  the combinatorial tree ${\textsc{p-shape}(T)}$ retains the combinatorial structure of $T \in \L$ and planar embedding,
and omits the edge length information.
Here $\textsc{shape}$ is a projection from $\cL$ or $\L$ to $\cT$, and $\textsc{p-shape}$ is a projection from $\L$ to $\T$.

A non-empty rooted tree is called {\it planted} if its root has degree $1$;
in this case the only edge connected to the root is called the 
{\it stem}\index{tree!planted}\index{stem}.
Otherwise the root has degree $\geq 2$ and a tree is called {\it stemless}\index{tree!stemless}.
We denote by $\cL^{|}$ and $\cL^{\vee}$ the subspaces of $\cL$ consisting of planted and stemless
trees, respectively. Hence $\cL=\cL^{|}\cup\cL^{\vee}$. 
Also, we let the empty tree $\phi$ to be contained in each of the spaces. 
Therefore, $\cL^{|}\cap\cL^{\vee}=\{\phi\}$. 
Similarly, we write $\L^{|}$ and $\L^{\vee}$ for the subspaces of $\L$ consisting of planted and stemless trees, respectively. 
Clearly, $\L=\L^{|}\cup\L^{\vee}$ and $\L^{|}\cap\L^{\vee}=\{\phi\}$.
Fig.~\ref{fig:planted} shows examples of a planted and a stemless tree.

For any space $\cS$ from the list $\{\cT,\T,\cL,\L\}$ we write 
$\cB\cS$ for the respective subspace of {\it binary} trees\index{tree!binary}, 
$\cS^{|}$ for the subspace of {\it planted} trees in $\cS$ including $\phi$, 
and $\mathcal{S}^{\vee}$ for the subspace of {\it stemless} trees in $\cS$ including $\phi$.
We also consider subspaces $\cB\cS^{|}=\cS^|\cap\cB\cS$ of {\it planted binary} trees 
and $\cB\mathcal{S}^{\vee}=\cS^{\vee}\cap\cB\cS$ of {\it stemless binary} trees.

\begin{figure}[t] 
\centering\includegraphics[width=0.7\textwidth]{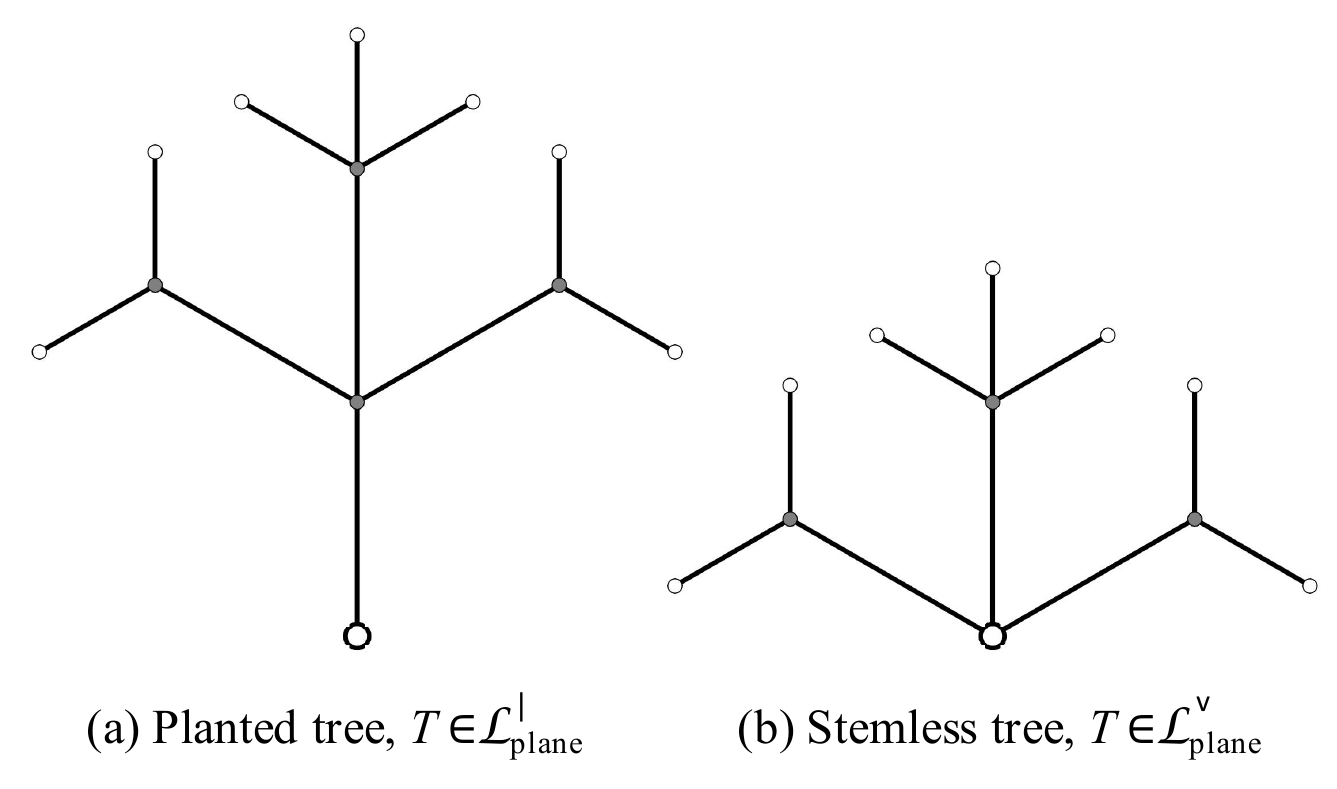}
\caption[Planted and stemless trees.]
{\small Examples of planted (a) and stemless (b) trees.
The combinatorial structure of both trees is the same, except the existence or absence of a stem.
Internal vertices are marked by gray circles. Leaves are marked by small empty circles. 
Root is marked by large empty circle.}
\label{fig:planted}
\end{figure}

\medskip
\noindent
Let $l_T=(l_1,\dots,l_{\#T})$ with $l_i>0$ be the vector of edge 
lengths of a tree $T \in \cL$ (or $\L$).
The length of a tree $T$ is the sum of the lengths of its edges: \index{tree!length}
\[\textsc{length}(T) = \sum_{i=1}^{\#T} l_i.\]
The height of a tree $T$ is the maximal distance between the root
and a vertex: \index{tree!height}
\[\textsc{height}(T) = \max_{1\le i \le \#T} d(v_i,\rho).\]

\subsection{Real trees}
\label{sec:Rsetup}

It is often natural to consider metric trees with structures more complicated
than that allowed by finite spaces $\cL$ and $\L$.
In such cases, we use the following general definition. 
\begin{Def}[{{\bf Metric tree \cite[Sect.~7]{Pitman}}}]
\label{def:treeL}
A metric space $(M,d)$ is called a {\it tree} if for each choice of $u,v\in M$
there is a unique continuous path $\sigma_{u,v}:[0,d(u,v)]\to M$ that travels
from $u$ to $v$ at unit speed, and for any simple continuous path $F:[0,L]\to M$ 
with $F(0)=u$ and $F(L)=v$, the ranges of $F$ and $\sigma_{u,v}$ coincide.\index{tree!metric}
\end{Def}

\noindent 
As an example of a metric tree that does not belong to $\L$, 
consider a unit disk in the complex plane $M= \{z\in\mathbb{C}:|z|\le 1\}$ 
and connect each point $z\in M$ to the origin ${\bf 0}$ by a linear segment $[z,{\bf 0}]$. 
Distances between points are computed in a usual way, but only along such segments.
This is a tree whose (uncountable) set of leaves coincides with the unit circle $\{|z|=1\}$.
We refer to a book of Steve Evans \cite{Evans2005} for a comprehensive discussion and further examples. 
Sects.~\ref{LST},\ref{sec:annihilation} of the present survey examine several natural
constructions of a metric $d$ on an $n$-dimensional manifold $M$ with $n\ge 1$,
such that $(M,d)$ becomes a (one-dimensional) tree according to Def.~\ref{def:treeL}.

\begin{figure}[h]
	\centering
	\includegraphics[width=0.7 \textwidth]{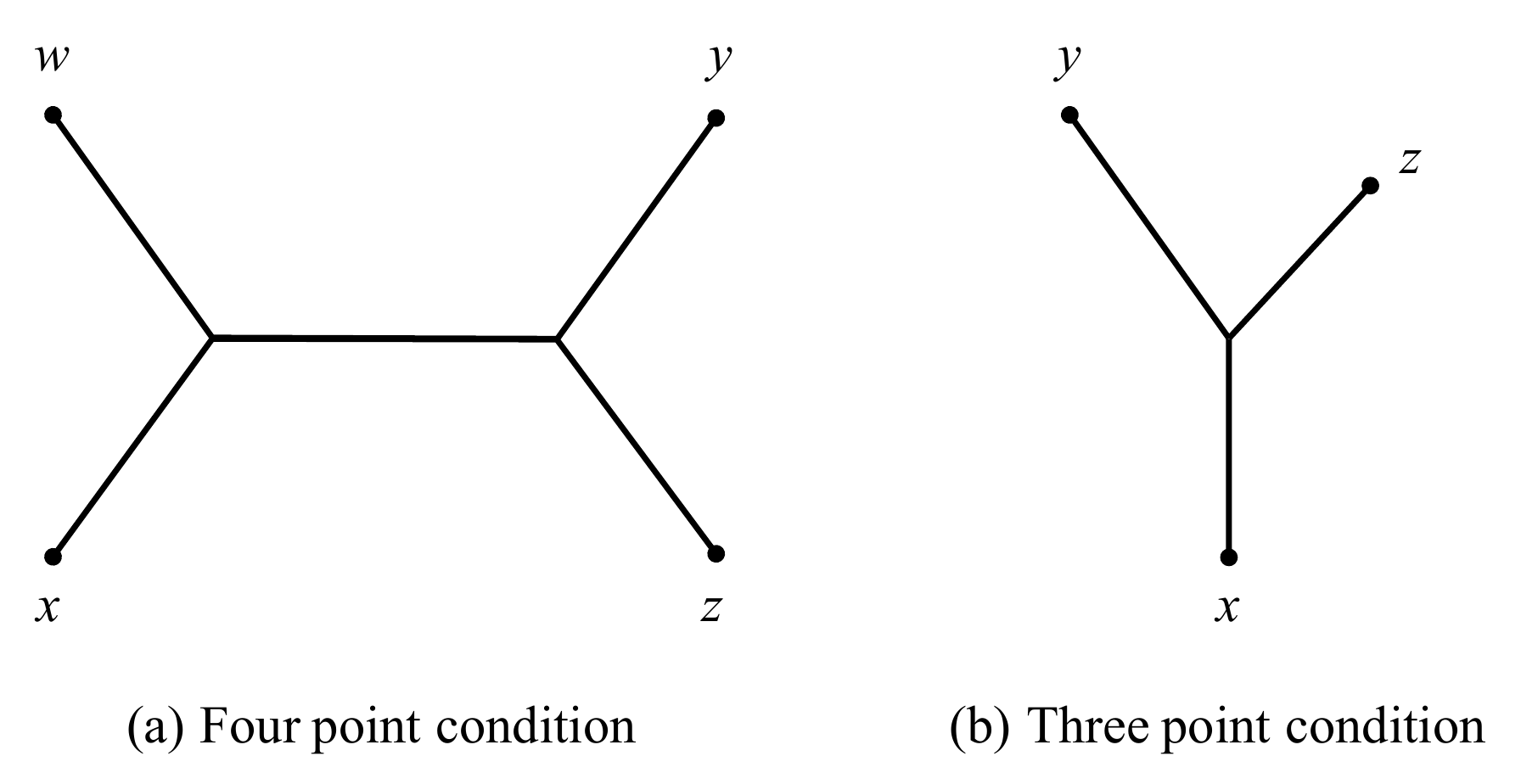}
	\caption{Equivalent conditions for $0$-hyperbolicity of a metric space $(M,d)$.
	(a) Four point condition: any quadruple $w,x,y,z\in M$ is geodesically connected as shown in the figure. 
	This configuration is algebraically expressed in Eq.~\eqref{four}.
	(b) Three point condition: any triplet $x,y,z\in X$ is geodesically connected as shown in the figure.
	There is no algebraic equivalent of the three point condition in terms of the 
	lengths of the shown segments.}
	\label{fig:hyp}
\end{figure}

Consider a metric tree $T=(M,d)$.
For any two points $x,y\in M$, we define a {\it segment}\index{segment} 
$[x,y]\subset M$ to be the image of the 
unique path $\sigma_{x,y}$ of the above definition. 
We call a point $y\in M$ a {\it descendant}\index{descendant point} 
of $x\in M$ if the path $[\rho,y]$ includes $x$.
Equivalently, removing $x$ from the tree separates its descendants from the root. 
To lighten the notations, we conventionally say $x\in T$ to indicate that point $x\in M$
belongs to tree $T$.

Metric trees benefit from an alternative characterization.
Recall that a metric space $(M,d)$ is called $0$-hyperbolic, if any quadruple
$w,x,y,z\in M$ satisfies the following {\it four point condition} \cite[Lemma 3.12]{Evans2005}:
\index{four point condition}
\be
\label{four}
d(w,x)+d(y,z)\le\max\{d(w,y)+d(x,z),d(x,y)+d(w,z)\}.
\ee
The four point condition is an algebraic description of an intuitive geometric constraint
on geodesic connectivity of quadruples that is shown in Fig.~\ref{fig:hyp}(a).
An equivalent way to define $0$-hyperbolicity is the three point condition
illustrated in Fig.~\ref{fig:hyp}(b).
It is readily seen that the four point condition is satisfied by any finite tree with 
edge lengths (considered as a metric space).
In general, a connected and $0$-hyperbolic metric space is called a {\it real tree}, 
\index{tree!real}or $\mathbb{R}$-tree \cite[Theorem 3.40]{Evans2005}.
Similarly to the case of finite trees, we say that a point $p\in T$ is an {\it ancestor} of
point $q\in T$ if the segment with endpoints $q$ and $\rho$ includes $p$: 
$p\in [p,\rho]\subset T$.
In this case, the point $q$ is called a {\it descendant} of point $p$. 
We denote by $\Delta_{p,T}$ the descendant tree at point $p$, that is the set 
of all descendants of point $p\in T$, including $p$ as the tree root.
The set of all descendant leaves of point $p$ is denoted by $\Delta^{\circ}_{p,T}$.
We use real trees in Sect.~\ref{sec:annihilation} to represent the dynamics of 
a continuum ballistic annihilation model. 


\subsection{Horton pruning}\label{pruning}

The concepts of Horton pruning and self-similarity under Horton pruning were originally developed for combinatorial binary trees $T\in\cBT$ 
\cite{Pec95,BWW00,ZK12, KZ16}. 
Here we provide a general definition of Horton pruning and Horton-Strahler orders for trees in $\cT$, their planar embeddings $\T$, and trees with edge lengths from $\cL$ and $\L$.
Horton pruning is illustrated in Fig.~\ref{fig:HST}.

\begin{Def}[{{\bf Series reduction}}]
The operation of {\it series reduction}\index{series reduction} on a rooted tree (with or without edge lengths, plane or not) removes each degree-two 
non-root vertex by merging its adjacent edges into one. 
For trees with edge lengths it adds the lengths of the two merging edges.  
The series reduction does not affect the left/right orientation in the planar trees. 
\end{Def}
\noindent
Thus, the series reduction is a mapping from the space of rooted trees  (with or without edge lengths, plane or not) to the corresponding space
of reduced rooted trees, which can be either $\cT,\T,\cL,$ or $\L$. Hence the term {\it reduced} in the definition of these spaces.

\begin{Def}[{{\bf Horton pruning}}]
\label{def:Horton_prune}
Horton pruning $\cR$ on either of the spaces $\cT,\T,\cL,$ or $\L$ is an onto 
function whose value $\cR(T)$ for a tree $T\ne\phi$ is obtained by removing
the leaves and their parental edges from $T$, followed by series reduction.
We also set $\cR(\phi)=\phi$.\index{Horton pruning}  
\end{Def}

Horton pruning induces a map on the underlying space of trees (Fig.~\ref{fig:HST}). 
The trajectory of each tree $T$ under $\cR(\cdot)$ is uniquely
determined and finite:
\be\label{TRR0}
T\equiv\cR^0(T)\to \cR^1(T) \to\dots\to\cR^k(T)=\phi,
\ee
with the empty tree $\phi$ as the (only) fixed point.
The pre-image $\cR^{-1}(T)$ of any non-empty tree $T$ consists of an infinite 
collection of trees.

\subsection{Horton-Strahler orders}
\label{sec:HS}
It is natural to think of the distance to $\phi$ under the Horton pruning map
and introduce the respective notion of tree order \cite{Horton45,Strahler}
(see Fig.~\ref{fig:HST}).

\begin{Def}[{{\bf Horton-Strahler order}}]\label{def:HSorder}
The Horton-Strahler order\index{Horton-Strahler order} 
${\sf ord}(T)\in\mathbb{Z}_+$ of a tree $T\in\cT$ ($\T,\cL,\L$) is defined as
the minimal number of Horton prunings necessary to eliminate the tree:
\[{\sf ord}(T)=\min\left\{k\ge 0 ~:~\cR^k(T)=\phi\right\}.\]
\end{Def}

\noindent 
In particular, the order of the empty tree is ${\sf ord}(\phi)=0$,
because $\cR^0(\phi)=\phi$.
Most of our discussion will be focused on non-empty trees with orders ${\sf ord}(T)>0$. 
We will often consider measures on tree spaces that assign probability zero to the empty tree $\phi$.

\begin{figure}[t] 
\centering\includegraphics[width=0.9\textwidth]{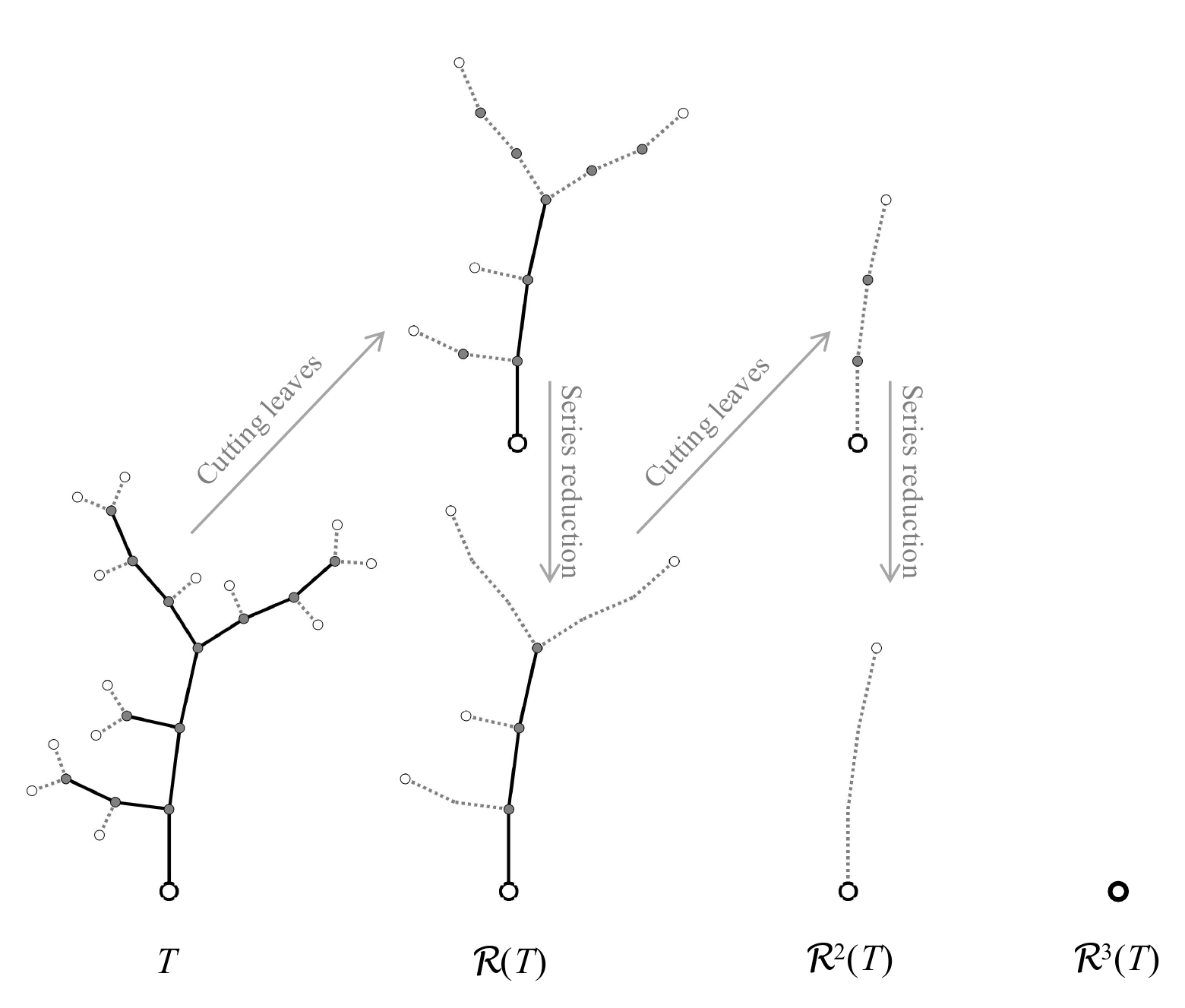}
\caption[Horton-Strahler indexing: an illustration]
{Example of Horton pruning and Horton-Strahler ordering for a tree $T\in\cBL_{\rm plane}$.
The figure shows the two stages of Horton pruning operation --
cutting the leaves (top row), and consecutive series reduction (bottom row).
The initial tree $T$ is shown in the leftmost position of the bottom row.
The edges pruned at the current step are shown by dashed gray lines.
The order of the tree is ${\sf ord}(T)=3$, since it is eliminated in three Horton prunings, $\cR^3(T)=\phi$.}
\label{fig:HST}
\end{figure}

Horton pruning partitions the underlying tree space into exhaustive and mutually exclusive collection 
of subspaces $\cH_K$ of 
trees of Horton-Strahler order $K\ge 0$ such that $\cR(\cH_{K+1})=\cH_K$.
Here $\cH_0=\{\phi\}$, $\cH_1$ consists of a single tree comprised of a root and a leaf
descendant to the root, and all other subspaces $\cH_K$, $K\ge 2$, consist of 
an infinite number of trees.
In particular, the tree size in these subspaces is unbounded from above:
for any $M>0$ and any $K\ge 2$, there exists a tree $T\in\cH_K$ such that $\#T>M$.
At the same time, the definition of Horton-Strahler orders implies, for any $K\ge 2$,
$\{\#T\big| T\in\cH_K\}\ge 2^{K-1}.$

\begin{figure}[h] 
\centering\includegraphics[width=\textwidth]{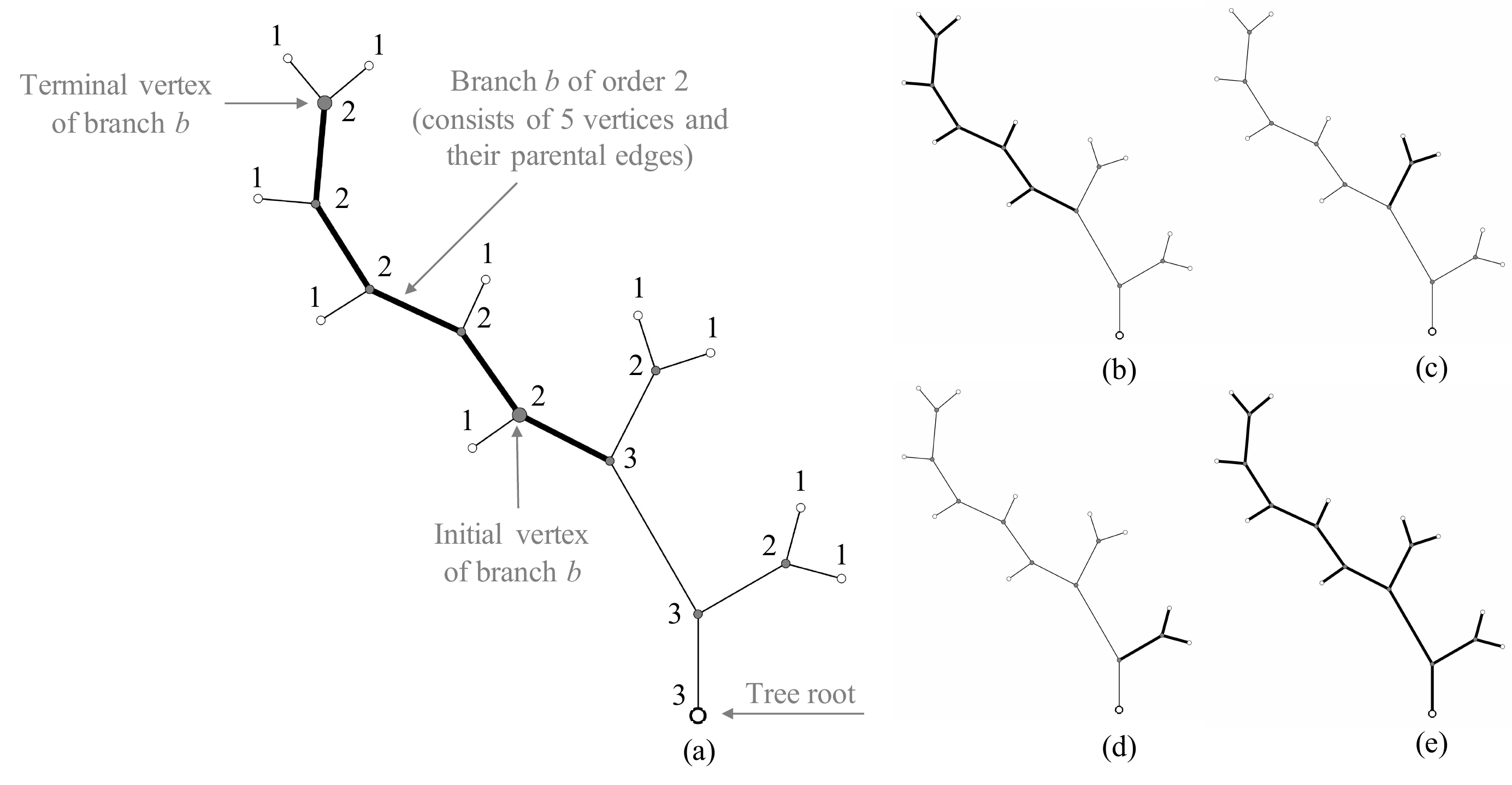}
\caption[Horton-Strahler terminology]
{Illustration of the Horton-Strahler terminology (Def.~\ref{def:HS}) in a 
tree $T \in \BL^|$ of order ${\sf ord}(T)=3$.
(a) Tree root, branch, initial and terminal vertex of a branch.
The numbers indicate the Horton-Strahler orders of the vertices
and their parental edges. 
The panel illustrates a branch of order 2, shown in bold.
Here $N_1=10$, $N_2=7$, $N_3=3$, $N_{1,2}=4$,
$N_{1,3}=0$, and $N_{2,3}=1$ (see Sect.~\ref{sec:mss}).
(b),(c),(d) Complete subtrees of order 2.
(e) Complete subtree of order 3 (coincides with the tree $T$).}
\label{fig:terminology}
\end{figure}

\medskip
\begin{Def}[{{\bf Horton-Strahler terminology}}]
\label{def:HS}
We introduce the following definitions related to the Horton-Strahler order
of a tree (see Fig.~\ref{fig:terminology}):
\begin{enumerate}
\item
{\bf (Subtree at a vertex)} For any non-root vertex $v$ in $T \ne \phi$, 
a {\it subtree}\index{subtree!combinatorial} $T_v\subset T$ 
is the only planted subtree in $T$ rooted at the parental vertex ${\sf parent}(v)$ of $v$, and comprised by $v$ and all 
its descendant vertices together
with their parental edges.
\item
{\bf (Vertex order)} For any vertex $v\in T\setminus\{\rho\}$
we set ${\sf ord}(v) = {\sf ord}(T_v)$ (Fig.~\ref{fig:terminology}a).
We also set ${\sf ord}(\rho)={\sf ord}(T)$.

\item 
{\bf (Edge order)} The parental edge of a non-root vertex has the same order as the vertex.

\item
{\bf (Branch)} A maximal connected component consisting of vertices and edges of the same order 
 is called a {\it branch}\index{branch} (Fig.~\ref{fig:terminology}a).
Note that a tree $T$ always has a single branch of the maximal order ${\sf ord}(T)$.
In a stemless tree, the maximal order branch may consist of a single root vertex. 

\item
{\bf (Initial and terminal vertex of a branch)} 
The branch vertex closest to the root is called the {\it initial vertex} 
of the branch\index{branch!initial vertex}.
The branch vertex farthest from the root is called 
the {\it terminal vertex} of a branch\index{branch!terminal vertex}.
See Fig.~\ref{fig:terminology}a.
\item
{\bf (Complete subtree of a given order)} 
Consider a connected component of tree $T$ that has been completely removed in $K$ pruning
operations (but has not been completely removed in $K-1$ prunings). 
This connected component together with the vertex used to connect it to the rest of the tree is a subtree of $T$
that will  be called a {\it complete subtree of order} 
$K$\index{subtree!complete subtree of a given order}. 

We observe that each subtree $T_v$ at the initial vertex $v$ of a branch of order $K\le {\sf ord}(T)$
is a complete subtree of order $K$, and vice versa (Fig.~\ref{fig:terminology}b-d).
A complete subtree of order ${\sf ord}(T)$ coincides with $T$ (Fig.~\ref{fig:terminology}e).
All subtrees of order ${\sf ord}=1$ are complete (and consist of a single leaf and its 
parental edge).
\end{enumerate} 
\end{Def}

\begin{figure}[t] 
\centering\includegraphics[width=0.9\textwidth]{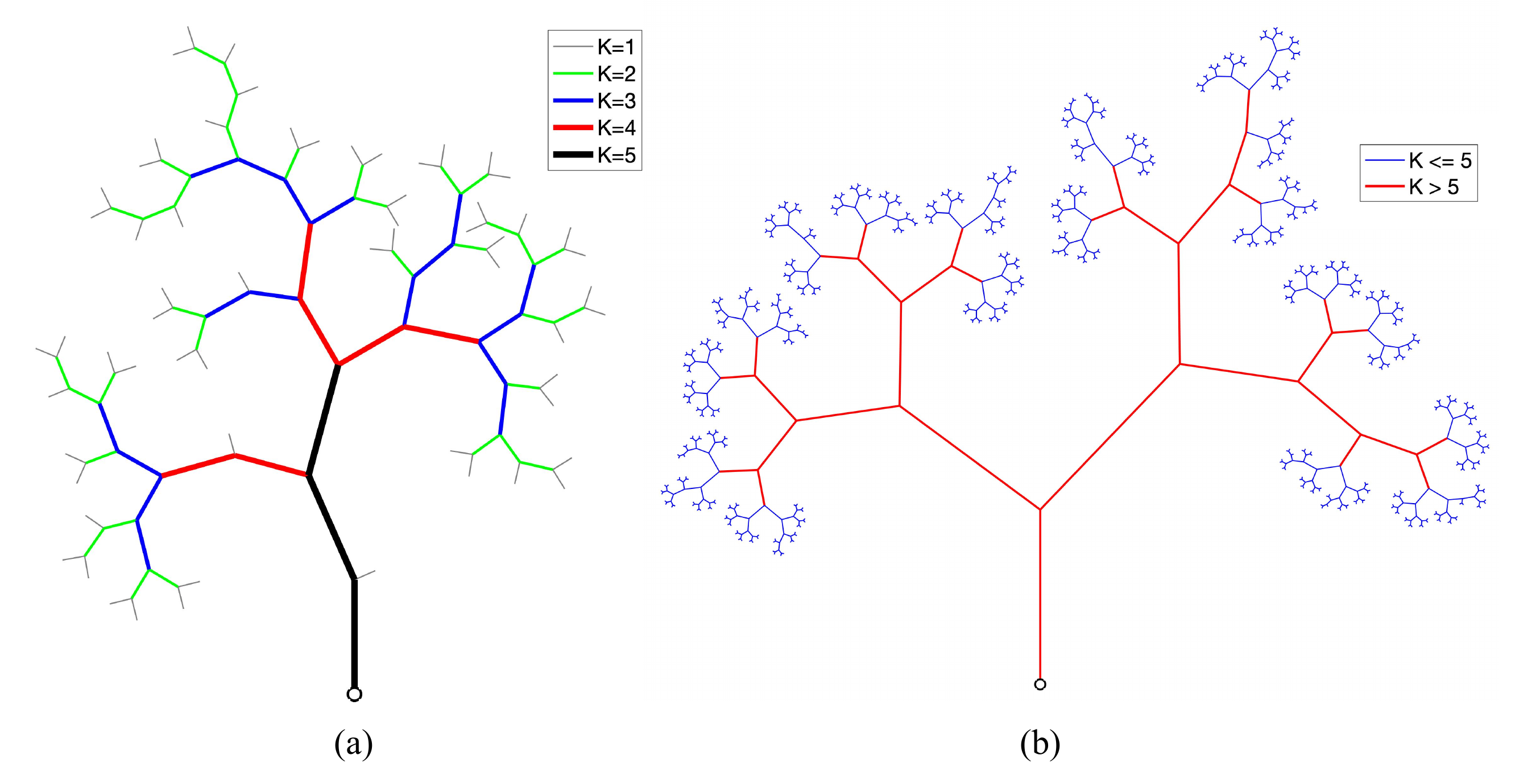}
\caption[Example of Horton-Strahler orders]
{Example of Horton-Strahler ordering of a binary tree $T\in\cBL^|$.
Different colors correspond to different orders of vertices and their parental edges, 
as indicated in legend. 
(a) $\#T = 121$, ${\sf ord}(T)=5$. 
(b) $\#T = 1233$, ${\sf ord}(T)=10$. }
\label{fig:HS_example}
\end{figure}

\noindent Figures~\ref{fig:HS_def},\ref{fig:Beaver},\ref{fig:HS_example},\ref{fig:Tok_example1} 
show examples of Horton-Strahler
ordering in binary trees.

\subsection{Alternative definitions of Horton-Strahler orders}
\label{Rem:altHS}
Definition~\ref{def:HSorder} connects the Horton-Strahler orders 
to the Horton pruning operation, which is the main theme of this survey.  
Here we give two alternative, equivalent, definitions of the Horton-Strahler
orders.
The proof of equivalence is straightforward and is left as an exercise.

The Horton-Strahler orders can be defined via hierarchical counting \cite{Horton45,Strahler,DK94,Pec95,NTG97,BWW00}.
In this approach, each leaf is assigned order $1$.
If an internal vertex $p$ has $m\ge 1$ offspring with orders $i_1,i_2,\hdots,i_m$ 
and $r=\max\left\{i_1,i_2,\hdots,i_m\right\}$, then 
\be\label{eq:HSorder_gen}
{\sf ord}(p)=\begin{cases}
    r & \text{ if }~ \#\left\{s:~i_s=r \right\}=1,\\
    r +1 & \text{ otherwise}.
\end{cases}
\ee
\noindent 
The parental edge of a non-root vertex has the same order as the vertex.
The Horton-Strahler order of a tree $T\ne\phi$ is 
${\sf ord}(T)=\max\limits_{v\in T} {\sf ord}(v)$, 
where the maximum is taken over all vertices in $T$.
This definition is most convenient for practical calculations, which
explains its popularity in the literature.

For instance, in a reduced binary tree, an internal vertex $p$ with two 
offspring of orders $i$ and $j$ has order
\be\label{eq:HSorder}
{\sf ord}(p)=\max\left(i,j\right)+\delta_{ij} =\lfloor \log_2(2^i+2^j)\rfloor,
\ee
where $\delta_{ij}$ is the Kronecker's delta and 
$\lfloor x\rfloor$ denotes the maximal integer less than or equal to $x$.
In words, the order increases by unity every time when two
edges of the same order meet at a vertex 
(Figs.~\ref{fig:HS_def},\ref{fig:Beaver},\ref{fig:HS_example},\ref{fig:Tok_example1}).

\medskip
Finally, we observe that ${\sf ord}(T)$ of a planted tree $T$ equals the depth
of the maximal planted perfect binary subtree of $T$ with the same root
(see Sect.~\ref{sec:ex}, Ex.~\ref{ex:perfect}).

\begin{figure}[!t] 
\centering\includegraphics[width=0.95\textwidth]{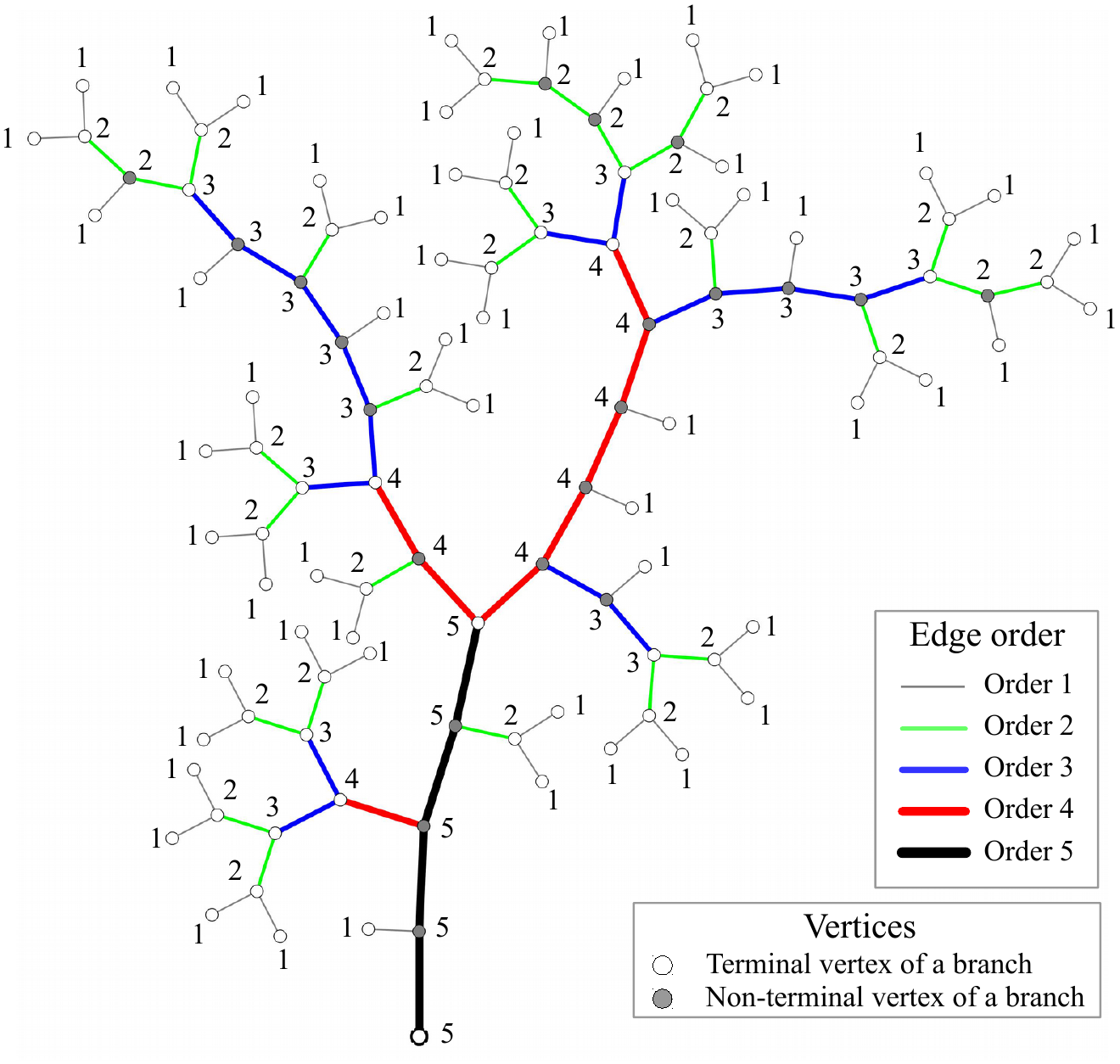}
\caption[Example of Horton-Strahler orders]
{Horton-Strahler orders of vertices in a binary tree: example. 
The order is shown next to every vertex.
Edge orders are indicated by colors (see legend).
Open circles mark terminal vertices of tree branches; they correspond either to 
leaves or mergers of principal branches.
Shaded circles mark vertices that correspond to side branches.
Here $N_1 = 56$, $N_2=22$, $N_3=8$, $N_4=3$, and $N_5=1$.
Figure~\ref{fig:Tok_example2} shows the Tokunaga indexing for the same tree.}
\label{fig:Tok_example1}
\end{figure}

\subsection{Tokunaga indices and side branching}
\label{sec:Tokunaga}
The Tokunaga indices complement the Horton-Strahler orders (Sects.~\ref{sec:HS},\ref{Rem:altHS}) by cataloging the mergers of branches according to their orders. 
In this work, we define and use the Tokunaga indices in binary trees.
It is straightforward to adopt these definitions for trees with general branching. 

Recall that a branch (Def.~\ref{def:HS}) is an uninterrupted sequence of vertices and edges 
of the same order (Fig.~\ref{fig:terminology}(a)).
According to the Horton-Strahler ordering rules, every time when two branches of the same
order $i$ meet at a vertex, this vertex (and hence the branch for which this is the terminal vertex) 
is assigned order $i\!+\!1$.
We refer to this as {\it principal branching}\index{principal branching}.
A merger of two branches of distinct orders at a vertex, however, does not result 
in assigning this vertex (and the corresponding branch) a higher order; in this case a higher-order branch 
absorbs the lower-order branch.
This phenomenon is known as {\it side branching} \index{side branching} \cite{NTG97}.
A branch of order $i$ that merges with (and is being absorbed by) a branch 
of a higher order $j>i$ is 
referred to as a {\it side branch} of Tokunaga index $\{i,j\}$.

Formally, for a non-root vertex $v$ in a reduced binary tree, 
we let ${\sf sibling}(v)$\index{vertex!sibling} denote the unique vertex of the tree that has the same 
parent as $v$, i.e.,
${\sf parent}(v) = {\sf parent}({\sf sibling}(v)).$

\begin{Def}[{{\bf Tokunaga indices}}]\label{def:TokunagaIndexing}
In a binary tree $T$, consider a branch $b$ of order $i  \in \{1,\dots,{\sf ord}(T)-1\}$,
and let $v$ denote the initial vertex of the branch $b$, whence ${\sf ord}(v)=i$. 
The branch $b$ is assigned the Tokunaga index $\{i,j\}$, where
$j={\sf ord}({\sf sibling}(v))$\index{Tokunaga!index}.
The Horton-Strahler ordering rules imply that $j\ge i$.
A branch with Tokunaga index $\{i,i\}$ is called principal branch\index{branch!principal branch}.
A branch with Tokunaga index $\{i,j\}$ such that $i< j$ is called side branch\index{branch!side branch}.
\end{Def}

\noindent
The definition of Tokunaga indices is illustrated in Fig.~\ref{fig:Tok_example2}.

\begin{Rem}
We emphasize that the Tokunaga indices refer to the tree branches, not to individual vertices and
edges as is the case with the Horton-Strahler orders. 
\end{Rem}

\begin{figure}[!t] 
\centering\includegraphics[width=\textwidth]{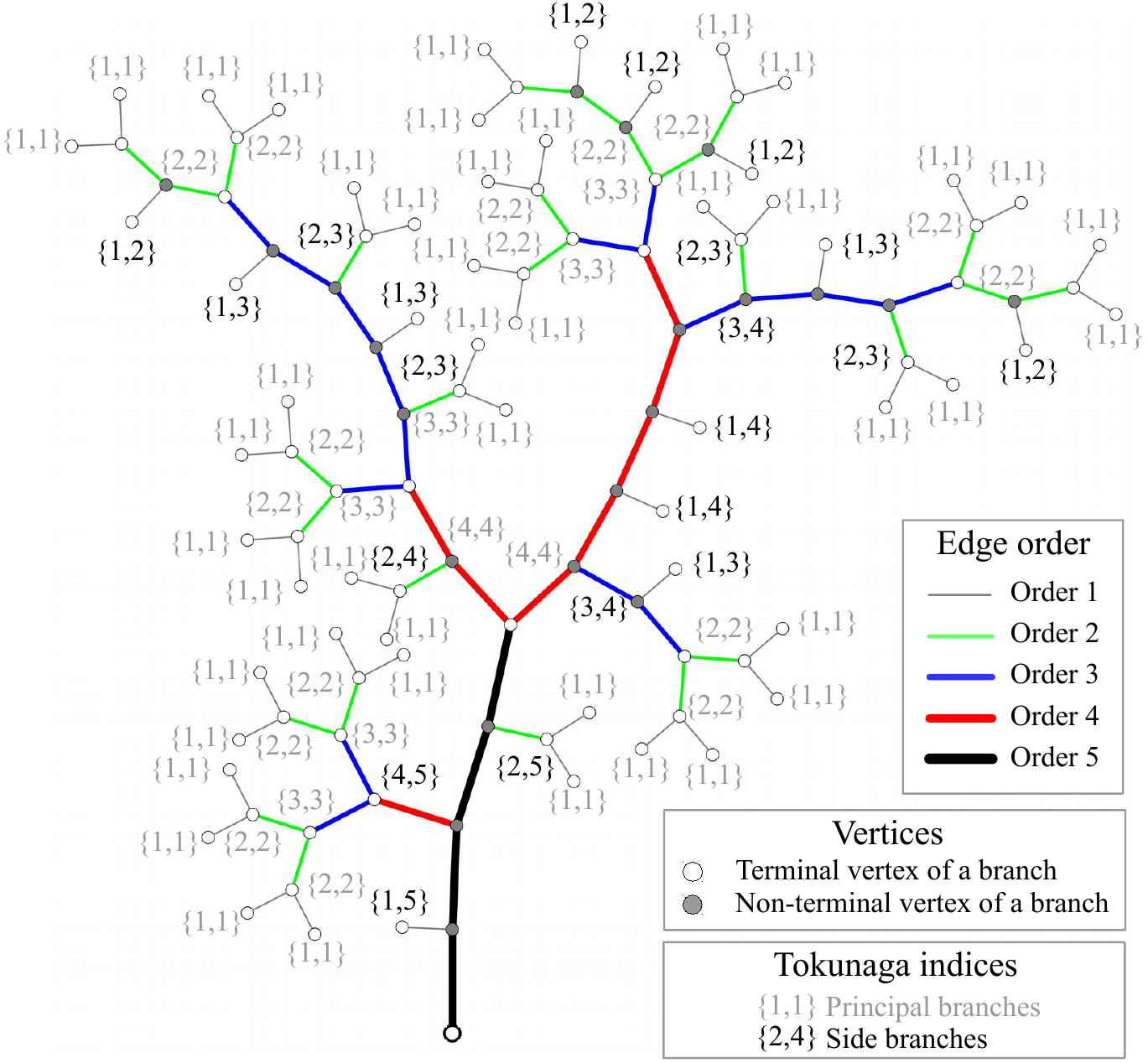}
\caption[Example of Tokunaga indices]
{Tokunaga indices for tree branches in a binary tree: example. 
The Tokunaga indices $\{i,j\}$ catalog mergers of tree branches, according 
to their Horton-Strahler orders.
Edge orders are indicated by colors (see legend).
Open circles mark terminal vertices of tree branches; they correspond either to 
leaves or mergers of principal branches.
Shaded circles mark vertices that correspond to side branches.
Here 
$N_{1,2}=5$, $N_{1,3}=4$, $N_{1,4}=2$, $N_{1,5}=1$,
$N_{2,3}=4$, $N_{2,4}=1$, $N_{2,5}=1$, 
$N_{3,4}=2$, $N_{3,5}=0$, and $N_{4,5}=1$.
Figure~\ref{fig:Tok_example1} shows the Horton-Strahler orders in the same tree.}
\label{fig:Tok_example2}
\end{figure}

\subsection{Labeling edges}
\label{sec:label}
The edges of a planar tree can be labeled by numbers $1,\dots,\#T$ in order of depth-first search.
For a tree with no embedding, labeling is done by selecting a 
suitable embedding and then using the depth-first search labeling as above.
Such embedding should be properly aligned with the Horton pruning
$\cR$, as we describe in the following definition.

\begin{Def}[{{\bf Proper embedding}}]
An embedding function $\textsc{embed}: \cT\to\cT_{\rm plane}$ ($\cL\to\cL_{\rm plane}$) 
is called {\it proper} if for any $T\in\cT$ $(T\in\cL)$
\[\cR\left(\textsc{embed}(T)\right)=\textsc{embed}\left(\cR(T)\right),\]
where the pruning on the left-hand side is in $\cT_{\rm plane}$ ($\cL_{\rm plane}$) and
pruning on the right-hand side is in $\cT$ ($\cL)$.
\end{Def}

\noindent An example of proper embedding is given in \cite{KZ18}.

\subsection{Galton-Watson trees}
\label{sec:GW}
The Galton-Watson distributions (aka Bienaym\'{e}-Galton-Watson distributions) over $\cT^|$ 
are pivotal in the theory of random trees\index{tree!Galton-Watson}.
Recall that a random Galton-Watson tree starts with a single progenitor represented by the tree root. 
The population then develops in discrete steps.
At every discrete step $d>0$ each existing population member (represented by a tree leaf 
at the maximal depth $d-1$) gives birth to $k\ge 0$ offspring with probability $q_k$, 
$\sum_{k\ge 0} q_k=1$, 
with $k=0$ representing no offspring, and terminates.
Hence, each member that terminates at step $d$ is represented by a tree vertex at depth $d-1$.
The process stops at step $d_{\rm max}$ when every leaf at depth $d_{\rm max}-1$ produces no offspring.  
\index{Galton-Watson tree}
\index{Galton-Watson tree!binary}
\index{Galton-Watson tree!critical}
\index{Galton-Watson tree!critical binary}
\index{Galton-Watson tree!combinatorial}

We denote the respective tree distribution on $\cT^{|}$ by $\mathcal{GW}(\{q_k\})$. 
Observe that $q_1=0$ in order to generate {\it reduced} trees.
Assuming that $q_1<1$, the resulting tree is finite with probability one if and only if 
$\sum k\,q_k\le 1$ \cite{Harris_book,AN_book}.
At the same time, it is well known that in the critical case (i.e., for $\sum kq_k=1$) 
the time to extinction (and hence the tree size) has infinite 
 first moment.

We write $\mathcal{GW}(q_0,q_2)$ for the probability distribution of (combinatorial) binary Galton-Watson trees in $\cBT^{|}$. 
The critical case (unit expected progeny) corresponds to $q_0=q_2=1/2$. 
Finally, we let $\mathcal{GW}_{\rm plane}(q_0,q_2)$ denote the probability distribution of (combinatorial) plane binary Galton-Watson trees in $\BT^{|}$. A random tree sampled from $\BT^{|}$ with distribution $\mathcal{GW}_{\rm plane}(q_0,q_2)$ is obtained from a random tree 
sampled from $\cBT^{|}$  with distribution $\mathcal{GW}(q_0,q_2)$ via the uniform planar embedding that assigns the left-right orientation to each pair of offsprings uniformly and independently for each node.

We conclude this section with a particular characterization of the critical binary
Galton-Watson distribution $\mathcal{GW}(1/2,1/2)$; it follows directly 
from the process definition and will be used later.
\begin{Rem}\label{rem:GWchar}
A distribution $\mu$ on $\cBT^|$ is $\mathcal{GW}(1/2,1/2)$ if and only if it can be constructed in the following way.
Start with a stem (root edge). With probability $1/2$ this completes the tree generation process.
With the complementary probability $1/2$, draw two trees independently from 
the distribution $\mu$, and attach them (as subtrees) to the non-root vertex of the stem.
This completes the construction.
\end{Rem}


\section{Self-similarity with respect to Horton pruning}
\label{TSS}

This section introduces self-similarity for finite combinatorial and metric trees.
The term {\it self-similarity} is associated with invariance of a tree distribution 
with respect to the Horton pruning $\cR$ introduced in Sect.~\ref{pruning}.
The prune-invariance alone, however, is insufficient to generate 
interesting families of trees. 
This calls for an additional property -- {\it coordination} among trees of different orders.
Coordination together with prune-invariance constitutes the self-similarity studied in this work. 
 
We start in Sects.~\ref{sec:dss}, \ref{TSSL} with a strong, distributional, self-similarity
for measures on the spaces $\cT$ and $\cL$, respectively. 
A weaker form of self-similarity that only considers the average values of selected branch statistics
it discussed in Sect.~\ref{sec:mss} for a narrower class of combinatorial binary trees from $\cBT$.

\subsection{Self-similarity of a combinatorial tree}
\label{sec:dss}

Let $\cH_{K}\subset \cT$ be the subspace of trees of Horton-Strahler order $K\ge 0$.
Naturally, $\cH_K\bigcap\cH_{K'}=\emptyset$ if $K\ne K'$, and 
$\bigcup\limits_{K\ge 1} \cH_K =\cT$.
Consider a set of conditional probability measures $\{\mu_{K}\}_{K\ge 0}$ each of which is defined on 
$\cH_K$ by 
\be\label{eq:muk}
\mu_K(T) = \mu(T\,|T\in\cH_K)
\ee
and let $p_K=\mu(\cH_K)$.
Then $\mu$ can be represented as a mixture of the conditional measures:
\be
\label{muk}
\mu=\sum_{K=1}^{\infty}p_K\mu_K.
\ee

\begin{Def}[{{\bf Horton prune-invariance}}]\label{def:prune}
Consider a probability measure $\mu$ on $\cT$ such that $\mu(\phi) = 0$.
Let $\nu$ be the pushforward measure, $\nu=\cR_*(\mu)$, i.e.,
$$\nu(T)=\mu \circ \cR^{-1}(T) = \mu \big(\cR^{-1}(T)\big).$$
Measure $\mu$ is called invariant with
respect to the Horton pruning (Horton prune-invariant)\index{Horton prune-invariance}
if for any tree $T\in\cT$ we have 
\be
\label{def:pi}
\nu\left(T\,|T\ne\phi\right)=\mu(T).
\ee
\end{Def}

\begin{Rem}
The pushforward measure $\nu$ is induced by the original measure $\mu$ 
via the pruning operation: 
if $T' \stackrel{d}{\sim} \mu$ then $T=\cR(T') \stackrel{d}{\sim} \nu$.
In particular, we observe that $\nu(\phi)=\mu(\cH_1)$ and this probability 
can be positive.
\end{Rem}

\begin{prop}\label{geom:p}
Let $\mu$ be a Horton prune-invariant measure on $\cT$.
Then the distribution of orders, $p_K=\mu(\cH_K)$, is geometric:
\be
\label{prop:1}
p_K=p\left(1-p\right)^{K-1},\quad K\ge 1,
\ee
where $p=p_1=\mu(\cH_1)$,
and for any $T\in\cH_K$
\be
\label{prop:2}
\mu_{K+1}\left(\cR^{-1}(T)\right) = \mu_K(T).
\ee
\end{prop}
\begin{proof}
Horton pruning $\cR$ is a shift operator on the sequence of subspaces $\{\cH_k\}$:
\be
\label{H_shift}
\cR^{-1}(\cH_{K-1})=\cH_K,~K\ge 2.
\ee
The only tree eliminated by pruning is the tree of order $1$: 
$\{\tau:\cR(\tau)=\phi\}=\cH_1.$
This allows to rewrite \eqref{def:pi} for any $T\ne\phi$ as
\be
\label{def:pi1}
\mu\left(\cR^{-1}(T)\right)=\mu(T)\left(1-\mu(\cH_1)\right).
\ee
Combining \eqref{H_shift} and \eqref{def:pi1} we find for any $K\ge 2$
\be
\label{mu_shift}
\mu\left(\cH_K\right)
\stackrel{{\rm by}~\eqref{H_shift}}{=}
\mu\left(\cR^{-1}(\cH_{K-1})\right)
\stackrel{{\rm by}~\eqref{def:pi1}}{=}
\left(1-\mu(\cH_1)\right)\mu(\cH_{K-1}),
\ee
which establishes \eqref{prop:1}.
Next, for any tree $T\in\cH_K$ we have
\[\mu(T) = \mu(\cH_1)\left(1-\mu(\cH_1)\right)^{K-1}\mu_K(T),\]
\[\mu\left(\cR^{-1}(T)\right) = \mu(\cH_1)\left(1-\mu(\cH_1)\right)^K\mu_{K+1}\left(\cR^{-1}(T)\right).\]
Together with \eqref{def:pi1} this implies \eqref{prop:2}.
\end{proof}

Proposition~\ref{geom:p} shows that a Horton prune-invariant measure $\mu$ is 
completely specified by its conditional measures $\mu_K$ and the mass 
$p=\mu(\cH_1)$ of the tree of order $K=1$.
The same result was obtained for Galton-Watson trees in \cite[Thm.~3.5]{BWW00}.

\medskip
Next, we introduce a (distributional) coordination property\index{coordination!distributional}.
Informally, we require that a complete subtree $T_K$ of a given order $K$
uniformly randomly selected from a random tree $T_H$ of order $H\ge K$ has 
a common distribution independent of $H$.
Since a tree $T_K$ of order $K$ has only one complete subtree of order $K$, which
coincides with $T_K$, this common distribution must be $\mu_K$.
Formally, consider the following process of selecting a {\it uniform random 
complete subtree} ${\sf subtree}_{K,H}$ of order $K$ from a random tree 
$T_H\in\cH_H$.
First, select a random tree $T_H$ according to the conditional measure $\mu_H$. 
Label all complete subtrees of order $K$ in $T_H$ in order
of proper labeling of Sect.~\ref{sec:label}, and
select a uniform random subtree, which we denote ${\sf subtree}_{K,H}$.
By construction,  ${\sf subtree}_{K,H}\in\cH_K$;
we denote the corresponding sampling measure on $\cH_K$ by $\mu^H_K$.

\begin{Def}[{{\bf Coordination}}]
\label{def:coord}
A set of measures $\{\mu_K\}_{K\ge 1}$ on $\{\cH_K\}_{K\ge 1}$ is called
{\it coordinated} if $\mu^H_K(T) = \mu_K(T)$ for any $K\ge 1$, $H\ge K$, and $T\in\cH_K$.
A measure $\mu$ on $\cT$ is called {\it coordinated} if the 
respective conditional measures $\{\mu_K\}$, as in Eq.~\eqref{muk}, are
coordinated.
\end{Def}

\begin{Def}[{{\bf Combinatorial Horton self-similarity}}]
\label{def:ss}
A probability measure $\mu$ on $\cT$ is called {\it self-similar with respect 
to Horton pruning} (Horton self-similar) if it is coordinated and Horton prune-invariant.
\index{Horton self-similarity!combinatorial}
\end{Def}

\subsection{Self-similarity of a tree with edge lengths}
\label{TSSL}
Consider a tree $T\in\cL$ with edge lengths given
by a positive vector $l_T=(l_1,\dots,l_{\#T})$ and let
$\textsc{length}(T)=\sum_il_i$.
We assume that the edges are labeled in a proper way as described in Sect.~\ref{sec:label}.
A tree is completely specified by its combinatorial shape
$\textsc{shape}(T)$ and edge length vector $l_T$.
The edge length vector $l_T$ can be specified by
distribution $\chi(\cdot)$ of a point $x_T=(x_1,\dots,x_{\#T})$ on the simplex
$\sum_i x_i = 1$, $0<x_i\le 1$, and conditional distribution $F(\cdot|x_T)$ of the tree length
$\textsc{length}(T)$, where
\[l_T = x_T\cdot\textsc{length}(T).\]
A measure $\eta$ on $\cL$ is a joint distribution of 
tree's combinatorial shape and its edge lengths; it has
the following component measures. 
\begin{align*}
&{\rm Combinatorial~shape:}\quad \mu(\tau)={\sf Law }\left(\textsc{shape}(T)=\tau\right),\\
&{\rm Relative~edge~lengths:}\quad \chi_\tau(\bar x) ={\sf Law }\left(x_T=\bar{x} \,|\,\textsc{shape}(T)=\tau \right),\\
&{\rm Total~tree~length:}\quad F_{\tau, \bar x} (\ell) ={\sf Law }\left(\textsc{length}(T)=\ell \,|\,x_T=\bar{x}, ~ \textsc{shape}(T)=\tau \right). 
\end{align*}

\noindent The definition of self-similarity for a tree with edge lengths builds
on its analog for combinatorial trees in Sect.~\ref{sec:dss}.
The combinatorial notions of coordination (Def.~\ref{def:coord}) 
and Horton prune-invariance (Def.~\ref{def:prune}), which we refer to as coordination 
and prune-invariance {\it in shapes}, are complemented with analogous properties 
{\it in edge lengths}. 
Formally, we denote by $\mu^H_K(\tau)$, $\chi^H_\tau(\bar x)$, and $F^H_{\tau, \bar x}(\ell)$ 
the component measures for a uniform complete subtree ${\sf subtree}_{K,H}$.
(Notice that the subtree order $K$ is completely specified by the tree shape $\tau$,
which explains the absence of subscript $K$ in the component measures for subtree length).
We also consider the distribution of edge lengths after pruning:
$$\Xi_\tau(\bar x) ={\sf Law }\left(x_{\cR(T)}=\bar{x}  \,|\,\textsc{shape}\big(\cR(T)\big)=\tau \right)$$
and
$$\Phi_{\tau, \bar x} (\ell)={\sf Law }\left(\textsc{length}\big(\cR(T)\big)=\ell \,|\,x_{\cR(T)}=\bar{x}, ~\textsc{shape}\big(\cR(T)\big)=\tau \right).$$
Finally, we adopt here the notation $\cH_K$ for a subspace of trees of order $K\ge 1$ from $\cL$, and 
consider conditional measures $\mu_K(\tau)=\mu(\tau|{\sf ord}(\tau)=K)$, $K\ge 1$, 
for a tree $\tau\in\cL$.

\begin{Def}[{{\bf Horton self-similarity of a tree with edge lengths}}]\label{def:distss}
We call a measure $\eta$ on $\cL$ {\it self-similar with respect to Horton pruning $\cR$} if the following conditions hold\index{Horton self-similarity!with edge lengths}: 
\begin{itemize}
\item[(i)] The measure is coordinated in shapes. This means that
for every $K\ge 1$ and every $H\ge K$ we have
\[\mu^H_K(\tau)=\mu_K(\tau)\qquad \forall \tau \in \cH_K .\] 
\item[(ii)] The measure is coordinated in lengths. This means that
for every $K\ge 1$, $H\ge K$, and $\tau \in \cH_K$ we have
\[\chi^H_\tau(\bar x)=\chi_{\tau}(\bar{x}),\]  
and for every given $\bar x$,
\[ F^H_{\tau, \bar x}(\ell)= F_{\tau, \bar x} (\ell). \] 
\item[(iii)]
The measure is Horton prune-invariant in shapes.
This means that for the pushforward measure $\nu=\cR_*(\mu)=\mu\circ\cR^{-1}$ we have
\[\mu(\tau)=\nu(\tau|\tau\ne\phi).\]
\item[(iv)]
The measure is Horton prune-invariant in lengths.
This means that 
\[\Xi_\tau(\bar x)=\chi_\tau(\bar x)\]
and there exists a {\it scaling exponent} $\zeta>0$ such that for any 
combinatorial tree $\tau\in\cT$ we have 
\[\Phi_{\tau, \bar x} (\ell)=\zeta^{-1}F_{\tau, \bar x} \left(\frac{\ell}{\zeta}\right).\]
\end{itemize}
\end{Def}

\subsection{Mean self-similarity of a combinatorial tree}
\label{sec:mss}
The discussion of this section refers to the space $\cBT$ of combinatorial binary trees.
Let $N_k=N_k[T]$ be the number of branches of order $k$ in a tree $T$, 
and $N_{i,j}=N_{i,j}[T]$ be the number of side branches with Tokunaga index $\{i,j\}$
with $1\le i <j\le{\sf ord}(T)$ in a tree $T$,
i.e., the number of instances when an order-$i$ 
branch merges with and is being absorbed by an order-$j$ branch.
Examples of counts $N_i[T]$ and $N_{i,j}[T]$ are given in 
Figs.~\ref{fig:terminology},\ref{fig:Tok_example1},\ref{fig:Tok_example2}.
We do not consider the numbers $N_{i,i}[T]$ of principal branches in $T$, 
since $N_{i,i}[T]=2N_{i+1}[T]$ and hence such counts are redundant with respect 
to the branch counts.

We write ${\sf E}_K[\cdot]$ for the mathematical expectation with respect to $\mu_{K}$ of Eq.~\eqref{eq:muk}.
As before, we adopt the notation $\cH_K$ for the subspace of trees of order $K$ in $\cBT$. 

\medskip
\noindent
We define the {\it average Horton numbers}\index{Horton numbers (branch counts)} 
for subspace $\cH_{K}$ as
\[\cN_{k}[K]= {\sf E}_K[N_k],\quad 1\le k\le K,\quad K\ge 1,\]
and the {\it average side-branch numbers}\index{side branch counts} 
of index $\{i,j\}$ as 
\[\cN_{i,j}[K]:={\sf E}_K[N_{i,j}], \quad 1\le i<j\le K,\quad K\ge 1.\]
We assume below that the average branch and side-branch numbers are finite
for any $K\ge 1$:
\[\cN_{i,j}[K]<\infty\text{ and } \cN_{j}[K]<\infty\text{ for all }1\le i<j\le K.\]

The {\it Tokunaga coefficient}\index{Tokunaga!coefficients} 
$T_{i,j}[K]$ for subspace $\cH_{K}$ is defined as the ratio of the average 
side-branch number of index $\{i,j\}$ to the average Horton number of order $j$:
\be
\label{def_tok}
T_{i,j}[K]=\frac{\cN_{i,j}[K]}{\cN_{j}[K]}, \quad 1\le i <j\le K.
\ee
The Tokunaga coefficient $T_{i,j}[K]$ is hence reflects the average number of side-branches
of index $\{i,j\}$ per branch of order $j$ in a tree of order $K$. 

\begin{Rem}\label{rem:branch}
Suppose that measure $\mu$ is coordinated (Def.~\ref{def:coord}). 
Then, all (complete) branches of order $j$ within a random tree $T\in\cH_K$ sampled with $\mu_K$ have the same
distribution. In particular, the numbers $n_{i,j}(b_k)$ of branches of order $i$ that
merge into a particular branch $b_k$, $k=1,\dots,N_j[T]$ of order $j$ in $T$ has the same
distribution for all $b_k$.
Let $n_{i,j}$ be a random variable such that $n_{i,j}(b_k)\stackrel{d}{=}n_{i,j}$.
Assume, furthermore, that the random counts $n_{i,j}(b_k)$ are independent of $N_j[T]$.
Then, by Wald's equation, we have
\begin{eqnarray*}
\cN_{i,j}[K] &=& {\sf E}_K[N_{i,j}[T]] = {\sf E}_K\left[\sum_{k=1}^{N_j[T]}n_{i,j}(b_k)\right]\\
&=& {\sf E}_K[N_j[T]]{\sf E}_K[n_{i,j}] = \cN_j[K] {\sf E}_K[n_{i,j}],
\end{eqnarray*}
and, accordingly,
\[T_{i,j}[K] = \frac{\cN_j[K] {\sf E}_K[n_{i,j}]}{\cN_j[K]}={\sf E}_K[n_{i,j}].\]
In other words, the Tokunaga coefficient in this case is the expected number 
of side-branches of appropriate index in a randomly selected branch.
This is how the Tokunaga coefficient is often defined (e.g., \cite{BWW00}).
The definition~\eqref{def_tok} adopted here is more general, as it does not require 
the distributional coordination and independence of side-branch numbers and
branch numbers.
\end{Rem}

\medskip
Next, we introduce a property that ensures independence of the side-branch structure 
of a tree order\index{coordination!mean}.
This is a weaker version of the distributional coordination (Def.~\ref{def:coord}).

\begin{Def}[{{\bf Mean coordination}}]
\label{coord}
A set of probability measures $\{\mu_K\}_{K\ge 1}$ on $\{\cH_K\}_{K\ge 1}$ is called {\it mean coordinated}
if 
\be\label{eq:mean_coord}
T_{i,j}:=T_{i,j}[K] \quad \text{ for all } K\ge 2 \text{ and } 1\le i< j\le K.
\ee
A measure $\mu$ on $\cBT$ is called {\it mean coordinated} if the respective conditional measures $\{\mu_K\}$, as in Eq.~\eqref{muk}, are mean coordinated.
\end{Def}
\noindent For a mean coordinated measure $\mu$, 
the Tokunaga matrix $\mathbb{T}_K$ is a $K \times K$ matrix
\[\mathbb{T}_K=\left[\begin{array}{ccccc}
0 & T_{1,2} & T_{1,3} & \hdots & T_{1,K} \\
0 & 0 & T_{2,3} & \hdots & T_{2,K} \\
0 & 0 & \ddots & \ddots & \vdots \\
\vdots & \vdots & \ddots & 0 & T_{K-1,K} \\
0 & 0 & \dots & 0 & 0\end{array}\right],\]
which coincides with the restriction of any larger-order Tokunaga matrix $\mathbb{T}_M$, $M>K$,
to the first $K\times K$ entries.

\begin{Def}[{{\bf Toeplitz property}}]
\label{Tsi}
A set of probability measures $\{\mu_K\}_{K\ge 1}$ on $\{\cH_K\}_{K\ge 1}$ is said to satisfy 
the {\it Toeplitz property}\index{Toeplitz property} if for every $K\ge 2$ there exists a sequence $T_k[K]\ge 0$, $k=1,2,\dots$
such that
\be\label{eq:Toeplitz}
T_{i,j}[K] = T_{j-i}[K]\quad\text{ for each } K\ge 2.
\ee
The elements of the sequences $T_k[K]$ are also referred to as Tokunaga coefficients, 
which does not create confusion with $T_{i,j}[K]$.
A measure $\mu$ on $\cBT$ is said to satisfy the {\it Toeplitz property} if the 
respective conditional measures $\{\mu_K\}$, as in Eq.~\eqref{muk}, satisfy
the Toeplitz property.
\end{Def}

\begin{Def}[{{\bf Mean Horton self-similarity}}]
\label{ss1}
A set of probability measures $\{\mu_K\}_{K\ge 1}$ on $\{\cH_K\}_{K\ge 1}$  is called {\it mean Horton self-similar} 
if it is mean coordinated and satisfies the Toeplitz property\index{Horton self-similarity!mean}. 
A measure $\mu$ on $\cBT$ is called {\it mean Horton self-similar} if the respective conditional measures $\{\mu_K\}$, as in Eq.~\eqref{muk}, are mean Horton self-similar.
\end{Def}
\noindent
An alternative definition Def. \ref{def:ss2} stated below will explain the name.

\medskip
\noindent
Combining Eqs.~\eqref{eq:mean_coord} and \eqref{eq:Toeplitz} we find that for 
a mean Horton self-similar measure there exists a nonnegative 
{\it Tokunaga sequence}\index{Tokunaga!sequence} $\{T_k\}_{k=1,2,\hdots}$ such that  
\be \label{eq:tok1}
T_{i,j}[K]=T_{j-i} \quad \text{ for all }0<i<j\leq K,
\ee
and the corresponding Tokunaga matrices $\mathbb{T}_K$ are Toeplitz:
$$\mathbb{T}_K=\left[\begin{array}{ccccc}
0 & T_1 & T_2 & \hdots & T_{K-1} \\
0 & 0 & T_1 & \hdots & T_{K-2} \\
0 & 0 & \ddots & \ddots & \vdots \\
\vdots & \vdots & \ddots & 0 & T_1 \\
0 & 0 & \dots & 0 & 0\end{array}\right].$$

\medskip
\noindent
Recall that Horton pruning $\cR$ decreases the Horton-Strahler order of each vertex (and hence 
of each branch) by unity; in particular
\be
\label{shift1}
N_k[T] = N_{k-1}\left[\cR(T)\right],\quad k\ge 2,
\ee
\be
\label{shift2}
N_{i,j}[T] = N_{i-1,j-1}\left[\cR(T)\right],\quad 2\le i<j.
\ee
Consider the pushforward probability measure $\cR_*(\mu)$ induced on $\cH_K$ by the pruning operator:
\[\cR_*(\mu)(A) = \mu_{K+1}\left(\cR^{-1}(A)\right)\quad \forall A\subset \cH_K.\]
The Tokunaga coefficients computed on $\cH_K$ using the pushforward measure
$\cR_*(\mu)$ are denoted by $T_{i,j}^{\cR}[K]$.  
Formally,
\be
T_{i,j}^{\cR}[K] = T_{i+1,j+1}[K+1] = \frac{\cN_{i+1,j+1}[K+1]}{\cN_{j+1}[K+1]}.
\ee

\begin{Def}[{{\bf Mean Horton prune-invariance}}]
\label{mpi}
A set of probability measures $\{\mu_K\}_{K\ge 1}$ on $\{\cH_K\}_{K\ge 1}$ is called 
{\it mean Horton prune-invariant} if\index{Horton prune-invariance!mean} 
\be \label{eq:mean_pi}
T_{i,j}[K] = T_{i,j}^{\cR}[K] = T_{i+1,j+1}[K+1]
\ee
for any $K\ge 2$ and all $1\le i < j\le K$.
A measure $\mu$ on $\cBT$ is called {\it mean Horton prune-invariant} if the 
respective conditional measures $\{\mu_K\}$, as in Eq.~\eqref{muk}, are
mean Horton prune-invariant.
\end{Def}

\begin{Def}[{{\bf Mean Horton self-similarity}}]
\label{def:ss2}
A set of probability measures $\{\mu_K\}_{K\ge 1}$ on $\{\cH_K\}_{K\ge 1}$ is called 
{\it mean self-similar with respect to Horton pruning}, or mean Horton self-similar, if it is mean coordinated and mean Horton prune-invariant\index{Horton self-similarity!mean}.
A measure $\mu$ on $\cBT$ is called {\it mean self-similar with respect to Horton pruning} if the respective conditional measures $\{\mu_K\}$, as in Eq.~\eqref{muk}, are mean self-similar with respect to Horton pruning.
\end{Def}

\begin{prop}
Definitions \ref{ss1} and \ref{def:ss2} of mean self-similarity are equivalent. 
\end{prop}

\noindent This equivalence was proven in \cite{KZ16}. 
Its validity is readily seen from the diagram of Fig.~\ref{fig:ss}a, which
shows relations among the quantities $T_{i,j}[K]$, $T_{i,j}[K+1]$, and
$T_{i+1,j+1}[K+1]$ involved in the definitions of mean coordination (Def.~\ref{coord}),
 Toeplitz property (Def.~\ref{Tsi}), and mean Horton prune-invariance (Def.~\ref{mpi}).
Moreover, we observe that if any two of these properties hold, the third 
also holds. 
The Venn diagram of Fig.~\ref{fig:ss}b illustrates the relation among 
mean coordination, mean prune-invariance, Toeplitz property and mean 
self-similarity in the binary tree space $\cBT$.

\begin{figure}[h] 
\centering\includegraphics[width=\textwidth]{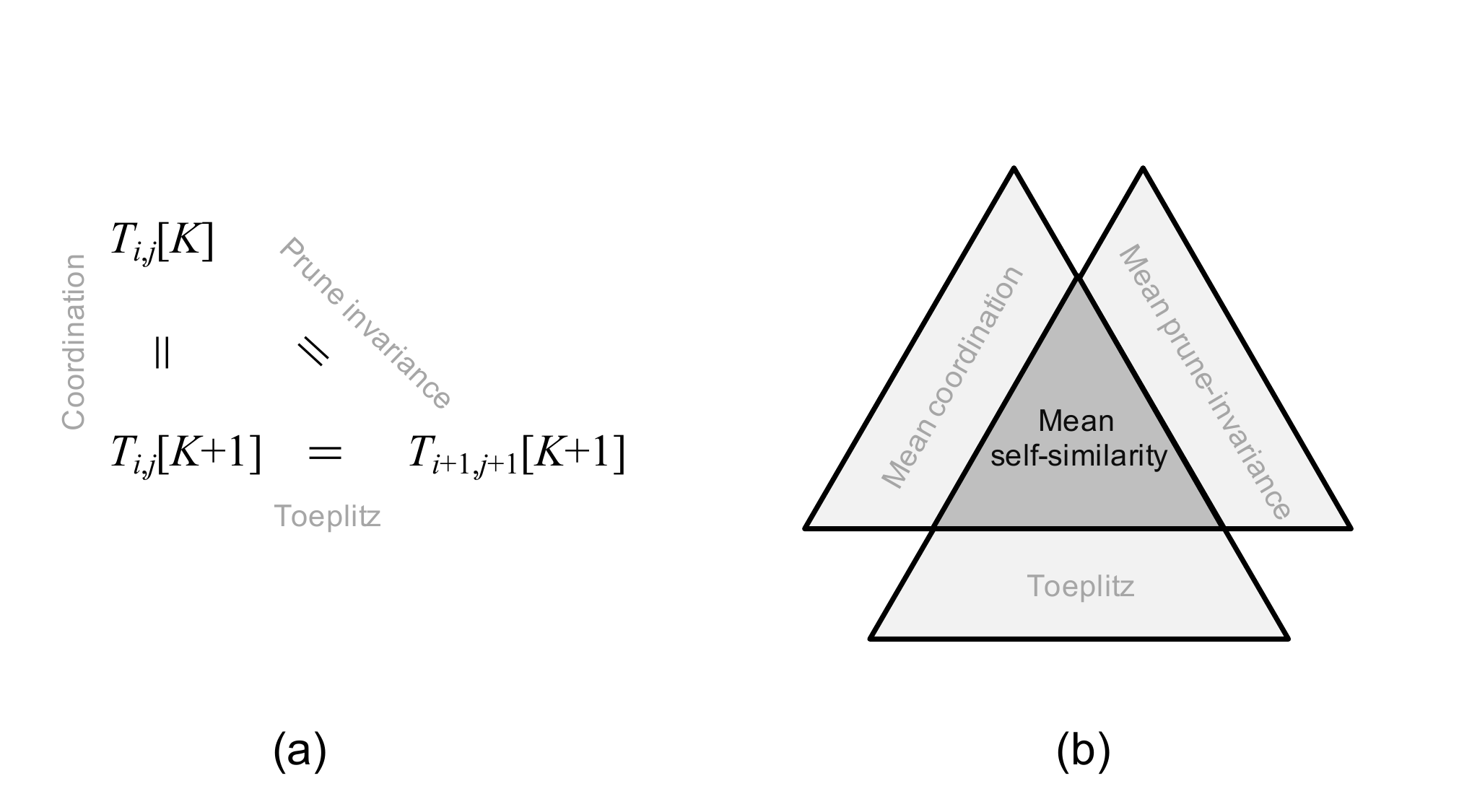}
\caption[Self-similarity]
{Relations among mean coordination, mean prune-invariance, and Toeplitz property.
(a) Pairwise equalities among the quantities $T_{i,j}[K]$, $T_{i,j}[K+1]$, and
$T_{i+1,j+1}[K+1]$ involved in the definitions of mean coordination,
mean prune-invariance, and Toeplitz property.
(b) Venn diagram of the space $\cBT$ illustrating the relation
among mean coordination (left triangle), mean prune-invariance (right triangle), 
and Toeplitz property (bottom triangle).
The mean self-similarity (inner dark triangle) is formed by the 
intersection of any pair of the three properties.}
\label{fig:ss}
\end{figure}

\medskip
\noindent 
Consider a mean Horton self-similar measure $\mu$. Observe that since exactly two branches of order $k$ are required to form a branch of order $(k+1)$,
the average number of side-branches of order $1\le k<K$ within $\cH_K$ is $\cN_{k}[K]-2\cN_{k+1}[K]$.
This number can also be computed by counting the average number of 
side-branches of order $k$ for all higher-order branches:
$$\sum\limits_{j=k+1}^K T_{k,j}\,\cN_{j}[K]=\sum\limits_{m=1}^{K-k}T_m\,\cN_{k+m}[K].$$
\noindent
Equalizing these two expressions we arrive at the main system of counting equations:
\be
\label{count1}
\cN_{k}[K] = 2\,\cN_{k+1}[K]+\sum_{j=1}^{K-k}T_j\,\cN_{k+j}[K],\quad 1\le k\le K-1, \quad K\ge2.
\ee

\noindent
Consider a $K \times K$ linear operator
\be\label{eq:Gk}
\mathbb{G}_K:=\left[\begin{array}{ccccc}-1 & T_1+2 & T_2 & \hdots & T_{K-1} \\0 & -1 & T_1+2 & \hdots & T_{K-2} \\0 & 0 & \ddots & \ddots & \vdots \\\vdots & \vdots & \ddots & -1 & T_1+2 \\0 & 0 & 0 & 0 & -1\end{array}\right].\ee

\noindent 
The counting equations \eqref{count1} rewrite as 
\be
\label{eq:count}
\mathbb{G}_K \left(\!\!\!\begin{array}{c}\cN_{1}[K] \\\cN_{2}[K] \\\vdots \\ \cN_{K}[K]\end{array}\!\!\!\right) = -e_K,\quad K\ge 1,
\ee
\noindent where $e_K$ is the $K$-th coordinate basis vector. 
Using this equation for $(K+1)$ and considering the last $K$ components we obtain
\[\mathbb{G}_K \left(\!\!\!\begin{array}{c}\cN_{2}[K\!\!+\!\!1] \\\cN_{3}[K\!\!+\!\!1] \\\vdots \\ \cN_{K\!+\!1}[K\!\!+\!\!1]\end{array}\!\!\!\right) = -e_K,\quad K\ge 1.\]
This proves the following statement.
\begin{prop}
\label{prop1_Horton}
Consider a mean Horton  self-similar measure $\mu$ on $\cBT$. 
Then for any $K\ge 1$ and $1 \leq k \leq K$ we have 
\[\cN_{k+1}[K\!\!+\!\!1]=\cN_{k}[K]\]
and
\[\cN_{i+1,j+1}[K\!\!+\!\!1] = \cN_{ij}[K], \quad 1\le i<j\le K,\quad K\ge 2.\]
\end{prop}

\medskip
\begin{Def}[{{\bf Tokunaga self-similarity}}]
\label{ss2}
A mean Horton self-similar measure $\mu$ on $\cBT$ is called 
{\it Tokunaga self-similar} with parameters $(a,c)$ if its 
Tokunaga sequence $\{T_j\}_{j=1,2,\hdots}$ is expressed as
\be \label{eq:tok}
T_j = a\,c^{j-1},\quad k\ge 1
\ee
for some constants $a\ge0$ and $c>0$.
\end{Def}\index{Tokunaga!self-similarity}
Tokunaga self-similarity \eqref{eq:tok} specifies a combinatorial tree 
shape (up to a permutation of side branch attachment within a given branch) with only two parameters $(a,c)$, 
hence suggesting a conventional modeling paradigm.
The empirical validity of the Tokunaga self-similarity  constraints \eqref{eq:tok} 
has been confirmed for a variety of river networks at
different geographic locations \cite{Pec95,Tarboton96,DR00,MTG10,ZZF13},
as well as in other types of data represented by trees, including 
botanical trees \cite{NTG97}, 
the veins of botanical leaves \cite{TPN98,PT00}, 
clusters of dynamically limited aggregation \cite{O92,NTG97}, 
percolation and forest-fire model clusters \cite{ZWG05,YNTG05}, 
earthquake aftershock sequences \cite{THR07,HTR08,Y13}, 
tree representation of symmetric random walks \cite{ZK12} (Sect. \ref{sec:erw}),
and hierarchical clustering \cite{GNT99}.
The conditions \eqref{eq:tok}, however, lacks a theoretical justification.
We make a step towards justifying this condition in Sect.~\ref{sec:TokTI}.

\begin{Rem}[{{\bf Mean self-similarity is a property of conditional measures}}]
\label{rem:muk}
The properties introduced in this section -- 
mean coordination (Def.~\ref{coord}),
Toeplitz (Def.~\ref{Tsi}),
mean Horton prune-invariance (Def.~\ref{mpi}), 
and mean Horton self-similarity (Def.~\ref{ss1},\ref{def:ss2}) -- are completely
specified by a set of conditional measures $\{\mu_K\}$, and are independent 
of the randomization probabilities $p_K = \mu(\cH_K)$, see Eq.~\eqref{muk}.
\end{Rem}

\begin{Rem}[{{\bf Terminology}}]
\label{rem:terminology}
The self-similarity concepts studied in this work refer to
a measure $\mu$, or a collection of conditional
measures $\{\mu_K\}$, on a suitable space of trees. 
For the sake of brevity, we sometimes use a common abuse of notations and discuss
self-similarity of a random tree $T\stackrel{d}{\sim}\mu$ 
(e.g., claiming that a tree $T$ is mean Horton self-similar, etc.).
Formally, such statements apply to the respective
tree distribution $\mu$.
\end{Rem}

\subsection{Examples of self-similar trees}
\label{sec:ex}

This section collects some examples (and non-examples) of self-similar trees
and related properties.

\begin{figure}[t] 
\centering\includegraphics[width=\textwidth]{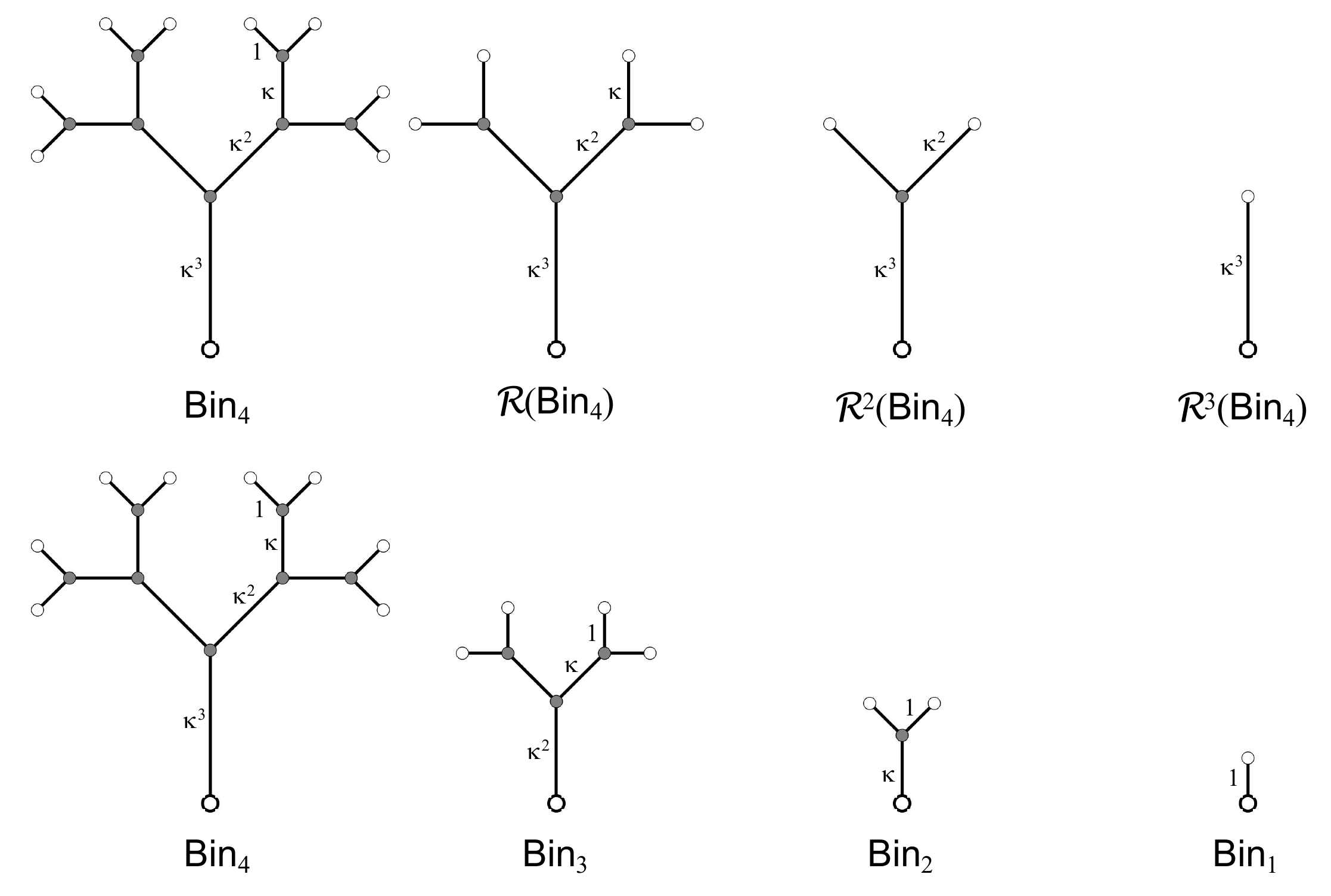}
\caption[Self-similarity of perfect binary trees]
{Self-similarity of perfect binary trees ${\sf Bin}(\kappa)\subset\cBL^|$
(Ex.~\ref{ex:perfect}).
The length of edges of order $i$ is $\kappa^{i-1}$ for some $\kappa>0$.
The space ${\sf Bin}(\kappa)$ is Horton self-similar with $\zeta=\kappa$ 
and Tokunaga sequence $T_j=0$, $k\ge 1$.
In this figure, $\kappa=1.5$.
We write ${\sf Bin}_K$ for the tree of order $K$.
Top row shows three consecutive Horton prunings of ${\sf Bin}_4$.
Bottom row shows trees ${\sf Bin}_{4,3,2,1}$.
Here, for any $K\ge 1$ and $m\ge 0$, the tree ${\sf Bin}_K$ is 
obtained by scaling all edges of the tree $\cR^m({\sf Bin}_{K+m})$ 
by a multiplicative factor $\kappa^{-m}$.
The four columns of the figure correspond to $m=0,1,2,3$ and $K+m=4$.
The lengths of selected edges are indicated in the figure.}
\label{fig:perfect}
\end{figure}

\begin{ex}[{{\bf Perfect binary trees}}]
\label{ex:perfect}
Recall that a binary tree is called {\it perfect}\index{tree!perfect binary} if it is reduced and  
all its leaves have the same depth (combinatorial distance from the root).
Consider space ${\sf Bin}\subset\cBT^|$ of finite planted perfect binary trees;
see Fig.~\ref{fig:perfect}.
We write $D=D[T]$ for the depth of a tree $T$ and ${\sf Bin}_D\subset{\sf Bin}$ 
for the subspace of trees of depth $D\ge 1$.
The subspace ${\sf Bin}_D$ consists of a single tree with $2^{D-1}$ leaves;
it has Horton-Strahler order $D$.
Every conditional measure $\mu_K$ in this case is a point measure on ${\sf Bin}_K$, $K\ge 1$.
Moreover, the order of a vertex at depth $1\le d\le D$ (and its parental edge) is $D-d+1$, and 
for the tree ${\sf Bin}_K$ we have \[N_k[{\sf Bin}_K] = 2^{K-k}, \quad K\ge 1, k\le K.\]

We write ${\sf Bin}(\kappa)\subset\cBL^|$ for the space of metric trees
with combinatorial shapes from ${\sf Bin}$ and 
length $\kappa^{i-1}$ assigned to edges of order $i\ge 1$.
The bottom row of Fig.~\ref{fig:perfect} shows trees ${\sf Bin}_i$, $i=4,3,2,1$,
that correspond to $\kappa=1.5$.

\begin{itemize}
\item[(a)] Coordination in shapes (Def.~\ref{def:coord} or \ref{def:distss}(i)) and 
in lengths (Def.~\ref{def:distss}(ii)).
The space ${\sf Bin}$ is coordinated in shapes and lengths, since
every subtree of order $K$ in a tree of order $H\ge K$ (not necessarily a 
uniform complete subtree) is  
the tree ${\sf Bin}_K$.
\item[(b)] Mean coordination (Def.~\ref{coord}) and Toeplitz property (Def.~\ref{Tsi}).
By construction, the space ${\sf Bin}$ has no side-branching ($N_{i,j}[T]=0$),
and so 
\[T_{i,j}[K]=T_{j-i}[K]=T_{j-i}=0,\quad i< j.\]
This implies mean coordination and Toeplitz property.
\item[(c)] Mean self-similarity (Def.~\ref{ss1}) follow from (b). 
\item[(d)] Mean Horton self-similarity (Def.~\ref{def:ss2}).
Recall that subspace ${\sf Bin}_K$ consists of a single tree for any $K\ge 1$.
Since
\[{\sf Bin}_K=\cR({\sf Bin}_{K+1}),\quad K\ge 1,\]
the space is mean Horton prune-invariant.
Together with mean coordination of (b) this implies mean Horton self-similarity.
\item[(e)] Combinatorial Horton self-similarity (Def.~\ref{def:ss}).
Observe that the argument used in (d) also implies Horton prune-invariance
in shapes (Def.~\ref{def:prune} or \ref{def:distss}(iii)). 
Together with coordination in shapes of (a) this gives combinatorial Horton
self-similarity.
\item[(f)] Tokunaga self-similarity with $a=0$ (Def.~\ref{ss2}) follows from (b). 
\item[(g)] Horton prune-invariance in lengths (Def.~\ref{def:distss}(iv)). 
By construction, the leaves of a pruned tree have length $\kappa$; and 
the edge lengths change by a multiplicative factor $\kappa$ with every
combinatorial step toward the root.
This implies Horton prune-invariance in lengths with $\zeta = \kappa$.
\item[(h)] Self-similarity (Def.~\ref{def:distss}) with $\zeta=\kappa$
follows from (a), (c) or (d), and (g). It implies that
for any $K\ge 1$ and $m\ge 0$, the tree ${\sf Bin}_K$ is 
obtained by scaling all edges of the tree $\cR^m({\sf Bin}_{K+m})$ 
by a multiplicative factor $\kappa^{-m}$.
The four columns of Fig.~\ref{fig:perfect} correspond to $m=0,1,2,3$ and $K+m=4$.
\end{itemize}
\end{ex}

\begin{ex}[{{\bf Combinatorial critical binary Galton-Watson trees}}]
\label{ex1}
The Galton-Watson distribution $\mathcal{GW}(\{q_k\})$ on $\cT^|$ has the coordination property
for any distribution $\{q_k\}$ with $p_1\ne 1$. 
Indeed, the Markovian branching mechanism (see Sect.~\ref{sec:GW}) creates subtrees 
of the same structure, 
independently of the tree order. 
This implies coordination. 
However, mean and distributional prune-invariance (and hence mean and combinatorial Horton self-similarity) only hold in the critical binary case $\mathcal{GW}(\frac{1}{2},\frac{1}{2})$ \cite{BWW00}.
The corresponding Tokunaga sequence is $T_j=2^{j-1}$, $j\ge 1$, which
implies Tokunaga self-similarity with parameters $(a,c)=(1,2)$.
\end{ex}

\begin{figure}[t] 
\centering\includegraphics[width=0.8\textwidth]{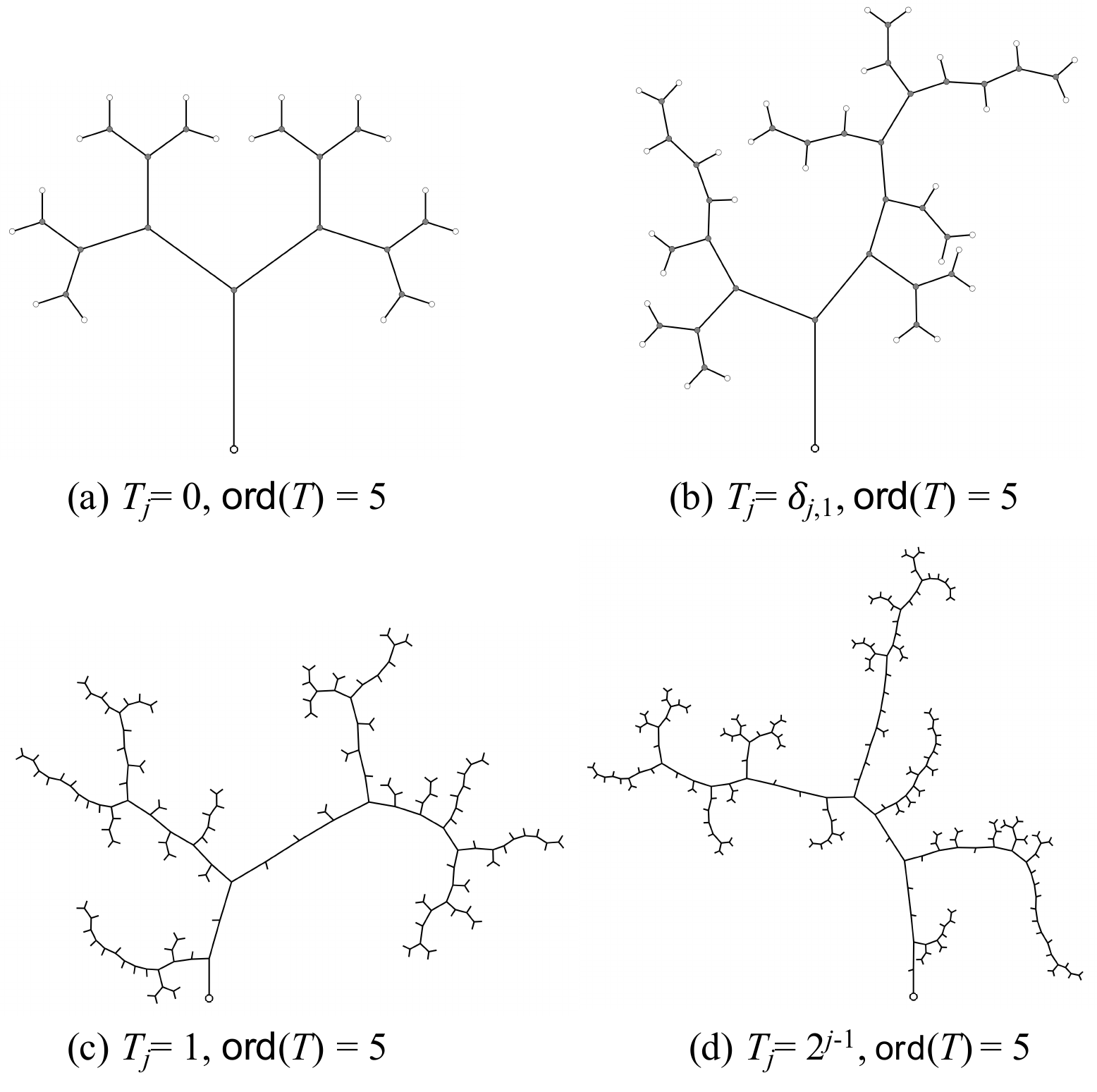}
\caption[Combinatorial Tokunaga trees]
{Tokunaga trees. Different panels correspond to different Tokunaga sequences $T_j=ac^{j-1}$.
(a) $(a,c)=(0,0), T_j = 0$,
(b) $(a,c)=(1,0), T_j = \delta_{j,1}$,
(c) $(a,c)=(1,1), T_j = 1$,
(d) $(a,c)=(1,2), T_j = 2^{j-1}$.
The lengths of edges of order $i$ equal $\kappa^{i-1}$, with $\kappa=1.5$.}
\label{fig:Tok_geom}
\end{figure}

\begin{ex}[{{\bf Critical binary Galton-Watson trees with i.i.d. exponential edge lengths}}]
\label{ex2}
The space of critical binary Galton-Watson trees with independent exponential edge lengths
is Horton self-similar with $\zeta=2$; this is shown in Sect.~\ref{ecgw}.
\end{ex}

\begin{ex}[{{\bf Hierarchical Branching Process}}]
\label{ex3}
Section~\ref{HBP} introduces a rich class of measures on $\cBL^|$ induced 
by the Hierarchical Branching Process (HBP). 
Notably, one can construct a version of the process that is Horton self-similar
(Def.~\ref{def:distss}) with an arbitrary Tokunaga sequence $\{T_j\}$ and 
for an arbitrary $\zeta>0$.
This class includes the critical binary Galton-Watson tree with independent 
exponential lengths as a special case.
\end{ex}

\begin{ex}[{{\bf Combinatorial Tokunaga trees}}]
\label{ex3a}
Tokunaga self-similar trees (Def.~\ref{ss2}) are specified by a particular form of the Tokunaga sequence:
\[T_j = ac^{j-1},\quad j\ge 1.\]
This is a very flexible model that can account for a variety of dendritic patterns.
Figure~\ref{fig:Tok_geom} shows four selected examples:
\begin{eqnarray*}
\text{Fig.~\ref{fig:Tok_geom}(a)} &:& (a,c) = (0,0),\quad T_j = 0,\\
\text{Fig.~\ref{fig:Tok_geom}(b)} &:& (a,c) = (1,0),\quad T_j = \delta_{j,1},\\
\text{Fig.~\ref{fig:Tok_geom}(c)} &:& (a,c) = (1,1),\quad T_j = 1,\\
\text{Fig.~\ref{fig:Tok_geom}(d)} &:& (a,c) = (1,2),\quad T_j = 2^{j-1}.
\end{eqnarray*}
The case $T_j=0$ corresponds to perfect binary trees with no side branching (see also Ex.~\ref{ex:perfect}). 
In this case, all branch mergers lead to increase of branch order by unity.
This results in a most symmetric deterministic tree structure.
Some side branching appears for $T_j=\delta_{j,1}$ 
(hence $T_1=1, T_2=0, T_3=0,\dots$): every branch of order $K$ has on average 
a single side branch of order $(K-1)$, and no side branches of lower orders.
This destroys symmetry and introduce randomness in tree shape.
The case $T_j=1$ corresponds to an average of one side branch of any order $1\le k\le K-1$
within a branch of order $K$, resulting
in tentacle-shaped formations of varying length.
The most complicated case illustrated here corresponds to $T_j=2^{j-1}$,
which is the Tokunaga sequence for critical binary Galton-Watson trees (but
not necessarily vice versa); see Ex.~\ref{ex1}.
In this case the number of side branches increases geometrically with
the difference of branch orders, hence producing branches with
widely varying lengths and shapes.
\end{ex}

\begin{ex}[{{\bf Tokunaga trees with i.i.d. exponential edge lengths}}]
\label{ex3b}
Random edge lengths often appear as an element of applied modeling. 
Figure~\ref{fig:Tok_exp} illustrates the same four Tokunaga models
as in Ex.~\ref{ex3a}, with i.i.d. exponential edge lengths.
Clearly, this additional random element substantially affects the 
tree outlook.
The edge length variability becomes a dominant element of the metric tree 
shape.
We notice, in particular, that the four types of trees with
exponential edge lengths in Fig.~\ref{fig:Tok_exp}
look much more similar that the same four types with deterministic
edge lengths related to branch order. 
\end{ex}

\begin{ex}[{{\bf Critical Tokunaga processes}}]
\label{ex4}
Section~\ref{sec:Tok} introduces a subclass of HBP, called critical Tokunaga processes,
with $T_j = (c-1)c^{j-1}$, $j\ge 1$ for an arbitrary $c\ge 1$.
These processes generate tree distributions that are Horton self-similar 
with $\zeta = c$ and have i.i.d. exponential edge lengths.
\end{ex}

\begin{figure}[t] 
\centering\includegraphics[width=0.8\textwidth]{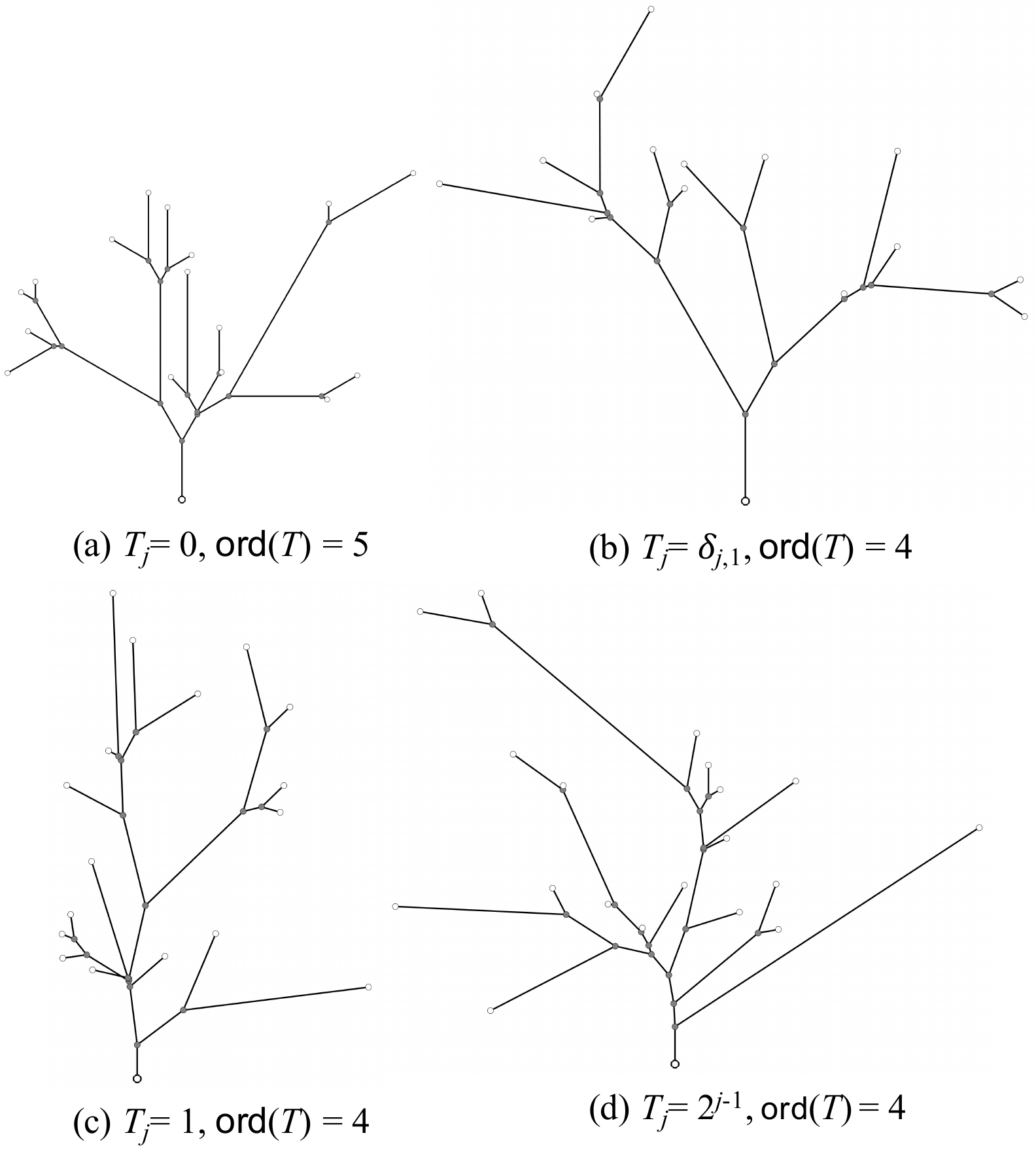}
\caption[Tokunaga trees with i.i.d exponential edge lengths]
{Tokunaga trees with i.i.d exponential edge lengths. 
Different panels correspond to different Tokunaga sequences $T_j=ac^{j-1}$.
(a) $(a,c)=(0,0), T_j = 0$,
(b) $(a,c)=(1,0), T_j = \delta_{k,1}$,
(c) $(a,c)=(1,1), T_j = 1$,
(d) $(a,c)=(1,2), T_j = 2^{j-1}$.}
\label{fig:Tok_exp}
\end{figure}

\begin{ex}[{{\bf Independent random attachment}}]
\label{ex5}
A variety of mean Horton self-similar measures on $\cT$ can be constructed for 
an arbitrary sequence of Tokunaga coefficients $\{T_j\}_{j=1,2,\hdots}$.
Here we give a natural example \cite{KZ16}.

Fix a sequence $\{T_j\}_{j=1,2,\hdots}$ of Tokunaga coefficients. 
By Remark~\ref{rem:muk}, it is sufficient to construct a set of Horton self-similar 
conditional measures $\mu_K$, $K\ge 1$.

The subspace $\cH_1$, which consists of a single-leaf tree $\tau_1$,
possesses a trivial unity mass conditional measure $\mu_1$.
To construct a random tree from $\cH_2$, we select a discrete probability
distribution $P_{1,2}(n)$, $n=0,1,\dots$, with the mean value $T_1$.
A random tree $T\in\cH_2$ is obtained from the single-leaf 
tree $\tau_1$ via the following two operations.
First, we attach two offspring vertices to the leaf of $\tau_1$.
This creates a tree of order $2$ with no side-branches -- one 
internal vertex of degree 3, two leaves, and the root.
Second, we draw the number $\tilde N_{1,2}$ from the distribution $P_{1,2}$, and
attach $\tilde N_{1,2}$ vertices to this tree so that they form side-branches 
of index $\{1,2\}$.

In general, we use a recursive construction procedure. 
Assume that a measure $\mu_{K-1}$, $K\ge 2$, is constructed. 
To construct a random tree $T\in\cH_{K}$ we select a set 
of discrete probability distributions $P_{k,K}(n)$, $k=1,...,K-1$, on $\mathbb{Z}_+$ 
with the respective mean values $T_j$.
A random tree $T\in\cH_{K}$ is constructed by adding branches of order $1$ (leaves)
to a random tree $\tau\in\cH_{K-1}$.
First, we add two new child vertices to every leaf of $\tau$ hence producing
a tree $\tilde T$ of order $K$ with no side-branches of order $1$.
Second, for each branch $b$ of order $2\le j\le K$ in $\tilde T$ we draw a random number 
$\tilde N_{1,j}(b)$ 
from the distribution $P_{j-1,K}$ and attach $\tilde N_{1,j}(b)$ new child vertices to 
this branch so that they form side-branches of index $\{1,j\}$. 
Each new vertex is attached in a random order with respect to
the existing side-branches.
Specifically, we notice that $m\ge 0$ side-branches attached to a branch of order $j$
are uniquely associated with $m+1$ edges within this branch. 
The attachment of the new $\tilde N_{1,j}(b)$ vertices among the $m+1$ edges
is given by the equiprobable multinomial distribution with $m+1$ categories
and $\tilde N_{1,j}(b)$ trials.

The procedure described above generates a set of mean-coordinated measures $\{\mu_K\}_{K\ge 1}$ 
on $\{\cH_K\}_{K\ge 1}$, since 
the mean values $T_j$ of the distributions $P_{k,K}$ are independent of $K$.
Furthermore, observe that
\[N_{i,j}=\sum_{b_i=1}^{N_j}\tilde N_{1,j-i+1}(b_i),\]
\begin{eqnarray}
\label{condarg}
\cN_{i,j}[K]&=&{\sf E}_K[N_{i,j}]={\sf E}_K\left[{\sf E}_K[N_{i,j}|N_j]\right]
={\sf E}_K[N_j\,T_{j-i}]\\
&=&T_{j-i}\,{\sf E}_K[N_j]=T_{j-i}\,\cN_j[K]\nonumber,
\end{eqnarray}
and hence
$T_{i,j}[K]=\cN_{i,j}[K]/\cN_{j}[K]=T_{j-i}$, so the tree is mean self-similar,
according to Def.~\ref{ss1}.

\noindent
Finally, to make that construction combinatorially Horton self-similar (Def. \ref{def:ss}), each tree $\tau_K\in\cH_K$ must be assigned the probability 
$p_K=p(1-p)^{K-1}$.
\end{ex}

\begin{ex}[{{\bf Why coordination?}}]
\label{rem_coord}
Relating mean Horton self-similarity (Def. \ref{def:ss2}) to 
mean prune-invariance (Def.~\ref{mpi}) is quite intuitive (see also \cite{BWW00}).
Much less so is the requirement of mean coordination of conditional measures 
(Def.~\ref{coord}), included in the definition of mean self-similarity.
This requirement is motivated by our goal to bridge the measure-theoretic 
definition of self-similarity via the pruning operation (Def.~\ref{def:ss2})
to a branch counting definition (Def.~\ref{ss1}).
In applications, when a handful of trees of different orders is observed, the coordination assumption
allows one to estimate the Tokunaga coefficients $T_{i,j}$ and make inference regarding
the Toeplitz property; see \cite{Pec95,NTG97,DR00,ZZF13}. 
The absence of coordination, at the same time, 
allows for a variety of prune-invariant
measures with no Toeplitz constraint, which are hardly treatable in applications. 
To give an example of such a measure, let select any tree $\tau_2$ from the pre-image
of the only tree $\tau_1\in\cH_1$ of order $K=1$ under the pruning operation: $\tau_2\in\cR^{-1}(\tau_1)=\cH_2$.
In a similar fashion, select any tree $\tau_{K+1}$ from the pre-image of $\tau_{K}$ for $K\ge 2$.
This gives us a collection of trees $\tau_K\in\cH_K$, $K\ge 1$ such that $\cR(\tau_{K+1}) = \tau_K$.
Assign the full measure on $\cH_K$ to $\tau_K$: $\mu_K(\tau_K)=1$.
By construction, the measures $\{\mu_K\}$ are mean prune-invariant.
They, however, may satisfy neither the mean coordination nor the Toeplitz property.
This example illustrates how one can produce rather obscure collections of 
mean prune-invariant measures, providing a motivation for the
coordination requirement.
\end{ex}

\section{Horton law in self-similar trees}\label{sec:HLSST}
In this section, we introduce the {\it strong Horton law} for the numbers of branches
of different orders in a combinatorial tree on $\cT$
(Def.~\ref{def:Horton_rv}) and for the respective averages (Def.~\ref{def:Horton_mean}).
The main result of this section (Thm.~\ref{thm:HLSST}) shows that the mean Horton self-similarity (Defs.~\ref{ss1} and \ref{def:ss2}) implies the strong Horton law for mean branch numbers
(Def. \ref{def:Horton_mean}).
\index{Horton law!strong Horton law}

Consider a measure $\mu$ on $\cT$ and its conditional measures $\mu_K$, each defined
on subspace $\cH_K\subset\cT$ of trees of Horton-Strahler 
order $K\ge 1$. 
We write $T\stackrel{d}{\sim}\mu_K$ for a random tree $T$ drawn from subspace 
$\cH_K$ according to measure $\mu_K$.

\begin{Def}[{{\bf Strong Horton law for branch numbers}}] 
\label{def:Horton_rv}
We say that a probability measure $\mu$ on $\cT$ satisfies a strong Horton law for branch numbers
if there exists such a positive (constant) Horton exponent $R\ge2$ that for any $k\ge 1$ 
\be
\label{eq:Horton_rv}
\left(\frac{N_k[T]}{N_1[T]};\,T\stackrel{d}{\sim}\mu_K\right)~\stackrel{p}{\longrightarrow}~ R^{1-k},\quad\text{ as }\quad K\to\infty,
\ee
\noindent that is, for any $\epsilon>0$
\be
\label{eq:Horton_rv1}
\mu_K\left(\left|\frac{N_k[T]}{N_1[T]}-R^{1-k}\right|>\epsilon\right)~{\to}~ 0\quad\text{ as }\quad K\to\infty.
\ee
\end{Def}

\noindent 
Corollary~\ref{cor:DoobMartingaleSHL} in Sect.~\ref{sec:martingale} is an example of the strong Horton law for branch numbers.
In the context of Horton laws, the adjective {\it strong} refers to the type of geometric decay,
while the convergence of random variables is in probability.
Section~\ref{sec:convergence} discusses weaker types of geometric convergence.
An alternative, weaker, definition of the Horton law is formulated in terms of 
expected branch counts.

\begin{Def}[{{\bf Strong Horton law for mean branch numbers}}] 
\label{def:Horton_mean}
We say that a probability measure $\mu$ on $\cT$ satisfies a strong Horton law for mean branch numbers
if there exists such a positive (constant) Horton exponent $R\ge2$ that for any $k\ge 1$ 
\be
\label{eq:Horton_mean}
\lim_{K\to\infty}\left(\frac{{\sf E}\left[N_k[T]\right]}{{\sf E}\left[N_1[T]\right]};\, T\stackrel{d}{\sim}\mu_K\right)
= \lim_{K\to\infty}\frac{\cN_k[K]}{\cN_1[K]}
= R^{1-k}.
\ee
\end{Def}

\medskip
\noindent 
\begin{lem}\label{lem:SHLbn-SHLmean}
The strong Horton law for branch numbers (Def. \ref{def:Horton_rv}) implies the strong Horton law for mean branch numbers (Def. \ref{def:Horton_mean}). 
\end{lem}
\begin{proof}
By construction, if ${\sf ord}(T)=K$, then $N_1[T] \geq 2^{K-1}$.
Accordingly, for any $k \leq K$ we have ${N_k[T] \over N_1[T]}\leq 2^{1-k}$. 
Assuming the strong Horton law \eqref{eq:Horton_rv1} for branch numbers, for any given $\epsilon>0$, we have
$$\mu_K\left(\left|{N_k[T] \over N_1[T]}-R^{1-k}\right|>\epsilon\right)<\epsilon$$
for all sufficiently large $K$.
Thus, for a given $k \in \mathbb{N}$ and for all sufficiently large $K$ exceeding $k$, we have
\begin{align*}
\left|{\cN_k[K] \over \cN_1[K]}-R^{1-k}\right| 
&= \left({\left| \E \left[N_1[T]\left({N_k[T] \over N_1[T]}-R^{1-k}\right)\right]\right| \over \E \big[N_1[T]\big]};\,T\stackrel{d}{\sim}\mu_K\right)\\
&\leq \left({\E \left[N_1[T]\left|{N_k[T] \over N_1[T]}-R^{1-k}\right|\right] \over \E \big[N_1[T]\big]};\,T\stackrel{d}{\sim}\mu_K\right) \\
&\leq \left({\epsilon\E \big[N_1[T]\big]+\epsilon 2^{1-k} \over \E \big[N_1[T]\big]};\,T\stackrel{d}{\sim}\mu_K\right)\\
&\leq \epsilon+\epsilon 2^{2-k-K} < 2\epsilon,
\end{align*}
as $\left|{N_k[T] \over N_1[T]}-R^{1-k}\right| \leq \max\Big(2^{1-k},\,R^{1-k}\Big) \leq 2^{1-k}$. 
This establishes \eqref{eq:Horton_mean}.
\end{proof}

\medskip
\noindent
A similar calculation allows us to establish the following result.
\begin{lem}\label{lem:SHLmeanP-SHLbn}
Consider a probability measure $\mu$ on $\cT$ and suppose the following properties hold: 
\begin{description}
  \item[(i)] $\mu$ satisfies the strong Horton law for mean branch numbers 
  (Def.~\ref{def:Horton_mean}), and
  \item[(ii)] $\forall k\ge 1$ $\exists L_k \in [0,\infty)$ such that $\left(\frac{N_k[T]}{N_1[T]};T\stackrel{d}{\sim}\mu_K\right)\stackrel{p}{\to} L_k$ as $K\to\infty$.
\end{description}
Then, the measure $\mu$ satisfies the strong Horton law for branch numbers (Def.~\ref{def:Horton_rv}), i.e., $L_k=R^{1-k}$.
\end{lem}

\medskip
\noindent
Sufficient conditions for the strong Horton law for mean branch numbers in binary trees 
were found in \cite{KZ16}, 
hence providing rigorous foundations for the celebrated regularity that has escaped a 
formal explanation for a long time.
These conditions are presented in Thm.~\ref{thm:HLSST} of this section.
It has been shown in \cite{KZ17ahp} that the tree that describes a trajectory of Kingman's coalescent process with $N$ particles obeys a weaker version of Horton law as $N\to\infty$ (Sect.~\ref{sec:Kingman}), and that the first pruning of this tree for any finite $N$
is equivalent to a level set tree of a white noise (see Sect.~\ref{LST} for definitions).

Consider a mean self-similar measure $\mu$ on $\cBT$ with a Tokunaga sequence $\{T_j\}_{j=1,2,\hdots}$.
Define a sequence $t(j)$ as
\be \label{seq:tj}
t(0)=-1,~t(1)=T_1+2,~\text{ and }t(j)=T_j \text{ for }j \geq 2,
\ee
and let $\hat{t}(z)$ denote the generating function of $\{t(j)\}_{j=0,1,\hdots }$:
\be\label{def:that}
\hat{t}(z)=\sum\limits_{j=0}^\infty z^j t(j)=-1+2z+\sum\limits_{j=1}^\infty z^j T_j.
\ee 
\noindent For a holomorphic function $f(z)$ represented by a power series $f(z)=\sum\limits_{j=0}^\infty a_j z^j$ in a nonempty disk $|z|\leq \rho$ we write 
\be\label{eq:check}
\check{f}(j)={1 \over 2\pi i}\oint\limits_{|z|=\rho} {f(z) \over z^{j+1}} dz=a_j.
\ee

\begin{thm}[{{\bf Strong Horton law in a mean self-similar tree}}]
\label{thm:HLSST}
Suppose $\mu$ is a mean Horton self-similar measure on $\cBT$ with a 
Tokunaga sequence $\{T_j\}_{j=1,2,\hdots}$ such that
\be
\label{eq:tamed}
\limsup_{j\to\infty} T_j^{1/j}<\infty.
\ee
Then the strong Horton law for mean branch numbers (Def.~\ref{def:Horton_mean}) holds
with the Horton exponent $R=1/w_0$, where $w_0$ is the only real zero of
the generating function $\hat{t}(z)$ in the interval $\left(0,{1 \over 2}\right]$. 
Moreover,
\begin{equation} \label{zetaT}
\cN_1[K+1]=
-\widecheck{\left(\frac{1}{\widehat ~ \!\! t}\right)}(K) 
\end{equation}
and
\be\label{eq:Nk}
\lim_{K\to\infty}\left(\cN_1[K]\,R^{-K}\right) = const. > 0.
\ee
Conversely, if 
$~\limsup\limits_{j \rightarrow \infty} T_j^{1/j} =\infty$, 
then the limit 
$\lim\limits_{K\to\infty}\frac{\cN_{k}[K]}{\cN_1[K]}$ 
does not exist at least for some $k$.
\end{thm}
\begin{proof}
The proof of Thm.~\ref{thm:HLSST} is given in  Sect.~\ref{sec:proof1}.
\end{proof}
\noindent That the Horton exponent $R$ is reciprocal to the real root
of $\hat t(z)$ was noticed by Peckham \cite{Pec95}, under the
assumption $\displaystyle\lim_{K\to\infty}\left(N_k R^{k-K}\right)=const. >0 $.

\begin{figure}[t] 
\centering
\includegraphics[width=0.49\textwidth]{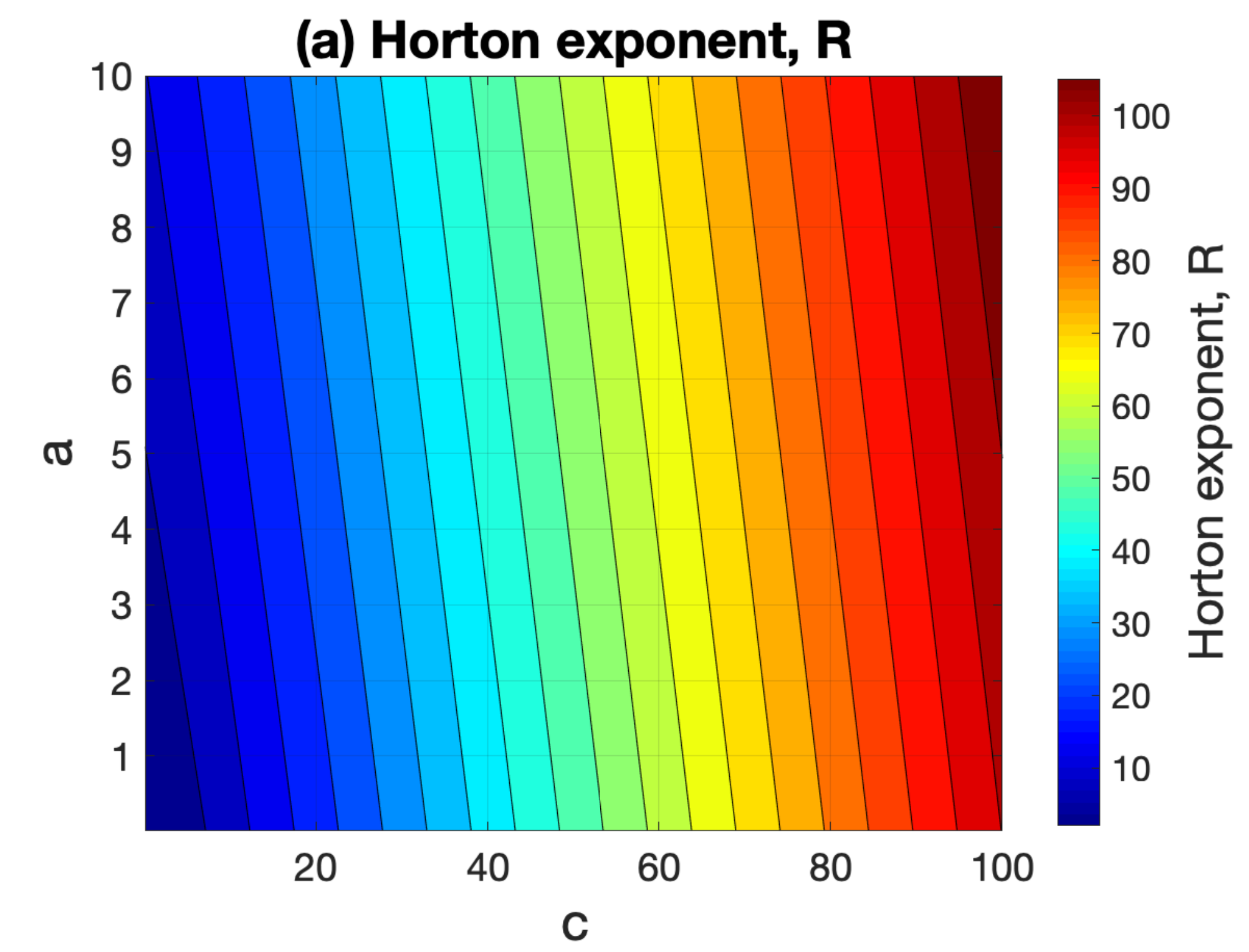}
\includegraphics[width=0.49\textwidth]{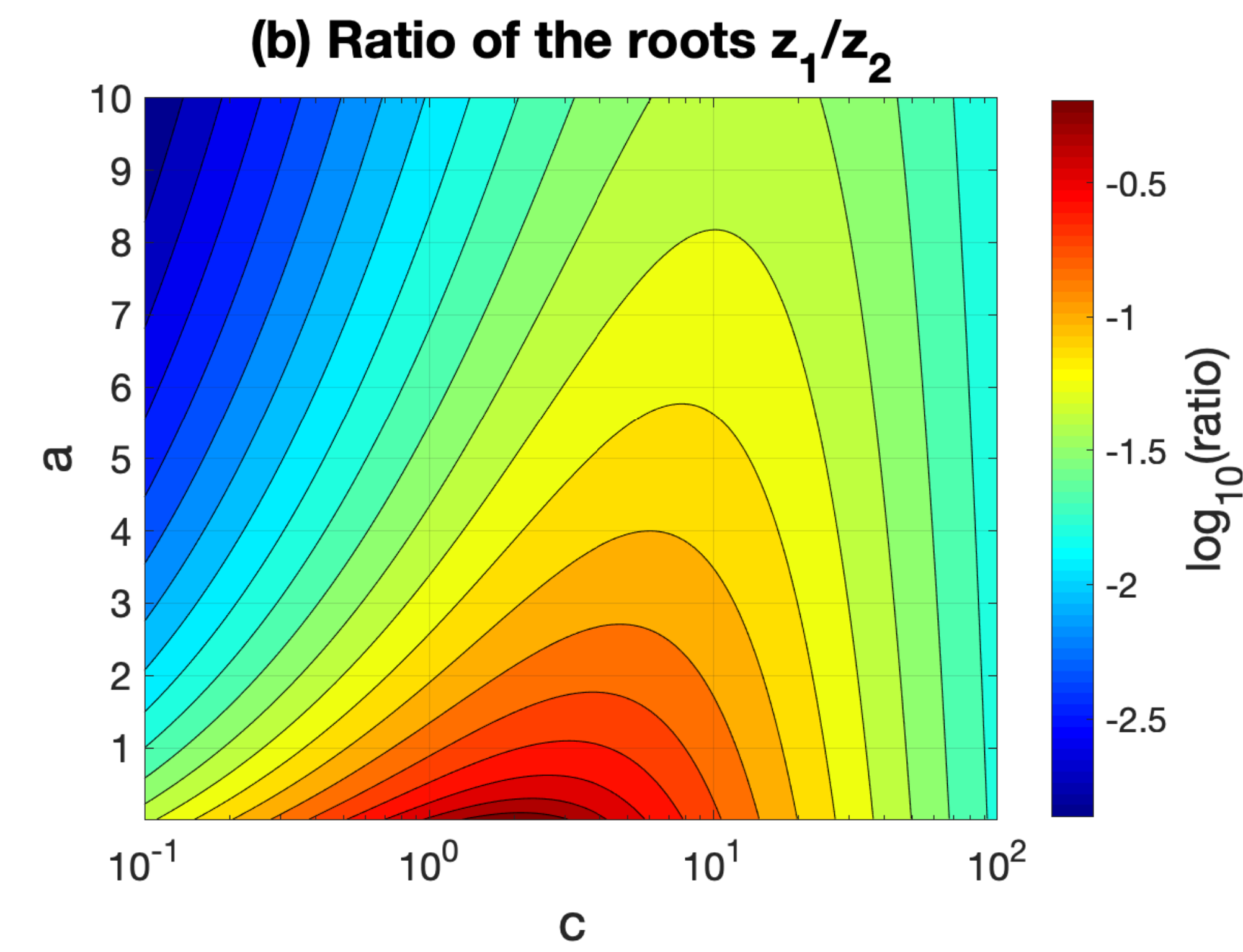}
\caption[Horton law in a Tokunaga tree]
{Strong Horton law in a Tokunaga mean self-similar tree with $T_j=ac^{j-1}$, $j\ge 1$.
(a) Horton exponent $R$ as a function of the Tokunaga parameters $(a,c)$.
(b) The ratio $0<z_1/z_2<1$ of the two roots of the equation $2cz^2-(a+c+2)z+1=0$
as a function of the Tokunaga parameters $(a,c)$. This ratio controls the
rate of convergence in the strong Horton law -- small values increase the rate.}
\label{fig:R}
\end{figure} 

\medskip
\noindent
Below we give two examples of using Theorem~\ref{thm:HLSST}.

\begin{ex}[{{\bf Tokunaga self-similar trees}}] 
Consider a Tokunaga self-similar tree (Def.~\ref{ss2}) with $T_j=a\,c^{j-1}$, where $a,c>0$. 
(We exclude the case $a=0\Rightarrow T_j=0$, which correspond to perfect binary trees with
no side branching.)
This model received considerable attention in the literature \cite{Pec95,Tok78,MG08}, in
part because of its ability to closely describe river networks \cite{ZZF13}. 
Here we have
\[\limsup_{j\to\infty} T_j^{1/j} = c<\infty\]
and
\begin{eqnarray}
\hat{t}(z)&=&-1+2z+az \sum\limits_{j=1}^\infty (cz)^{j-1}
=-1+2z+{az \over 1-cz}\nonumber\\
&=&{-1+(a+c+2)z-2cz^2 \over 1-cz}\text{   for  } |z| < 1/c.
\end{eqnarray}
The discriminant of the quadratic polynomial in the numerator is positive,  
\[(a+c+2)^2-8c > (c+2)^2-8c=(c-2)^2 \geq 0.\] 
Therefore, there exist two real roots, $z_1<z_2$, of the numerator.
It is easy to check that 
\[z_1z_2=(2c)^{-1},0< z_1 <\min\{2^{-1},c^{-1}\}, \text{ and } z_2>\max\{2^{-1},c^{-1}\}.\]
Hence, there is a single root of $\hat{t}(z)=0$ for $|z|<1/c$ of algebraic multiplicity one:
\[ z_1 \equiv w_0 = {a+c+2 -\sqrt{(a+c+2)^2-8c} \over 4c},\]
and the respective Horton exponent is
\begin{equation} \label{HortonR}
R=1/w_0 = {a+c+2 +\sqrt{(a+c+2)^2-8c} \over 2}
\end{equation}
as was observed in earlier works \cite{Tok78,Pec95,MG08}. 
A map of the values of the Horton exponent $R(a,c)$ is shown in Fig.~\ref{fig:R}a.
As suggested by \eqref{HortonR}, the level sets of $R(a,c)$ are fairly approximated
by $a+c = const.$

\begin{figure}[t] 
\centering
\includegraphics[height=1.9in]{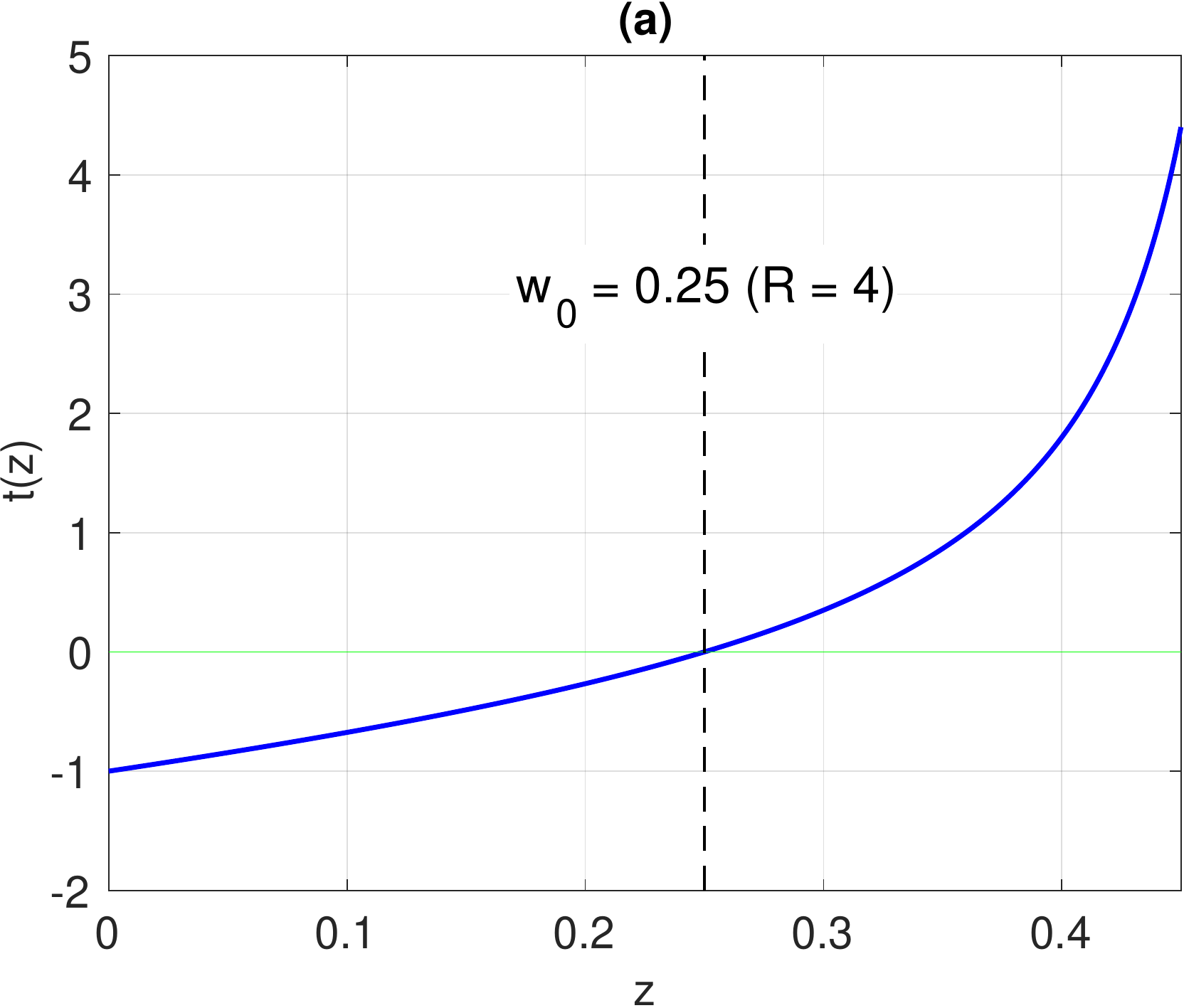}
\includegraphics[height=1.9in]{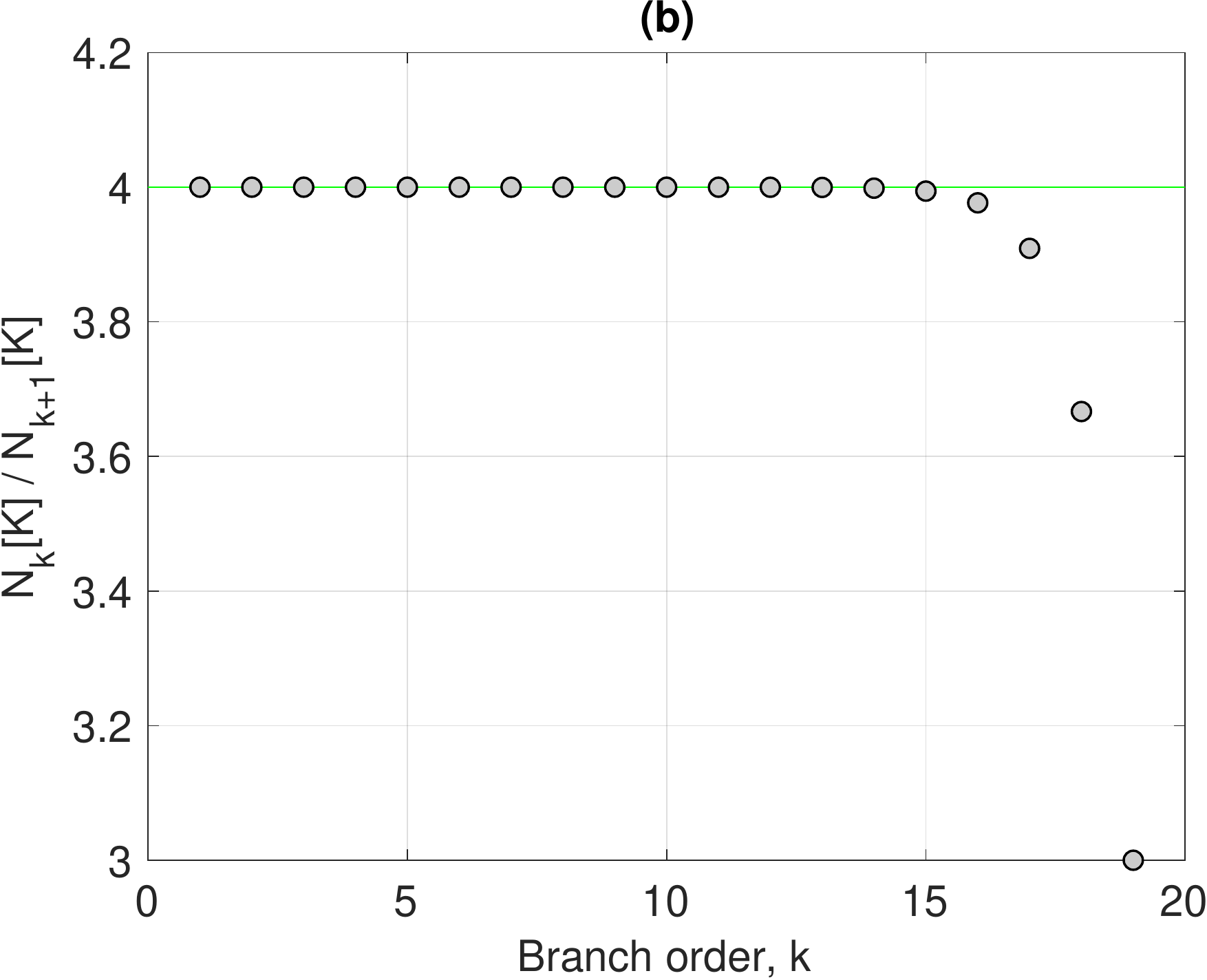}
\caption[Strong Horton law in a mean self-similar tree]
{The strong Horton law in a mean self-similar tree: an illustration.
The figure refers to a Tokunaga mean self-similar measure $\mu$ with
 $T_j = 2^{j-1}$, $j\ge 1$.
(a) Characteristic function $\hat{t}(z)$ (solid blue). 
The zero level is marked by a green horizontal line. 
The real solution $w_0=0.25$ is depicted by a vertical dashed line.
(b) Ratio $\cN_k[K]/\cN_{k+1}[K]$ for tree order $K=20$ and branch orders $k=1,\dots,19$.
The strong Horton law suggests $\cN_k[K]/\cN_{k+1}[K]\approx R = 4$ for large $K$ and
$k$ not too close to $K$.}
\label{fig:that}
\end{figure} 

To examine the rate of convergence in the strong Horton law, we use \eqref{zetaT}.
The reciprocal generating function is given by 
\begin{eqnarray}
-{1 \over \hat{t}(z)}&=&\frac{1-cz}{2c(z-z_1)(z-z_2)}\nonumber\\
&=&{1-cz\over 2c(z_2-z_1)}\left({1\over z-z_1}-{1 \over z-z_2}\right).
\end{eqnarray}
Thus, since ${1 \over z-p}=-\sum\limits_{k=0}^\infty {1 \over p^{k+1}}z^k$ for $|z|<|p|$, formula (\ref{zetaT}) implies
\begin{eqnarray}\label{HortonZeta}
\cN_1[K+1]&=&
{1 \over 2c(z_2-z_1)}
\left({1-cz_1 \over z_1^{K+1}} -{1-cz_2 \over z_2^{K+1}} \right)\nonumber\\
&=&
{1-cz_1 \over 2c(z_2-z_1)}{1\over z_1^{K+1}}\left(1-\left(\frac{z_1}{z_2}\right)^{K+1}
\frac{1-cz_2}{1-cz_1}\right).
\end{eqnarray}
Accordingly, the rate of convergence in \eqref{eq:Nk} is determined by the ratio $z_1/z_2<1$
-- values farther away from 1 lead to faster convergence.
Recall (Prop.~\ref{prop1_Horton}) that 
\[\cN_1[m+1] = \cN_{K-m}[K],\quad 0\le m\le K-1, \quad K\ge 1.\]
Hence, the ratio $z_1/z_2$ also determines the rate of convergence in \eqref{eq:Horton_mean}.
Figure~\ref{fig:R}(b) shows the ratio $z_1/z_2$ as a function of $(a,c)$.
The only region when the ratio is approaching 1, hence slowing down the
convergence rate in the strong Horton law, corresponds to $\{c\approx 2, a<1\}$.

\begin{figure}[t] 
\centering
\includegraphics[width=0.75\textwidth]{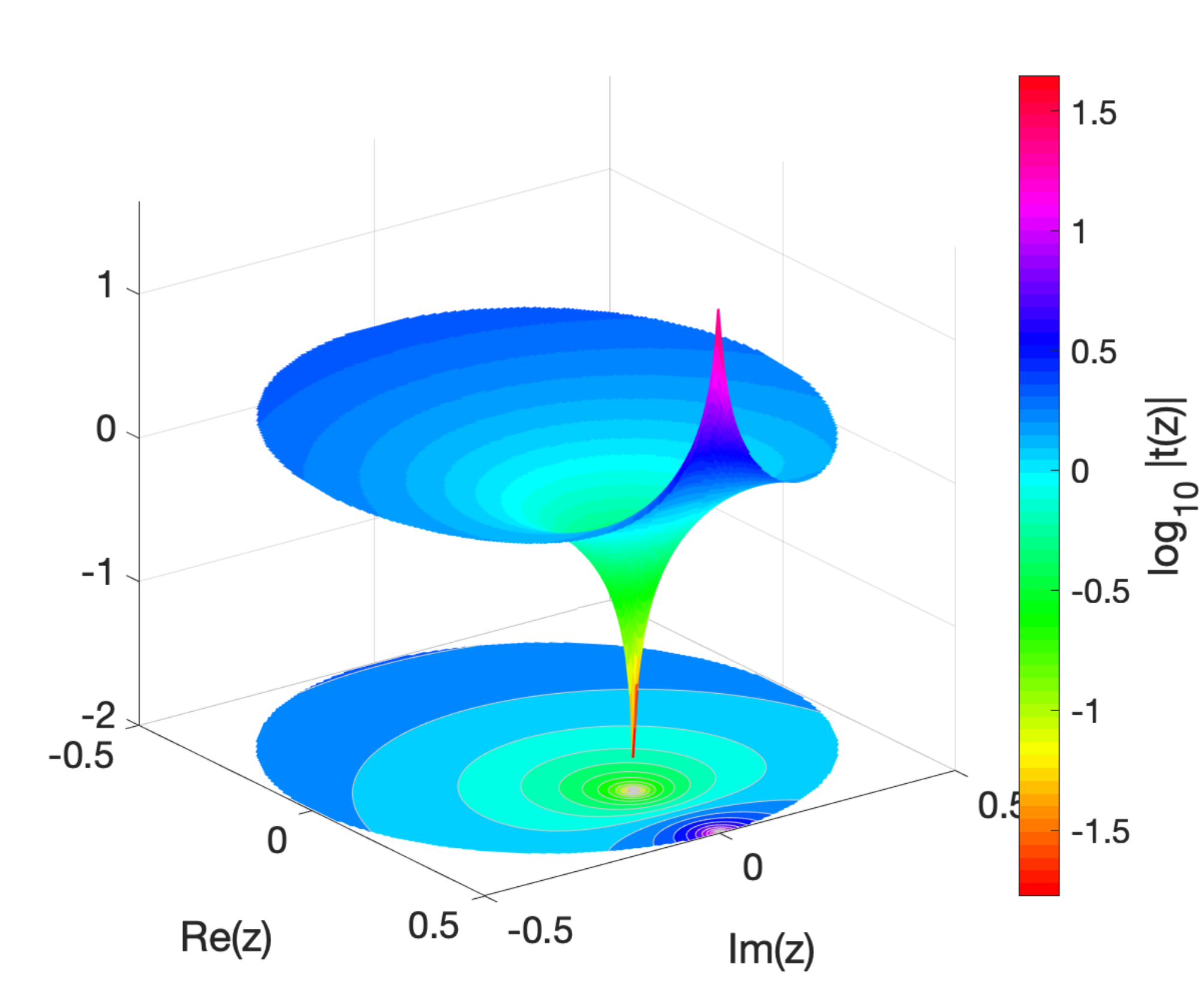}
\caption[Generating function $\hat{t}(z)$ for a Tokunaga tree with $(a,c)=(1,2)$.]
{Generating function $\hat{t}(z)$ for a mean Horton self-similar tree 
with Tokunaga sequence $T_j=2^{j-1}$, $j\ge 1$, see \eqref{eq:that_Tok}.
The figure shows the value $\log_{10}|\hat{t}(z)|$ for visual convenience.
The values of $\hat{t}(z)$ are well separated from its only zero at $z=1/4$,
ensuring a hight convergence rate in the strong Horton law.}
\label{fig:that0}
\end{figure}

Figure~\ref{fig:that} illustrates the strong Horton law in a Tokunaga mean self-similar tree with $a=1, c=2$,
which corresponds to $T_j=2^{j-1}$, $j\ge 1$.
In this case (Figs.~\ref{fig:that}(a),\ref{fig:that0})
\be\label{eq:that_Tok}
\hat{t}(z) = \frac{-1+5z-4z^2}{1-2z} = \frac{-4(z-1)(z-1/4)}{1-2z},\quad |z|<1/2.
\ee
The ratios $\cN_k[K]/\cN_{k+1}[K]$ for $K=20$ are shown in Fig.~\ref{fig:that}(b).
The ratios are very close to the theoretical value $R=1/w_0=4$, except for 
the branch orders $k$ close to the tree order $K$, $k>15$.
As suggested by Fig.~\ref{fig:R}(b), for most of the choices $(a,c)$ the convergence 
rate is higher, so we expect to have a larger number of ratios in a close vicinity
of the limit value $R$. 
As we discussed above, the convergence in \eqref{eq:Nk} has the same rate,
with first terms (small $k$) deviating from the limit value rather
then the last ones, as was the case in \eqref{eq:Horton_mean} and 
Fig.~\ref{fig:that}(b). 

We show below in Eq.~\ref{eq:error_est} that, in general, the rate of convergence in 
the strong Horton law \eqref{eq:Horton_mean}, \eqref{eq:Nk} is controlled by 
\[\min\limits_{|z|<\gamma}|\hat{t}(z)|,\]
where $\gamma$ separates $w_0$ from other possible zeros of $\hat{t}(z)$ 
-- higher values lead to faster convergence. 
Figure~\ref{fig:that0} shows the value $\log_{10}|\hat{t}(z)|$ on its disk on convergence 
for the Tokunaga tree of this example.
Here, the only zero of $\hat{t}(z)$ at $z=1/4$ (downward peak) is well isolated so that 
the surrounding values are separated from zero; this suggests a high rate of 
convergence that we already 
illustrated more directly in
\eqref{HortonZeta} and Figs.~\ref{fig:R}(b),\ref{fig:that}(b).
\end{ex}

\begin{ex}[{{\bf Shallow side-branching}}]
Suppose $T_j=0$ for $j\ge 3$, that is we only have 
``shallow'' side-branches of orders $\{k-2,k\}$ and $\{k-1,k\}$.
Then
\[\hat t(z) = -1 + (T_1+2)\,z + T_2\,z^2.\]
The only root of this equation within $[0,1/2]$ is
\[w_0 = \frac{\sqrt{(T_1+2)^2+4T_2}-(T_1+2)}{2\,T_2},\]
which leads to
\[R=1/w_0=\frac{\sqrt{(T_1+2)^2+4T_2}+(T_1+2)}{2}.\]
In particular, if $T_j=0$ for $j\ge 2$, then $R=T_1+2$;
such trees are called ``cyclic'' \cite{Pec95}.
This shows that the entire range of Horton exponents $2\le R < \infty$
can be achieved by trees with only very shallow side-branching.
\end{ex}

\noindent
We conclude this section with a linear algebra construction that
clarifies the essence of Horton law in a mean self-similar tree. 
Define a vector $\zeta_K\in\mathbb{R}^K$ of average Horton numbers  and 
a respective normalized vector $\xi_K\in\mathbb{R}^\infty$ as
$$\zeta_K=\left(\!\!\!\begin{array}{c}\cN_{1}[K] \\\cN_{2}[K] \\\vdots \\ \cN_{K}[K]\end{array}\!\!\!\right) 
\quad \text{ and } \quad
\xi_K:={1 \over \cN_1[K]}\left(\!\!\begin{array}{c}\zeta_K \\0 \\0 \\\vdots \end{array}\!\!\right)=\left(\!\!\!\!\begin{array}{c}1 \\\cN_{2}[K]/\cN_{1}[K] \\\vdots  \\\cN_{K}[K]/\cN_{1}[K]\\0 \\0 \\\vdots \end{array}\!\!\!\!\right)$$ 
and consider an infinite dimensional extension to operator $\mathbb{G}_K$ of \eqref{eq:Gk}:
\begin{equation} \label{gen}
\mathbb{G}:=\left[\begin{array}{ccccc}-1 & T_1+2 & T_2 & T_3 & \hdots  \\0 & -1 & T_1+2 & T_2 & \hdots  \\0 & 0 & -1 & T_1+2 & \ddots  \\0 & 0 & 0 & -1  & \ddots \\\vdots & \vdots & \ddots & \ddots  & \ddots \end{array}\right].
\end{equation}

\noindent
Using these notations, the main counting equations \eqref{eq:count} becomes
$\mathbb{G}_K \zeta_K = -e_K,$ 
and therefore 
\[\mathbb{G} \xi_K = -{e_K \over \cN_1[K]}.\] 

\noindent
Here $\cN_{1}[K] \geq (T_1+2)^{K-1}\to\infty$ as $K\to\infty$, and hence the
strong Horton law for mean branch numbers (Def.~\ref{def:Horton_mean})
is equivalent to the existence of a limit solution
 $\lim\limits_{K \rightarrow \infty} \xi_K = \xi$ to an infinite 
 dimensional linear operator equation
 \[\mathbb{G} \xi=0\] 
with coordinates $\xi(k)=R^{1-k}$.

\subsection{Proof of Theorem \ref{thm:HLSST}}
\label{sec:proof1}

First, we establish (Prop.~\ref{propmain_Horton}) necessary and sufficient conditions for the 
existence of the strong Horton law. 
Then we show that these conditions are satisfied and express the value of the 
Horton exponent $R$ via the Tokunaga coefficients $\{T_j\}$.

\begin{prop} \label{propmain_Horton}
Let $\mu$ be a mean Horton self-similar measure on $\cBT$.
Suppose that the limit
\be
\label{R}
R=\lim\limits_{K \rightarrow \infty} {\cN_{1}[K+1] \over \cN_{1}[K]}
\ee
exists and is finite.
Then, the strong Horton law for mean branch numbers holds; that is, for each positive integer $k$,
\be\label{Horton_here}
\lim\limits_{K\to\infty} \frac{\cN_{k}[K]}{\cN_{1}[K]}=R^{1-k}.
\ee
Conversely, if the limit \eqref{R} does not exist, then the limit
in the left hand side of \eqref{Horton_here} also does not exist, at least for some $k$.
\end{prop}

\begin{proof}
Suppose the limit \eqref{R} exists and is finite.
Proposition~\ref{prop1_Horton} implies that for any fixed integer $m \ge 1$
$${\cN_{m+1}[K] \over \cN_m[K]}={\cN_1[K-m] \over \cN_1[K-m+1]} \rightarrow R^{-1},\text{ as }K\to\infty.$$
Thus, for any fixed integer $k\geq 2$,
$${\cN_k[K]\over \cN_1[K]}
=\prod\limits_{m=1}^{k-1} {\cN_{m+1}[K]\over \cN_m[K]} \rightarrow  R^{1-k}, \text{ as }K\to\infty.$$
\noindent
Conversely, suppose the limit $\lim\limits_{K\rightarrow\infty} {\cN_1[K+1]\over \cN_1[K]}$ does not exist. Taking $k=2$, we obtain by Prop.~\ref{prop1_Horton}
$${\cN_2[K] \over \cN_1[K]}={\cN_1[K\!-\!1]\over \cN_1[K]}.$$
Thus $\lim\limits_{K \rightarrow \infty}{\cN_2[K]\over\cN_1[K]}$ diverges.
\end{proof}

Next, we express $\cN_1[K]$ via
the elements of the Tokunaga sequence $\{T_j\}_{j=1,2,\hdots}$ that satisfy condition~\eqref{eq:tamed}. 
The quantity $\cN_1[K+1]$ can be computed by counting, 
and expressed via convolution products as follows:
\begin{eqnarray*}
\cN_1[K+1]& = & \sum\limits_{r=1}^K \sum\limits_{\substack{j_1,j_2,\hdots,j_r \geq 1\\ j_1+j_2+\hdots+j_r =K}}t(j_1)t(j_2)\hdots t(j_r)\\
& =  & \sum\limits_{r=1}^K \underbrace{(t+\delta_0)\ast (t+\delta_0) \ast \hdots \ast (t+\delta_0)}_{r \text{ times }}(K)\\
& =  & \sum\limits_{r=1}^\infty \underbrace{(t+\delta_0)\ast (t+\delta_0) \ast \hdots \ast (t+\delta_0)}_{r \text{ times }}(K),\\
\end{eqnarray*}
where $\delta_0(j)$ is the Kronecker delta, and therefore, $(t+\delta_0)(0)=0$. Hence, taking the \mbox{$z$-transform} of $\cN_1[K]$, we obtain
\begin{equation} \label{Tzeta}
\sum_{K=1}^\infty z^{K-1}\cN_1[K]=1+\sum\limits_{r=1}^\infty \Big[\widehat{(t+\delta_0)}(z)\Big]^r=1+\sum\limits_{r=1}^\infty \Big[\hat{t}(z)+1\Big]^r=-{1 \over \hat{t}(z)}
\end{equation}
for $|z|$ small enough.
Recalling the definition \eqref{eq:check} establishes \eqref{zetaT}:
\[\cN_1[K+1]=
-\widecheck{\left(\frac{1}{\widehat ~ \!\! t}\right)}(K).\]

\bigskip
Since $T_j\ge 0$ for any $j\ge 1$, the function $\hat{t}(z)=-1+2z+\sum\limits_{j=1}^\infty z^j T_j$
has a single real root $w_0$ in the interval $(0,1/2]$. 
Our goal is to show that the Horton exponent $R$ is reciprocal to $w_0$. 
We begin by showing that $w_0$ is the root of $\hat{t}(z)$ closest to the origin.

\noindent
\begin{lem}\label{minmod}
Let $w_0$ be the only real root of  $\hat{t}(z)=-1+2z+\sum\limits_{j=1}^\infty z^j T_j$ in the interval $\left(0,1/2\right]$. Then, for any other root $w$ of  $~\hat{t}(z)$, we have $|w|>w_0.$
\end{lem}
\begin{proof}
Since $\{T_j\}$ are all nonnegative reals, we have $\overline{\hat{t}(\bar{z})}=\hat{t}(z)$. 
The radius of convergence of $\sum\limits_{j=1}^\infty z^j T_j$ must be greater than $w_0$. Suppose $w=re^{i\theta}$ \mbox{($0 \leq \theta <2\pi$)} is a root of magnitude at most $w_0$. 
That is $~\hat{t}(w)=0~$ and $r:=|w| \leq w_0.$ 
Then $~\hat{t}(\bar{w})=0~$ and
$$0 = {1 \over 2}\Big[\hat{t}(w)+\hat{t}(\bar{w})\Big]=-1+2r\cos(\theta)+\sum\limits_{j=1}^\infty r^j T_j \cos(j\theta).$$
If $r<w_0$, then
$$0 = -1+2r\cos(\theta)+\sum\limits_{j=1}^\infty r^j T_j \cos(j\theta)\leq -1+2r+\sum\limits_{j=1}^\infty r^j T_j  < -1+2w_0+\sum\limits_{j=1}^\infty w_0^j T_j =0$$
arriving to a contradiction. Thus $r=w_0$.

\medskip
\noindent
Next we show that $\theta=0$. Suppose not. Then
$$0 = -1+2r\cos(\theta)+\sum\limits_{j=1}^\infty r^j T_j \cos(j\theta) < -1+2r+\sum\limits_{j=1}^\infty r^j T_j  = -1+2w_0+\sum\limits_{j=1}^\infty w_0^j T_j =0$$
arriving to another contradiction.
Hence $r=w_0$, $\theta=0$, and $w=w_0$.
\end{proof}

\bigskip
\noindent
Let $L=\limsup\limits_{j \rightarrow \infty} T_j^{1/j}$.
Then $L^{-1}$ is the radius of convergence of $\hat{t}(z)$ (we set $L^{-1}=\infty$ if $L=0$), 
and $L^{-1}  >w_0.$ 
Lemma \ref{minmod} asserts that there exists a positive real $\gamma \in (w_0,L^{-1})$ such that  
\begin{equation}\label{eq:gamma}
\gamma<w ~\text{ for all }~w \not= w_0~\text{ such that }~\hat{t}(w)=0.
\end{equation}
Accordingly, for $0<\rho <w_0$
\be\label{eq:error}
\cN_1[K]={-1 \over 2\pi i}\oint\limits_{|z|=\rho}{dz \over \hat{t}(z)z^K}
=-Res\left({1 \over \hat{t}(z)z^K}; w_0 \right)-{1 \over 2\pi i}\oint\limits_{|z|=\gamma}{dz \over \hat{t}(z)z^K}.
\ee
Observe that $Res\left({1 \over \hat{t}(z)z^K}; w_0 \right)$ is a constant multiple of $w_0^{-K}$ since $w_0$ is a root of $\hat t(z)$ of algebraic multiplicity one. 
Thus, since $w_0 <\gamma$ and
\be\label{eq:error_est}
\left|{1 \over 2\pi i}\oint\limits_{|z|=\gamma}{dz \over \hat{t}(z)z^K}\right| \leq {1 \over \gamma^K \min\limits_{|z|=\gamma}|\hat{t}(z)|}=o\left(w_0^{-K}\right),\quad K\to\infty,
\ee
we have
$${\cN_1[K+1] \over \cN_1[K]}=\left|{\cN_1[K+1] \over \cN_1[K]}\right| \rightarrow {1 \over w_0} \quad \text{ as } K \rightarrow \infty.$$
Proposition \ref{propmain_Horton} now implies the following lemma.
\begin{lem} \label{maincor1}
Suppose $~\limsup\limits_{j \rightarrow \infty} T_j^{1/j} <\infty$. 
Then, for each positive integer $k$ 
$$\lim\limits_{K \rightarrow \infty}{\cN_k[K]\over\cN_{1}[K]}=w_0^{k-1}.$$
Moreover,
\[\lim_{K\to\infty}\left(\cN_1[K]\,w_0^{K}\right)= const. >0.\]
\end{lem}

\medskip
\noindent
To establish the converse we need the following statement.
\begin{prop}\label{prop:NT}
Suppose $\mu$ is a mean Horton self-similar measure on $\cBT$
with Tokunaga sequence $\{T_j\}_{j\ge 1}$. Then
\[\cN_1[K] \geq T_j^{(K-1)/j}\]
for all $j \in \mathbb{N}$ and $(K-1) \in j\mathbb{N}.$
\end{prop}
\begin{proof}
Fix any $j\ge 1$. 
The main counting equations \eqref{count1} show that for any integer $m\ge 0$ 
\[\cN_{mj+1}[K]\ge T_j\cN_{(m+1)j+1}[K].\]
Accordingly,
\[\cN_1[K]\ge T_j^{m}\cN_{mj+1}[K],\]
given $mj+1\le K$.
Choosing $m = (K-1)/j$ we obtain
\[\cN_1[K]\ge T_j^{(K-1)/j}\cN_K[K] = T_j^{(K-1)/j}.\]
\end{proof}

\noindent Suppose the limit
$$R=\lim\limits_{K \rightarrow \infty} {\cN_1[K+1]\over \cN_1[K]}$$
exists and is finite. 
Proposition~\ref{prop:NT} asserts that
$\cN_1[K]^{1/(K-1)} \geq T_j^{1/j}~$ for all $j \in \mathbb{N}$ and $(K-1) \in j\mathbb{N}$.
Hence,
$$\limsup\limits_{j \rightarrow \infty} T_j^{1/j} \leq \lim\limits_{K \rightarrow \infty} \cN_1[K]^{1/(K-1)} = R < \infty.$$
We summarize this in a lemma. 

\begin{lem} \label{maincor2}
Suppose $~\limsup\limits_{j \rightarrow \infty} T_j^{1/j} =\infty$. 
Then, the limit 
$~\lim\limits_{K \rightarrow \infty}{\cN_k[K]\over \cN_1[K]}$ 
does not exist at least for some $k$.
\end{lem}

\medskip
\noindent
Finally, Thm.~\ref{thm:HLSST} follows from Lem.~\ref{maincor1} and Lem.~\ref{maincor2}.

\subsection{Well-defined asymptotic Horton ratios}
\label{sec:convergence}

The setting for Horton law in \eqref{eq:Horton_rv} and \eqref{eq:Horton_mean} can be generalized beyond randomizing the tree measure with respect to Horton-Strahler orders as in \eqref{muk}. For instance, as it will be the case with the combinatorial critical binary Galton-Watson trees $\mathcal{GW}\left({1 \over 2},{1 \over 2}\right)$ in \eqref{eqn:ratio4j1}, the tree measure may be randomized with respect to the number of leaves in a tree.
A general set up for the Horton laws is described below.

\medskip
\noindent
Let $\cQ_n$, $n\in\mathbb{N}$, be a sequence of probability measures on $\cT$. 
We write $N_j^{( \cQ_n)}$ for the number of branches of Horton-Strahler order $j\ge 1$ in a tree generated
according to $\cQ_n$. 

\begin{Def}[{\bf Well-defined asymptotic Horton ratios}]\label{def:HortonWellDef} \index{asymptotic Horton ratio}  \index{asymptotic Horton ratio!well-defined}
We say that a sequence of probability measures $\{ \cQ_n\}_{n \in \mathbb{N}}$ has {\it well-defined asymptotic Horton ratios} if for each $j\ge 1$
\be\label{eqn:HortonWellDef}
{N_j^{( \cQ_n)} \over N_1^{( \cQ_n)}} \,\stackrel{p}{\to} \,\cN_j \quad \text{ as }~n \rightarrow \infty,
\ee
where $\cN_j$ is a constant, called the {\it asymptotic Horton ratio} of the branches of order $j$. 
\end{Def} 

\noindent
Sometimes it is possible to establish a stronger limit than in \eqref{eqn:HortonWellDef}. One such example is the almost sure convergence in equation \eqref{eqn:TokMartingale6} of Sect. \ref{sec:martingale}.

\medskip
\noindent
For a sequence of well-defined asymptotic Horton ratios $\cN_j$, 
the {\it Horton law} states that $\cN_j$ decreases in a geometric fashion as $j$ goes to infinity. 
We consider three particular forms of geometric decay.

\begin{Def}[{\bf Root, ratio, and strong Horton laws}]\label{def:HortonList}
Consider a sequence $\{ \cQ_n\}_{n \in \mathbb{N}}$ of probability measures on $\cT$ with well-defined asymptotic Horton ratios (Def. \ref{def:HortonWellDef}).
Then, the sequence $\{ \cQ_n\}$ is said to obey 
\begin{itemize}
  \item a {\it root-Horton law} if the following limit exists:
$\lim\limits_{j \rightarrow \infty} \Big( \cN_j \Big)^{-{1 \over j}}=R$;
  \item a {\it ratio-Horton law} if the following limit exists:
$\lim\limits_{j \rightarrow \infty} {\cN_j \over \cN_{j+1}}=R$;
  \item  a {\it strong Horton law} if the following limit exists:
$\lim\limits_{j \rightarrow \infty} \big(\cN_j R^j\big)=const$.
\end{itemize}
The constant $R$ is called the {\it Horton exponent}.
In each case, we require the Horton exponent $R$ to be finite and positive.
\end{Def}
\index{Horton law!root-Horton law}
\index{Horton law!ratio-Horton law}
\index{Horton law!strong Horton law}
\noindent Observe that the Horton laws in Def. \ref{def:HortonList} above are listed in the order from weaker to stronger.

\subsection{Entropy and information theory}\label{sec:entropy}
The information theoretical aspects of self-similar trees were not addressed until very recently.
This section reviews recent results by Chunikhina \cite{EVC18,EVCthesis}, 
where the entropy rate is computed for trees that satisfy the strong Horton law for branch numbers (Def.~\ref{def:Horton_rv}) and for Tokunaga self-similar trees (Def.~\ref{ss2})
as a function of the respective parameters, $R$ and $(a,c)$.

\medskip
\noindent
Consider a subspace $\cT_{N_1,\hdots,N_K}$ of $\BT^|$ of trees of a given order ${\sf ord}(T)=K$ 
and given admissible ($N_K=1$, $N_j \geq 2N_{j+1}$) branch counts  $N_1,N_2,\hdots,N_K$:
$$\cT_{N_1,\hdots,N_K}=\big\{T \in \BT^| :\,{\sf ord}(T)=K,~N_1[T]=N_1,\hdots,N_K[T]=N_K=1\big\}.$$
In \cite{EVC18}, Chunikhina finds the size of $\cT_{N_1,\hdots,N_K}$, providing 
an alternative form of expression that was first derived by Shreve \cite{Shreve66}.

\begin{lem}[{{\bf Branch counting lemma, \cite{EVC18}}}]\label{lem:ChunikhinaShreve}
$$\big|\cT_{N_1,\hdots,N_K}\big|=2^{N_1-1-\sum_{j=2}^K N_j}\prod\limits_{j=1}^{K-1}\binom{N_j-2}{2N_{j+1}-2}.$$
\end{lem}

\medskip
\noindent
Subsequently, Lem.~\ref{lem:ChunikhinaShreve} is used to find the entropy rate for trees that satisfy the strong Horton law (Def.~\ref{def:Horton_rv}) with exponent $R>2$.
\begin{thm}[{\bf Entropy rate for Horton self-similar trees, \cite{EVC18}}]\label{thm:Chunikhina1}
For a given $R>2$, let $T$ be a random tree, uniformly sampled from the space
$$\cT_{R,K}\!=\!\big\{T \in \BT^| \!:{\sf ord}(T)\!=\!K,~
\big|N_j[T]-\!R^{K-j}\big| < (R-\epsilon)^{K-j}~~\forall j \leq K\big\},$$
where $\epsilon \in (0,R)$ is a given small quantity. Then, the entropy rate
\be\label{eqn:Chunikhina1}
\cH_{\infty}(R):=\lim\limits_{K \rightarrow \infty}{-{\sf E}[\log_2{\sf P}(T)] \over 2R^{K-1}}=1-{1-H(2/R) \over 2-2/R},
\ee
where
\be\label{eqn:BEF}
H(z)=-z\log_2{z}-(1-z)\log_2(1-z),\quad 0< z< 1
\ee
is the binary entropy function illustrated in Fig.~\ref{fig:R1}(a).
The entropy rate $\cH_{\infty}(R)$ is illustrated in Fig.~\ref{fig:R1}(b).
\end{thm}

\begin{figure}[t] 
\centering
\includegraphics[width=0.45\textwidth]{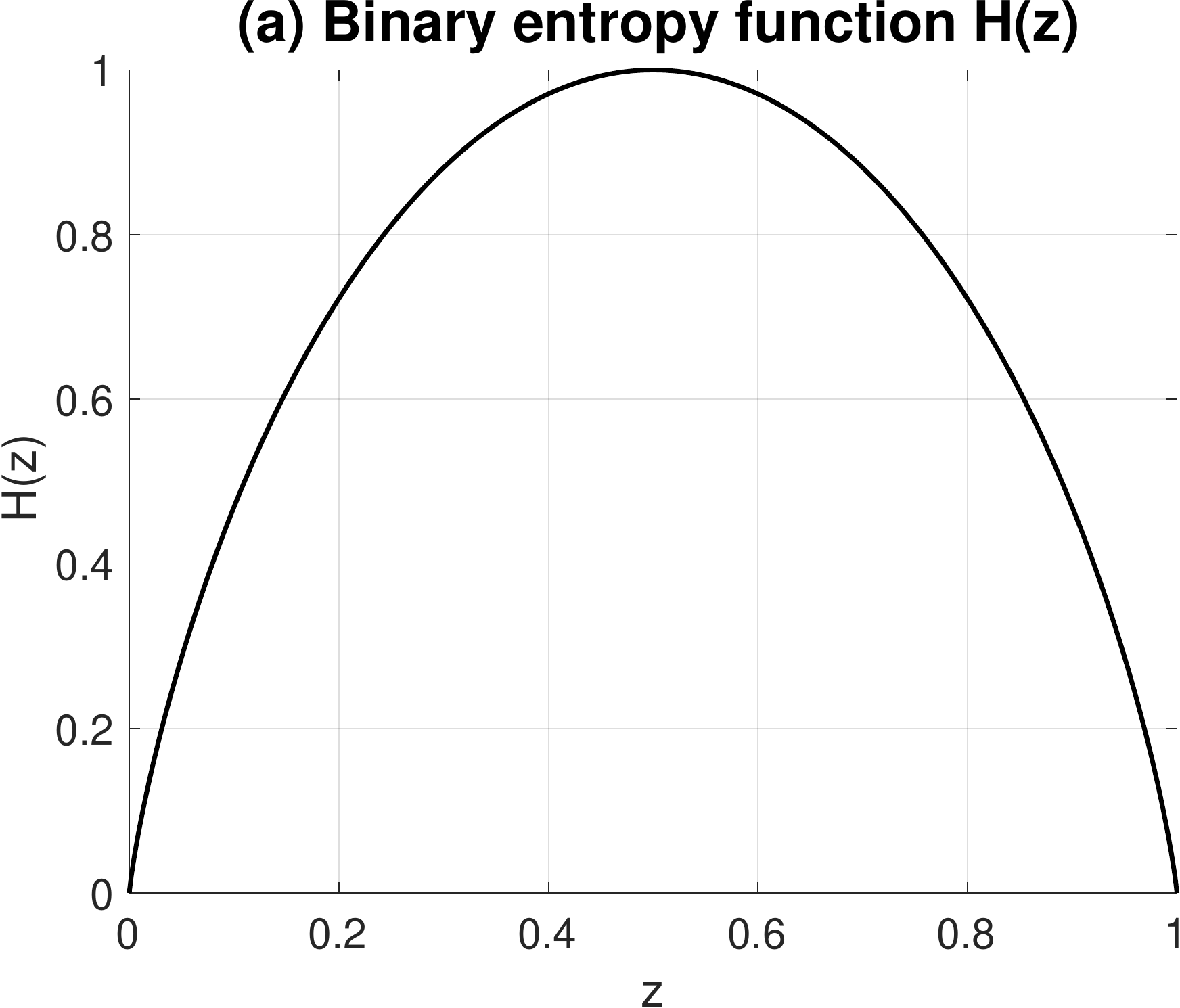}
\includegraphics[width=0.45\textwidth]{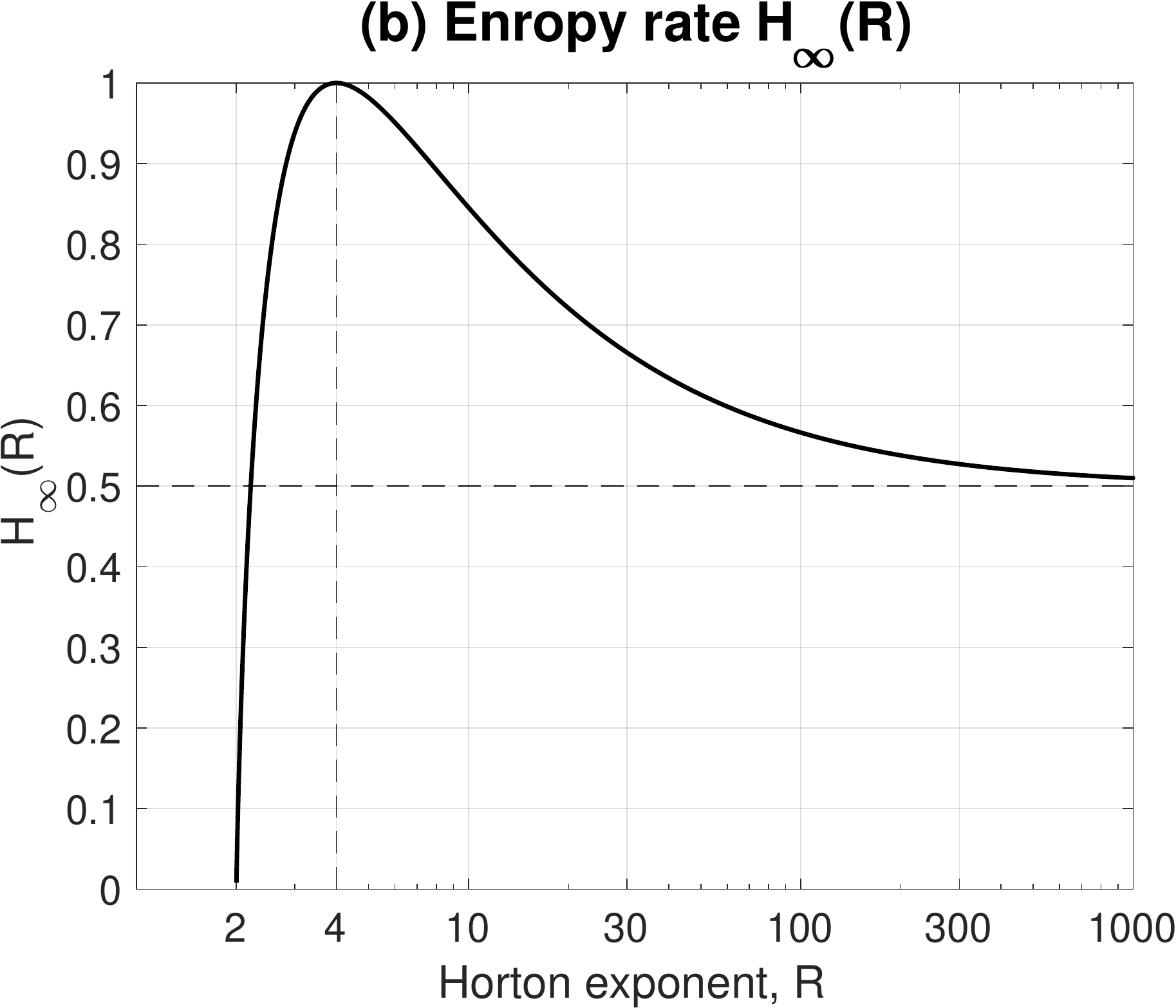}
\caption[Entropy rate in trees with strong Horton law]
{Entropy rate in trees that satisfy the strong Horton law with exponent $R$.
(a) Binary entropy function $H(z)$.
(b) Entropy rate $\cH_{\infty}(R)$.}
\label{fig:R1}
\end{figure}

\medskip
\noindent
Notice that the trees in $\cT_{R,K}$ satisfy the strong Horton law (Def.~\ref{def:Horton_rv}) 
with the Horton exponent $R$, and
$2R^{K-1}$ is the asymptotic number of nodes  in a tree $T$ from $\cT_{R,K}$. 
\begin{Rem}\label{rem:binary_treeGW}
It is an easily verified fact that a random tree $T$ selected uniformly from the subspace 
\be\label{eqn:spaceBTIN}
\BT^|(N):=\{T \in \BT^| \,:\, \#T=2N-1\}
\ee
of $\BT^|$ containing only the trees with $N$ leaves ($2N$ nodes and $2N-1$ edges) 
is distributed as a random tree sampled from the critical plane Galton-Watson distribution $\mathcal{GW}_{\rm plane}\left({1 \over 2},{1 \over 2}\right)$, conditioned on $\#T=2N-1$,
i.e., 
\be\label{eqn:binary_treeGWplane}
{\sf Unif}\big(\BT^|(N)\big) ~\stackrel{d}{=}~\left(\mathcal{GW}_{\rm plane}\left({1 \over 2},{1 \over 2}\right) \Big| \#T=2N-1\right).
\ee
Consequently, we have that
\be\label{eqn:binary_treeGW}
T \stackrel{d}{\sim} {\sf Unif}\big(\BT^|(N)\big) ~\Rightarrow~\textsc{shape}(T) \stackrel{d}{\sim}\left(\mathcal{GW}\left({1 \over 2},{1 \over 2}\right) \Big| \#T=2N-1\right).
\ee
\end{Rem}

\medskip
\noindent
The number  $\big|\BT^|(N)\big|$ of different combinatorial shapes of rooted planted plane binary trees with $N$ leaves and $2N-1$ edges, is given by $C_{N-1}$,
where $C_n$ denotes the {\it Catalan number} defined as
\be\label{eqn:catalan}
C_n={1 \over n+1}\binom{2n}{n}.
\ee 
\index{Catalan number}

\medskip
\noindent
Using $\big|\BT^|(N)\big|=C_{N-1}$ and Stirling's formula, 
it is observed in \cite{EVC18} that the entropy rate for a tree $T'$, selected uniformly from $\BT^|(N)$ is
\be\label{eqn:ChunikhinaA}
\cH_{\infty}^{\rm GW}:=\lim\limits_{N \rightarrow \infty}{-{\sf E}[\log_2{\sf P}(T')] \over 2N}=1.
\ee
Thus, scaling by the asymptotic number of nodes $2R^{K-1}$ in Thm.~\ref{thm:Chunikhina1} implies 
$$\cH_{\infty}(R) \leq \cH_{\infty}^{\rm GW}=1.$$
Indeed, by definition of the corresponding spaces,  
$$\cT_{R,K} \subseteq \bigcup_N \BT^|(N),$$ 
where the union is taken over $N$ ranging from 
$$\lceil R^{K-1}-(R-\epsilon)^{K-1}\rceil ~~\text{ to }~~\lfloor R^{K-1}+(R-\epsilon)^{K-1}\rfloor ,$$
and therefore
$$\Big|\cT_{R,K}\Big| \leq 2(R-\epsilon)^{K-1}\,\Big|\BT^|\big(2R^{K-1}+2(R-\epsilon)^{K-1}\big)\Big|.$$
Hence, for the following limits known to converge, we have
\begin{align*}
\cH_{\infty}(R) &=\lim\limits_{K \rightarrow \infty}{\log_2 \Big|\cT_{R,K}\Big| \over 2R^{K-1}} \\
&\leq \cH_{\infty}^{\rm GW}
=\lim\limits_{K \rightarrow \infty}{\log_2 \left(2(R-\epsilon)^{K-1}\,\Big|\BT^|\big(2R^{K-1}+2(R-\epsilon)^{K-1}\big)\Big|\right) \over 2R^{K-1}}.
\end{align*}
\medskip
\noindent
Moreover, scaling by the asymptotic number of nodes $2R^{K-1}$ in 
Thm.~\ref{thm:Chunikhina1} enables representing $\cH_{\infty}(R)$
as the limit ratio of the entropy for Horton self-similar trees with 
parameter $R$ to the entropy for uniformly selected binary trees. Specifically, let $T$ be a random tree sampled uniformly from the space $\cT_{R,K}$ and let $T'$ be a random tree sampled uniformly from the space $\BT^|(N)$ with $N=R^{K-1}$. Then, equations \eqref{eqn:Chunikhina1} and \eqref{eqn:ChunikhinaA} imply 
that $\cH_{\infty}(R)$ is the the limit ratio of entropies as the space sizes grow with $K \rightarrow \infty$:
\begin{equation}\label{eqn:ChunikhinaB}
\cH_{\infty}(R)=\lim\limits_{K \rightarrow \infty}{-{\sf E}[\log_2{\sf P}(T)] \over -{\sf E}[\log_2{\sf P}(T')] }=1-{1-H(2/R) \over 2-2/R}.
\end{equation}

\medskip
\noindent
As an important consequence of Thm.~\ref{thm:Chunikhina1}, a special place of the parameter $R=4$ is established amongst all Horton exponents $R \in [2,\infty)$ as
$${\rm argmax}_R \cH_{\infty}(R)=4 \quad \text{ and }\quad \max_R \cH_{\infty}(R)=\cH_{\infty}(4)=1.$$
Not surprisingly, $R=4$ is the parameter value for the strong Horton law results we will encounter in Sect.~\ref{cgw}, 
primarily in the context of the critical binary Galton-Watson tree $\mathcal{GW}\left({1 \over 2},{1 \over 2}\right)$.
Indeed, as stated in Rem.~\ref{rem:binary_treeGW}, the tree $T''=\textsc{shape}(T') \in \cBT^|$ in \eqref{eqn:ChunikhinaB} is a random tree sampled from the Galton-Watson distribution $\mathcal{GW}\left({1 \over 2},{1 \over 2}\right)$ conditioned on $\#T''=2N-1$.

\medskip
\noindent
In \cite{EVCthesis}, Chunikhina extended the results in \cite{EVC18} by counting the number of trees with the given merger numbers $N_{i,j}$ (see Sect. \ref{sec:mss}), 
and finding the entropy rates for the Tokunaga self-similar trees (Def.~\ref{ss2}) represented as a function of the parameters $(a,c)$. 
For a given integer $K>1$, consider a finite sequence of admissible branch counts  
$\{N_i\}_{i=1,\hdots,K}$, and a finite sequence of admissible branch 
numbers $\{N_{i,j}\}_{1 \leq i <j\leq K}$. 
Admissibility means that for all $i\leq K-1$,
$$N_i=2N_{i+1} +\sum\limits_{j=i+1}^K N_{i,j}$$
as all $N_i$ branches of Horton-Strahler order $i$ have to merge into a higher order branch
(either two branches of order $i$ merge and originate a branch of order $i+1$, or a branch of order $i$ merges into a branch of order $j>i$).
Consider the subspace 
$$\cT_{K,\{N_i\},\{N_{i,j}\}}=\big\{T \in \cT_{N_1,\hdots,N_K} :\,N_{1,2}[T]=N_{1,2},\hdots,N_{K-1,K}[T]=N_{K-1,K}\big\}.$$
\begin{lem}[{{\bf Side branch counting lemma, \cite{EVCthesis}}}] \label{thm:Chunikhina2}
$$\Big|\cT_{K,\{N_i\},\{N_{i,j}\}}\Big|=\prod\limits_{j=2}^K \prod\limits_{i=1}^{j-1} 2^{N_{i,j}}\binom{N_j-1+\sum\limits_{k=i}^{j-1}N_{k,j}}{N_{i,j}}.$$
\end{lem}

\medskip
\noindent
Lemma \ref{thm:Chunikhina2} is used to obtain the following asymptotic results.
Consider Tokunaga self-similar tree with parameters $(a,c)$. 
Such a tree satisfies the strong Horton law for mean branch numbers (Def.~\ref{def:Horton_mean}) 
with the Horton exponent \eqref{HortonR}
\[R=R(a,c)= {a+c+2 +\sqrt{(a+c+2)^2-8c} \over 2}.\] 
Next, similarly to $\cT_{R,K}$, one can define the space $\cT_{a,c,K}$ of asymptotically 
Tokunaga self-similar trees of order $K$. 
Informally, this space includes the trees in $\BT^|$ such that 
\[{\sf ord}(T)=K,\quad N_j[T]\sim R^{K-j},\quad \text{ and }\quad \frac{N_{i,j}(T)}{N_j(T)}\sim ac^{j-i-1},\]
where $R=R(a,c)$, and the asymptotic equality $\sim$ is taken as $K\to\infty$.
\begin{thm}[{\bf Entropy rate for Tokunaga self-similar trees, \cite{EVCthesis}}]\label{thm:Chunikhina3}
For given $a,c>0$, let $T$ be a random tree, uniformly sampled from the space $\cT_{a,c,K}$. Then, the entropy rate
\begin{align}\label{eqn:Chunikhina2}
\cH_{\infty}(a,c):=&\lim\limits_{K \rightarrow \infty}{-{\sf E}[\log_2{\sf P}(T)] \over 2R^{K-1}}  \nonumber \\
=& {a \over 2} \sum\limits_{j=1}^\infty R^{-j}\left({1-c^j \over 1-c}+{1 \over a} \right) \, \log_2\left({1-c^j \over 1-c}+{1 \over a} \right) \nonumber \\
&+{aR \over 2(R-c)(R-1)}+{\log_2{a} \over 2(R-1)}-{aRc\log_2{c} \over 2(R-c)^2(R-1)}.
\end{align}
\end{thm}
\noindent Figure~\ref{fig:H}(a) illustrates the entropy rate $\cH_{\infty}(a,c)$.

\begin{figure}[t] 
\centering
\includegraphics[width=\textwidth]{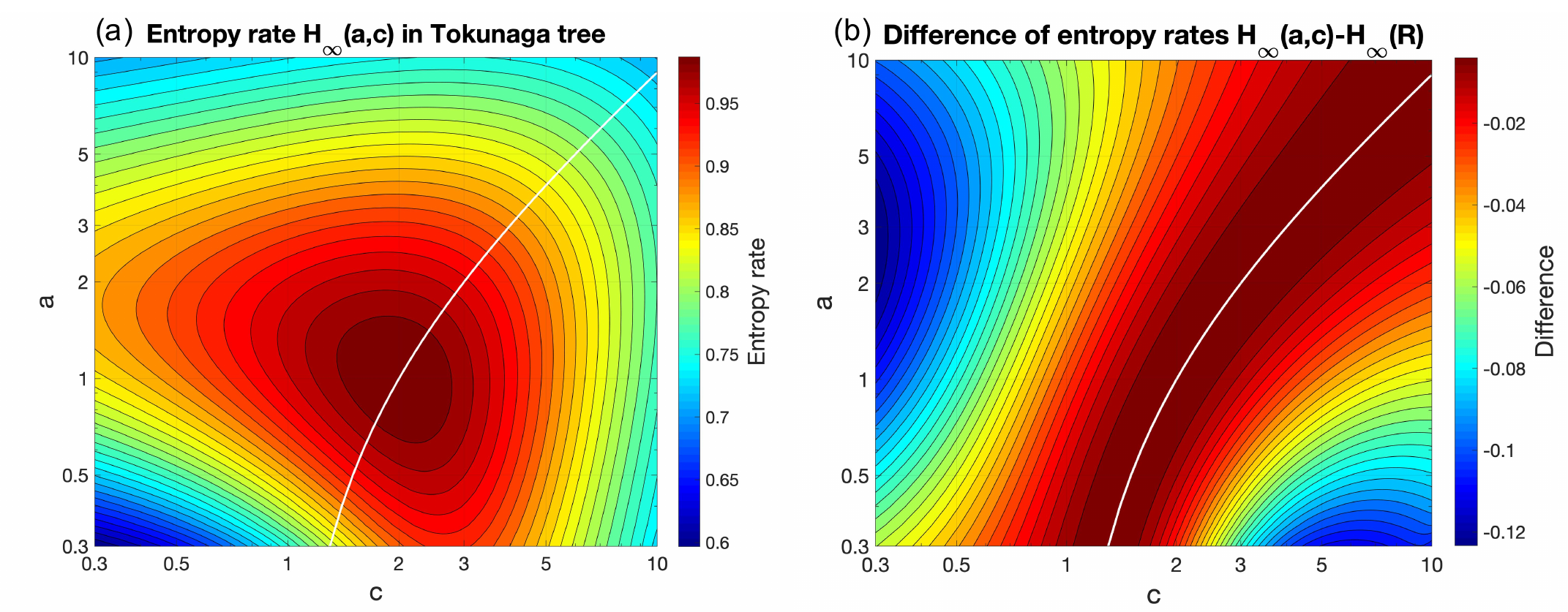}
\caption[Entropy rate in Tokunaga trees.]
{Entropy rate in Tokunaga trees. 
(a) Entropy rate $\cH_{\infty}(a,c)$ in a Tokunaga self-similar tree with parameters $(a,c)$.
(b) The difference $\cH_{\infty}(a,c)-\cH_{\infty}(R)$ of entropy rates in a Tokunaga 
tree with parameters $(a,c)$ and in a tree satisfying Horton law with Horton exponent $R(a,c)$. 
A double-logarithmic scale is used to emphasize behavior of the plots at the
boundary values. 
White line corresponds to $a=c-1$.}
\label{fig:H}
\end{figure}

\medskip

If $a=c-1$, then $R=2c$ by \eqref{HortonR}, and the equation \eqref{eqn:Chunikhina2} simplifies, leading to the following corollary.
\begin{cor}[{{\bf \cite{EVCthesis}}}]\label{cor:Chunikhina} 
Let $T$ be a random tree, uniformly sampled from the space $\cT_{a,c,K}$ with $c>1$ and $a=c-1$. Then $T$ satisfies the strong Horton law \eqref{eq:Horton_mean} with $R=2c$,
and the entropy rate is given by
\be\label{eqn:Chunikhina3}
\cH_{\infty}(c-1,c)=1-{1-H(1/c) \over 2-1/c}=\cH_{\infty}(R),
\ee
where $H(z)$ is the binary entropy function \eqref{eqn:BEF} and $\cH_{\infty}(R)$ 
is defined by \eqref{eqn:Chunikhina1}.
\end{cor}
\noindent 
Figure~\ref{fig:H}(b) illustrates this result, by showing how the difference of entropy rates 
$\cH_{\infty}(a,c)-\cH_{\infty}(R)$ decreases away from the line $a=c-1$.
The special place for the line $a=c-1$ within the parameter space of 
the Tokunaga self-similar random trees was observed earlier in \cite{VG00,KZ18a,KZ18}. 
See Remark \ref{rem:a_cminus1}.
The constraint $a=c-1$ will reappear in many instances in Sect. \ref{HBP} of the present work.

\medskip
\noindent
Finally, the maximum value $\max \cH_{\infty}(a,c)=1$ is achieved at the special point $(a,c)=(1,2)$ of the special line $a=c-1$.
Once again, this is not surprising as $(a,c)=(1,2)$ is the parameter value for the Tokunaga self-similarity results 
of Sect.~\ref{cgw}, presented in the context of the critical binary Galton-Watson trees $\mathcal{GW}\left({1 \over 2},{1 \over 2}\right)$ and related processes.
We recall that the combinatorial shape $T''=\textsc{shape}(T') \in \cBT^|$ of the random binary tree $T'$ in \eqref{eqn:ChunikhinaA} is distributed according to $\mathcal{GW}\left({1 \over 2},{1 \over 2}\right)$ conditioned on $\#T''=2N-1$.

\subsection{Applications}
\label{sec:appl}
A quantitative understanding of the branching 
patterns is instrumental in hydrology \cite{RIR01,TBR88,Mar96,BMR99,BM10,Kir00},
geomorphology \cite{DR00,HSP15},
statistical seismology \cite{BP04,THR07,HTR08,ZGKW08,G13,ZBZ13,Y13},
statistical physics of fracture \cite{RTS+03},  
vascular analysis \cite{Kassab00},
brain studies \cite{Cetal06},
ecology \cite{CLF07}, biology \cite{TPN98}, 
and beyond, encouraging a rigorous treatment.
Introduced in hydrology to describe the dendritic 
structure of river networks, which is among the most evident examples 
of natural branching, Horton-Strahler \cite{Horton45, Strahler} and Tokunaga \cite{Tok78} 
indexing schemes have been rediscovered and used in other fields.
Subsequently, the Horton law (Def.~\ref{def:Horton_rv}) and Tokunaga self-similarities (Def.~\ref{ss2})
have been empirically 
or rigorously established in numerous observed and modeled systems \cite{NTG97}.
This includes hydrology (see Sect.~\ref{sec:hydrology}),
vein structure of botanical leaves \cite{NTG97,TPN98},
diffusion limited aggregation \cite{O92,MT93,YM94dla},\index{diffusion limited aggregation}
two dimensional site percolation \cite{TMMN99,YNTG05,ZWG05,ZWG06},\index{percolation}
a hierarchical coagulation model of Gabrielov et al. \cite{GNT99}\index{hierarchical coagulation model} 
introduced in the framework of self-organized criticality, 
and a random self-similar network model of Veitzer and Gupta \cite{VG00}\index{random self-similar network}
developed as an alternative to the Shreve's random topology model for river networks.
The Horton exponent commonly reported in empirical studies 
is within the range $3 < R < 6$.
Curiously, it has been observed in \cite{KZ18a} that the critical Tokunaga model 
(Sect.~\ref{sec:Tok}) with this range of Horton exponents generates
trees with fractal dimension in the range $\approx (1.6,3)$, which
includes all the trees that may exist in a three-dimensional world, excluding
the range $<1.6$ that corresponds to almost ``linear'', and 
probably less studied, trees.

\subsubsection{Hydrology}
\label{sec:hydrology}

An illuminating natural example of Horton laws and Tokunaga self-similarity is given
by the combinatorial structure of river networks (Figs.~\ref{fig:Beaver},\ref{fig:Beaver_a}).
The hydrological {\it Horton law} was first described by Robert E. Horton \cite{Horton45}
who noticed that the empirical ratio $N_K/N_{K+1}$ in river streams is close to $4$.
This observation has been strongly corroborated in numerous observational studies
\cite{Kirchner93,Shreve66,Leopold,Pec95,Tarboton96,GW98,RIR01,Mesa18,Turcotte_book}.
See Barndorff-Nielsen \cite{B-N93} for a 1993 survey for probabilists.
\index{Horton law!hydrological Horton law}

Write $Z_K$ for the value of a selected statistic $Z$ averaged over basins/channels 
of order $K$.
This can be basin area, basin magnitude (number of leaves in the tree that describes the basin),
the lengths of the longest channel, the total channel lengths, etc.
The Horton law approximates the growth of $Z_K$ with order as a geometric sequence:
$Z_K ~\propto~ R_Z^K$ with $R_Z>1$. 
Informally, this suggests that the order $K$ of a channel (branch) or a subbasin (subtree) 
is proportional to $\ln(Z_K)$, where $Z_K$ can be
interpreted as the channel/basin ``size''. 
%
If statistic $Z$ satisfies the Horton law with exponent $R_Z$, and the branch
counts $N_K$ satisfy the Horton law \eqref{eq:Horton_emp} with Horton exponent $R$,
then
\[Z_K~\propto~N_K^{-\alpha},\quad \text{with}\quad \alpha = \frac{\ln{R_Z}}{\ln{R}}.\]
A similar power relation holds for any pair of statistics that satisfy the Horton law.
A well studied example is the Hack's law that relates the length $L$
of the longest stream to the basin area $A$ via $L~\propto~A^h$ with $h\approx 0.6$
\cite{Rigon+96}.

Furthermore, it has been shown that river networks 
are closely approximated by a two-parametric Tokunaga self-similar model (Def.~\ref{ss2}) with 
parameters that are independent of river's geographic location \cite{Tok78,Pec95,DR00,ZZF13}.
The Tokunaga model closely predicts values of the Horton exponents for multiple
basin statistics with only two parameters (see Fig.~\ref{fig:Beaver_a}).

Discovery of the Horton law prompted exploration of various branching models,
most popular of which is the critical binary Galton-Watson tree (Sect.~\ref{cgw}),
also known in hydrology as Shreve's random topology model \cite{Shreve66,Shreve69};
it is conditionally equivalent to the uniform distribution on planar binary trees 
with a fixed number of leaves \cite{Pitman}.
This model has the Horton exponent $R=4$ and Tokunaga parameters $(a,c)=(1,2)$;
see Thm.~\ref{thm:BWW00}.
For long time, the critical binary Galton-Watson tree has remained the only well-known 
probability model for which the Horton and Tokunaga self-similarity was 
rigorously established, and whose Horton-Strahler ordering has received 
attention in the literature 
\cite{Shreve66,Shreve69,Kemp79,B-N93,DK94,Pec95,OWZ97,WW91,YM94,BWW00}.
The model has been particularly popular in hydrology as an
approximation to the topology of the observed river networks \cite{TBR88}.
Scott Peckham \cite{Pec95} has first explicitly noticed, by performing a high-precision 
extraction of river channels for Kentucky River, Kentucky and Powder River, Wyoming, 
that the Horton exponents and Tokunaga parameters
for the observed rivers significantly deviate from that for the Galton-Watson model.
He reported values $R\approx 4.6$ and $(a,c)\approx (1.2, 2.5)$ and
emphasized the importance of studying a broad range of Horton exponents and 
Tokunaga parameters.
The general interest to fractals and self-similar structures in natural sciences
during the $1990$s resulted in a quest, mainly inspired and led by Donald Turcotte, for 
Tokunaga self-similar tree graphs of diverse origin.
As a result, the Horton and Tokunaga self-similarity, with a broad range
of respective parameters, have been empirically or rigorously established in numerous observed and modeled systems,
well beyond river networks.

\subsubsection{Computer science}
The Horton-Strahler orders are known in computer science as the {\it register function} or
{\it register number}. 
They first appeared in the $1958$ paper by Ershov \cite{Ershov1958} 
as the minimal number of memory registers required for 
evaluating a binary arithmetic expression. 
\index{register function}
\index{register number}

\medskip
\noindent
A study of Flajolet et al. \cite{FRV79} concerns calculating the average register function in a random plane planted binary tree with $n$ leaves.
That is, let the random tree $T$ be uniformly sampled from all $C_{n-1}$ trees in the subspace $\BT^{|}(n)$ of $\BT^{|}$  defined in \eqref{eqn:spaceBTIN}, where $C_n$ is the Catalan number \eqref{eqn:catalan}. Following Rem.~\ref{rem:binary_treeGW}, we know that the combinatorial shape $\textsc{shape}(T) \in \cBT^|$ of such binary tree $T$ can also be obtained by sampling from the Galton-Watson distribution $\mathcal{GW}\left({1 \over 2},{1 \over 2}\right)$ conditioned on $\#T=2n-1$.
The work \cite{FRV79} finds the average register function (Horton-Strahler order) in a random binary tree $T \stackrel{d}{\sim} {\sf Unif}\big(\BT^|(n)\big)$, 
$${\sf E}\big[{\sf ord}(T)\big]=1+{1 \over C_n}\sum\limits_{j=1}^{n-1} v_2(j)\left[\binom{2n}{n+j+1}-2\binom{2n}{n+j}+\binom{2n}{n+j-1}\right],$$
where $v_2(n)$ is known as the {\it dyadic valuation} of $n$. 
Specifically, the  {\it dyadic valuation} of $n$ is the cardinality of the inverse image of 
$$f(p,k)=k2^p ~:~\mathbb{Z}_+ \times \mathbb{N} \rightarrow \mathbb{N},$$
i.e., $v_2(n)=\big|\{(p,k) \in \mathbb{Z}_+ \times \mathbb{N} ~:~ k2^p=n\}\big|$.

\medskip
\noindent
In addition, Flajolet et al. \cite{FRV79} proved that for $T \stackrel{d}{\sim} {\sf Unif}\big(\BT^|(n)\big)$,
\be\label{eqn:Flajolet}
{\sf E}\big[{\sf ord}(T)\big]=\log_4{n}+D\big(\log_4{n} \big)+o(1),\quad\text{as}\quad  n\to\infty,
\ee
where $D(\cdot )$ is a particular continuous periodic function of period one, explicitly derived in \cite{FRV79}.
We illustrate Eq.~\eqref{eqn:Flajolet} below in Fig.~\ref{fig:D}(a), which closely reproduces 
Fig.~6 from the original paper by Flajolet et al. \cite{FRV79}.
Equation \eqref{eqn:Flajolet} is related to the tree size asymptotic \eqref{eq:Nk}
of Thm.~\ref{thm:HLSST}, with the Horton exponent $R=4$.

\medskip
\noindent
For more on register functions see \cite{FP86,Prodinger87,Nebel2002,DP2006,HHP2018} and references therein.


\section{Critical binary Galton-Watson tree}
\label{cgw}
The critical binary Galton-Watson tree is pivotal for the theory of random trees 
and for diverse applications because of its transparent generation process 
and multiple symmetries.
This section summarizes some properties of this tree used in our 
further discussion.

\subsection{Combinatorial case}
\label{ecgw}
Here we discuss the combinatorial binary Galton-Watson trees.

\subsubsection{Horton and Tokunaga self-similarities}
\label{sec:BWW}
Burd, Waymire, and Winn \cite{BWW00} have first recognized 
a special position held by the critical binary tree with respect to the Horton pruning
in the space of Galton-Watson distributions $\mathcal{GW}(\{q_k\})$ on $\cT^|$.
We now state the main result of \cite{BWW00} using the language of the
present work.

\begin{thm}[{{\bf Horton self-similarity of Galton-Watson trees, \cite{BWW00}}}]
\label{thm:BWW00}
Consider a collection of Galton-Watson measures $\mathcal{GW}(\{q_k\})$ on $\cT^|$.
The following statements are equivalent:
\begin{itemize}
\item[(a)] A distribution is Horton self-similar (Def.~\ref{def:ss});
\item[(b)] A distribution is mean Horton self-similar (Def.~\ref{ss1},\ref{def:ss2});
\item[(c)] A distribution is Tokunaga self-similar (Def.~\ref{ss2});
\item[(d)] A distribution is critical binary: $q_0=q_2=1/2$.
\end{itemize}
Furthermore, the critical binary distribution has Tokunaga sequence
$T_j=2^{j-1}$, $j\ge 1$, which corresponds to Tokunaga self-similarity with 
$(a,c)=(1,2)$ and strong Horton law with exponent $R=4$. 
\end{thm}

\noindent The following statement provides a 
useful characterization of the critical binary Galton-Watson tree.

\begin{prop}[{{\bf \cite{BWW00}}}]
\label{prop:BWW}
Suppose $T\stackrel{d}{\sim}\mathcal{GW}(1/2,1/2)$.
Then, the tree order ${\sf ord}(T)$ has geometric distribution:
\[{\sf ord}(T)\stackrel{d}{\sim}{\sf Geom}_1(1/2).\]
Furthermore, let $b_j$ be a branch of order $j\ge 2$ in $T$ selected uniformly and 
randomly among all branches of order $j$ in $T$.
Then, the total number $m_j\ge0$ of side branches within the 
branch $b_j$ is geometrically distributed:
\[m_j\stackrel{d}{\sim}{\sf Geom}_0(2^{1-j}),\quad j\ge 2.\] 
In particular, 
\[{\sf E}(m_j) = 2^{j-1}-1 = T_1+\dots+T_{j-1},\quad j\ge 2,\]
where $T_i = 2^{i-1}$, $i\ge 1$, are the Tokunaga coefficients. 
Conditioned on $m_j$, each side branch within $b_j$ is assigned order $i$ 
independently of other side branches with probability
\[\frac{T_{j-i}}{T_1+\dots+T_{j-1}},\quad i=1,\dots,j-1.\]

\end{prop}

\noindent Notably, critical {\it non-binary}  Galton-Watson trees converge to the 
critical {\it binary} tree under
consecutive Horton pruning, as described in the following statement.
\begin{thm}[{{\bf Attraction property of critical binary Galton-Watson tree, \cite{BWW00}}}]
\label{BWW00_1}
Suppose a Galton-Watson measure $\mu\equiv\mathcal{GW}(\{q_k\})$ on $\cT^|$ 
satisfies the following conditions:
\begin{itemize}
\item The measure $\mu$ is critical, i.e. $q_1\ne1$ and $\sum_k kq_k=1$;
\item The measure $\mu$ has a.s. bounded offspring number, i.e.
there exists such $j_0\ge 2$ that $q_j=0$ for any $j\ge j_0$.
\end{itemize} 
Then, for any $\tau\in\cT^|$
\[\lim_{n\to\infty}\mu\big(\cR^n(T)=\tau|\cR^n(T)\ne\phi\big)=\mu^*(\tau),\]
where $\mu^*$ denotes the critical binary Galton-Watson measure on $\cT^|$:
\be\label{eqn:mustarGWonT}
\mu^* = \left\{
\begin{array}{ll}
\mathcal{GW}({1\over 2},{1\over 2})& \text{ on } \cBT^|, \\
0&\text { on } \cT^|\setminus\cBT^|.
\end{array}
\right. 
\ee
\end{thm}

The Markov structure of the Galton-Watson tree $T\stackrel{d}{\sim}\mathcal{GW}(\{q_k\})$ 
ensures the existence of the following additional properties:
\begin{itemize}
\item[(i)]  The forest of trees obtained by removing the edges and the vertices 
below combinatorial depth $d\ge 0$ has the same frequency structure as the original 
space $\mathcal{GW}(\{q_k\})$;
\item[(ii)] A subtree rooted in a uniform random vertex of $T$ has the same 
distribution as $T$; 
and
\item[(iii)] The forest of trees obtained by considering subtrees rooted
at every vertex of $T$ 
approximates the frequency structure of the entire space of trees when 
the order of $T$ increases.
\end{itemize}
\noindent
We define these properties more formally in Sect.~\ref{sec:combHBP}.
Combined with the Horton self-similarity of Thm.~\ref{thm:BWW00}, they further 
highlight very special symmetries of the critical binary Galton-Watson distribution 
$\mathcal{GW}({1\over 2},{1\over 2})$.
Stated loosely, this distribution is invariant with respect to various form of 
cutting, either from the leaves down or from the root up.
Moreover, this is the only distribution that enjoys all these invariances 
in the family of Galton-Watson distributions $\mathcal{GW}(\{q_k\})$.  
Analysis of real world data (e.g.~\cite{Pec95,NTG97}), however, reveals 
self-similar tree-like structures with Tokunaga parameters and 
Horton exponents different from those in the critical binary Galton-Watson model.
This motivates one to look for invariant tree models outside of 
the Galton-Watson family.  
In Sect.~\ref{sec:Tok}, we construct a one parameter family of trees, called
{\it critical Tokunaga trees}, that inherit all the invariant properties mentioned 
in this section and include the critical binary Galton-Watson tree as a special case.
In particular, it generates self-similar trees with Horton exponents $2\le R<\infty$.

\subsubsection{Dynamics of branching probabilities under Horton pruning}

The following result of Burd et al. \cite{BWW00} clarifies the Horton self-similarity 
of the critical binary Galton-Watson tree and absence of such in non-critical case.

\begin{thm}[{{\bf Dynamics of branching {\rm \cite[Proposition 2.1]{BWW00}}}}]
\label{thm:BWW2}
Consider a critical or subcritical combinatorial binary Galton-Watson probability 
measure $\mu_0=\mathcal{GW}(q_0,q_2)$ on $\cBT^{|}$, i.e. require $q_0+q_2=1$ and $q_2\le 1/2$. 
Construct a recursion by repeatedly applying Horton pruning operation $\cR$ as follows. Starting with $k=0$, and for each consecutive integer,
let 
$\nu_k=\cR_*(\mu_k)$ be the pushforward probability measure induced by the pruning operator, i.e.,
$$\nu_k(T)=\mu_k \circ \cR^{-1}(T) = \mu_k \big(\cR^{-1}(T)\big),$$
and set
$$\mu_{k+1}(T)=\nu_k\left(T~|T\ne\phi\right).$$
Then for each $k \geq 0$, distribution $\mu_k(T)$ is a binary Galton-Watson distribution $\mathcal{GW}(q_0^{(k)},q_2^{(k)})$ with
$q_0^{(k)}$ and $q_2^{(k)}$ constructed recursively as follows: start with $q_0^{(0)}=q_0$ and $q_2^{(0)}=q_2$, and let
\be\label{eq:iterateDFS}
q_2^{(k+1)}=\frac{\left[q_2^{(k)}\right]^2}{\left[q_0^{(k)}\right]^2+\left[q_2^{(k)}\right]^2},
\quad q_0^{(k+1)}=1-q_2^{(k+1)}.
\ee

Consequently, a combinatorial binary Galton-Watson probability distribution $\mathcal{GW}(q_0,q_2)$ is prune-invariant as in the Def. \ref{def:prune} if and only if it is critical, i.e.,
$$q_0=q_2=1/2.$$
\end{thm}

\subsubsection{The Central Limit Theorem and the strong Horton law for branch counts}
For a given $N \in \mathbb{N}$, consider $T \stackrel{d}{\sim} {\sf Unif}\big(\BT^|(N)\big)$. Following Remark \ref{rem:binary_treeGW}, we know that
$\textsc{shape}(T) \stackrel{d}{\sim}  \left(\mathcal{GW}\left({1 \over 2},{1 \over 2}\right) \Big| \#T=2N-1\right)$. The branch counts
$$N_j^{(N)}[T]:=\left(N_j[T];\,T\stackrel{d}{\sim}{\sf Unif}\big(\BT^|(N)\big)\right)$$ 
are integer valued random variables induced by $T$. They are the same for $T$ and $\textsc{shape}(T)$, i.e.,
$N_j^{(N)}[\textsc{shape}(T)]=N_j^{(N)}[T]$.
 The following Law of Large Numbers was proved in Wang and Waymire  \cite{WW91} (Theorem 2.1).
\begin{thm}[{\bf LLN for order two branches, \cite{WW91}}]\label{thm:WW91_LLN}
For a random tree $T \stackrel{d}{\sim} {\sf Unif}\big(\BT^|(N)\big)$,
\be\label{eqn:LLN_N2}
{N_2^{(N)}[T] \over N} ~\overset{p}{\rightarrow} ~4^{-1} \quad \text{ as }\, N\rightarrow \infty.
\ee
\end{thm}

\noindent
Recall that we know from Theorem \ref{thm:BWW2} that the critical binary Galton-Watson tree is invariant under the Horton pruning operation $\cR$.
Thus, the {\it strong Horton law} for branch numbers is deduced from Theorem \ref{thm:WW91_LLN} as follows.
\begin{cor}[{\bf The strong Horton law for branch counts}]\label{cor:ratio4j1}
For a random tree $T \stackrel{d}{\sim} {\sf Unif}\big(\BT^|(N)\big)$ and for all $j \in \mathbb{N}$,
\be\label{eqn:ratio4j1}
{N_j^{(N)}[T] \over N} ~\overset{p}{\rightarrow} ~4^{-(j-1)}\quad \text{ as }\, N\rightarrow \infty.
\ee 
\end{cor}
\begin{proof}
For a fixed integer $k>1$ and a tree  $T^{\rm GW} \stackrel{d}{\sim} \mathcal{GW}\left({1 \over 2},{1 \over 2}\right)$, we have 
for any positive integers $N$ and $M \leq 2^{-(k-1)}N$,
\begin{align}\label{eqn:pruningGW_NM}
\Big(\cR^{k-1}&\big(T^{\rm GW}\big) \,\Big| \, N_1^{(N)}[T^{\rm GW}]=N,\,  N_k^{(N)}[T^{\rm GW}]=M \Big) \\
&\stackrel{d}{=} \Big(\cR^{k-1}\big(T^{\rm GW}\big)\,\Big|\, N_k^{(N)}[T^{\rm GW}]=M \Big)
\stackrel{d}{=} \Big(T^{\rm GW}\,\Big|\, N_1^{(N)}[T^{\rm GW}]=M \Big)  \nonumber
\end{align}
as $\cR^{k-1}(T^{\rm GW})\stackrel{d}{=} T^{\rm GW}$ by the Horton prune-invariance Theorem \ref{thm:BWW2} (and a more general statement in Theorem \ref{thm:main} of Sect. \ref{sec:pi}). The first equality in \eqref{eqn:pruningGW_NM} can be easily verified from permutability of attachments of smaller order branches to the larger order branches. 
Specifically, the event $N_k^{(N)}[T^{\rm GW}]=M$ is equivalent to the event that the pruned tree $\cR^{k-1}\big(T^{\rm GW}\big)$ will have $\#\cR^{k-1}\big(T^{\rm GW}\big)=2M-1$ edges. 
Thus, conditioned of the combinatorial shape $\cR^{k-1}\big(T^{\rm GW}\big)$, all complete subtrees $T_v$ (see Def. \ref{def:HS}(6)) of $T$ such that ${\sf ord}(T_v)={\sf ord}(v)<k$ and ${\sf ord}({\sf parent}(v))\geq k$ will be attached to the edges and leaves of $\cR^{k-1}\big(T^{\rm GW}\big)$ in the same number of ways, for each $\cR^{k-1}\big(T^{\rm GW}\big)$ satisfying $\#\cR^{k-1}\big(T^{\rm GW}\big)=2M-1$ edges.

\medskip
\noindent
Thus, for a fixed $k \in \mathbb{N}$ and a random tree 
$$T \stackrel{d}{\sim} {\sf Unif}\big(\BT^|(N)\big),$$
we have by \eqref{eqn:pruningGW_NM},
$$\Big(\cR^{k-1}(T)\Big| N_k^{(N)}[T]=M \Big) \stackrel{d}{\sim}  {\sf Unif}\big(\BT^|(M)\big)$$
for all $M \leq 2^{-(k-1)}N$.
 Hence, Thm.~\ref{thm:WW91_LLN} implies
 $$\left({N_k^{(N)}[T] \over N_{k-1}^{(N)}[T]} \, \Big| \, {\sf ord}(T) \geq k \right)=\left({N_2^{(N)}\big[\cR^{k-1}(T)\big] \over N_1^{(N)}\big[\cR^{k-1}(T)\big]} \, \Big| \, {\sf ord}(T) \geq k \right)
 \overset{p}{\rightarrow} 4^{-1} \,\text{ as }\, N\rightarrow \infty.$$
 Next, we let ${0 \over 0}=0$ as here $N_k^{(N)}[T] \leq N_{k-1}^{(N)}[T]$, and 
 $$N_{k-1}^{(N)}[T]=0 ~\text{ implies }~N_k^{(N)}[T]=0.$$ 
 Then, as
 $\lim\limits_{N\rightarrow \infty,}{\sf P}\big({\sf ord}(T) < k\big)=0$ we have
 \be\label{eqn:LL_NkNT}
 {N_k^{(N)}[T] \over N_{k-1}^{(N)}[T]}  ~\overset{p}{\rightarrow} ~4^{-1}\quad \text{ as }\, N\rightarrow \infty.
 \ee
Finally, iterating \eqref{eqn:LL_NkNT}, we obtain
$${N_j^{(N)}[T] \over N}={N_j^{(N)}[T] \over N_{j-1}^{(N)}[T]} \, {N_{j-1}^{(N)}[T] \over N_{j-2}^{(N)}[T]} \hdots {N_2^{(N)}[T] \over N} ~\overset{p}{\rightarrow} ~4^{-(j-1)}\quad \text{ as }\, N\rightarrow \infty.$$
\end{proof}

\medskip
\noindent
Following Theorem \ref{thm:WW91_LLN}, the corresponding Central Limit Theorem was proved in Wang and Waymire  \cite{WW91} (Theorem 2.4).
\begin{thm}[{\bf CLT for order two branches, \cite{WW91}}]\label{thm:WW91}
For a random tree $T \stackrel{d}{\sim} {\sf Unif}\big(\BT^|(N)\big)$,
\be\label{eqn:CLT_N2}
\sqrt{N}\left({N_2^{(N)}[T] \over N} -{1 \over 4}\right) \overset{d}{\rightarrow} N(0,4^{-2}) \quad \text{ as }\, N\rightarrow \infty.
\ee
\end{thm}

\medskip
\noindent
Next, using the pruning framework, the following Central Limit Theorem for $N_j^{(N)}[T]$ is readily obtained as a direct consequence of the original Theorem \ref{thm:WW91} of Wang and Waymire \cite{WW91} and the Horton prune-invariance (Def. \ref{def:prune}) of $\mathcal{GW}\left({1 \over 2},{1 \over 2}\right)$ as stated in Theorem \ref{thm:BWW2}, and a more general statement that will appear in Theorem \ref{thm:main} of Sect. \ref{sec:pi}.
\begin{cor}[{\bf CLT for branch numbers, \cite{Yamamoto17}}]\label{cor:CLT_Nj}
For a random tree $T \stackrel{d}{\sim} {\sf Unif}\big(\BT^|(N)\big)$,
\be\label{eqn:CLT_Nj}
\sqrt{N}\left({N_{j+1}^{(N)}[T] \over N_j^{(N)}[T]}  -{1 \over 4}\right) \overset{d}{\rightarrow} N(0,4^{r-3})  \quad \text{ as }\, N\rightarrow \infty,
\ee
where we set ${0 \over 0}=0$.
\end{cor}
\begin{proof}
Pruning $T \stackrel{d}{\sim} {\sf Unif}\big(\BT^|(N)\big)$ iteratively $j-1$ times, we obtain $T \stackrel{d}{\sim} {\sf Unif}\Big(\BT^|\big(N_j^{(N)}[T]\big)\Big)$,
where  for the case when $j>{\sf ord}(T)$ and $N_j^{(N)}[T]=0$, we set $\BT^|(0):=\{\phi\}$. 
Hence, Theorem \ref{thm:WW91} immediately implies
\be\label{eqn:CLT_Nj_v0}
\sqrt{N_j^{(N)}[T]}\left({N_{j+1}^{(N)}[T] \over N_j^{(N)}[T]} -{1 \over 4}\right) \overset{d}{\rightarrow} N(0,4^{-2}) \quad \text{ as }\, N\rightarrow \infty.
\ee
Thus, substituting \eqref{eqn:ratio4j1} into \eqref{eqn:CLT_Nj_v0}, we obtain \eqref{eqn:CLT_Nj}.
\end{proof}
The limit \eqref{eqn:CLT_Nj} was derived by Yamamoto \cite{Yamamoto17}  directly, after a series of technically involved calculations.


\subsection{Metric case}\label{sec:GWmetric}
In this section we turn to the trees in $\cBL^|$. In particular, we will assign i.i.d. exponential lengths to the edges of a critical plane binary Galton-Watson tree 
$\mathcal{GW}_{\rm plane}({1\over 2},{1\over 2})$ in $\cT^|$, thus obtaining what will be called 
the {\it exponential critical binary Galton-Watson tree}.

\medskip
\begin{Def} [{{\bf Exponential critical binary Galton-Watson tree}}]
\label{def:cbinary}
We say that a random tree 
$T\in\BL^{|}$ is
an exponential critical binary Galton-Watson tree with (edge length) parameter $\lambda>0$, 
and write $T\stackrel{d}{\sim}{\sf GW}(\lambda)$,
if the following conditions are satisfied:
\begin{itemize}
\item[(i)] \textsc{p-shape}($T$) is a critical plane binary Galton-Watson tree $\mathcal{GW}_{\rm plane}({1\over 2},{1\over 2})$;
\item[(ii)] conditioned on a given \textsc{p-shape}($T$), the edges of $T$ are sampled as independent ${\sf Exp}(\lambda)$
random variables, i.e., random variables with  probability density function (p.d.f.)
\be
\label{exp}
\phi_\lambda (x)=
\lambda e^{-\lambda x} {\bf 1}_{\{x\ge0\}}.\\ 
\ee
\end{itemize}
\end{Def}
\noindent
The branching process that generates an exponential critical binary Galton-Watson tree is known as the continuous time 
 Galton-Watson process, and is sometimes simply called {\it Markov  branching process} \cite{Harris_book}.
 \index{Galton-Watson tree!critical!exponential critical binary Galton-Watson tree}

\subsubsection{Length of a Galton-Watson random tree ${\sf GW}(\lambda)$}\label{sec:lengthGW}
Recall the modified Bessel functions of the first kind
$$I_\nu(z)=\sum\limits_{n=0}^\infty {\left({z \over 2}\right)^{2n+\nu} \over \Gamma(n+1+\nu)\, n!}.$$
\index{Bessel function}
\index{modified Bessel functions of the first kind}

\begin{lem}\label{lem:ell}
Suppose $T\stackrel{d}{\sim}{\sf GW}(\lambda)$ is an exponential critical binary Galton-Watson tree
with parameter $\lambda$.
The {\it total length} $\textsc{length}(T)$ of the tree $T$ has the p.d.f. 
\begin{equation} \label{eq:pdfs}
\ell(x) = {1 \over x}  e^{-\lambda x}  I_1 \big(\lambda x\big), \quad x>0. 
\end{equation}
\end{lem}
\begin{proof}
Recall that the number of different combinatorial shapes of a planted plane binary tree with $n+1$ leaves, and therefore $2n+1$ edges, is given by 
the {\it Catalan number} \eqref{eqn:catalan}, i.e.,
$$\big|\BT^|(n+1)\big|=C_n={1 \over n+1}\binom{2n}{n}={(2n)! \over (n+1)! n! }.$$
The total length of $2n+1$ edges is a gamma random variable with parameters 
$\lambda$ and $2n+1$ and density function
\[\gamma_{\lambda,2n+1}(x)={\lambda^{2n+1} x^{2n} e^{-\lambda x} \over \Gamma(2n+1)}, \quad x>0.\]
Hence, the total length of the tree $T$ has the p.d.f.
\begin{align}\label{eq:ell}
\ell(x) &= \sum\limits_{n=0}^\infty {C_n \over 2^{2n+1}} \cdot {\lambda^{2n+1} x^{2n} e^{-\lambda x} \over (2n)!} 
= \sum\limits_{n=0}^\infty {\lambda^{2n+1} x^{2n} e^{-\lambda x} \over 2^{2n+1} (n+1)! n!}  \nonumber \\
&=  {1 \over x}  e^{-\lambda x} \sum\limits_{n=0}^\infty {\left({\lambda x \over 2}\right)^{2n+1} \over \Gamma(n+2)\, n!} = {1 \over x}  e^{-\lambda x}  I_1 \big(\lambda x\big).
\end{align}
\end{proof}

\medskip
\noindent
Next, we compute the Laplace transform of $\ell(x)$. By the summation formula in (\ref{eq:ell}),
\begin{align*}
\mathcal{L}\ell(s) &= \int\limits_0^\infty  \sum\limits_{n=0}^\infty {C_n \over 2^{2n+1}} \cdot {\lambda^{2n+1} x^{2n} e^{-(\lambda +s) x} \over (2n)!}  \,dx\\
&=\sum\limits_{n=0}^\infty {C_n \over 2^{2n+1}} \cdot \left({\lambda \over \lambda +s}\right)^{2n+1}\int\limits_0^\infty {(\lambda +s)^{2n+1} x^{2n} e^{-(\lambda +s) x} \over (2n)!}  \,dx\\
&=\sum\limits_{n=0}^\infty {C_n \over 2^{2n+1}}\cdot \left({\lambda \over \lambda +s}\right)^{2n+1} =Z \cdot c(Z^2),
\end{align*}
where we let $Z={\lambda \over 2(\lambda +s)}$, and the characteristic function of Catalan numbers 
\be\label{eq:CatalanZ}
c(z)=\sum\limits_{n=0}^\infty C_n z^n={2 \over 1+\sqrt{1-4z}}
\ee 
is well known.
Therefore
\begin{equation}\label{eq:LaplaceL}
\mathcal{L}\ell(s)=Z \cdot c(Z^2)={\lambda \over \lambda + s+\sqrt{(\lambda +s)^2 -\lambda^2}}.
\end{equation}

Note that the Laplace transform $\mathcal{L}\ell(s)$ could be derived from the total probability formula
\begin{equation}\label{recursionEll}
\ell(x)={1 \over 2}\phi_\lambda (x)+{1 \over 2}\phi_\lambda \ast \ell \ast \ell(x),
\end{equation}
where $\phi_\lambda (x)$ is the exponential p.d.f. (\ref{exp}). Thus, $\mathcal{L}\ell(s)$ solves
\begin{equation}\label{recursionLaplace}
\mathcal{L}\ell(s)={1 \over 2} {\lambda \over \lambda +s}\Big(1+\big(\mathcal{L}\ell(s)\big)^2\Big).
\end{equation}

\begin{cor}
The p.d.f. $f(x)$ of the length of an excursion
in an exponential symmetric random walk with parameter $\lambda$
is given by
\begin{equation}
f(x)={1 \over 2} \ell(x/2).
\end{equation}
\end{cor}
\begin{proof}
Observe that the excursion has twice the length of a tree ${\sf GW}(\lambda)$.
\end{proof}
\bigskip

\subsubsection{Height of a Galton-Watson random tree ${\sf GW}(\lambda)$}\label{sec:heightGW}

\begin{lem}[{{\bf \cite{KZ19}}}]\label{lem:Hx}
Suppose $T\stackrel{d}{\sim}{\sf GW}(\lambda)$ is an exponential critical binary Galton-Watson tree
with parameter $\lambda$.
Then, the height $\textsc{height}(T)$ of the tree $T$ has the cumulative distribution function 
\begin{equation} \label{eq:pdfht}
{\sf H}(x) = {\lambda x \over \lambda x +2}, \quad x>0. 
\end{equation}
\end{lem}
\begin{proof}
The proof is based on duality between trees and positive real excursions
that we introduce in Sect. \ref{LST}.
In particular, Thm.~\ref{Pit7_3} establishes equivalence 
between the level set tree (Sect. \ref{sec:level}) of a positive excursion of an 
{\it exponential random walk} (Sect. \ref{sec:erw}) 
and an exponential critical binary Galton-Watson tree ${\sf GW}(\lambda)$.
This implies, in particular, that for a tree $T\stackrel{d}{\sim}{\sf GW}(\lambda)$ 
the $\textsc{height}(T)$ has the same distribution as the height of a positive excursion of an exponential random walk
$Y_k$ with $Y_0=0$ and independent increments $Y_{k+1}-Y_k$ distributed according to the Laplace density function
${\phi_\lambda(x)+\phi_\lambda(-x) \over 2}={\lambda \over 2}e^{-\lambda |x|}$, with 
$\phi_\lambda(x)$ defined in \eqref{exp}.

Notice that $Y_k$ is a martingale. We condition on $Y_1>0$,  and consider an excursion  $Y_0,Y_1,\hdots, Y_{\tau_-}$
with $\tau_-=\min\{k>1 ~:~ Y_k \leq 0\}$ denoting the termination step of the excursion. 
For $x>0$, we write
$$p_x=1-{\sf H}(x) ={\sf P}\left(\max\limits_{j:~ 0< j <\tau_-} Y_j >x ~\Big|~Y_1>0 \right)$$
for the probability that the height of the excursion exceeds $x$.
The problem of finding $p_x$ is solved using the Optional Stopping Theorem. Let 
$$\tau_x=\min\{k >0~:~Y_k \geq x\} \qquad \text{ and } \qquad \tau:=\tau_x \wedge \tau_-.$$
Observe that
$$p_x={\sf P}(\tau=\tau_x ~|~Y_1>0).$$
For a fixed $y \in (0,x)$, by the Optional Stopping Theorem, we have
\begin{eqnarray*}
y & = & {\sf E}[Y_\tau ~|~Y_1=y]\\
\\
& = & {\sf E}[Y_\tau ~|~\tau=\tau_-, Y_1=y] \,{\sf P}(\tau=\tau_- ~|~ Y_1=y)\\
& & \qquad +\,{\sf E}[Y_\tau ~|~\tau=\tau_x, Y_1=y] \,{\sf P}(\tau=\tau_x ~|~ Y_1=y)\\
\\
& = & {\sf E}[Y_\tau ~|~Y_\tau \leq 0, Y_1=y] {\sf P}\,(\tau=\tau_- ~|~ Y_1=y)\\
& & \qquad +\,{\sf E}[Y_\tau ~|~Y_\tau \geq x, Y_1=y] \,{\sf P}(\tau=\tau_x ~|~ Y_1=y)\\
\\
& = & -{1 \over \lambda}{\sf P}(\tau=\tau_- ~|~ Y_1=y)+
\left(x+{1 \over \lambda}\right){\sf P}(\tau=\tau_x ~|~ Y_1=y)\\
\\
& = & \left(x+{2 \over \lambda}\right){\sf P}(\tau=\tau_x ~|~ Y_1=y)-{1 \over \lambda}.\\
\end{eqnarray*}

\noindent
Hence, $${\sf P}(\tau=\tau_x ~|~ Y_1=y)={y+{1 \over \lambda} \over x +{2 \over \lambda}}.$$
Thus,
\begin{eqnarray*}
{\sf P}\Big(\tau=\tau_x,  0<Y_1<x ~|~Y_1>0\Big) & = &\int\limits_0^x {\sf P}(\tau=\tau_x ~|~ Y_1=y)~\lambda e^{-\lambda y} dy\\
& = & \int\limits_0^x {y+{1 \over \lambda} \over x +{2 \over \lambda}}~\lambda e^{-\lambda y} dy\\
& = & {2 \over \lambda x +2}-e^{-\lambda x},
\end{eqnarray*}

\noindent
and therefore,
\begin{eqnarray*}
p_x & = & {\sf P}\left(\max\limits_{j:~ 0< j <K } Y_j >x ~|~Y_1>0 \right)\\
& = & {\sf P}\Big(\tau=\tau_x,  0<Y_1<x ~|~Y_1>0\Big)+{\sf P}\Big(\tau=\tau_x,  Y_1 \geq x ~|~Y_1>0\Big)\\
& = & {2 \over \lambda x +2}-e^{-\lambda x}+{\sf P}\Big(Y_1 \geq x ~|~Y_1>0\Big) 
=  {2 \over \lambda x +2}.
\end{eqnarray*}
Hence, 
\[{\sf H}(x)=1-p_x={\lambda x \over \lambda x +2}.\]
\end{proof}

We continue examining the height function $\textsc{height}(T)$ for $T\stackrel{d}{\sim}{\sf GW}(\lambda)$. This time, we condition on
$\#T=2n-1$, i.e., the tree $T$ has $n$ leaves and $n-1$ internal non-root vertices. We let ${\sf H}_n(x)$ denote the corresponding conditional cumulative distribution function,
\be\label{eqn:HnDef}
{\sf H}_n(x)={\sf P}\big(\textsc{height}(T) \leq x ~\big| ~\#T=2n-1 \big).
\ee
There, for a one-leaf tree,
\be\label{eqn:H1}
{\sf H}_1(x)=1-e^{-\lambda x},
\ee
and for $n \geq 2$, the following recursion follows from conditioning on the length of the stem (root edge),
\be\label{eqn:HnRec}
{\sf H}_n(a)=\sum\limits_{k=1}^{n-1}{C_{k-1}C_{n-k-1}\over C_{n-1}} \int\limits_0^a {\sf H}_k(a-x){\sf H}_{n-k}(a-x) \, \lambda e^{-\lambda x}\, dx ,
\ee
where $C_n$ is the Catalan number as defined in \eqref{eqn:catalan}.

\medskip
\noindent
Next, we consider the following $z$-transform:
\be\label{eqn:HnZtr}
{\sf h}(a;z)=\sum\limits_{n=1}^\infty {\sf H}_n(a)\, C_{n-1}\, z^n \quad \text{ for }~|z|<1/4.
\ee
Then, \eqref{eqn:H1} and \eqref{eqn:HnRec} imply
$${\sf h}(a;z)=(1-e^{-\lambda a})z+\int\limits_0^a {\sf h}^2(a-x;z)\, \lambda e^{-\lambda x}\, dx$$
which, if we let $y=a-x$, simplifies to
$$e^{\lambda a}{\sf h}(a;z)-e^{\lambda a}z=\int\limits_0^a {\sf h}^2(y;z)\, \lambda e^{\lambda y}\, dy ~-z.$$
We differentiate the above equation, obtaining
\be\label{eqn:HnDE1}
{\partial \over \partial a}{\sf h}(a;z)=\lambda\, \Big({\sf h}^2(a;z)-{\sf h}(a;z)+z\Big).
\ee
Let
$${\sf x}_1(z)={1+\sqrt{1-4z} \over 2} ~~\text{ and }~~{\sf x}_2(z)={1-\sqrt{1-4z} \over 2}$$
be the two roots of ${\sf x}^2-{\sf x}+z=0$. 
Here, ${\sf x}_2(z)/z=1/{\sf x}_1(z)=c(z)$ is the $z$-transform of the Catalan sequence $C_n$, introduced in \eqref{eq:CatalanZ}.
Then, \eqref{eqn:HnDE1} solves as
$${\sf h}(a;z)-{\sf x}_1(z)=\Phi(z)e^{\lambda a \sqrt{1-4z}}\big({\sf h}(a;z)-{\sf x}_2(z)\big),$$
where due to the initial conditions ${\sf h}(0;z)=0$, we have $\Phi(z)={\sf x}_1(z)/{\sf x}_2(z)$, and
\be\label{eqn:HnDEsol}
{\sf h}(a;z)-{\sf x}_1(z)={{\sf x}_1(z) \over {\sf x}_2(z)}e^{\lambda a \sqrt{1-4z}}\big({\sf h}(a;z)-{\sf x}_2(z)\big).
\ee
Solution \eqref{eqn:HnDEsol} implies
\be\label{eqn:HnDEsolh}
{\sf h}(a;z)
={2\left(e^{\lambda a \sqrt{1-4z}}-1\right)z \over e^{\lambda a \sqrt{1-4z}}-1+\left(e^{\lambda a \sqrt{1-4z}}+1\right)\sqrt{1-4z}}.
\ee
Here and throughout we use $-\pi < \arg(z) \leq \pi$ branch of the logarithm when defining $\sqrt{1-4z}$ for $|z|<1/4$.

\medskip
\noindent
Now, since ${\sf P}\big(\#T=2n-1 \big)=2C_{n-1} 4^{-n}$, the series expansion \eqref{eqn:HnZtr} implies
\be\label{eqn:Hvsh}
{\sf H}(a)=\lim\limits_{z \uparrow {1 \over 4}} 2 \, {\sf h}(a;z),
\ee
where $z \in (-1/4,\,1/4)$ is real.
We substitute \eqref{eqn:HnDEsolh} into the limit \eqref{eqn:Hvsh},
\begin{align}\label{eqn:Hviah}
\lim\limits_{z \uparrow {1 \over 4}} 2 \, {\sf h}(a;z) &= \lim\limits_{z \uparrow {1 \over 4}} {4\left(e^{\lambda a \sqrt{1-4z}}-1\right)z \over e^{\lambda a \sqrt{1-4z}}-1+\left(e^{\lambda a \sqrt{1-4z}}+1\right)\sqrt{1-4z}} \nonumber \\
&= \lim\limits_{z \uparrow {1 \over 4}} {4z \over 1+\sqrt{1-4z}+{2\sqrt{1-4z} \over e^{\lambda a \sqrt{1-4z}}-1}}  ~~={1 \over 1+{2 \over \lambda a}}\nonumber \\
&={\lambda a \over \lambda a +2},
\end{align}
thus obtaining an alternative proof of formula \eqref{eq:pdfht} in Lemma \ref{lem:Hx}.

\medskip
\noindent
The asymptotic of the height distribution ${\sf H}_n(a)$ for a given number of leaves $n$ was the object of analysis in \cite{Kolchin78,Waymire89,GMW90,DKW91}. In particular,
Gupta et al. \cite{GMW90} extended the results of Kolchin \cite{Kolchin78}, by showing that
\be\label{eqn:WaymireKolchin}
\lim\limits_{n \rightarrow \infty}{\sf H}_n\left({a\sqrt{n} \over \lambda}\right)={\sf H}_\infty(a):=1+2\sum_{k=1}^\infty (1-4k^2a^2)\exp\big\{-2k^2a^2\big\}.
\ee
It was also observed in \cite{GMW90} that ${\sf H}_\infty\left({a \over 2\sqrt{2}}\right)$ is the distribution function for the maximum of the Brownian excursion
as shown in the work of Durrett and Iglehart \cite{DI77}. The results of \cite{GMW90} were 
further developed in \cite{DKW91} for more general trees with edge lengths.

\section{Hierarchical Branching Process}\label{HBP}

Tree self-similarity has been studied primarily in terms of the average values
of selected branch statistics, as defined in Sect. \ref{sec:mss}.
Until recently, the only rigorous results have been obtained only for 
a very special classes of Markov trees (e.g., binary Galton-Watson trees with no
edge lengths, as in Sect. \ref{ecgw}).
At the same time, solid empirical evidence motivates a search for a 
flexible class of self-similar models that
would encompass a variety of observed combinatorial and metric structures 
and rules of tree growth.
In Sec.~\ref{TSSL} we introduced a general concept of self-similarity
that accounts for both combinatorial and metric tree structure. 
In this section we will describe a model called {\it hierarchical branching process}
that generates a broad range of self-similar trees (Thm.~\ref{HBPmain}) and includes the 
critical binary Galton-Watson tree with exponential edge lengths as a special case (Thm.~\ref{HBPmain3}).
We will also introduce a class of critical self-similar {\it Tokunaga processes} (Sect.~\ref{sec:Tok})
that enjoy additional symmetries --- their edge lengths are i.i.d. random
variables (Prop.~\ref{Tok_tree}), and subtrees of large Tokunaga trees 
reproduce the probabilistic structure of 
the entire random tree space (Prop.~\ref{prop:Tok1}).
The results of this section are derived in \cite{KZ18}.

\medskip
\noindent
The results of Sect.~\ref{cgw} concerned a very narrow class of mean self-similar
trees with $T_j=2^{j-1}$.
Among such trees, the self-similarity is established 
only for the critical binary Galton-Watson tree ${\sf GW}(\gamma)$ 
with independent exponential edge lengths, i.e., continuous parameter Galton-Watson 
binary branching Markov processes; 
this case corresponds to the scaling exponent $\zeta=2$. 
Next, we construct a multi-type branching process \cite{Harris_book,AN_book} that generates self-similar 
trees for an arbitrary sequence $T_j\ge0$ and for any $\zeta > 0$; 
it includes the critical binary Galton-Watson tree as a special case.

\subsection{Definition and main properties}
\label{HBP:def}
Consider a probability mass function $\{p_K\}_{K\ge1}$, 
a sequence $\{T_k\}_{k\ge 1}$ of nonnegative Tokunaga coefficients, 
and a sequence $\{\lambda_j\}_{j\ge 1}$ of positive termination rates.
We now define a hierarchical branching process $S(t)$.
\begin{Def}[{{\bf Hierarchical Branching Process (HBP)}}]
\label{def:HBP}
We say that $S(t)$ is a hierarchical branching process with a triplet 
of parameter sequences $\{T_k\}$, $\{\lambda_j\}$, and $\{p_K\}$, and 
write
\[S(t) \stackrel{d}{\sim} {\sf HBP}\big(\{T_k\},\{\lambda_j\},\{p_K\}\big)\]
if $S(t)$ is a multi-type branching process that develops in continuous time 
$t>0$ according to the following rules:
\begin{itemize}
\item [(i)] The process $S(t)$ starts at $t=0$ with a single progenitor (root branch) 
whose Horton-Strahler order (type) is $K\ge 1$ with probability $p_K$.
\item [(ii)] Every branch of order $j\le K$ produces offspring (side branches) 
of every order $i<j$ with rate $\lambda_j T_{j-i}$.
Each offspring (side branch) is assigned a uniform random orientation (right or left). 
\item [(iii)] A branch of order $j$ terminates with rate $\lambda_j$.
\item [(iv)] At its termination time, a branch of order $j\geq 2$ splits into 
two independent branches of order $j-1$.
The two branches are assigned uniform random orientations, i.e., a uniformly randomly 
selected branch becomes right and the other becomes left.
\item[(v)] A branch of order $j=1$ terminates without leaving offspring.
\item[(vi)] Generation of side branches and termination of distinct branches are
independent. 
\end{itemize}
\end{Def}
\index{hierarchical branching process}

\noindent The definition implies that the process $S(t)$ terminates a.s. in finite time.
Accordingly, the branching history of $S(t)$ creates a random binary tree $T[S]$ in the space 
$\BL^|$ of planted binary trees with edge lengths and planar embedding.
To avoid heavy notations, we sometimes use the process distribution name
${\sf HBP}(\cdot,\cdot,\cdot)$, as well as its
various special cases introduced below, to also denote the 
measures induced by the process on suitable tree spaces ($\T^|$, $\L^|$ $\BL^|$, etc.)
\medskip

The next statement describes the branching structure of $T[S]$.
\begin{prop}[{\bf Side-branching in hierarchical branching process, \cite{KZ18}}]
\label{HBP:branch}
Consider a hierarchical branching process $S(t)\stackrel{d}{\sim} {\sf HBP}\big(\{T_k\},\{\lambda_j\},\{p_K\}\big)$ and let $T[S]$ be the tree generated by $S(t)$ in $\BL^|$.
For a branch $b\subset T[S]$ of order $K\ge 1$, let $m_i:=m_i(b)\ge 0$ be the
number of its side branches of order $i=1,\dots,K-1$, and $m:=m(b)=m_1+\dots+m_{K-1}$ be
the total number of the side branches.
Conditioned on $m$, let $l_i:=l_i(b)$ be the lengths of $m+1$ edges within $b$, counted sequentially from 
the initial vertex,  
and $l:=l(b)=l_1+\dots+l_{m+1}$ be the total branch length.
Define
\[S_K:=1+T_1+\dots+T_K\]
for $K\ge 0$ by assuming $T_0=0$.
Then the following statements hold:
\begin{enumerate}
\item The tree order satisfies
\be\label{orddist}
{\sf P}\left({\sf ord}(T[S])=K\right)=p_K,\quad K\ge 1.
\ee
\item The total number $m(b)$ of side branches within a branch $b$ of order $K$
has geometric distribution:
\be
\label{mdist}
m(b)\stackrel{d}{\sim}{\sf Geom}_0\left(S_{K-1}^{-1}\right),\quad K\ge 1,
\ee
with ${\sf E}[m(b)]=S_{K-1}-1=T_1+\dots+T_{K-1}.$

\item Conditioned on the total number $m$ of side branches, the distribution 
of vector $(m_1,\dots,m_{K-1})$ is multinomial with $m$ trials and success probabilities
\be
\label{multinomial}
{\sf P}(\text{side~branch~has~order~}i)=\frac{T_{K-i}}{S_{K-1}-1}.
\ee
The vector $({\sf ord}_1,\dots,{\sf ord}_m)$ of side branch orders, where the side branches
are labeled sequentially starting from the initial vertex of $b$, is obtained
from the sequence
\[{\sf orders}=\underbrace{(1,\dots,1}_{m_1\text{~times}},\underbrace{2,\dots,2}_{m_2\text{~times}},
\dots\underbrace{K-1,\dots,K-1)}_{m_{K-1}\text{~times}}\]
by a uniform random permutation $\sigma_m$ of indices $\{1,\dots,m\}$:
\[({\sf ord}_1,\dots,{\sf ord}_m) = {\sf orders}\circ\sigma_m.\]

\item The total numbers of side branches and orders of side branches are
independent in distinct branches.

\item The branch length $l$ has exponential
distribution with rate $\lambda_{K}$, independent of the lengths of
any other branch (of any order). 
The corresponding edge lengths $l_i$ are i.i.d. random variables; they have
a common exponential distribution with rate
\be
\label{edges}
\lambda_{K}S_{K-1}.
\ee
\end{enumerate}
\end{prop}
\noindent
\begin{proof}
All the properties readily follow from Def.~\ref{def:HBP}.
\end{proof}
Combining properties 2 and 3 of Prop.~\ref{HBP:branch} we find that 
the number $m_i$ of side branches of order $i$ within a branch $b$ of order $K$
has geometric distribution:
\be
\label{midist}
m_i(b)\stackrel{d}{\sim}{\sf Geom}_0\left(\left[1+T_{K-i}\right]^{-1}\right),
\quad K\ge 1, i\le K-1,
\ee
with ${\sf E}\left[m_i\right] = T_{K-i}.$
We also notice that the numbers $m_i(b)$ for $i=1,\dots,K-1$ within the same branch $b$
are dependent.

Proposition~\ref{HBP:branch} provides an alternative definition of the 
hierarchical branching process and suggests a recursive construction of $T[S]$ 
that does not require time-dependent simulations.
Specifically, a tree of order $K=1$ consists of two vertices (root and leaf) 
connected by an edge of exponential length with rate $\lambda_1$.
Assume now that we know how to construct a random tree of any order below $K\ge 2$.
To construct a tree of order $K$, we start with a perfect (combinatorial)
planted binary tree of depth $K$, which we call {\it skeleton}.
The combinatorial shapes of such trees is illustrated in Fig.~\ref{fig:perfect}.
All leaves in the skeleton have the same depth $K$, and all vertices at depth $1\le \kappa\le K$
have the same Horton-Strahler order $K-\kappa+1$.
The root (at depth 0) has order $K$.
Next, we assign lengths to the branches of the skeleton. 
Recall (Ex.~\ref{ex:perfect}) that each branch in a perfect tree consists of 
a single edge. 
To assign length to a branch $b$ of order $\kappa$, with $1\le \kappa\le K$,
we generate a geometric number $m\stackrel{d}{\sim}{\sf Geom}_0(S^{-1}_{\kappa-1})$ 
according to
\eqref{mdist} and then $m+1$ i.i.d. exponential lengths $l_i$, $i=1,\dots,m+1$,
with the common rate $\lambda_{\kappa}S_{\kappa-1}$ according to \eqref{edges}.
The total length of the branch $b$ is $l_1+\dots+l_{m+1}$.
Moreover, branch $b$ has $m$ side branches that are attached along $b$
with spacings $l_i$, starting from the branch point closest to the root.
The order assignment for the side branches is done according to \eqref{multinomial}.
We generate side branches (each has order below $K$) independently 
and attach them to the branch $b$. This completes the construction of a 
random tree of order $K$.
To construct a random HBP tree, one first generates a random
order $K\ge 1$ according to \eqref{orddist} and then constructs
a tree of order $K$ using the above recursive process.

\medskip
Next, we establish various forms of self-similarity for the hierarchical branching process.
\begin{thm}[{\bf Self-similarity of hierarchical branching process, \cite{KZ18}}]
\label{HBPmain}
Consider a hierarchical branching process $S(t)\stackrel{d}{\sim} {\sf HBP}\big(\{T_k\},\{\lambda_j\},\{p_K\}\big)$
and let $T:=T[S]$ be the tree generated by $S(t)$ on $\BL^|$.
The following statements hold.
\begin{enumerate}

\item The combinatorial tree $\textsc{shape}(T)$ is mean Horton self-similar
(according to Def.~\ref{ss1},\ref{def:ss2}) with Tokunaga coefficients $\{T_k\}$.

\item The combinatorial tree $\textsc{shape}(T)$ is Horton self-similar
(according to Def.~\ref{def:ss}) with Tokunaga coefficients $\{T_k\}$
if and only if 
\[p_K = p(1-p)^{K-1}\text{ for all }K\ge 1\text{ and some } 0<p< 1.\]

\item The tree $T$ is Horton self-similar (according to Def.~\ref{def:distss})
with scaling exponent $\zeta>0$ if and only if
\[p_K = p(1-p)^{K-1}, ~K\ge 1,\quad{ and }\quad\lambda_j = \gamma\,\zeta^{-j},~j\ge 1,\]
for some positive $\gamma$ and $0<p< 1$.
\end{enumerate} 
\end{thm}
\begin{proof}
By process construction, the tree $T$ is coordinated in shapes and lengths 
(according to Def.~\ref{def:distss}), with independent complete subtrees.

(1) Proposition~\ref{HBP:branch}, part (3) implies that the expected value of the number
$\tilde N_{i,j}$ of side branches of order $i\ge 1$ within a branch of order $j>i$ is 
given by ${\sf E}\left[\tilde N_{i,j}\right]=T_{j-i}$.
The mean self-similarity of Def.~\ref{ss1} with coefficients $T_k$ immediately follows, using 
a conditional argument as in \eqref{condarg}.

(2) Assume that $\textsc{shape}\left(T\right)$ is self-similar. 
A geometric distribution of orders is then established in Prop.~\ref{geom:p}.
Inversely, a geometric distribution of orders ensures that the total mass
$\mu\left(\cH_K\right)$, $K\ge 1$, is invariant with respect to pruning.
The conditional distribution of trees of a given order is completely specified
by the side branch distribution, described in Proposition~\ref{HBP:branch},
parts (1)-(3). 
Consider a branch of order $K+1$, $K\ge 1$. 
Pruning decreases the orders of this branch, and all its side branches, by unity.
Pruning eliminates a random geometric number $m_1$ of side-branches of order $1$
from the branch.
It acts as a thinning with removal probability $T_{K}/(S_K-1)$
on the total side branch count $m$.
Accordingly, the total side branch count after pruning has geometric distribution
with success probability
\[q^{\cR} = S_{K-1}^{-1}.\]
The order assignment among the remaining side branches (with possible orders $i=1,\dots,K-1$) 
is done according to multinomial distribution with probabilities proportional to $T_{K-i}$.
This coincides with the side branch structure in the original tree, hence completing 
the proof of (2).

(3) Having proven (2), it remains to prove the statement for the length structure of the tree.
Assume that $T$ is self-similar
with scaling exponent $\zeta$. 
The branches of order $j\ge 2$ become branches of order $j-1$ after
pruning, which necessitates $\lambda_{j} = \zeta\,\lambda_{j-1}$.
Inversely, pruning acts as a thinning on the side branches within a
branch of order $K+1$, eliminating the side branches of order ${\sf ord}=1$. 
Accordingly, the spacings between the remaining side branches are
exponentially distributed with a decreased rate
\[\lambda_{K+1}S_{K-1}=\zeta\,\lambda_KS_{K-1}.\]
Comparing this with \eqref{edges}, and recalling the self-similarity of $\textsc{shape}\left(T\right)$, we conclude that Def.~\ref{def:distss}
is satisfied with scaling exponent $\zeta$.
\end{proof}

\subsection{Hydrodynamic limit}
\label{HBP:hydro}
Here we analyze the average numbers of branches of different orders in 
a hierarchical branching process, using a hydrodynamic limit.
Specifically, let $n\,x^{(n)}_j(s)$ be the number of branches of order $j$ at time 
$s$ observed in $n$ independent copies of the process $S$. 
Let $N_j(s)$ be the number of branches of order $j\ge 1$ in the process $S$
at instant $s\ge 0$. 
We observe that, by the law of large numbers,
\[ x^{(n)}_j(s)\stackrel{\text{a.s.}}{\longrightarrow}{\sf E}\left[N_j(s)\right]= :x_j(s).\]

\begin{thm}[{\bf Hydrodynamic limit for branch dynamics, \cite{KZ18}}]
\label{thm:hydro}
Suppose that the following conditions are satisfied:
\be\label{eq:L}
L:=\limsup_{k\to\infty} T_k^{1/k}<\infty,
\ee
and 
\be\label{con}
\sup\limits_{j \geq 1} \lambda_j <\infty,\quad
\limsup\limits_{j\to\infty} \lambda_j ^{1/j}\le 1/L.
\ee
Then, for any given $T>0$, the empirical process 
\[x^{(n)}(s)=\Big(x^{(n)}_1(s),x^{(n)}_2(s),\hdots \Big)^T, \qquad s \in [0,T],\] 
converges almost surely, as $n\to\infty$, to the process 
\[x(s)=\Big(x_1(s),x_2(s),\hdots \Big)^T, \qquad s \in [0,T],\] 
that satisfies
\begin{equation}\label{GODE}
\dot{x}=\mathbb{G}\Lambda x \quad
\text{ with the initial conditions } \quad x(0)=\pi:=\sum\limits_{K=1}^\infty p_K e_K,
\end{equation}
where $\Lambda=\text{diag}\{\lambda_1,\lambda_2,\hdots\}$ is a diagonal operator with 
the entries $\lambda_1,\lambda_2,\hdots~$,
$e_i$ are the standard basis vectors, and operator $\mathbb{G}$ defined in Eq. (\ref{gen}).
\end{thm}
\begin{proof} 
The process $~x^{(n)}(s)~$ evolves according to the transition rates 
$$q^{(n)}(x,x+\ell)=n\beta_\ell\left({1 \over n} x\right)$$ 
with
$$\beta_\ell(x)=\begin{cases}
      \lambda_1x_1 & \text{  if } \ell=-e_1, \\
      \lambda_{i+1}x_{i+1} & \text{  if } \ell=2e_i-e_{i+1}, i\ge 1,\\
      \sum\limits_{j=i+1}^\infty \lambda_j T_{j-i}x_j & \text{ if } \ell=e_i, i\ge 1.
\end{cases}$$
Here the first term reflects termination of branches of order $1$;
the second term reflects termination of branches of orders $i+1>1$, each of 
which results in creation of two branches of order $i$; and
the last term reflects side-branching.
Thus, the infinitesimal generator of the stochastic process $x^{(n)}(s)$ is
\begin{eqnarray}\label{eq:Ln}
L_nf(x) & = & n\lambda_1 x_1\left[f\left(x-{1 \over n}e_1\right)-f(x)\right]  \nonumber \\
&  &  ~~+\sum\limits_{i=1}^\infty n\lambda_{i+1}x_{i+1}\left[f\left(x-{1 \over n}e_{i+1}+{2 \over n}e_i\right)-f(x)\right]  \nonumber \\
&  &  ~~+  \sum\limits_{i=1}^\infty\left(\sum\limits_{j=i+1}^\infty n \lambda_jT_{j-i}x_j\right) \left[f\left(x+{1 \over n}e_i\right)-f(x)\right].
\end{eqnarray}
Let
$$F(x):=\sum\limits_\ell \beta_\ell(x)=-\lambda_1 x_1e_1 +\sum\limits_{i=1}^\infty \lambda_{i+1}x_{i+1}(2e_i-e_{i+1})+\sum\limits_{i=1}^\infty \left(\sum\limits_{j=i+1}^\infty \lambda_j T_{j-i}x_j\right) e_i.$$
The convergence result of Kurtz (\cite[Theorem 2.1, Chapter 11]{EK86}, \cite[Theorem 8.1]{Kurtz81}) given here in Appendix \ref{sec:kurtz} extends (without changing the proof) to the Banach space $\ell^1(\mathbb{R})$ provided the same conditions are satisfied for $\ell^1(\mathbb{R})$ as for $\mathbb{R}^d$ in Theorem \ref{kurtzT}. 
Specifically, we require that for a compact set $\mathcal{C}$ in $\ell^1(\mathbb{R})$, 
\begin{equation}\label{eq:beta}
\sum_\ell \|\ell\|_1 \sup_{x \in \mathcal{C}} \beta_\ell(x) <\infty,
\end{equation} 
and there exists $M_\mathcal{C} >0$ such that 
\begin{equation}\label{eq:Lipschitz}
\|F(x)-F(y)\|_1 \leq M_\mathcal{C} \|x-y\|_1, \qquad x,y \in \mathcal{C} .
\end{equation}
Here the condition (\ref{eq:beta}) follows from
$$\sum_i \sup_{x \in \mathcal{C}} |\lambda_i x_i| <\infty \qquad \text{ and } \qquad \sum_i \sup_{x \in \mathcal{C}} \sum\limits_{j=i+1}^\infty \lambda_j T_{j-i}|x_j| <\infty,$$
which in turn follow from conditions \eqref{con}. 
Similarly, Lipschitz conditions (\ref{eq:Lipschitz}) are satisfied in 
$\mathcal{C}$ due to conditions \eqref{con}.
Thus, by Theorem \ref{kurtzT} extended for $\ell^1(\mathbb{R})$,  
the process $x^{(n)}(s)$ converges almost surely to $x(s)$ that satisfies $\dot{x}=F(x)$,
which expands as the following system of ordinary differential equations:
\begin{equation}\label{gFODEsys}
\begin{cases}
x'_1(s) & =  -\lambda_1 x_1 +\lambda_2 (T_1+2)x_2+\lambda_3 T_2x_3+\hdots \\ 
x'_2(s) & =  -\lambda_2 x_2 +\lambda_3 (T_1+2)x_3+\lambda_4 T_2x_4+\hdots \\  
& \vdots \\
x'_k(s) & =  -\lambda_k x_k +\lambda_{k+1} (T_1+2)x_{k+1}+\lambda_{k+2} T_2x_{k+2}+\hdots \\ 
& \vdots   
\end{cases}
\end{equation}
with the initial conditions $x(0)=\lim\limits_{n \rightarrow \infty} x^{(n)}(0)=\pi:=\sum\limits_{K=1}^\infty p_K e_K$ by the law of large numbers. 
Finally, we observe that $\|\pi\|_1=1$, and conditions \eqref{con} imply that 
$\mathbb{G}\Lambda$ is a bounded operator in $\ell^1(\mathbb{R})$.
\end{proof}

\subsection{Criticality and time invariance}
\subsubsection{Definitions}
\label{sec:width}
Assume that the hydrodynamic limit $x(s)$, and hence the averages $x_j(s)$, exist. 
Write $\pi=\sum\limits_{K=1}^\infty p_K e_K$ for the initial distribution of the process.
Consider the average progeny of the process, that is the average 
number of branches of any order alive at instant $s\ge0$:
$$C(s)=\sum\limits_{j=1}^\infty x_j(s)=\Big\|e^{\mathbb{G}\Lambda s} \pi \Big\|_1.$$

\begin{Def}\label{criticalpi}
A hierarchical branching process $S(s)$ is said to be 
{\it critical} if its average progeny is constant: $C(s)=1$ for all $s \geq 0$.
\end{Def}

\begin{Def}
A hierarchical branching process $S(s)$ is said to be
{\it time-invariant} if
\begin{equation}\label{eq:fipc}
e^{\mathbb{G}\Lambda s}\pi=\pi\quad\text{ for all}\quad s\ge 0.
\end{equation}
\end{Def}

\begin{prop}\label{ti2c}
Suppose the hydrodynamic limit $x(s)$ exists, and the hierarchical branching process 
$S(s)$ is time-invariant. 
Then the process $S(s)$ is critical.
\end{prop}
\begin{proof}
$C(s)=\|x(s)\|_1=\| e^{\mathbb{G}\Lambda s} \pi \|_1=\| \pi \|_1=1.$
\end{proof}

\medskip
\noindent
Recall the function  $~\hat{t}(z) = -1+2z+\sum_j z^j\,T_j~$ defined in Eq. (\ref{def:that}) for all complex $|z| <1/L$, where the inverse radius of convergence $L$ is defined in Eq. (\ref{eq:L}).
We also recall that there is a unique real root $w_0$ of $\hat{t}(z)$ within $(0,\frac{1}{2}]$.
We formulate some of the results below in terms of  $\hat{t}(z)$ and the {\it Horton exponent} $R:=w_0^{-1}$; see Theorem \ref{thm:HLSST}.

\begin{prop}
\label{critcond}
Suppose 
$\Lambda\pi$ is a constant multiple of the geometric vector 
$v_0=\sum\limits_{K=1}^\infty R^{-K} e_K$. 
Then the process $S(s)$ is time-invariant.
\end{prop}
\begin{proof}
Observe that since $\hat{t}\left(R^{-1}\right)=0$ and $\mathbb{G}$ is a Toeplitz operator,
$$\mathbb{G}v=\hat{t}(w)v \quad \text{ for }~v=\sum\limits_{K=1}^\infty w^K e_K, ~~|w|<L.$$
and 
$$\mathbb{G}v_0=\hat{t}\left(R^{-1}\right)v_0=0 \quad \text{ for }~v_0:=
\sum\limits_{K=1}^\infty R^{-K} e_K.$$
Hence
$\mathbb{G}\Lambda \pi = \hat{t}\left(R^{-1}\right) \Lambda \pi=0$
and
\[e^{\mathbb{G}\Lambda s}\pi = \pi +\sum_{m=1}^{\infty}\frac{s^m}{m!} (\mathbb{G}\Lambda)^m\pi=\pi.\]
\end{proof}

\begin{Rem}
Proposition~\ref{critcond} states that the condition
\be\label{lpR}
\lambda_K\,p_K = b\,R^{-K}, K\ge 1
\ee
is sufficient for time-invariance, for any proportionality constant $b>0$.
This implies that a time-invariant process can be constructed for 
\begin{itemize}
\item[(i)] an arbitrary sequence of Tokunaga coefficients $\{T_k\}$
satisfying \eqref{eq:L} -- by selecting $\lambda_K\,p_K = b\,R^{-K}$;
\item[(ii)] arbitrary sequences $\{T_k\}$ satisfying \eqref{eq:L} and $\{p_K\}$ -- by selecting
$\lambda_K = b\,R^{-K}\,p_K^{-1}$;
\item[(iii)] arbitrary sequences $\{T_k\}$ satisfying \eqref{eq:L} and $\{\lambda_K\}$ -- by selecting
$p_K = b\,R^{-K}\,\lambda_K^{-1}$.
\end{itemize}
At the same time, arbitrary sequences $\{\lambda_K\}, \{p_K\}$ 
will not, in general, satisfy \eqref{lpR} and hence will not
correspond to a time-invariant process.  
\end{Rem}

\subsubsection{Criticality and time-invariance in a self-similar process}
\label{HBP:ss}
A convenient characterization of criticality can be established for 
self-similar hierarchical branching processes.
Recall that by Theorem~\ref{HBPmain}, part (3), a self-similar process
$S(s)$ is specified by parameters $\gamma>0$, $0<p<1$ and length self-similarity 
constant $\zeta>0$ such that $p_K=p(1-p)^{K-1}$ and $\lambda_j = \gamma\,\zeta^{-j}$. 
We refer to a self-similar process by its parameter triplet, and write 
$S(s)\stackrel{d}{\sim}S_{p,\gamma,\zeta}(s)$.
We denote the respective average progeny by $C_{p,\gamma,\zeta}(s)$.
Observe that in the self-similar case the first of the conditions~\eqref{con}
is equivalent to $\zeta\ge 1$, and the second is equivalent to $\zeta\ge L$.
Hence, the conditions \eqref{con} are equivalent to
$\zeta \ge  1\vee L$.

\begin{thm}[{\bf Average progeny of a self-similar process, \cite{KZ18}}]
\label{wfss}
Consider a self-similar process $S_{p,\gamma,\zeta}(s)$
with $0<p<1$ and $\gamma>0$.
Suppose that \eqref{eq:L} is satisfied and $\zeta \ge 1\vee L$.
Then
$$C_{p,\gamma,\zeta}(s)~\begin{cases}
   \text{decreases}    & \text{ if } p>1-{\zeta \over R}, \\
   = 1 & \text{ if  } p=1-{\zeta \over R}, \\
   \text{increases} & \text{ if  } p<1-{\zeta \over R}.
\end{cases}$$
\end{thm}
\begin{proof}
The choice of the limits for $\zeta$ ensures that the conditions \eqref{con}
are satisfied and hence, by Theorem~\ref{thm:hydro}, the hydrodynamic limit $x(s)$
exists
and the function $C_{p,\gamma,\zeta}(s)$ is well defined.
Now we have
\[~\Lambda \pi={\gamma p \over 1-p}\sum\limits_{K=1}^\infty \big(\zeta^{-1}(1-p)\big)^K e_K,\]
and  therefore 
\begin{equation}\label{eq:GLt}
\mathbb{G}\Lambda \pi = \hat{t}\big(\zeta^{-1} (1-p)\big) \Lambda \pi.
\end{equation} 
Iterating recursively, we obtain
$$(\mathbb{G}\Lambda)^2 \pi = \hat{t}\big(\zeta^{-1} (1-p)\big) \mathbb{G}\Lambda^2\pi= \hat{t}\big(\zeta^{-1} (1-p)\big) \hat{t}\big(\zeta^{-2} (1-p)\big) \Lambda^2 \pi,$$
and in general,
$$(\mathbb{G}\Lambda)^m \pi = \hat{t}\big(\zeta^{-1}(1-p)\big) \mathbb{G}\Lambda^m\pi= \left[\prod\limits_{i=1}^m \hat{t}\big(\zeta^{-i} (1-p)\big)\right] \Lambda^m \pi.$$
Thus, taking $x(0)=\pi$,
\begin{equation} \label{eq:solx}
x(s)=e^{\mathbb{G}\Lambda s} \pi=\pi+\sum\limits_{m=1}^\infty {s^m \over m!} \left[\prod\limits_{i=1}^m \hat{t}\big(\zeta^{-i}(1-p)\big)\right] \Lambda^m \pi.
\end{equation}
The average progeny function for fixed values of $p \in (0,1)$, $\gamma>0$ and $\zeta \ge 1$ can therefore be expressed as
\begin{align} \label{eq:Cpgz}
C_{p,\gamma,\zeta}(s)&=\sum\limits_{j=1}^\infty x_j(s) \nonumber\\
&=1+\sum\limits_{m=1}^\infty {s^m \over m!} \left[\prod\limits_{i=1}^m \hat{t}\big(\zeta^{-i}(1-p)\big)\right] \sum\limits_{j=1}^\infty \big(\Lambda^m \pi\big)_j \nonumber \\
&=1+\sum\limits_{m=1}^\infty {\big(s\gamma/\zeta \big)^m \over m!} \left[\prod\limits_{i=1}^m \hat{t}\big(\zeta^{-i}(1-p)\big)\right] {p \over 1-\zeta^{-m}(1-p)}, 
\end{align}
since 
\begin{align*}
\sum\limits_{j=1}^\infty \big(\Lambda^m \pi\big)_j&=\sum\limits_{j=1}^\infty \lambda_j^m \pi_j=\sum\limits_{j=1}^\infty \gamma^m \zeta^{-jm} p(1-p)^{j-1}\\
&=\gamma^m \zeta^{-m} {p \over 1-\zeta^{-m}(1-p)}.
\end{align*}

Next, notice that by letting $p'=1-\zeta^{-1}(1-p)$, we have from (\ref{eq:Cpgz}) and the uniform convergence of the corresponding series for any fixed $M>0$ and $s \in [0,M]$, that
\begin{equation}\label{eq:Cprime}
{d \over ds}C_{p,\gamma,\zeta}(s)={\gamma \over \zeta} \hat{t}(1-p') C_{p',\gamma,\zeta}(s) \quad \text{ with }~ C_{p,\gamma,\zeta}(0)=C_{p',\gamma,\zeta}(0)=1.
\end{equation}
Observe that  $\zeta\geq 1$ implies $p' \geq p$ and $~C_{p',\gamma,\zeta}(s)\leq C_{p,\gamma,\zeta}(s)$. Also, observe that  
$$\hat{t}(1-p') ~\begin{cases}
   <0   & \text{ if } p>1-{\zeta \over R}, \\
   =0 & \text{ if  } p=1-{\zeta \over R}, \\
   >0 & \text{ if  } p<1-{\zeta \over R}, 
\end{cases}$$
as $\hat{t}$ is an increasing function on $[0,\infty)$ and  $\hat{t}\big(1/R\big)=0$. 
This leads to the statement of the theorem.
\end{proof}

\begin{Rem}\label{spec_case}
If $\zeta=1$, then $p'=p$ and equation (\ref{eq:Cprime}) 
has an explicit solution 
\[C_{p,\gamma,1}(s)=\exp\big\{s\gamma \hat{t}(1-p)\big\}.\]
Accordingly,
$$C_{p,\gamma,1}(s)~\begin{cases}
   \text{ exponentially decreases}    & \text{ if } p>1-R^{-1}, \\
   =1 \text{ for all }s \geq 0 & \text{ if  } p=1-R^{-1}, \\
   \text{ exponentially increases} & \text{ if  } p<1-R^{-1}.
\end{cases}$$
This case is further examined in Sect.~\ref{DTM}.
In general, the average progeny $C_{p,\gamma,\zeta}(s)$ may increase 
sub-exponentially for $p<1-{\zeta \over R}$. 
For example, if there is a nonnegative integer $d$ such that $\zeta^{d+1} < R$,
then for $p=1-{\zeta^{d+1} \over R}$ we have
$\hat{t}\big(\zeta^{-d-1} (1-p)\big)=0$.
Accordingly, \eqref{eq:solx} implies that $C_{p,\gamma,\zeta}(s)$ is a polynomial of degree $d$.
\end{Rem}

\begin{thm}[{\bf Criticality of a self-similar process, \cite{KZ18}}]
\label{HBPmain2}
Consider a self-similar process $S_{p,\gamma,\zeta}(s)$
with $0<p<1$, $\gamma>0$.
Suppose that \eqref{eq:L} is satisfied and $\zeta \ge 1\vee L$. 
Then the following conditions are equivalent:
\begin{itemize}
\item[(i)]
The process is {\it critical}.
\item[(ii)]
The process is {\it time-invariant}.
\item[(iii)]
The following relations hold:
$\zeta<R\quad\text{and}\quad p=p_c: = 1-\frac{\zeta}{R}.$
\end{itemize}
\end{thm}
\begin{proof}
(i)$\leftrightarrow$(iii) is established in Theorem~\ref{wfss}.
(ii)$\rightarrow$(i) is established in Prop~\ref{ti2c}.
(iii)$\rightarrow$(ii):
Observe that $\hat{t}\left(\zeta^{-1}(1-p)\right)=\hat{t}\left(R^{-1}\right)=0$.
Time invariance now follows from (\ref{eq:solx}).
\end{proof}

\begin{Rem}
By Thm.~\ref{HBPmain}, the product $\lambda_K\,p_K$ in a self-similar process
is given by
\[\lambda_K\,p_K = \frac{\gamma\,p}{1-p}\left(\frac{1-p}{\zeta}\right)^K\]
for some $0<p<1$, $\gamma>0$, and $\zeta\ge 1\vee L$.
Hence, a time-invariant process can be constructed, according to
Prop.~\ref{critcond} and \eqref{lpR}, by selecting
any sequence $\{T_k\}$ such that the unique real zero $w_0$ on $[0,1/2)$ of the respective 
function $\hat{t}(z)$ is given by
\[w_0 = R^{-1} = \zeta^{-1}\,(1-p).\]
Theorem~\ref{HBPmain2} states that this is the only possible way
to construct a time-invariant process, given that the process
is self-similar.
\end{Rem}

\subsection{Closed form solution for equally distributed branch lengths}
\label{DTM}
Consider a self-similar hierarchical branching process with $\Lambda=I$ and $x(0)=e_K$ 
for a given integer $K \geq 1$. 
In other words, we assume $\lambda_j=1$ for all $j\ge 1$, which implies
$\gamma = \zeta=1$.

In this case, the system of equation (\ref{gFODEsys}) is finite dimensional,
\begin{equation}\label{FODEsys}
\begin{cases}
x'_1(s) & =  -x_1 +(T_1+2)x_2+T_2x_3+\hdots +T_{K-1}x_K \\ 
x'_2(s) & =  -x_2 +(T_1+2)x_3+T_2x_4+\hdots +T_{K-2}x_K \\ 
& \vdots \\
x'_{K-1}(s) & =  -x_{K-1} +(T_1+2)x_K \\
x'_K(s) & =  -x_K 
\end{cases}
\end{equation}
with the initial conditions $x(0)=e_K$. 

Recall the sequence $t(j)$ defined in Eq. (\ref{seq:tj}), and let $y(s)=e^s x(s)$. Then (\ref{FODEsys}) rewrites in terms of the coordinates of $y(s)$ as follows
\begin{equation}\label{yODEsys}
\begin{cases}
y'_1(s) & =  t(1)y_2+t(2)y_3+\hdots +t(K-1)y_K  \\ 
y'_2(s) & =  t(1)y_3+t(2)y_4+\hdots +t(K-2)y_K  \\ 
& \vdots \\
y'_{K-2}(s) & =  t(1)y_{K-1}+t(2)y_K \\
y'_{K-1}(s) & =  t(1)y_K \\
y'_K(s) & =  0 
\end{cases}
\end{equation}
with the initial conditions $y(0)=e_K$. The ODEs (\ref{yODEsys}) can be solved recursively in a reversed order of equations in the system obtaining for $m=1,\hdots,K-1$,
$$y_{K-m}(s)=\sum\limits_{n=1}^m \left(\sum\limits_{\substack{i_1,\hdots,i_n \geq 1\\ i_1+\hdots+i_n=m}} t(i_1)\cdot \hdots \cdot t(i_n) \right){s^n \over n!}.$$
Let $\delta_0(j)={\bf 1}_{\{j=0\}}$ be the Kronecker delta function. 
Then we arrive with the closed form solution
\begin{align}\label{eq:convt}
x_{K-m}(s)&=e^{-s}y_{K-m}(s) \nonumber\\
&=e^{-s}\sum\limits_{n=1}^{\infty} \underbrace{(t+\delta_0)*(t+\delta_0)*\hdots*(t+\delta_0)}_{n \text{ times}} (m) {s^n \over n!}.
\end{align} 
Observe that if we randomize the orders of trees by assigning an order $K$ to a tree with geometric probability $p_K=p(1-p)^{K-1}$, 
then the above closed form expression (\ref{eq:convt}) would yield an expression 
for the average progeny that was observed in Remark \ref{spec_case} of this section:
\begin{eqnarray*}
C(s) & = & e^{-s}+e^{-s}\sum\limits_{n=1}^\infty \sum_{m=1}^\infty (1-p)^m ~\underbrace{(t+\delta_0)*(t+\delta_0)*\hdots*(t+\delta_0)}_{n \text{ times}} (m) ~{s^n \over n!} \\ 
&=&  e^{-s}+e^{-s}\sum\limits_{n=1}^\infty \Big(\hat{t}(1-p)+1\Big)^n ~{s^n \over n!}   
=  \exp\left\{s\hat{t}(1-p)\right\}.
\end{eqnarray*}

\subsection{Critical Tokunaga process}
\label{sec:Tok}
We introduce here a class of hierarchical branching processes that enjoy all 
of the symmetries discussed in this work -- Horton self-similarity, criticality, 
time-invariance, strong Horton law, Tokunaga self-similarity, and also have 
independently distributed edge lengths. 
Despite these multiple constraints, the class is sufficiently broad,
allowing the self-similarity constant $\zeta$ (Def.~\ref{def:distss}, part (iv))
to take any value $\zeta\ge1$, and the Horton exponent to take
any value $R\ge 2$.
The critical binary Galton-Watson process is a special case of this class. 

\begin{Def}[{{\bf Critical Tokunaga process}}]\label{def:TokProcess}
We say that $S(t)$ is a {\it critical Tokunaga process} with
parameters ($\gamma$, $c$), and
write $S(t)\stackrel{d}{\sim}S^{\rm Tok}(t;c,\gamma)$, if it is 
a hierarchical branching process with the following parameter triplet:
\begin{equation}\label{eq:Tok}
\lambda_j=\gamma\,c^{1-j}, ~ p_K=2^{-K}, ~ \text{ and }~T_{k}=(c-1)\,c^{k-1}
\end{equation}
for some $\gamma>0,~c\ge1$. 
\end{Def}
\index{critical Tokunaga process}

\begin{prop}[{{\bf Critical Tokunaga process}}]
\label{Tok_tree}
Suppose $S(t)\stackrel{d}{\sim}S^{\rm Tok}(t;c,\gamma)$ and let $T[S]$ be the tree of $S(t)$.
Then, 
\begin{enumerate}
\item $S(t)$ is a Horton self-similar, critical, and time invariant process
\[S(t)\stackrel{d}{\sim}S_{{1\over2},\gamma,c}(t).\]
\item Independently of the combinatorial shape of $T[S]$, its edge lengths 
are i.i.d. exponential random variables with rate $\gamma$.
\item We have
\[\hat{t}(z) = \frac{(1-2\,c\,z)(z-1)}{1-c\,z},~
R = w_0^{-1} = 2\,c,~\zeta = L = c,\text{ and }
p = 2^{-1}.\]
\end{enumerate}
\end{prop}
\begin{proof}
1. Self-similarity follows from Thm.~\ref{HBPmain}, part (3).
Specification of parameters \eqref{eq:Tok} implies $p=2^{-1}$ and $\zeta = c$.
The Horton exponent $R=2c$ is found from \eqref{HortonR}.  
Criticality and time-invariance now follow from Thm.~\ref{HBPmain2},
since here 
\[2^{-1} = p = 1 - \frac{\zeta}{R} = 1-\frac{c}{2c} = 2^{-1}.\]

2. To establish the edge lengths property, observe that
\[\{T_0=1,T_k = (c-1)c^{k-1},k\ge1\}\Rightarrow S_K = 1+T_1+\dots +T_K = c^K, K\ge 0.\]
Recall from Prop.~\ref{HBP:branch}, part(4) that the edge lengths within a branch
of order $K\ge 1$ are i.i.d. exponential r.v.s with rate
\[\lambda_KS_{K-1} = \gamma\,c^{1-K}c^{K-1} = \gamma.\]

3. The values of $R$, $p$, and $\zeta$ are found in 1. 
The expression for $\hat{t}(z)$ and equality $L=c$ are readily obtained from the
geometric form of the Tokunaga coefficients $T_k$.
\end{proof}

Criticality and i.i.d. edge length distribution property characterize the critical 
Tokunaga process, as we explain in the following statement.
\begin{lem}
\label{lem:iidedges}
Consider a self-similar hierarchical branching process 
$S(t) \!\!\stackrel{d}{\sim} \!\!S_{p,\gamma,\zeta}(t)$
with $p\in(0,1)$ and $\gamma>0$. Suppose that \eqref{eq:L} holds and $\zeta\ge 1\vee L$. 
Let $T[S]$ be the tree of $S(t)$.
Then, the following conditions are equivalent:
\begin{enumerate}
\item $S(t)$ is critical and the edges in $T$ have i.i.d. exponential lengths 
with rate $\gamma>0$. 
\item $S(t)$ is a critical Tokunaga process:
$S(t)\stackrel{d}{\sim}S_{{1\over2},\gamma,c}(t).$ 
\end{enumerate}
\end{lem}
\begin{proof}
The implication ($2\Rightarrow1$) was established in Prop.~\ref{Tok_tree}.
To show ($1\Rightarrow2$), recall from Prop.~\ref{HBP:branch}, Eq. \eqref{edges}, that the 
edge lengths within a branch of order $K$ are i.i.d. with rate $\lambda_KS_{K-1}$.
If the rate is independent of $K$, we have for any $K\ge1$:
\[\lambda_KS_{K-1} = \lambda_{K+1}S_K\]
or
\[\zeta=\frac{\lambda_K}{\lambda_{K+1}}=\frac{S_K}{S_{K-1}}.\]
Given $S_0=1$, we find $S_K = \zeta^K$, and hence 
$T_K = (\zeta-1)\zeta^{K-1}.$
By \eqref{HortonR}, the Horton exponent is $R = 2\zeta$.
Criticality implies (Prop.~\ref{HBPmain2}, part (iii)):
\[p_c = 1-\frac{\zeta}{R} = 2^{-1},\]
which completes the proof.
\end{proof}

It follows from the proof of Lemma~\ref{lem:iidedges} that the i.i.d. edge length
property alone (and no criticality) is equivalent to the following constraints on
the process parameters:
\[\lambda_j=\gamma\,\zeta^{1-j},\quad\text{ and } \quad T_k=(\zeta-1)\zeta^{k-1},\]
while allowing an arbitrary choice of $p\in(0,1)$.
The tree of such process is Tokunaga self-similar, although not critical
unless $p=2^{-1}$.

The next results shows that the critical binary Galton-Watson tree ${\sf GW}(\lambda)$ 
with i.i.d. exponential edge lengths is a 
special case of the critical Tokunaga process.

\begin{thm}[{\bf Critical binary Galton-Watson tree, \cite{KZ18}}]
\label{HBPmain3}
Suppose $S(t)$ is a critical Tokunaga process with parameters
\begin{equation}\label{eq:GW}
\lambda_j=\gamma 2^{1-j}, ~ p_K=2^{-K},~~ \text{and}~~ T_{k}=2^{k-1}
~\text{for~some~}\gamma>0,
\end{equation} 
which means $S(t)\stackrel{d}{\sim}S^{\rm Tok}(t;2,\gamma)$.
Let $T[S]$ be the tree of $S(t)$.
Then $T[S]$ has the same distribution on $\BL^|$ as the critical binary Galton-Watson tree 
with i.i.d. edge lengths: $T[S]\stackrel{d}{\sim}{\sf GW}(\gamma)$.
\end{thm}
 
\begin{proof}
Consider a tree $T\stackrel{d}{\sim}{\sf GW}(\gamma)$ in $\BL^|$. 
We show below that this tree can be dynamically generated according to
Def.~\ref{def:HBP} of the hierarchical branching process
with parameters \eqref{eq:GW}.

First, notice that by Prop.~\ref{prop:BWW}
\[{\sf P}({\sf ord}(T) = K) = 2^{-K}.\]
We will establish later in  Corollary~\ref{cor:GW} that the length of every 
branch of order $j$ in $T$ is exponentially 
distributed with parameter $\lambda_j=\gamma 2^{1-j}$, which matches the
branch length distribution in the hierarchical branching process \eqref{eq:GW}.
Furthermore, by Corollary~\ref{cor:GW}, conditioned on $\cR^{i}(T)\ne \phi$ (which happens with a positive probability), 
we have $\cR^{i}(T)\stackrel{d}{\sim}{\sf GW}(2^{-i}\gamma)$.
This means that the distribution of Galton-Watson trees pruned $i$ times is a 
linearly scaled version of the original distribution 
(the same combinatorial structure, linearly scaled edge lengths).
Recall (Prop.~\ref{prop:BWW}) the total number $m_j$ of side branches 
within a branch of order $j\ge 2$ in $T$ is geometrically distributed
with mean $T_1+\dots+T_{j-1}=2^{j-1}-1$, where $T_i=2^{i-1}$, $i\ge 1$.
Conditioned on $m_j$, the assignment of orders among the $m_j$ side-branches 
is done according to the multinomial distribution with $m_j$ trials and 
success probability for order $i=1,\dots,j-1$ given by
$T_{j-i}/(T_1+\dots+T_{j-1})$.
This implies that 
the leaves of the original tree merge into every branch of the pruned tree as a Poisson 
point process with intensity $\gamma=\lambda_j T_{j-1}$.
Iterating this pruning argument, 
the branches of order $i$ merge into any branch of order $j$ in the pruned tree $\cR^i(T)$ as 
a Poisson point process with intensity $\gamma\,2^{-i}=\lambda_j T_{j-i}$ for
every $j>i$.

Finally, the orientation of the two offspring of the same parent in ${\sf GW}(\gamma)$
is uniform random,
by Def.~\ref{def:cbinary}.
We conclude that tree ${\sf GW}(\gamma)$ has the same distribution  
on $\BL^|$ as the critical Tokunaga process with parameters (\ref{eq:GW}).
\end{proof}

\begin{Rem}\label{rem:a_cminus1}
The condition $T_{i,i+k}=T_k=a\,c^{k-1}$ was first introduced in hydrology by
Eiji Tokunaga \cite{Tok78} in a study of river networks, hence the process 
name.
The additional constraint $a=c-1$ is necessitated here by the self-similarity
of tree lengths, which requires the sequence $\lambda_j$ to be geometric. 
The sequence of the Tokunaga coefficients then also has to be geometric, and
satisfy $a=c-1$, to ensure identical distribution of the edge lengths, see Prop.~\ref{HBP:branch}(4).
Recall the special place case $a=c-1$ plays for the entropy rate of Tokunaga self-similar trees as elaborated in Sect. \ref{sec:entropy}.
See Cor.~\ref{cor:Chunikhina}.
Interestingly, the constraint $a=c-1$ appears in the {\it random self-similar network} (RSN) model introduced by Veitzer
and Gupta \cite{VG00}, which uses a purely topological algorithm
of recursive local replacement of the network generators to 
construct random self-similar trees. 
The importance of the constraint $a=c-1$ in purely combinatorial
context is revealed in Sect.~\ref{sec:combHBP}. 
\end{Rem}

\subsection{Martingale approach}
\label{sec:HBPmartingale}
In this section, we propose a martingale representation for the size and length
of a critical Tokunaga tree of a given order.
This leads, via the martingale techniques, to the strong Horton laws for both these quantities, and allows us to
find the asymptotic order of a tree of a given size. 
The proposed martingale representation is related to an alternative construction of a 
critical Tokunaga tree, via a Markov tree process on $\BL^|$.

\subsubsection{Markov tree process}
\label{sec:Markov_tree}
Consider a critical Tokunaga process $S^{\rm Tok}(t;c,\gamma)$ 
(Def. \ref{def:TokProcess}) with $c>1$ (hence excluding a trivial case $c=1$ of perfect
binary trees), and let $\mu$ be the measure induced by this process on $\BL^|$.
Following the notations introduced in Sect. \ref{sec:dss}, Eq. \eqref{eq:muk}, we 
consider conditional measures
$$\mu_K(T) = \mu(T\,|{\sf ord}(T)=K).$$ 
Next, we construct a discrete time Markov tree process $\big\{\Upsilon_K\big\}_{K \in \mathbb{N}}$
on $\BL^|$ such that for each $K \in \mathbb{N}$,
\be\label{eqn:UpsilonK}
{\sf ord}(\Upsilon_K) = K, \quad \Upsilon_K \stackrel{d}{\sim} \mu_K, ~\text{ and }~\cR(\Upsilon_{K+1})=\Upsilon_K.
\ee 
Let
$$X_K=N_1[\Upsilon_K]={1+\# \Upsilon_K \over 2} \in \mathbb{N}$$ 
be the number of leaves in $\Upsilon_K$ and 
$Y_K=\textsc{length}(\Upsilon_K) \in \mathbb{R}_+$ be the tree length.
We let $\Upsilon_1$ be an I-shaped tree of Horton-Strahler order one, with
the edge length $Y_1\stackrel{d}{\sim}{\sf Exp}(\gamma)$.
This tree has one leaf, $X_1=1$.

Conditioned on $\Upsilon_K$, the tree $\Upsilon_{K+1}$ is 
constructed according to the following transition rules.    
Denote by $\Upsilon_K'$ the tree $\Upsilon_K$ with edge length scaled by $c$. 
That is, the tree $\Upsilon_K'$ is obtained by multiplying the edge lengths in $\Upsilon_K$ by $c$, 
while preserving the combinatorial shape and planar embedding:
$$\textsc{p-shape}(\Upsilon_K')=\textsc{p-shape}(\Upsilon_K).$$
Next, we attach new leaf edges to $\Upsilon_K'$ at the points 
sampled by a Poisson point process with intensity $\gamma(c-1)c^{-1}$ along $\Upsilon_K'$.  
We also attach a pair of new leaf edges to each of the leaves in $\Upsilon_K'$;
there is exactly $2X_K$ such attachments ($X_K$ pairs).
The lengths of all the newly attached leaf edges are i.i.d. exponential random variables with parameter $\gamma$. 
The left-right orientation of the newly added edges is determined independently and uniformly.
Finally, the tree $\Upsilon_{K+1}$ consists of $\Upsilon_K'$ and all the attached leaves and leaf edges. 

\begin{lem}
\label{lem:Markov_tree}
The process $\big\{\Upsilon_K\big\}_{K \in \mathbb{N}}$ is a Markov process that 
satisfies \eqref{eqn:UpsilonK}.
\end{lem}
\begin{proof}
The process construction readily implies the Markov property, and
ensures that ${\sf ord}(\Upsilon_K)=K$ and $\cR(\Upsilon_{K+1})=\Upsilon_K$.
Next, we show that a random tree $\Upsilon_K$ satisfies 
Def.~\ref{def:HBP}, conditioned on the tree order $K\ge 1$, with the critical
Tokunaga parameters 
\[\lambda_j=\gamma c^{1-j}\quad\text{and}\quad T_k = (c-1)c^{k-1}.\]
The tree $\Upsilon_1$ has exponential edge length with parameter $\lambda_1=\gamma$
and no side branching, hence $\Upsilon_1\stackrel{d}{\sim} \mu_1$.
Assume now that $\Upsilon_K\stackrel{d}{\sim} \mu_K$ for some $K\ge 1$
and establish each of the properties of Def.~\ref{def:HBP}, except the tree order property (i), 
for $\Upsilon_{K+1}$.

Property Def.~\ref{def:HBP}(ii). 
Fix any $j$ such that $1<j\le K$. 
Every branch of order $j$ in $\Upsilon_{K+1}$ is formed by 
a branch of order $j-1$ in $\Upsilon_{K}$.
In particular, the length of the branch is multiplied by $c$.
Accordingly, every branch of order $j$ within $\Upsilon_{K+1}$ 
produces offspring of every order $i$ such that $1<i<j$ with rate 
\[c^{-1}\left(\lambda_{j-1} T_{(j-1)-(i-1)}\right)  = c^{-1}\gamma c^{1-(j-1)} T_{j-i} 
= \gamma c^{1-j} T_{j-i} =\lambda_j T_{j-i}.\] 
By construction, the side branches of order $i=1$ are generated with rate 
\[\gamma(c-1)c^{-1} = \lambda_j T_{j-1}.\] 
This establishes property (ii).

Property Def.~\ref{def:HBP}(iii). 
Using the same argument as above, each branch of order $j>1$ in $\Upsilon_{K+1}$ terminates with rate
$c^{-1}\lambda_{j-1} = \lambda_j.$
By construction, each branch of order $i=1$ terminates with rate
$\gamma = \lambda_1.$
This establishes property (iii).

Properties Def.~\ref{def:HBP}(iv,v,vi) 
follow trivially from the process construction.
This completes the proof.

%
\end{proof}

\noindent
Notice that sampling a random variable $\kappa \stackrel{d}{\sim} {\sf Geom}_1\left({1 \over 2}\right)$ independently of the process $\Upsilon_K$, we have
the stopped process $\Upsilon_\kappa \stackrel{d}{\sim} \mu$.

\subsubsection{Martingale representation of tree size and length}
\label{sec:martingale}
By construction, the pairs $(X_K,Y_K)$ and $(X_{K+1},Y_{K+1})$ are related in an iterative way as follows.
Conditioned on the values of $(X_K,Y_K)$, we have
 \be\label{eqn:TokMartingaleX}
X_{K+1}=2X_K+V_K,
\ee
where 
$V_K\stackrel{d}{\sim}{\sf Poi}\big(\gamma (c-1)Y_K\big)$ is the number of 
side branches of order one attached to $\Upsilon'_K$.  
Next, conditioning on $X_{K+1}$, we have
 \be\label{eqn:TokMartingaleY}
Y_{K+1}=U_K+cY_K,
\ee
where 
$U_K\stackrel{d}{\sim}{\sf Gamma}\big(X_{K+1},\gamma\big)$
is the sum of $X_{K+1}$ i.i.d. edge lengths, 
each exponentially distributed with parameter $\gamma$.

\medskip
\noindent
\begin{lem}[{{\bf Martingale representation}}]
\label{lem:TokMartingale}
The sequence
\be\label{eqn:TokMartingale4}
M_K=R^{1-K}\left(\vphantom{I^{I^I}}X_K+\gamma(c-1)Y_K\right)\text{ with } K\in\mathbb{N}
\ee
is a martingale with respect to the Markov tree process $\big\{\Upsilon_K\big\}_{K \in \mathbb{N}}$.
\end{lem}
\begin{proof}
Taking conditional expectations in \eqref{eqn:TokMartingaleX} and \eqref{eqn:TokMartingaleY} gives 
\begin{eqnarray}
\label{eqn:TokMartingale1}
\E [X_{K+1}\,|\Upsilon_K]&=&2X_K+\gamma (c-1)Y_K,\\
\label{eqn:TokMartingale2}
\E [Y_{K+1}\,|\Upsilon_K]&=&\gamma^{-1} \E [X_{K+1}\,|\Upsilon_K]+cY_K\nonumber\\
&=&2\gamma^{-1}X_K+(2c-1)Y_K.
\end{eqnarray}
This can be summarized as
\be\label{eqn:TokMartingale3}
\E\left[\left(\!\!\begin{array}{c}X_{K+1} \\Y_{K+1}\end{array}\!\!\right) \,\Big|\Upsilon_K \right]=
\mathbb{M}\left(\!\!\begin{array}{c}X_K \\Y_K\end{array}\!\!\right),
\ee
where
\[\mathbb{M} = 
\left[\!\!\begin{array}{cc}2 & \gamma (c-1) \\2\gamma^{-1} & 2c-1\end{array}\!\!\right].\]
The eigenvalues of the matrix $\mathbb{M}$ are $R=2c$ and $1$.
The largest eigenvalue equals the Horton exponent $R$; the respective
eigenspace is  
$y=2\gamma^{-1}x$.
Equation \eqref{eqn:TokMartingale3} implies that
$$\mathbb{M}^{1-K}\!\left(\!\!\begin{array}{c}X_K \\Y_K\end{array}\!\!\right)={1 \over 2c-1}\left(\!\!\begin{array}{c}[R^{1-K}+2(c-1)]X_K+\gamma(c-1)[R^{1-K}-1]Y_K \\  \\ 2\gamma^{-1}[R^{1-K}-1]X_K+[2(c-1)R^{1-K}+1]Y_K\end{array}\!\!\right)$$
 is a vector valued martingale with respect to the Markov tree 
 process $\big\{\Upsilon_K\big\}_{K \in \mathbb{N}}$. 
Multiplying this martingale by the left eigenvector $\Big(1,\,\gamma(c-1)\Big)$ of
$\mathbb{M}$ that corresponds to the largest eigenvalue $R$, 
 we obtain a scalar martingale with respect to $\big\{\Upsilon_K\big\}_{K \in \mathbb{N}}$:
$$\Big(1,\,\gamma(c-1)\Big)\mathbb{M}^{1-K}\!\left(\!\!\begin{array}{c}X_K \\Y_K\end{array}\!\!\right)=R^{1-K}\left(X_K+\gamma(c-1)Y_K\right).$$
This completes the proof.
\end{proof}
 
\medskip
\noindent
\begin{lem}\label{lem:YoverXas}
Suppose $\mu=S^{\rm Tok}(t;c,\gamma)$ is the distribution of a critical Tokunaga process,
and $\big\{\Upsilon_K\big\}_{K \in \mathbb{N}}$ is the corresponding Markov tree process.
Then, 
\be\label{eqn:TokMartingale5}
Y_K/X_K \rightarrow 2\gamma^{-1}~a.s. ~\text{ as }~K \rightarrow \infty.
\ee 
\end{lem}
\begin{proof}
Recall that $Y_K$ is a sum of $2X_K-1$ independent edge lengths, each being exponentially distributed with parameter $\gamma$.
Thus, since $X_K=N_1[\Upsilon_K] \geq 2^{K-1}$, the Chebyshev inequality implies for any $\epsilon>0$,
\begin{align*}
\sum\limits_{k=1}^\infty {\sf P}\left(\left|{Y_K \over X_K} -2\gamma^{-1}\right| \geq \epsilon \right) &\leq \epsilon^{-2} \sum\limits_{k=1}^\infty \Var \left({Y_K \over X_K} -2\gamma^{-1}\right) \\
&\leq \epsilon^{-2} \sum\limits_{k=1}^\infty \E \left[ \E \left[\left({Y_K \over X_K} -2\gamma^{-1}\right)^2 \,\Big| X_K \right] \right]\\
&= \epsilon^{-2} \gamma^{-2} \sum\limits_{k=1}^\infty \E \left[ X_K^{-1}+X_K^{-2}\right]\\
&\leq \epsilon^{-2} \gamma^{-2} \sum\limits_{k=1}^\infty \left(2^{1-K}+2^{2(1-K)}\right) ~~<\infty.
\end{align*}
as $\E[Y_K\,|X_K]=2\gamma^{-1}X_K-\gamma^{-1}$ and $\E[Y_K^2\,|X_K]=4\gamma^{-2}X_K^2-3\gamma^{-2}X_K+\gamma^{-1}$.

\noindent
Hence, by the Borel-Cantelli lemma, we arrive with the almost sure convergence in \eqref{eqn:TokMartingale5}.
\end{proof} 

\medskip
\noindent
\begin{lem}\label{lem:LimNotZero}
Suppose $\mu=S^{\rm Tok}(t;c,\gamma)$ is the distribution of a critical Tokunaga process,
and $\big\{\Upsilon_K\big\}_{K \in \mathbb{N}}$ is the corresponding Markov tree process.
Then,
$${\sf P}\left(\lim\limits_{K \rightarrow \infty} R^{1-K}X_K=0 \right)=0.$$
\end{lem}
\begin{proof}
For a given integer $x \geq 2^{K-1}$, we condition on the event $X_K =x$.
Then, $Y_K$ is a sum of $2X_K-1=2x-1$ i.i.d. exponential edge lengths. 
Hence,  $Y_K\stackrel{d}{\sim}{\sf Gamma}\big(2x-1,\gamma\big)$. Finally, recall that in the setup of \eqref{eqn:TokMartingaleX}, 
$V_K\stackrel{d}{\sim}{\sf Poi}\big(\gamma (c-1)Y_K\big)$. 
Therefore, we can compute the moment generating function of $V_K$ conditioned on the event $X_K =x$ as follows
\begin{align}\label{eqn:mgfVK}
\mathcal{M}_{\sf v}(s;x)&:=\E\left[e^{sV_K}\,\big|X_K=x\right] \nonumber \\
&=\int\limits_0^\infty \sum\limits_{k=0}^\infty e^{sk} e^{-\gamma (c-1)y} {\big(\gamma (c-1)y\big)^k \over k!} {\gamma^{2x-1}y^{2x-2}e^{-\gamma y} \over \Gamma(2x-1)}\,dy \nonumber \\
&=\int\limits_0^\infty e^{-\gamma\big(c-(c-1)e^s\big)} {\gamma^{2x-1}y^{2x-2} \over \Gamma(2x-1)}\,dy \nonumber \\
&={1 \over \big(c-(c-1)e^s\big)^{2x-1}},
\end{align}
with the domain $s \in \left(-\infty,\,\log{c \over c-1}\right)$.

\medskip
\noindent
Next, we use \eqref{eqn:mgfVK} in the exponential Markov inequality (a.k.a. Chernoff bound).
For a given $\varepsilon \in \Big(0,(c-1)c^{-1}\Big)$ and $x \geq 2^{K-1}$, by \eqref{eqn:TokMartingaleX} we have, for all $s \geq 0$,
\begin{eqnarray}\label{eqn:VK_Chernoff1}
\lefteqn{{\sf P}\left({X_{K+1} \over RX_K} \leq 1-\varepsilon \,\big|X_K=x\right)}\nonumber\\ 
&=&{\sf P}\left(-sV_K \geq -2s\big((1-\varepsilon)c-1\big)x \,\big|X_K=x\right) \nonumber \\
&\leq& e^{2s\big((1-\varepsilon)c-1\big)x} \mathcal{M}_{\sf v}(-s;x) \nonumber \\
&=& {e^{2s\big((1-\varepsilon)c-1\big)x} \over \big(c-(c-1)e^{-s}\big)^{2x-1}} \nonumber \\
&=&\big(c-(c-1)e^{-s}\big)\left({e^{(1-\varepsilon)cs} \over ce^s-(c-1)} \right)^{2x}.
\end{eqnarray}

\medskip
\noindent
We find the extreme value of ${e^{(1-\varepsilon)cs} \over ce^s-(c-1)}$ in \eqref{eqn:VK_Chernoff1}, 
and substitute 
$$e^s={(1-\varepsilon)(c-1) \over (1-\varepsilon)c-1}={1-\varepsilon \over 1-{c \over c-1}\varepsilon}$$
into the right hand side of \eqref{eqn:VK_Chernoff1}, obtaining
\begin{eqnarray}\label{eqn:VK_Chernoff2}
\lefteqn{{\sf P}\left({X_{K+1} \over RX_K} \leq 1-\varepsilon \,\big|X_K=x\right)}\nonumber\\
&\leq& \big(c-(c-1)e^{-s}\big)\left({e^{(1-\varepsilon)cs} \over ce^s-(c-1)} \right)^{2x} 
\nonumber \\
&=&(1-\varepsilon)^{-1}\left(\left(1-{c \over c-1}\varepsilon\right)\left({1-\varepsilon \over 1-{c \over c-1}\varepsilon}\right)^{(1-\varepsilon)c} \right)^{2x} \nonumber \\
&=&(1-\varepsilon)^{-1} \exp\left\{-x\left({c \over c-1}\varepsilon^2+O(\varepsilon^3)\right)\right\}.
\end{eqnarray}

\medskip
\noindent
Now, since $X_K \geq 2^{K-1}$, \eqref{eqn:VK_Chernoff2} implies
\begin{align}\label{eqn:XKupperbound}
{\sf P}\left({X_{K+1} \over RX_K} \leq 1-\varepsilon \right)
&=\sum\limits_{x=2^{K-1}}^\infty {\sf P}\left({X_{K+1} \over RX_K} \leq 1-\varepsilon \,\big|X_K=x\right) {\sf P}(X_K=x) \nonumber \\
&\leq  \exp\left\{-2^{K-1}\left({c \over c-1}\varepsilon^2+O(\varepsilon^3)\right)\right\}.
\end{align}

\medskip
\noindent
Next, plugging $\varepsilon =1-e^{-1/K^2}$ into \eqref{eqn:XKupperbound}, we find that
\be\label{eqn:ratioBC1}
\sum\limits_{K=1}^\infty {\sf P}\left({X_{K+1} \over RX_K} \leq e^{-1/K^2}\right) <\infty,
\ee
and equivalently, 
\be\label{eqn:ratioBC2}
\sum\limits_{K=1}^\infty {\sf P}\left(\log\left({R^{1-K}X_K \over R^{-K}X_{K+1}}\right) \geq {1 \over K^2}\right) <\infty,
\ee

\medskip
\noindent
Therefore, by the Borel-Cantelli lemma, 
\be\label{eqn:ratioBC}
{\sf P}\left(\left|\left\{K \in \mathbb{N}\,:\,\log\left({R^{1-K}X_K \over R^{-K}X_{K+1}}\right) \geq {1 \over K^2}\right\}\right|<\infty \right)=1,
\ee
where $|\cdot |$ denotes the magnitude of sets. Hence, as $\sum\limits_{K=1}^\infty  {1 \over K^2} <\infty$,
\begin{align}\label{eqn:ratioBCmain}
&{\sf P}\left(\lim\limits_{K \rightarrow \infty} R^{1-K}X_K=0 \right)
= {\sf P}\left(\lim\limits_{S \rightarrow \infty} \prod\limits_{K=1}^S {R^{1-K}X_K \over R^{-K}X_{K+1}}=\infty \right) \nonumber \\
&= {\sf P}\left(\prod\limits_{K=1}^\infty {R^{1-K}X_K \over R^{-K}X_{K+1}}=\infty \right) 
= {\sf P}\left(\sum\limits_{K=1}^\infty \log\left({R^{1-K}X_K \over R^{-K}X_{K+1}}\right)=\infty \right)=0.
\end{align}
This completes the proof.
\end{proof}

\subsubsection{Strong Horton laws in a critical Tokunaga tree}
\label{sec:Tok_Horton}
The martingale representation of Lemma~\ref{lem:TokMartingale}
has an immediate implication for the asymptotic behavior of the 
average size of a critical Tokunaga tree, stated below.

\begin{cor}[{{\bf Strong Horton law for mean branch numbers}}]
\label{cor:TokMartingaleSHL}
Suppose $\mu=S^{\rm Tok}(t;c,\gamma)$ is the distribution of a critical Tokunaga process with 
$c\ge 1$. Then, the following closed form expression holds
for all $1 \leq  k \leq K$:
\be\label{eqn:NkClosedForm}
{(2c-1)\cN_k[K]-(c-1) \over (2c-1)\cN_1[K]-(c-1)}= R^{1-k},\text{ with } R=2c.
\ee
Consequently, $\mu=S^{\rm Tok}(t;c,\gamma)$ satisfies the strong Horton law for mean branch numbers (Def. \ref{def:Horton_mean}).
The equation \eqref{eqn:NkClosedForm} implies, in particular, 
\be\label{eq:N1K}
\cN_1[K] = \frac{R^{K-1}c+c-1}{2c-1} = \frac{R^K+R-2}{2(R-1)}.
\ee
\end{cor}
\begin{proof}
Since $Y_K$ is a sum of $2X_K-1$ independent edge lengths,
each exponentially distributed with parameter $\gamma$, 
we have 
$\E[Y_K]=\gamma^{-1}(2\E[X_K]-1)$.
Therefore, 
\begin{align*}
\E[M_K]&=R^{1-K}\E[X_K]+\gamma(c-1)R^{1-K} \E[Y_K] \\
&=(2c-1)R^{1-K}\E[X_K]-(c-1)R^{1-K}.
\end{align*}
Furthermore, for all $1 \leq k \leq K$, substituting $K-k+1$ instead of $K$ in the above equation, we obtain
$$\E[M_{K-k+1}]=(2c-1)R^{k-K}\E[X_{K-k+1}]-(c-1)R^{k-K}.$$
Since $M_K$ is a martingale (see Lemma \ref{lem:TokMartingale}), we have
$\E[M_{K-k+1}]=\E[M_K]$. 
Hence,  
\begin{align*}
1&={\E[M_{K-k+1}] \over \E[M_K]}=R^{k-1}{(2c-1)\E[X_{K-k+1}]-(c-1) \over (2c-1)\E[X_K]-(c-1)} \\
&=R^{k-1}{(2c-1)\E\big[N_k[\Upsilon_K]\big]-(c-1) \over (2c-1) \E\big[N_1[\Upsilon_K]\big]-(c-1)} 
\end{align*}
as
$\E[X_{K-k+1}]=\E\big[N_k[\Upsilon_K]\big]$ and $\E[X_K]=\E\big[N_1[\Upsilon_K]\big]$.
This establishes \eqref{eqn:NkClosedForm}.
The strong Horton law \eqref{eq:Horton_mean} for mean branch numbers follows from \eqref{eqn:NkClosedForm}.
The expression \eqref{eq:N1K} is obtained by using $k=K$ in \eqref{eqn:NkClosedForm}.
This completes the proof.
\end{proof}

\noindent We also suggest an alternative proof that emphasizes the spectral property
of the transition matrix $\mathbb{M}$ of \eqref{eqn:TokMartingale3}.
\begin{proof}[Alternative proof of Corollary~\ref{cor:TokMartingaleSHL}]
Taking expectation in \eqref{eqn:TokMartingale3} we obtain, for any $K> 1$,
\be
\label{eq:EX}
\left(\begin{array}{l}\E[X_{K}]\ \\ \E[Y_{K}]\end{array}\right)=
\mathbb{M}\left(\begin{array}{l}\E[X_{K-1}] \\\E[Y_{K-1}]\end{array}\right)
= \mathbb{M}^{K-1}\left(\begin{array}{l}\E[X_1] \\\E[Y_1]\end{array}\right).
\ee
Since $Y_K$ is a sum of $2X_K-1$ independent edge lengths,
each exponentially distributed with parameter $\gamma$, 
we have 
$\E[Y_K]=\gamma^{-1}(2\E[X_K]-1)$.
Recall also that $\Big(1,\gamma(c-1)\Big)$ is the left eigenvector of $\mathbb{M}$ that corresponds
to the eigenvalue $R$. 
Accordingly,
\[\Big(1,\gamma(c-1)\Big)\mathbb{M}^{K-1} = R^{K-1}\Big(1,\gamma(c-1)\Big).\]
Premultiplying \eqref{eq:EX} by the eigenvector $\Big(1,\gamma(c-1)\Big)$ we hence obtain
\[(2c-1)\E[X_K]-(c-1) = R^{K-1}\Big((2c-1)\E[X_1]-(c-1)\Big),\]
which establishes \eqref{eqn:NkClosedForm}, since
$\E[X_{1}]=\E\big[N_K[\Upsilon_K]\big]$ and $\E[X_K]=\E\big[N_1[\Upsilon_K]\big]$.
The strong Horton law \eqref{eq:Horton_mean} for mean branch numbers follows from \eqref{eqn:NkClosedForm}.
The expression \eqref{eq:N1K} is obtained by using $k=K$ in \eqref{eqn:NkClosedForm}.
This completes the proof.
\end{proof}

\noindent The sizes of trees of distinct orders have fixed asymptotic ratios in 
a much stronger (almost sure) sense, as we show below.
\begin{thm}
\label{thm:DoobMartingaleSHL}
Suppose $\mu=S^{\rm Tok}(t;c,\gamma)$ is the distribution of a critical Tokunaga process,
and $\big\{\Upsilon_K\big\}_{K \in \mathbb{N}}$ is the corresponding Markov tree process.
Then,
\be\label{eqn:TokMartingale6}
{N_k[\Upsilon_K] \over N_1[\Upsilon_K]}\stackrel{a.s.}{\to}R^{1-k} \quad \text{as } K\to\infty.
\ee
\end{thm}
\begin{proof}
Recall that by Lemma \ref{lem:TokMartingale}, $M_K$ defined in \eqref{eqn:TokMartingale4} is a martingale. 
Also, $M_K>0$ and is in $L^1$ for all $K \in \mathbb{N}$. Thus, by the Doob's Martingale Convergence Theorem,  $M_K$ converges almost surely. 
Hence, by \eqref{eqn:TokMartingale5}, $R^{1-K}X_K$ also converges almost surely, and 
\be\label{eq:Xlimit}
\lim\limits_{K \rightarrow \infty} R^{1-K}X_K=\lim\limits_{K \rightarrow \infty} {M_K \over 2c-1}.
\ee
In other words, for almost every trajectory of the process $\big\{\Upsilon_K\big\}_{K \in \mathbb{N}}$, we have
$R^{1-K}X_K=R^{1-K}N_1[\Upsilon_K]$ converging to a finite limit $V_\infty$, 
where $V_\infty$ is a random variable.
Hence, for any $k \in \mathbb{N}$, the random sequences 
$$R^{1-K}X_K=R^{1-K}N_1[\Upsilon_K]~~\text{ and }~~R^{k-K}X_{K-k+1}=R^{k-K}N_k[\Upsilon_K]$$ 
converge almost surely to the same finite $V_\infty$, where $V_\infty>0~a.s.$  by Lemma \ref{lem:LimNotZero}. The almost sure convergence \eqref{eqn:TokMartingale6} follows.
\end{proof}

\noindent
The almost sure convergence \eqref{eqn:TokMartingale6} in Theorem \ref{thm:DoobMartingaleSHL} implies the corresponding week convergence
$${\sf P}\left(\left|{N_k[\Upsilon_K] \over N_1[\Upsilon_K]}-R^{1-k}\right|>\epsilon\right) \to \,0\quad\text{ as }\quad K\to\infty,$$
via the Bounded Convergence Theorem. We restate it as the following corollary.
\begin{cor}[{{\bf Strong Horton law for branch numbers}}]
\label{cor:DoobMartingaleSHL}
The distribution $\mu=S^{\rm Tok}(t;c,\gamma)$ of a critical Tokunaga process satisfies the strong Horton law for branch numbers (Def. \ref{def:Horton_rv}). That is, for any $\epsilon>0$,
$$\mu_K\left(\left|\frac{N_k[T]}{N_1[T]}-R^{1-k}\right|>\epsilon\right) \to \,0\quad\text{ as }\quad K\to\infty.$$
\end{cor}

\begin{cor}[{{\bf Asymptotic tree order}}]
\label{cor:Tok_order}
Suppose $\mu=S^{\rm Tok}(t;c,\gamma)$ is the distribution of a critical Tokunaga process,
and $\big\{\Upsilon_K\big\}_{K \in \mathbb{N}}$ is the corresponding Markov tree process.
Then,
\[\frac{\log_R \#\Upsilon_K}{K} \stackrel{a.s.}{\to}1,\quad\text{as }K\to\infty.\]
\end{cor}
\begin{proof}
Recall from \eqref{eq:Xlimit} that 
\[R^{1-K}X_K\stackrel{a.s.}{\to}V_\infty,\]
where $V_\infty$ is finite (by Doob's Martingale Convergence Theorem) 
and $V_\infty >0~ a.s.$ by Lemma~\ref{lem:LimNotZero}.
Accordingly,
\be\label{eqn:logN1-K}
\log_R X_K-K ~\stackrel{a.s.}{\to} ~\log_R{V_\infty}\,-1,
\ee
with $-\infty<\log_R{V_\infty}<\infty~ a.s.$
Recalling that $ \#\Upsilon_K=2X_K-1$ completes the proof.
\end{proof}

\medskip
\noindent 
The almost sure convergence \eqref{eqn:TokMartingale5} allows to restate 
the limit results of this section in terms of the tree length $Y_K$.
\begin{cor}[{{\bf Strong Horton laws for tree lengths}}]\label{cor:Ymartingale}
Suppose $\mu=S^{\rm Tok}(t;c,\gamma)$ is the distribution of a critical Tokunaga process,
and $\big\{\Upsilon_K\big\}_{K \in \mathbb{N}}$ is the corresponding Markov tree process.
Then, for a tree $T\stackrel{d}{\sim}\mu$,
\be\label{cor:YK2}
\E\big[\textsc{length}(T)\,|{\sf ord}(T)=K\big]=\E[Y_K] = \frac{R^K-1}{\gamma(R-1)},\quad K\ge 1.
\ee
Furthermore, we have, for any $k\ge 1$,
\be\label{cor:YK3}
\frac{Y_{K-k}}{Y_K}\stackrel{a.s.}{\to}R^{-k},\text{ as } K\to\infty,
\ee
which implies the strong Horton law for tree lengths: for any $\epsilon>0$,
\be\label{cor:YK4}
\mu_K\left(\left|{\textsc{length}\big(\cR^k(T)\big) \over \textsc{length}(T)}-R^{-k}\right|>\epsilon\right){\to} 0
\quad\text{ as }\quad K\to\infty.
\ee
\end{cor}

\begin{ex}[{{\bf Critical binary Galton Watson tree}}]
Theorem \ref{HBPmain3} asserts that the critical binary Galton-Watson tree with 
exponential i.i.d. edge lengths, $T\stackrel{d}{\sim}{\sf GW}(\lambda)$, has the same distribution
as a critical Tokunaga branching process with $c=2$ and $\gamma=\lambda$. 
In this case $R=2c=4$ and the expressions \eqref{eqn:NkClosedForm}, \eqref{eq:N1K} give,
for any $K\ge 1$,
\[\cN_1[K] = \frac{4^{K}+2}{6}.\]
Fixing $\lambda=1$, by the expression
\eqref{cor:YK2} we have, for any $K\ge 1$,
\[
\E\big[\textsc{length}(T)\,|{\sf ord}(T)=K\big]= \frac{4^K-1}{3}.\]
Table~\ref{tab1} shows the values of the mean size and mean length
of a critical binary Galton-Watson tree $T\stackrel{d}{\sim}{\sf GW}(1)$,
conditioned on selected values of tree order.
\end{ex}

\begin{table}
\captionof{table}{Mean size, $\E[X_K] = \cN_1[K]$, and length, $\E[Y_K]=\E[\textsc{length}(T)]$, of a critical binary 
Galton-Watson tree $T\stackrel{d}{\sim}{\sf GW}(1)$; here $c=2$, $R=4$.} \label{tab1}
\begin{tabular}{ccccc}
\hline
\hline
\vspace{.1cm}
${\sf ord}(T)$  &  $\cN_1[K]$ &  $\E[\textsc{length}(T)]$ &  $2-\frac{\E[Y_K]\vphantom{I^I}}{\E[X_K]}$ &  $4-\frac{\E[X_K]}{\E[X_{K-1}]}$\\
\hline
$1 $ &                       $1$&                         $1$ &        $1$  &                     --				\\
$2 $ &                       $3$&                         $5$ &        $1/3$ &                    $1$				\\
$3 $ &                      $11$&                        $21$ &       $9\times 10^{-2}$ &        $1/3$				\\
$4 $ &                      $43$&                        $85$ &       $2\times 10^{-2}$ &        $9\times 10^{-2}$ 	\\ 
$5 $ &                     $171$&                       $341$ &       $6\times 10^{-3}$ &        $2\times 10^{-2}$ 	\\
$6 $ &                     $683$&                      $1365$ &       $1\times 10^{-3}$ &        $6\times 10^{-3}$ 	\\
$7 $ &                    $2731$&                      $5461$ &       $4\times 10^{-4}$ &        $1\times 10^{-3}$ 	\\ 
$8 $ &                   $10923$&                     $21845$ &       $9\times 10^{-5}$ &        $4\times 10^{-4}$ 	\\
$9 $ &                   $43691$&                     $87381$ &       $2\times 10^{-5}$ &        $9\times 10^{-5}$	\\
$10$ &                   $174763$&                    $349525$ &      $6\times 10^{-6}$ &        $2\times 10^{-5}$	\\
$11$ &                   $699051$&                   $1398101$ &      $1\times 10^{-6}$ &        $6\times 10^{-6}$	\\
$12$ &                  $2796203$&                   $5592405$ &      $4\times 10^{-7}$ &        $1\times 10^{-6}$	\\
$13$ &                 $11184811$&                  $22369621$ &      $9\times 10^{-8}$ &        $4\times 10^{-7}$	\\
$14$ &                 $44739243$&                  $89478485$ &      $2\times 10^{-8}$ &        $9\times 10^{-8}$	\\
$15$ &                $178956971$&                 $357913941$ &      $6\times 10^{-9}$ &        $2\times 10^{-8}$ 	\\
$16$ &                $715827883$&                $1431655765$ &      $1\times 10^{-9}$ &        $6\times 10^{-9}$	\\
$17$ &               $2863311531$&                $5726623061$ &      $3\times 10^{-10}$ &       $1\times 10^{-9}$	\\
$18$ &              $11453246123$&               $22906492245$ &      $9\times 10^{-11}$ &       $3\times 10^{-10}$	\\
$19$ &              $45812984491$&               $91625968981$ &      $2\times 10^{-11}$ &       $9\times 10^{-11}$	\\
$20$ &             $183251937963$&              $366503875925$ &      $5\times 10^{-12}$ &       $2\times 10^{-11}$	\\
\hline
\hline
\end{tabular}
\end{table}


\subsection{Combinatorial HBP: Geometric Branching Process}
\label{sec:combHBP}
This section focuses on combinatorial structure 
of a Horton self-similar hierarchical branching process
\cite{KZ18}
\[S(t)\stackrel{d}{\sim} {\sf HBP}\big(\{T_k\},\{\lambda_j\},\{p(1-p)^{K-1}\}\big).\]
Let $T[S]$ be the tree generated by $S(t)$ in $\BL^|$.
Section~\ref{sec:GBPdef} introduces a discrete 
time multi-type {\it geometric branching process} 
$\cG(s)={\cG}(s;\{T_k\},p)$ whose trajectories induce
a random tree $\cG(\{T_k\},p)$ on $\cBT^|$ such that
\be\label{G2T}
{\cG}(\{T_k\},p) \stackrel{d}{=}\textsc{shape}\big(T[S]\big)\in\cBT^|.
\ee
We then show in Sect.~\ref{sec:TokTI} that geometric branching process 
is time invariant (in discrete time) if and only if it is Tokunaga self-similar with $T_k=(c-1)c^{k-1}$
and $p=1/2$.

\subsubsection{Definition and main properties}
\label{sec:GBPdef}
Our goal is to consider combinatorial shape of a self-similar hierarchical branching process.
The following definition suggests an explicit time dependent construction of such a process,
which we denote $\cG(s;\{T_k\},p)$. \index{geometric branching process}
\begin{Def}[{{\bf Geometric Branching Process}}]
\label{def:GBM}
Consider a sequence of {\it Tokunaga coefficients} $\{T_k\ge0\}_{k\ge 1}$ and $0<p<1$.
Define 
$$S_{K}:=1+T_1+\dots+T_{K}$$ 
for $K\ge 0$ by assuming $T_0=0$.
The {\it Geometric Branching Process} (GBP) $\cG(s)={\cG}(s;\{T_k\},p)$ describes a discrete time 
multi-type population growth:
\begin{itemize}
\item[(i)] The process starts at $s=0$ with a progenitor of order ${\sf ord}(\cG)$
such that ${\sf ord}(\cG)\stackrel{d}{\sim}{\sf Geom}_1(p)$.
\item[(ii)] At every integer time instant $s>0$, each population member of order
$K\in\{1,\dots,{\sf ord}(\cG)\}$ terminates 
with probability $q_K=S^{-1}_{K-1}$, independently of other members. 
At termination, a member of order $K>1$ produces two offspring of
order $(K-1)$; and a member of order $K=1$ terminates with leaving no offspring.
\item[(iii)] At every integer time instant $s>0$, each population member of order
$K\in\{1,\dots,{\sf ord}(\cG)\}$ survives (does not terminate) with probability $$1-q_K=1-S^{-1}_{K-1},$$
independently of other members.
In this case, it produces a single offspring (side branch). 
The offspring order $i\in\{1,\dots,K-1\}$, is drawn from the distribution
\be
\label{eq:pj}
p_{K,i}=\frac{T_{K-i}}{T_1+\dots+T_{K-1}}.
\ee
\end{itemize}
The geometric tree ${\cG}(\{T_k\},p)$ is a combinatorial tree generated by the trajectories 
of ${\cG}(s;\{T_k\},p)$ in $\cBT^|$.
\end{Def}

By construction, the distribution of a geometric tree $\cG(\{T_k\}, p)$ 
coincides with the combinatorial shape of the tree of a combinatorially Horton self-similar 
hierarchical branching process $S(t)$ with Tokunaga coefficients $\{T_k\}$,
initial order distribution $p_K=p(1-p)^{K-1}$ and an arbitrary 
positive sequence of termination rates $\{\lambda_i\}$.
Accordingly, the branching structure of a geometric tree is described by 
Prop.~\ref{HBP:branch}, items (1)-(4).
The essential elements of the geometric trees (tree order, 
total number of side branches within a branch, 
numbers of side branches of a given order within a branch) are described by 
geometric laws, hence the model name.

Similarly to the tree of an HBP, a geometric tree can be constructed 
without time-dependent simulations,
following a suitable modification of the algorithm given after Prop.~\ref{HBP:branch}.
Specifically, the step that involves generation and assignment of the edge lengths 
$l_i$ should be skipped.  
 
\medskip

\begin{figure}
\includegraphics[width=0.75\textwidth]{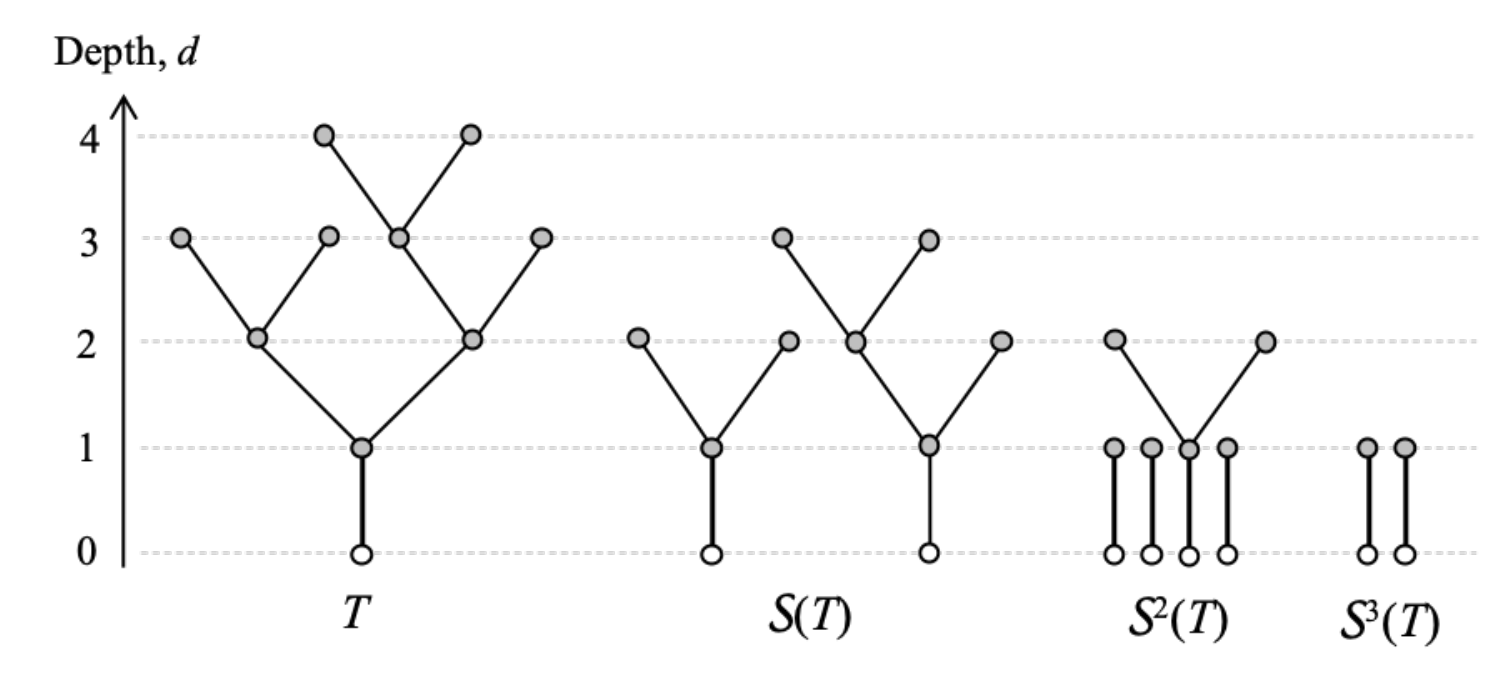}
\caption{\label{fig:Depth} Time shift $\cS$: an illustration. 
The figure shows forests obtained by consecutive
application of the time shift operator $\cS$ to a tree $T$ shown on the left.
At every step, we remove the stem from each existing tree. 
This terminates the trees of order ${\sf ord}=1$, and splits any other
tree in two new trees.
The operation $\cS^d(T)$ removes all vertices at depth $\le d$, together
with their parental edges.}
\end{figure}

Consider a geometric tree $\cG=\cG(\{T_k\},p)$ and its two subtrees, $T^a$ and $T^b$, 
rooted at the internal vertex closest to the root, randomly and uniformly permuted.
We call $T^a$ and $T^b$ the {\it principal} subtrees of $\cG$.
Let $K$ be the order of $\cG$, and, 
conditioned on  $K>1$, let $K_a,K_b$ be the orders 
of the principal subtrees $T^a$ and $T^b$, respectively.
Observe that the pair $K_a,K_b$ uniquely defines the tree order $K$:
\[K=\begin{cases}
K_a\vee K_b&\text{ if } K_a\ne K_b,\\
K_a+1&\text{ if } K_a = K_b.
\end{cases}\]
We write $K_1\le K_2$ for the order statistics of $K_a$, $K_b$.

\begin{lem}[{{\bf Order of principal subtrees}}]
\label{lem:K}
Conditioned on the tree order $K$, the joint distribution 
of the order statistics $(K_1,K_2)$ is given by
\be
\label{joint}
{\sf P}\left(K_1=j,K_2=m|K=k\right)
=
\begin{cases}
S^{-1}_{k-1}&\text{ if }j=m=k-1,\\
T_{k-j}S^{-1}_{k-1}&\text{ if }j<m=k,
\end{cases}
\ee
where
$${\sf P}(K=k|K>1)=(1-p)p^{k-2},\quad k\ge 2.$$
\end{lem}
\begin{proof}
Definition~\ref{def:GBM}, part (ii) states that a branch of order $K$
splits into two branches of order $K-1$ with probability $S^{-1}_{K-1}$,
which establishes the first case of \eqref{joint}.
Definition~\ref{def:GBM}, part (iii) states that, otherwise, 
with probability $1-S^{-1}_{K-1}$, a side branch is created 
whose order equals $j$ with probability $T_{K-j}(S_{K-1}-1)^{-1}$.
This gives
\begin{eqnarray*}
\lefteqn{{\sf P}\left(K_1=j,K_2=k|K=k\right)}\\
&=&{\sf P}\left(K_1=j|K=k,K_2=k\right)
{\sf P}\left(K_2=k|K=k\right)\\
&=&\frac{T_{k-j}}{S_{k-1}-1}\left(1-\frac{1}{S_{k-1}}\right)
=T_{k-j}S_{k-1}^{-1},
\end{eqnarray*}
which establishes the second case.
\end{proof}

\subsubsection{Tokunaga self-similarity of time invariant process}
\label{sec:TokTI}
Let $x_i(s)$, $i\ge 1$, denote the average number of vertices of order $i$
at time $s$ in the process $\cG(s)$, and
${\bf x}(s)=(x_1(s),x_2(s),\dots)^T$ be the state vector.
By definition we have
\[{\bf x}(0)=\pi:=\sum\limits_{K=1}^\infty p(1-p)^{K-1} {\bf e}_K,\]
where ${\bf e}_K$ are standard basis vectors. 
Furthermore, if $q_{a,b}$, $a\ge b$, denotes the probability that a vertex
of order
${\sf ord}=a+{\bf 1}_{\{a=b\}}$
that exists at time $s$ splits into a pair of vertices of 
orders $(a,b)$ at time $(s+1)$, then
\begin{eqnarray}
\label{eq:dyn1}
\lefteqn{x_K(s+1)= 2\,x_{K+1}(s)q_{K,K}}\nonumber\\
&+&x_K(s)(1-q_{K-1,K-1})
+\sum_{i=K+1}^{\infty}x_i(s)\,q_{i,K}.
\end{eqnarray}
The first term in the right-hand side of \eqref{eq:dyn1} corresponds to a split of an order-$(K+1)$
vertex into two vertices of order $K$, 
the second -- to a split of an order-$K$ vertex into a vertex
of order $K$ and a vertex of a smaller order, and
the third -- to a split of a vertex of order $i>K$ into 
a vertex of order $K$ and a vertex of order $i$. 
The geometric branching implies (see Lemma~\ref{lem:K}, Eq. (\ref{joint}))
\begin{eqnarray}
\label{jointGBM}
q_{a,b}
=\left\{
\begin{array}{rl}
S_{a}^{-1} & \text{ if }a=b,\\
T_{a-b}S_{a-1}^{-1}& \text{ if }b<a.
\end{array}\right.
\end{eqnarray}
Accordingly, the system \eqref{eq:dyn1} rewrites as
\be
\label{eq:dyn2}
{\bf x}(s+1)-{\bf x}(s)~=~\mathbb{G}\mathbb{S}^{-1}{\bf x}(s),
\ee
where $\mathbb{G}$ is defined in Eq. (\ref{gen}),
and 
\[\mathbb{S} = {\sf diag}\{S_0,S_1,\dots\}.\]

\noindent
In this setup, the {\it unit time shift} operator $\cS$, which advances the 
process time by unity, can be applied to 
individual trees and forests (collection of trees).
For each tree $T\in \cT^|$, the operator removes the root and stem, resulting in two principal subtrees $T^a$ and $T^b$.
A consecutive applications of $d$ time shifts to a tree $T$ is equivalent to removing
the vertices at depth $\le d$  from the root together with their 
parental edges (Fig.~\ref{fig:Depth}).
Next we define time invariance with respect to the shift $\cS$.


\begin{Def}[{\bf{Time invariance}}]
\label{def:TI}
Geometric branching process $\cG(s)$, $s\in\mathbb{Z}_+$, is called 
{\it time invariant} if the state vector ${\bf x}(s)$ is 
invariant with respect to a unit time shift $\cS$:
\be
\label{eq:TI}
{\bf x}(s) = {\bf x}(0)\equiv \pi ~~\forall s~\Longleftrightarrow~
\mathbb{G}\mathbb{S}^{-1}\pi = {\bf 0}.
\ee
\end{Def}

%
 
Now we formulate the main result of this Section.
\begin{thm}[{{\bf \cite{KZ18a}}}]
\label{thm:mainGBM}
A geometric branching process $\cG(s;T_k,p)$ is time invariant if and only if
\be
\label{eq:cTok}
p=1/2\text{  and } T_k=(c-1)c^{k-1} \text{ for any } c\ge 1.
\ee  
We call this family a (combinatorial) critical Tokunaga process, and the respective trees -- (combinatorial) critical Tokunaga trees.
\end{thm}
\noindent 
Theorem~\ref{thm:mainGBM} is proven in Sect.~\ref{sec:proofGBM} via solving a 
nonlinear system of equations that writes \eqref{eq:TI} in terms of ratios $S_k/S_{k+1}$.

\begin{cor}
\label{cor1}
Let $\cG$ be a combinatorial critical Tokunaga tree. 
Then the distribution of the principal subtree $T^a$ 
(and hence $T^b$) matches that of the initial tree $\cG$.
The distributions of $T^a$ and $T^b$ are independent if and only if 
$c=2$.
\end{cor}
\begin{proof}
Let ${\sf ord}(\cG)$ denote the (random) order of a random geometric tree $\cG$. 
Conditioned on ${\sf ord}(\cG)>1$, at instant $s=1$ (equivalently, after
a unit time shift $\cS$) there exist exactly two 
vertices that are the roots of the principal subtrees $T^a$ and $T^b$. 
Since the trees $T^a$ and $T^b$ have the same distribution, their roots
have the same order distribution.
Denote by $y_k$ the probability that the tree $T^a$ has order $k\ge 1$
and let ${\bf y}=(y_1,y_2,\dots)^T$.
By Thm.~\ref{thm:mainGBM}, the process $\cG(s)$ is time invariant.
We have $p=\pi_1=1/2$, which, together with time invariance,
implies
\[{\bf x}(0)={\bf x}(1)=2{\bf y}(1-\pi_1)+{\bf 0}\pi_1
={\bf y}.\]
This establishes the first statement.

The second statement follows from examining the joint distribution $q_{a,b}$
of \eqref{jointGBM}.
Recall that we write $K$ for the order of tree $T$, $K_a$, $K_b$ for the orders of 
the principal subtrees $T^a$, $T^b$, and $K_1<K_2$ for the
order statistics of $K_a$, $K_b$. Observe that for $k>1$,
\begin{align*}
{\sf P}(K_a=m ~&|~K=k)\\
=&
\begin{cases} 
\frac{1}{2} \sum\limits_{j:j<k} {\sf P}(K_1=j,K_2=k|K=k)& \qquad \text{ if } ~m=k,\\
{\sf P}(K_1=K_2=k-1|K=k) \\  
+\frac{1}{2} {\sf P}(K_1=k-1,K_2=k|K=k)
& \qquad \text{ if } ~m=k-1,\\
\frac{1}{2} {\sf P}(K_1=m,K_2=k|K=k)& \qquad \text{ if } ~m<k-1,
\end{cases}\\
\\
=&
\left\{
\begin{array}{lll}
\frac{1}{2}(S_{k-1}-1)S_{k-1}^{-1} &={1\over2}(1-c^{1-k}) 
\vphantom{\left(I^{I^I}\right)}&\text{ if }~m=k,\\
\left(1+\frac{1}{2}T_1\right)S_{k-1}^{-1} &={1\over 2}(c+1)c^{1-k} 
\vphantom{\left(I^{I^I}\right)}&\text{ if }~m=k-1,\\
{1\over 2}T_{k-m}S_{k-1}^{-1}  &={1\over 2}(c-1)c^{-m}&\text{ if }~m<k-1.
\end{array}
\right.
\end{align*}
Furthermore,
\begin{eqnarray*}
\lefteqn{{\sf P}(K_a=m, ~K_b=j ~|~K>1)}\\
&=&\sum\limits_{k=m}^\infty {\sf P}(K_a=m, ~K_b=j ~|~K=k){\sf P}(K=k|K>1)\\
&=& 
\begin{cases}
(c-1)c^{-j}\,2^{-m} &\text{ if }j<m,\\
c^{-m}\,2^{-m}&\text{ if }j=m.
\end{cases}
\end{eqnarray*}
Accordingly, the joint distribution of $K_a$, $K_b$ equals the product 
of their marginals if and only if $c=2$. 
This establishes the second statement.

We also notice that 
\begin{eqnarray}
\label{eq:Kam}
\lefteqn{{\sf P}(K_a=m ~|~K>1)=\sum\limits_{k=m}^\infty {\sf P}(K_a=m ~|~K=k){\sf P}(K=k|K>1)}\nonumber\\
&=& (1-c^{1-m})2^{-m}+c^{-m}2^{-m}+{(c-1)c^{-m} \over 2}\sum\limits_{k=m+1}^\infty 2^{1-k} = 2^{-m},
\end{eqnarray}
which provides an alternative, direct proof of the first statement of the corollary
that does not use the time invariance.
\end{proof}

\begin{Rem} 
Corollary~\ref{cor1} asserts that the principal
subtrees in a random critical Tokunaga tree are dependent, except 
the critical binary Galton-Watson case $c=2$.
This implies that, in general, non-overlapping subtrees within
a critical Tokunaga tree are dependent.
Accordingly, the increments of the Harris path $H$ of a critical Tokunaga 
process have (long-range) dependence.
The only exception is the case $c=2$ that will be discussed in Sect.~\ref{sec:erw}.
The structure of $H$ is hence reminiscent of a self-similar random process
with long-range dependence
\cite{MS,ST}.
Establishing the correlation structure of the Harris paths of critical Tokunaga
processes is an interesting open problem (see Sect. \ref{open}). 
\end{Rem}

\subsubsection{Frequency of orders in a large critical Tokunaga tree}

Combinatorial trees of the critical Tokunaga processes (Def.~\ref{def:TokProcess},
Prop.~\ref{Tok_tree}),
and hence the time invariant geometric trees (also called combinatorial
critical Tokunaga trees) of Thm.~\ref{thm:mainGBM},
have an additional important property: the frequencies of vertex
orders in a large-order tree approximate the tree order distribution $p_K = 2^{-K}$
in the space $\cBT^|$.
To formalize this observation, let $\mu$ be a measure on $\cBT^|$ induced by 
a combinatorial critical Tokunaga tree $\cG$ of \eqref{eq:cTok}.
For a fixed $K \geq 1$, let $\mu_K(\cG) = \mu(\cG|{\sf ord}(\cG)=K)$.
We write $V_k[\cG]$ for the number of non-root vertices of order $k$ in a tree $\cG$, 
and let $\cV_k[K]={\sf E}_K\big[V_k[\cG]\big].$ 
Finally, we denote by $V[\cG]= \sum\limits_{k=1}^{{\sf ord}(\cG)} V_k[T]$ the total number of non-root 
vertices in $\cG$, and notice that $V[\cG]= 2V_1[\cG]-1$. 
Thus,  $\cV[K]:={\sf E}_K\big[V[\cG]\big]= 2\cV_1[K]-1$.

\begin{prop}
\label{prop:Tok1}
Let $\cG$ be a combinatorial critical Tokunaga tree \eqref{eq:cTok}.
Then
\be
\label{v1}
\lim_{ K\to\infty}\frac{\cV_k[K]}{\cV_1[K]}=2^{1-k}.
\ee
Let $v\in\cG$ be a vertex selected by uniform random drawing from the
non-root vertices of $\cG$. Then, for any $k\ge 1$, 
\be
\label{v2}
\lim_{K\to\infty}{\sf P}({\sf ord}(v)=k|{\sf ord}(\cG)=K)=2^{-k}.
\ee

\end{prop}
\begin{proof}
Theorem \ref{thm:HLSST} asserts that a critical Tokunaga tree $\cG$ satisfies
the strong Horton law \eqref{eq:Horton_mean} with Horton exponent $R=2c$:
\[\lim_{K\to\infty}\frac{\cN_k[K]}{\cN_1[K]}=(2c)^{1-k},\text{ for any }k\ge 1.\]
Conditioned on ${\sf ord}(\cG)=K$ we have, for any $k \in \{1,\hdots, K\}$,
\[V_k[K] = \sum_{i=1}^{N_k(\cG)}(1+m_i(\cG)),\]
where $m_i(\cG)$ is the number of side branches that merge the $i$-th branch
of order $k$ in $\cG$, according to the proper branch labeling of 
Sect.~\ref{sec:label}.
Proposition~\ref{HBP:branch} gives
\[\cV_k[K] = \cN_k[K](1+T_1+\dots+T_{k-1}).\]
For a critical Tokunaga tree with $T_k=(1-c)c^{k-1}$ this implies 
\begin{eqnarray*}
\lim_{K\to\infty}\frac{\cV_k[K]}{\cV_1[K]}
&=&\lim_{K\to\infty}\frac{\cN_k[K](1+T_1\dots+T_{k-1})}{\cN_1[K]}
=(2c)^{1-k}c^{k-1}=2^{1-k}.
\end{eqnarray*}
To show \eqref{v2}, we write
\[V_k[\cG] = N_k[\cG]+\sum_{i=1}^{N_k[\cG]}m(i),\quad \E[m(i)] = S_{k-1}-1,\]
where $m(i)$ is a random variable that represents the total number of side branches
within $i$-th branch of order $k$ within $\cG$.
Since $N_k[\cG]\stackrel{p}{\to}\infty$
for any $k\ge 1$ as ${\sf ord}(\cG)\to\infty$, the Weak Law of Large Numbers gives
\[\frac{V_k[\cG]}{N_k[\cG]}\stackrel{p}{\to}S_{k-1}=c^{k-1}
\text{ as }{\sf ord}(\cG)\to\infty.\]
Finally, the strong Horton law of Cor.~\ref{cor:DoobMartingaleSHL} gives
\begin{eqnarray*}
\frac{V_k[\cG]}{\#T}
=\frac{V_k[\cG]}{N_k[\cG]}\frac{N_k[\cG]}{2N_1[\cG]-1}
~\stackrel{p}{\longrightarrow}~c^{k-1}\frac{1}{2}(2c)^{1-k} = 2^{-k}.
\end{eqnarray*}
This implies \eqref{v2} and completes the proof.
\end{proof}

Proposition~\ref{prop:Tok1} has an immediate extension to trees with edge lengths, which we
include here for completeness.
Recall (Def.~\ref{def:treeL}) that a tree $\cG\in\cBL^|$ can be considered a metric
space with distance $d(a,b)$ between two points $a,b\in \cG$ defined as 
the length of the shortest path within $\cG$ connecting them.
\begin{prop}
\label{prop:Tok2}
Let $\cG$ be a combinatorial critical Tokunaga tree \eqref{eq:cTok}.
Let point $u\in \cG$ be sampled from a uniform density function on the metric space $\cG$, 
and let ${\sf ord}(u)$ denote the order of the edge to which the point $u$ belongs.
Then
\be
\label{p2}
\lim_{K\to\infty}{\sf P}\Big({\sf ord}(u)=k~\Big|~{\sf ord}(\cG)=K\Big)=2^{-k}.
\ee
\end{prop}
\begin{proof}
Proposition~\ref{Tok_tree} establishes that the edge lengths in $\cG$ are i.i.d. exponential random variables. 
Thus we can generate $\cG$ by first sampling a combinatorial critical Tokunaga tree 
$\textsc{shape}(\cG)$, 
and then assigning i.i.d. exponential edge lengths. 
Provided that we already sampled $\textsc{shape}(\cG)$, selecting the i.i.d. edge lengths and then selecting the point $u\in \cG$ uniformly at random, and marking the edge that $u$ belongs to, is equivalent to selecting a random edge uniformly from the edges of $\textsc{shape}(\cG)$, in order of proper labeling of Sect.~\ref{sec:label}.
The order ${\sf ord}(u)$ is uniquely determined by the edge to which $u$ belongs.
The statement now follows from Prop.~\ref{prop:Tok1}. 
\end{proof}

\subsubsection{Proof of Theorem~\ref{thm:mainGBM}}
\label{sec:proofGBM}

\begin{lem}[{{\bf \cite{KZ18a}}}]
\label{lem:system}
A geometric branching process $\cG(s)$ is time invariant if and only if
$p=1/2$ and the sequence $\{T_k\}$ solves the following (nonlinear) system
of equations:
\be
\label{eq:system}
\frac{S_0}{S_k}=\sum_{i=1}^{\infty} 2^{-i}\frac{S_i}{S_{k+i}}\quad\text{for all }k\ge 1.
\ee
\end{lem}
\begin{proof}
Assume that the process is time invariant. 
Then the process progeny
is constant in time and equals unity:
\[\|\pi\|_1=\sum_{k=1}^{\infty}p(1-p)^{k-1} = 1.\]
Observe that in one time step, every vertex of order ${\sf ord}=1$
terminates, and any vertex of order ${\sf ord}>1$ splits in two.
Hence, the process progeny at $s=1$ is
\[2\sum_{k=2}^{\infty}p(1-p)^{k-1} = 2(1-p)=1,\]
which implies $p=1/2$.
Accordingly, $p(1-p)^{k-1}=2^{-k}$ and
the time invariance \eqref{eq:TI} takes the following coordinate form
\be
\label{eq:cdn}
-\frac{2^{-k}}{S_{k-1}}+2^{-(k+1)}\frac{T_1+2}{S_k} +\sum_{i=k+2}^{\infty}2^{-i}\frac{T_{i-k}}{S_{i-1}}=0,\text{ for all } k\ge 1.
\ee
Multiplying \eqref{eq:cdn} by $2^k$ and observing that $T_k=S_k-S_{k-1}$ ,
we obtain
$$-\frac{1}{S_{k-1}}+\frac{1}{2}\frac{T_1+2}{S_k} +\sum_{i=2}^{\infty}2^{-i}\frac{T_{i}}{S_{k+i-1}}=0,$$

$$\frac{1}{S_{k-1}}-\sum_{i=1}^{\infty}2^{-i}\frac{S_i}{S_{k+i-1}}=\frac{1}{S_k}-\frac{1}{2\,S_k} -\sum_{i=2}^{\infty}2^{-i}\frac{S_{i-1}}{S_{k+i-1}},$$
and
\be\label{eq:10}
\frac{1}{S_{k-1}}-\sum_{i=1}^{\infty}2^{-i}\frac{S_i}{S_{k+i-1}}=\frac{1}{2}\left(\frac{1}{S_k}-\sum_{i=1}^{\infty}2^{-i}\frac{S_i}{S_{k+i}}\right),
\ee
We prove \eqref{eq:system} by induction.
For $k=1$ we have
\begin{eqnarray*}
\frac{1}{2}&=&\frac{1}{2\,S_1}+\sum_{i=1}^{\infty}2^{-(i+1)}\frac{S_i-S_{i-1}}{S_i},\\
1&=&\frac{1}{S_1}+\sum_{i=1}^{\infty}2^{-i}-\sum_{i=1}^{\infty}2^{-i}\frac{S_{i-1}}{S_i},
\end{eqnarray*}
which establishes the base case
\[\frac{1}{S_1}=\sum_{i=1}^{\infty}2^{-i}\frac{S_i}{S_{i+1}}.\]
Next, assuming that the statement is proven for $(k-1)$, the left-hand side of \eqref{eq:10} vanishes, 
and the right-hand part rewrites as \eqref{eq:system}.
This establishes necessity.

\medskip
\noindent
Conversely, we showed that the system \eqref{eq:system} 
is equivalent to \eqref{eq:TI} in case $p=1/2$.
This establishes sufficiency.
\end{proof}

\noindent Let $a_k=S_k/S_{k+1}\le 1$ for all $k\ge 0$.
Then, for any $i\ge 0$ and any $k>0$ we have $S_i/S_{k+i} = a_i\,a_{i+1}\dots a_{i+k-1}$.
The system \eqref{eq:system} rewrites in terms of $a_i$ as
\begin{align*}
{1 \over 2}a_1+{1 \over 4}a_2+{1 \over 8}a_3+\hdots &=a_0, \nonumber \\
{1 \over 2}a_1a_2+{1 \over 4}a_2a_3+{1 \over 8}a_3a_4+\hdots &=a_0a_1, \nonumber \\
{1 \over 2}a_1a_2a_3+{1 \over 4}a_2a_3a_4+{1 \over 8}a_3a_4a_5+\hdots &=a_0a_1a_2,\nonumber  
\end{align*}
and so on, which can be summarized  as
\begin{equation}\label{eq:a}
\sum\limits_{j=1}^\infty {1 \over 2^j}\prod\limits_{k=j}^{n+j-1}a_k =  \prod\limits_{k=0}^{n-1}a_k,
\text{ for all } n \in \mathbb{N}.
\end{equation}

\begin{lem}[{{\bf \cite{KZ18a}}}]
\label{lem:NL}
The system \eqref{eq:a} with the initial value $a_0=1/c >0$ has a unique solution
$$a_0=a_1=a_2=\hdots=1/c.$$
\end{lem}

\begin{proof}[Proof of Lemma~\ref{lem:NL}] 
Suppose $\{a_0,a_1,a_2,\dots\}$ is a solution to system \eqref{eq:a}.
Then $\{1,a_1/a_0,a_2/a_0,\dots\}$ is also a solution, since each equation only includes multinomial terms of the same degree.
Thus, without loss of generality we assume $a_0=1$, and we need to prove that
$$a_1=a_2=\hdots=1.$$
We consider two cases. 
\medskip
\noindent

{\it Case I.} Suppose the sequence $\{a_j\}$ has a maximum:
there exists an index $i \in \mathbb{N}$ such that 
$a_i =\max\limits_{j \in \mathbb{N}}a_j.$
Define
\[w_{j,\ell}:= {1 \over 2^j}\prod\limits_{k=j}^{\ell+j-1}a_k \left[\prod\limits_{k=0}^{\ell-1}a_k\right]^{-1}.\]
Using $n=\ell$ in (\ref{eq:a}) we obtain that for any $\ell \in \mathbb{N}$,
\begin{equation}\label{eq:L1}
\sum\limits_{j=1}^\infty w_{j,\ell}  =  1,
\end{equation}
and using $n=\ell+1$ we find that an arbitrary $a_\ell$ is the weighted average of $\{a_{\ell+j}\}_{j=1,2,\hdots}$:
\begin{equation}\label{eq:L+1}
\sum\limits_{j=1}^\infty w_{j,\ell}\,a_{\ell+j} =  a_\ell.
\end{equation}
Hence, since $a_i =\max\limits_{j \in \mathbb{N}}a_j$,
$$a_i=a_{i+1}=a_{i+2}=a_{i+3}=\hdots =a.$$
Similarly, letting $\ell=i-1$ in (\ref{eq:L1}) and (\ref{eq:L+1}), we obtain $a_{i-1}=a$. Recursively, by plugging in $\ell=i-2, ~i-3, \hdots$, we show that
$$a_1=a_2=\hdots =a_{i-1}=a_i=a_{i+1}=\hdots =a.$$
Finally, ${1 \over 2}a_1+{1 \over 4}a_2+{1 \over 8}a_3+\hdots=1$ implies $a=1$.

\medskip
\noindent
{\it Case II.} Suppose there is no $\max\limits_{j \in \mathbb{N}}a_j$. Let 
$U:=\limsup\limits_{j \rightarrow \infty}a_j.$
From (\ref{eq:a}) we know via cancelation that
\begin{eqnarray}
\label{eq:U2}
{1 \over 2}a_n+{1 \over 4}{a_n a_{n+1} \over a_1}&+&{1 \over 8}{a_n a_{n+1}a_{n+2} \over a_1a_2}+\hdots\nonumber\\
&+&{1 \over 2^{n-1}}{ \prod\limits_{k=n}^{2n-2}a_k \over \prod\limits_{k=0}^{n-2}a_k}
+\sum\limits_{j=n}^\infty {1 \over 2^j}{ \prod\limits_{k=j}^{n+j-1}a_k \over \prod\limits_{k=0}^{n-1}a_k}=  1.
\end{eqnarray}
Thus, $2^{-1}\,a_n <1$ and $U \leq 2$.
The absence of maximum implies $a_j <U \leq 2$ for all $j \in \mathbb{N}$.

\medskip
\noindent
Plugging $n+1$ in (\ref{eq:a}), we obtain
\begin{eqnarray*}\label{eq:Un+1}
\lefteqn{\left({1 \over 2}a_n\right) a_{n+1}+\left({1 \over 4}{a_n a_{n+1} \over a_1}\right)a_{n+2}+\hdots}\\
&+&{1 \over 2^{n-1}}{ \prod\limits_{k=n}^{2n-2}a_k \over \prod\limits_{k=0}^{n-2}a_k}a_{n+j-1}+\sum\limits_{j=n}^\infty {1 \over 2^j}{ \prod\limits_{k=j}^{n+j-1}a_k \over \prod\limits_{k=0}^{n-1}a_k}a_{n+j}=  a_n.
\end{eqnarray*}

Thus, since $a_j <U$ for all $j \in \mathbb{N}$,
\begin{eqnarray*}
\lefteqn{\left({1 \over 2}a_n\right) a_{n+1}+\left({1 \over 4}{a_n a_{n+1} \over a_1}\right)U+\hdots}\\
&+&{1 \over 2^{n-1}}{ \prod\limits_{k=n}^{2n-2}a_k \over \prod\limits_{k=0}^{n-2}a_k}U+\sum\limits_{j=n}^\infty {1 \over 2^j}{ \prod\limits_{k=j}^{n+j-1}a_k \over \prod\limits_{k=0}^{n-1}a_k}U > a_n
\end{eqnarray*}
which simplifies via (\ref{eq:U2}) to
\begin{equation}\label{eq:anan+1}
\left({a_n \over 2}\right) a_{n+1}+\left(1-{a_n \over 2}\right)U  > a_n.
\end{equation}
For all $\varepsilon \in (0,1)$, there are infinitely many $n \in \mathbb{N}$ such that $a_n >(1-\varepsilon)U$.
Then, for any such $n$, the above inequality (\ref{eq:anan+1}) implies
$$a_{n+1}>2-{2 \over a_n}U+U >2-{2\varepsilon \over 1-\varepsilon}+U =\big(1-\varphi(\varepsilon)\big)U,$$
where 
$$\varphi(x):={2x \over (1-x)U}.$$
Let $\varphi^{(k)}=\varphi \circ \hdots \circ \varphi$. 
Repeating the argument for any given number of iterations $K \in \mathbb{N}$, we obtain
\begin{equation*}
a_{n+2}>\big(1-\varphi^{(2)}(\varepsilon)\big)U, ~~ a_{n+3}>\big(1-\varphi^{(3)}(\varepsilon)\big)U,\hdots, ~~a_{n+K}>\big(1-\varphi^{(K)}(\varepsilon)\big)U.
\end{equation*}
Thus, given any $K \in \mathbb{N}$, fix $\varepsilon \in (0,1)$ small enough so that such that $\varphi^{(k)}(\varepsilon)\in (0,1)$ for all $k=1,2,\hdots,K$. Then, taking $n>K$ such that $a_n >(1-\varepsilon)U$, we obtain from (\ref{eq:U2}) that
\begin{align*}
1 &> {1 \over 2}a_n+{1 \over 4}{a_n a_{n+1} \over a_1}+{1 \over 8}{a_n a_{n+1}a_{n+2} \over a_1a_2}+\hdots +{1 \over 2^{K+1}}{ \prod\limits_{k=n}^{n+K}a_k \over \prod\limits_{k=0}^Ka_k}\\
>& {1 \over 2}(1-\varepsilon)U+{1 \over 4}{(1-\varepsilon)\big(1-\varphi(\varepsilon)\big)U^2 \over U}+\hdots +{1 \over 2^{K+1}}{ U^{K+1}\prod\limits_{k=0}^K\big(1-\varphi^{(k)}(\varepsilon)\big) \over U^K}.
\end{align*}
Now, since $\varepsilon$ can be chosen arbitrarily small,
$$1 \geq \left(1-{1 \over 2^{K+1}}\right)U.$$
Finally, since $K$ can be selected arbitrarily large, we have proven that $1\geq U$. However, this will contradict the assumption of {\it Case II}. Indeed, if
$a_j <U \leq 1$ for all $j \in \mathbb{N}$, then
$${1 \over 2}a_1+{1 \over 4}a_2+{1 \over 8}a_3+\hdots <1,$$
contradicting the first equation in the statement of the theorem. Thus, the assumptions  of {\it Case II} cannot be satisfied.  
We conclude that there exists a maximal element in the sequence $\{a_j \}_{j=1,2,\hdots}$ as assumed in {\it Case I}, implying the statement of the theorem.
\end{proof}

\bigskip
\begin{proof}[Proof of Theorem~\ref{thm:mainGBM}]
Lemma~\ref{lem:NL} implies
$a_k = S_k/S_{k+1} = 1/c$ for some $c\ge 1$.
Hence
$S_1 = 1+T_1 = c$ and $T_1 = c-1.$
Furthermore, 
\[S_{k+1} = c\,S_k =c^k\]
and, accordingly,
\[T_{k+1} = S_{k+1}-S_k = (c-1)c^{k-1},\]
which completes the proof.
\end{proof}

\section{Tree representation of continuous functions}
\label{LST}

We review here the results of \cite{LeGall93,NP,Pitman,ZK12} on tree 
representation of continuous functions. 
This representation allows us to apply the self-similarity concepts to
time series.

\begin{figure}[t] 
\centering\includegraphics[width=0.7\textwidth]{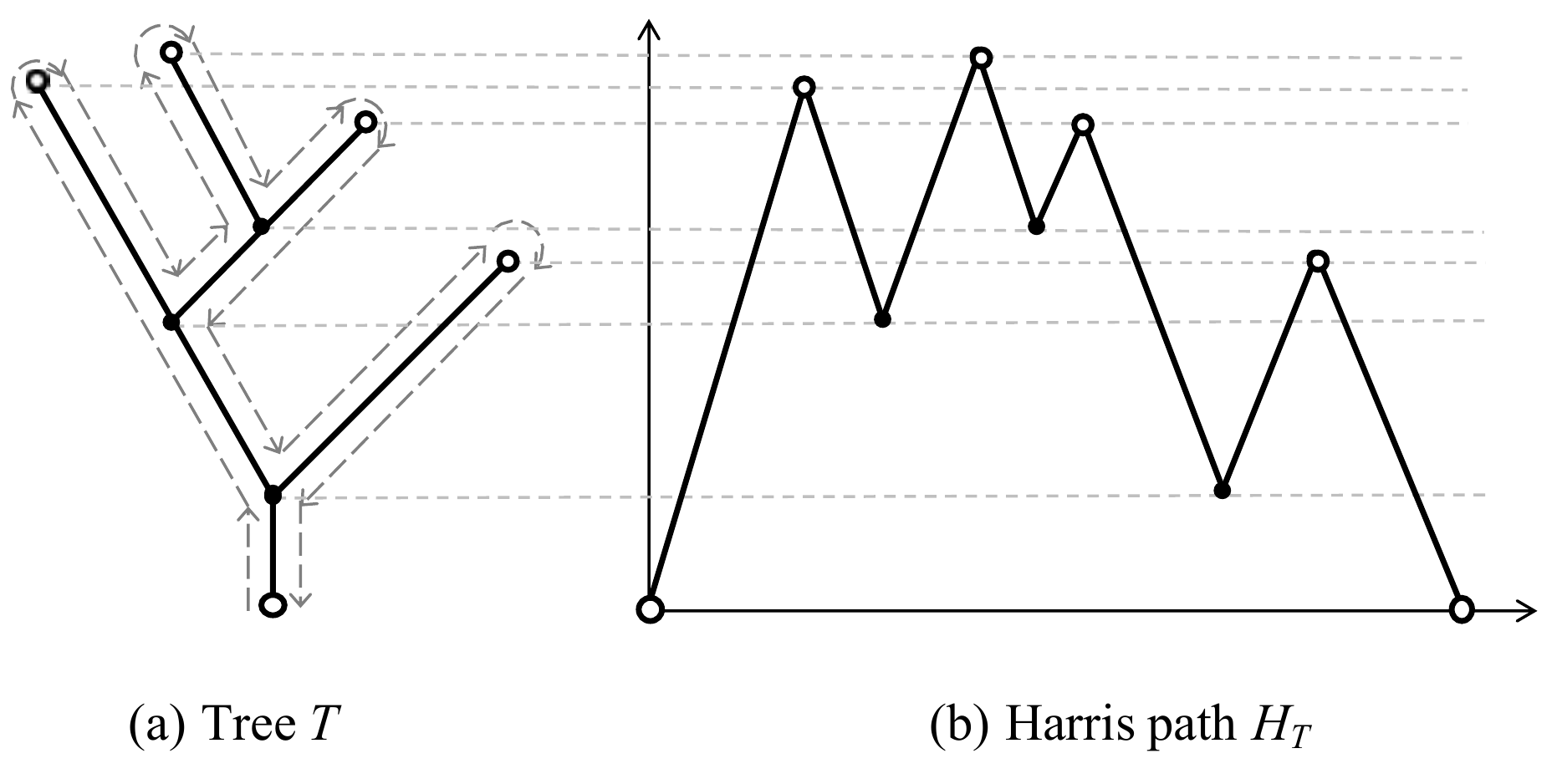}
\caption[Harris path]
{(a) Tree $T$ and its depth-first search illustrated by dashed arrows.
(b) Harris path $H_T(t)$ for the tree $T$ of panel (a). 
In this figure, the distances on a tree (edge lengths)
are measured along the $y$-axis.
Dashed horizontal lines illustrate correspondence between
vertices of $T$ and local extrema of $H_T(t)$.}
\label{fig:Harris}
\end{figure} 

\subsection{Harris path}
\label{sec:Harris}
For any embedded tree  $T\in\cL_{\rm plane}$ with edge lengths, the {\it Harris path}
(also known as the contour function, or Dyck path) 
is defined as a piece-wise linear function \cite{Harris,Pitman}  \index{Harris path}
\[H_T(t)\,:\,[0,2\cdot\textsc{length}(T)]\to\mathbb{R}\]
that equals the distance from the root traveled along the tree $T$ 
in the depth-first search, as illustrated in Fig.~\ref{fig:Harris}.
For a tree $T$ with $n$ leaves, the Harris path 
$H_T(t)$ is a piece-wise linear positive excursion
that consists of $2n$ linear segments with alternating slopes $\pm 1$.

\subsection{Level set tree}
\label{sec:level}
This section introduces a tree representation of continuous functions,
which we call a {\it level set tree}.
We begin in Sect.~\ref{sec:level1} by assuming
a finite number of local extrema; this construction is more intuitive
and is sufficient for analysis of finite trees from $\L$.  
A general definition for continuous functions follows in Sect.~\ref{sec:level2}.
\index{level set!tree}
\index{tree!level set}

\subsubsection{Tamed functions: finite number of local extrema}
\label{sec:level1}

Consider a closed interval $I\subset\mathbb{R}$ and function $f(x) \in C(I)$, where $C(I)$ is the 
space of continuous functions from $I$ to $\mathbb{R}$. 
Suppose that $f(x)$ has a finite number of distinct local minima.
The {\it level set} $\mathcal{L}_\alpha\left(f\right)$ is defined 
as the pre-image of the function values equal to or above $\alpha$: 
\[\mathcal{L}_\alpha =\mathcal{L}_\alpha (f) = \{x \in I\,:\,f(x)\ge\alpha\}.\]
The level set $\mathcal{L}_\alpha$ for each $\alpha$ is
a union of non-overlapping intervals; we write $|\mathcal{L}_\alpha|$ for their number.
Notice that 
$|\mathcal{L}_\alpha| = |\mathcal{L}_{\beta}|$ 
as soon as the interval $[\alpha,\,\beta]$ does not contain a value of
local extrema of $f(x)$ and  
$0\le |\mathcal{L}_\alpha| \le n$, where $n$ is the total number 
of the local maxima of $f(x)$ over $I$. \index{level set}

\begin{figure}[t] 
\centering\includegraphics[width=0.7\textwidth]{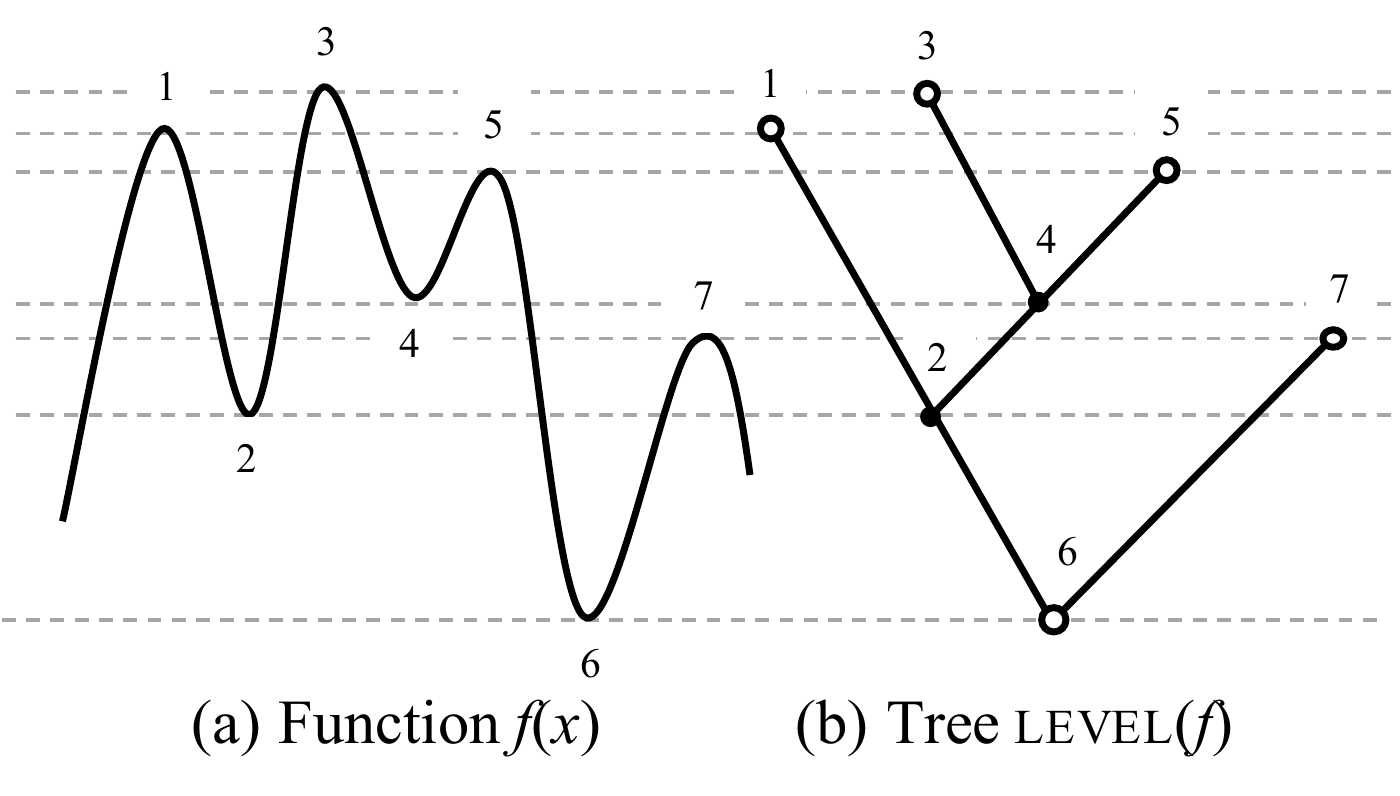}
\caption{Function $f(x)$ (panel a) with a finite number of local
extrema and its level set tree $\textsc{level}(f)$ (panel b).
In this figure, the distances on a tree (edge lengths)
are measured along the $y$-axis.
Dashed horizontal lines and numbers $1,\dots,7$ illustrate correspondence between
the local extrema of $f(x)$ and vertices of $\textsc{level}(f)$.
}
\label{fig:LST}
\end{figure}

The {\it level set tree} $\textsc{level}(f)\in\cL_{\rm plane}$ is a 
tree that describes the structure of the level sets $\mathcal{L}_\alpha$ 
as a function of threshold $\alpha$, as illustrated in Fig.~\ref{fig:LST}.
Specifically, there are bijections between 
\begin{description}
  \item[(i)] the leaves of $\textsc{level}(f)$ and the local maxima of $f(x)$;
  \item[(ii)] the internal (parental) vertices of $\textsc{level}(f)$ 
and the local minima of $f(x)$, excluding possible local minima achieved on 
the boundary $\partial I$;
  \item[(iii)] a pair of subtrees of $\textsc{level}(f)$ rooted in the parental vertex that corresponds to a local
minima $f(x^*)$ and the adjacent positive excursions (or meanders bounded by $\partial I$) of $f(x)-f(x^*)$ to the right and left of $x^*$.
\end{description}
Furthermore, every edge in the tree is assigned a length equal the difference
of the values of $f(x)$ at the local extrema that correspond to the
vertices adjacent to this edge according to the bijections (i) and (ii) above.
The tree root corresponds to the global minimum of $f(x)$ on $I$.
If the minimum is achieved at $x \in I \setminus \partial I$, then the 
level set tree is stemless, $\textsc{level}(f)\in\cL_{\rm plane}^{\vee}$;
this case is shown in Fig.~\ref{fig:LST}.
Otherwise, if the minimum is on the boundary $\partial I$, then the
level set tree is planted, $\textsc{level}(f)\in\cL_{\rm plane}^{|}$.

\begin{figure}[t] 
\centering\includegraphics[width=\textwidth]{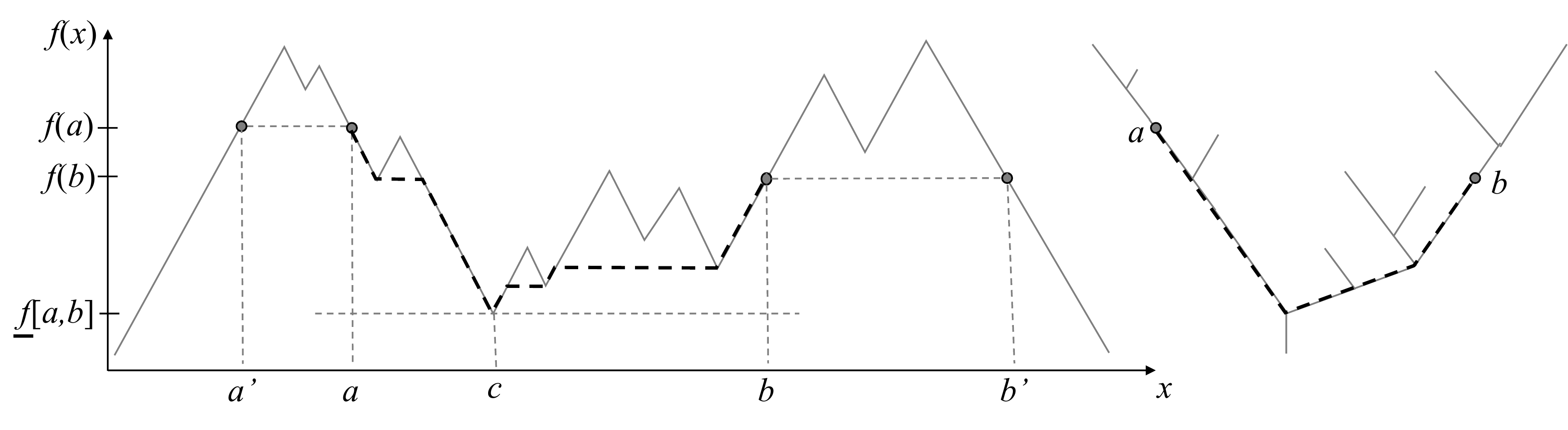}
\caption{Tree metric $d_f$ on a real interval $I$ defined by a function $f(x)$.
(Left panel): The graph of a function $f(x)$, $x\in I$ is shown by solid gray 
line. 
The distance $d_f(a,b)$ between points $a,b\in I$ is given by \eqref{eqn:tree_dist};
it equals the vertical distance along the path between $f(a)$ and $f(b)$ shown by
black dashed line.
The panel also illustrates equivalence in metric $d_f$: 
here $a\sim_f a'$ and $b\sim_f b'$, since $d_f(a,a')=d_f(b,b')=0$. 
(Right panel): The level set tree $\textsc{level}(f)$ of function $f(x)$ is shown
by solid gray line.
The distance $d_f(a,b)$ equals the length of the minimal 
tree path between points $a$ and $b$, shown by dashed black line.
Here, the tree distance is measured along the vertical axis.
}
\label{fig:dist}
\end{figure}

\subsubsection{General case}
\label{sec:level2}
For a function $f(x)\in C(I)$ on a closed interval $I\subset\mathbb{R}$,
the level set tree is defined 
via the framework of Def.~\ref{def:treeL}, following Aldous \cite{AldousI,AldousIII}
and Pitman \cite{Pitman}.
Specifically, let
$\underline{f}[a,b]:=\inf_{x\in[a,b]} f(x)$ for any subinterval $[a,b]\subset\,I$.
We define a {\it pseudo-metric} on $I$ as \cite{AldousIII,Pitman}
\begin{equation}
\label{eqn:tree_dist}
d_f(a,b):=\left(f(a)-\underline{f}[a,b]\right)
+ \left(f(b)-\underline{f}[a,b]\right),\quad a,b\in\,I.
\end{equation}
We write $a\sim_f b$ if $d_f(a,b)=0$.
Here $d_f$ is a metric on the quotient space $I_f\equiv I/\!\sim_f$. 
It can be shown \cite{Pitman} that $\left(I_f,d_f\right)$ is a tree by Def.~\ref{def:treeL}.
Figure~\ref{fig:dist} illustrates this construction for a particular piece-wise
function (left panel), and shows the respective tree $(I_f,d_f)$ as an element of
$\L^|$ (right panel).

We describe now the unique path $\sigma_{a,b}\subset I_f$
between a pair of points $a,b$. 
Let $c\in[a,b]$ be the leftmost point where $f(x)$ achieves the minimum on $[a,b]$:
\[c=\min\{x\in[a,b]: f(x) = \underline{f}[a,b]\}.\]
We define a function $\underline{f}(x)$ on $[a,b]$ as
\[\underline{f}(x) = \left\{\begin{array}{lr}
\inf_{y\in[a,x]}f(y),&\text{ if }x\in[a,c],\\
\inf_{y\in[x,b]}f(y),&\text{ if }x\in[c,b].
\end{array}
\right.
\]
By construction, $\underline{f}(x)$ is a continuous function that
is monotone non-increasing on $[a,c]$ and
monotone nondecreasing on $[c,b]$.
Furthermore,  $\underline{f}(x)\le f(x)$ and, in particular, $\underline{f}(x)=f(x)$
for $x\in\{a,b,c\}$.

\begin{lem}[{{\bf Rising Sun Lemma, F.\,Riesz \cite{Riesz}}}]
\label{sun}
Let 
\[S=\{x:\underline{f}(x)< f(x)\}\subset[a,b].\]
Then $S$ is an open set that can be represented as a countable union
of disjoint intervals
\[S = \bigcup_k(a_k,b_k)\]
such that $f(a_k)=f(b_k)=\underline{f}(a_k)=\underline{f}(b_k)$ 
and $f(x)>f(a_k)$ for any $x\in(a_k,b_k)$.
\end{lem}\index{rising sun lemma}
\begin{proof}
The statement is equivalent to that of the {\it Rising Sun Lemma} of Riesz \cite{Riesz,Tao}
applied to the functions $-f(x)$ on $[c,b]$ and $-f(-x)$ on $[a,c]$.
We just notice that $f(c)$ is the global minimum of $f(x)$ on $[a,b]$ and so
$c$ cannot be a part of $S$.
The union of two open sets, each represented as a countable union of disjoint intervals,
is itself an open interval represented as a countable union of disjoint intervals.
This completes the proof.
\end{proof}

\begin{figure}[t] 
\centering\includegraphics[width=\textwidth]{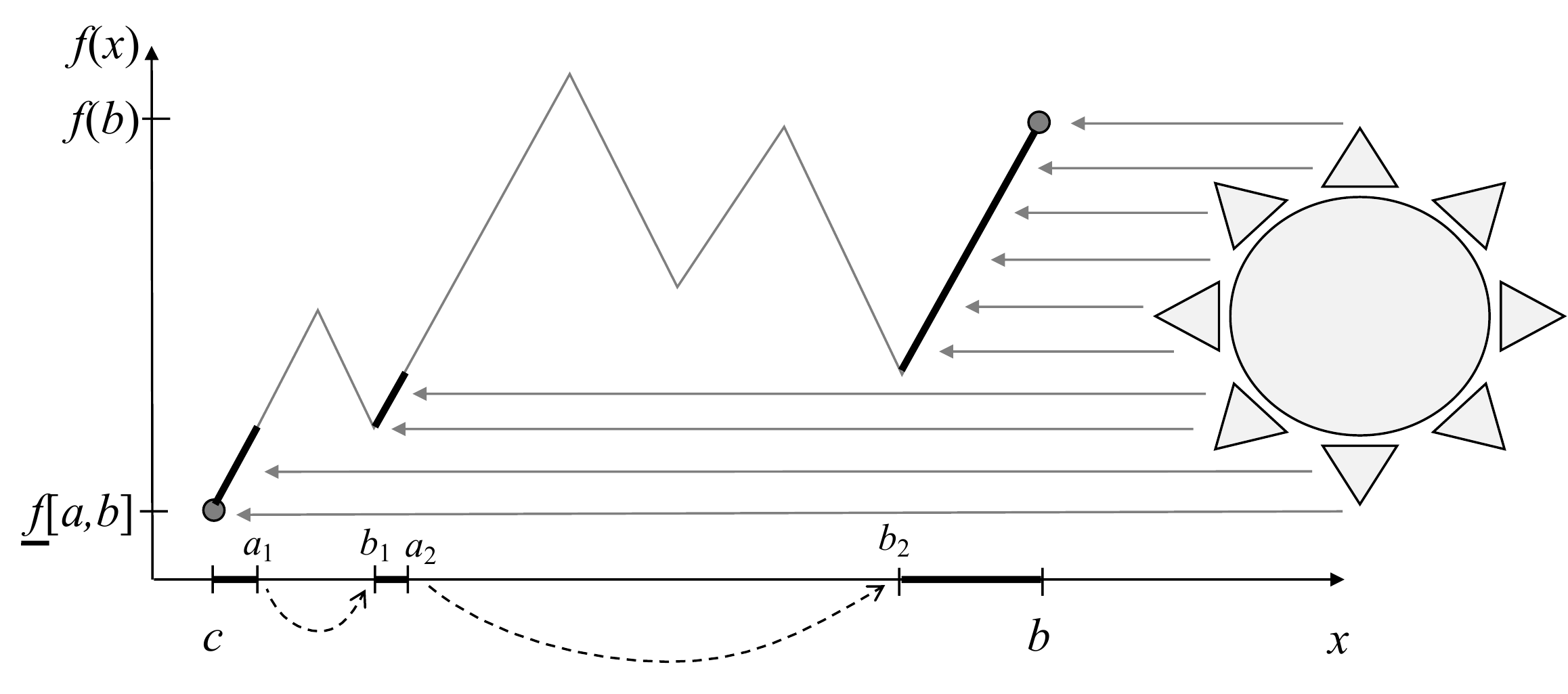}
\caption[Rising Sun Lemma]
{Rising Sun Lemma: an illustration.
The figure shows the graph of the function $f(x)$ of Fig.~\ref{fig:dist}
on the interval $[c,b]$. 
Lighted segments of the graph are shown by solid black lines; their pre-image
comprises the path $\sigma_{c,b}$ (solid black segments on the $x$-axis).
Shaded segments of the graph are shown by thin gray lines; their pre-image
comprises the set $S$ of Lemma~\ref{sun}.
The path $\sigma_{c,b}$ jumps over the intervals $(a_k,b_k)$ that form the set $S$,
as shown by dashed arrows, using the equivalence of the interval ends, $a_k\sim_f b_k$.}
\label{fig:Riesz}
\end{figure}

\noindent Figure~\ref{fig:Riesz} illustrates the Rising Sun Lemma in our setting
on the interval $[c,b]$.
As the sun rises from east (right), it lightens some segments of the 
graph of $f(x)$, and leaves the other segments in shade. 
The pre-image of the shaded segments is the set $S$, while
the pre-image of the lighted segments is the path $\sigma_{c,b}$.
The path, considered as a set in $[c,b]$, is making at most a countable number of jumps over
the intervals $(a_k,b_k)$ that comprise the set $S$ of Lemma~\ref{sun}.

For a tamed function with a finite number of local extrema, the path 
$\sigma_{a,b}$ is the pre-image of the graph 
of $\underline{f}(x)$ excluding the constant intervals.
The Rising Sun Lemma ensures that this statement generalizes to any continuous
function: 
\[\sigma_{a,b} = [a,b]\setminus S = \{x:\underline{f}(x)=f(x)\}\subset[a,b],\] 
which is travelled at unit speed left to right.
As a real set, the path $\sigma_{a,b}$ may have quite complicated structure.
For instance, it can be the Cantor set.
This, however, does not disturb the continuity of the map $[0,d_f(a,b)]\to I_f$ 
in Def.~\ref{def:treeL}.

The Rising Sun Lemma asserts that the function $\underline{f}(x)$ on $[a,b]$
has at most a countable set of constant disjoint intervals $I^{(k)}=(a_k,b_k)$, each of which 
corresponds to a positive excursion of $f(x)$. 
The end points of these intervals are equivalent in $I_f$, hence each
interval generates a tree $(I^{(k)}_f,d_f)$ whose root
corresponds to the equivalence class on $I$ consisting of $\{a_k,b_k\}$. 
This observations leads to the following statement.
\begin{cor}
The level set tree $\textsc{level}(f)$ of a continuous function $f(x)$ on
a real closed interval $[a,b]\subset\mathbb{R}$ consists of a segment
of length $d_f(a,b)$ and at most a countable number of 
trees attached to this segment with the same orientation. 
There is a one-to-one correspondence 
between these trees and the intervals $(a_k,b_k)$ from the Rising Sun Lemma. 
\end{cor}

It is straightforward to observe that the tree $(I_f,d_f)$ is equivalent to the above defined level 
set tree $\textsc{level}(f)$ for a function  $f(x) \in C\left(I\right)$ with 
a finite number of distinct local minima.
We just notice that for any subinterval $[a,b] \subset I$, the correspondence $a\sim_f b$ implies $\{f(x):~x\in [a,b]\}$ is a nonnegative excursion i.e., 
$$[a,b] \subset \mathcal{L}_\alpha(f) = \{x\,:\,f(x)\ge\alpha\} \quad \text{ where } \quad \alpha=f(a)=f(b).$$
In other words, every point in $\left(I_f,d_f\right)$ is an equivalence class
of points on $I$ with respect to $\sim_f$.
There exist three types of equivalence classes, depending on the number 
of distinct points from $I$ they include: 
(i) each single point class corresponds to a leaf vertex (local maximum), 
(ii) each two point class corresponds to an internal edge point (positive excursion), and
(iii) each three point class corresponds to an interval vertex 
(two adjacent positive excursions). 
For a general $f(x)\in C(I)$ there may exist equivalence classes that include an arbitrary
number $n$ of points from $I$, corresponding to $(n-1)$
adjacent positive excursions; 
and classes that consist of an infinite (countable or uncountable) number of points.
Conversely, for every $\alpha$, the level set $\mathcal{L}_\alpha(f)$ is a union of non-overlapping intervals $[a_j,b_j]$, i.e.,
\[\mathcal{L}_\alpha(f)=\bigcup_j \, [a_j,b_j],\]
where for each $j$, $a_j\sim_f b_j$. 

Representing level sets of a continuous function as a tree goes back to
works of Menger \cite{Menger} and Kronrod \cite{Kronrod}.
A multivariate analog of level set tree is among the key tools in
proving the celebrated Kolmogorov-Arnold representation theorem
(every multivariate continuous function can be represented as 
a superposition of continuous functions of two variables) that gives
a positive answer to a general version of the Hilbert's thirteenth problem
\cite{Arnold1959, Vit04}.
Such trees have also been discussed by Vladimir Arnold in connection to
topological classification of Morse functions and generalizations of 
Hilbert's sixteenth problem \cite{Arnold2006,Arnold2007}.
Level set trees for multivariate Morse functions (albeit
slightly different from those considered by Arnold) are discussed in Sect.~\ref{sec:morse}.

\subsection{Reciprocity of Harris path and level set tree}
\label{sec:rec}
Consider a function $f(x)\in C(I)$ with a finite number of distinct local minima. 
By construction, the level set tree $\textsc{level}(f)$ is completely
determined by the sequence of the values of local extrema of $f$,
and is not affected by timing of those extrema, as soon as their order
is preserved.
This means, for instance, that if $g(x)$ is a continuous and monotone increasing 
function on $I$, then the trees $\textsc{level}(f)$ and $\textsc{level}(f \circ g)$ are equivalent in $\L$.
Hence, without loss of generality we can focus on the level set trees
of continuous functions with alternating slopes $\pm 1$.
We write $\cE^{\rm ex}$ for the space of all positive piece-wise linear continuous finite excursions with alternating slopes $\pm 1$ and a finite number of segments (i.e., a finite number 
of local extrema).

The level set tree of an excursion from $\cE^{\rm ex}$ and 
Harris path are reciprocal to each other as described in the following statement.

\begin{prop}[{{\bf Reciprocity of Harris path and level set tree}}]\label{prop:reciprocityHP}
The Harris path
$H: \cL_{\rm plane}^|\to \cE^{\rm ex}$ and the level set tree
${\textsc{level}}: \cE^{\rm ex}\to\cL_{\rm plane}^|$ are reciprocal
to each other.
This means that for any $T\in\cL_{\rm plane}^|$ we have
$\textsc{level}(H_T(t))\equiv T,$
and for any $g(t)\in\cE^{\rm ex}$ we have
$H_{\textsc{level}(g)}(t)\equiv g(t).$
\end{prop}
\noindent This statement is readily verified by examining the excursions and 
trees in Figs.~\ref{fig:Harris},\ref{fig:dist}.

\begin{figure}[t] 
\centering\includegraphics[width=0.8\textwidth]{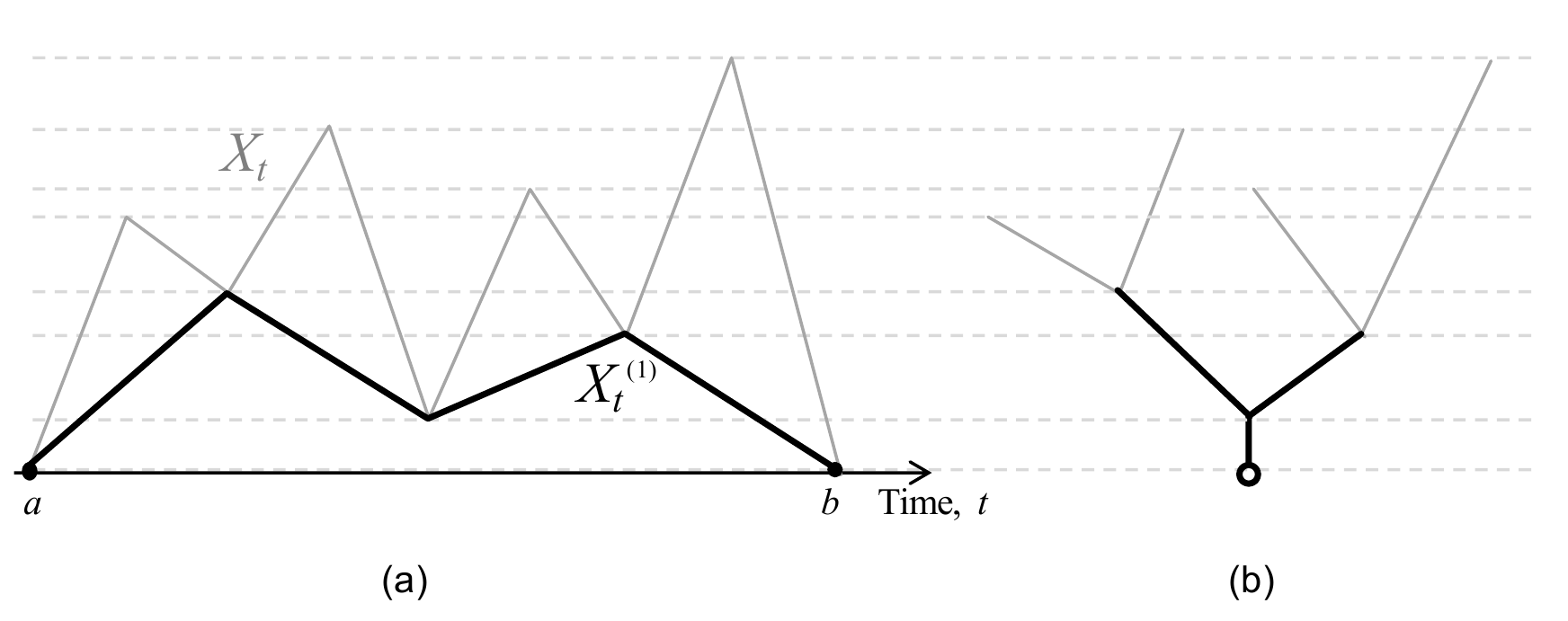}
\caption{Horton pruning of a positive excursion: transition to the local 
minima of an excursion $X_t$ corresponds to pruning of the corresponding 
level set tree.
(a) An original excursion $X_t$ (gray line) and linearly interpolated sequence 
$X^{(1)}_t$ of the respective local minima (black line).
(b) The level set tree $\textsc{level}(X^{(1)}_t)$  of the local minima sequence 
(black lines) is obtained by pruning of the level set tree $\textsc{level}(X_t)$
of the original excursion (whole tree). The pruned edges are shown in gray -- each of them
corresponds to a local maximum of the original excursion.}
\label{fig:prune1}
\end{figure}

\subsection{Horton pruning of positive excursions}\label{sec:excursions}
This section examines the level set tree and its Horton pruning for a positive 
excursion on a finite real interval.
We use these results for analysis of random walks $X_k$, $k\in\mathbb{Z}$,
which motivates us to write here $X_t$, $t\in\mathbb{R}$, for a continuous function.
 
Consider a continuous positive excursion $X_t$, $t\in[a,b]$, with a finite number
of distinct local minima and such that $X_a=X_b=0$ and $X_t>0$ for $a<t<b$.
Furthermore, consider excursion $X^{(1)}_t$, $t\in[a,b]$, obtained by a linear interpolation
of the boundary values and the local minima of $X_t$; as well as
functions $X^{(m)}_t$, $t\in[a,b]$, for $m\ge 1$, obtained by taking the local minima of $X_t$ 
iteratively $m$ times, and linearly interpolating their values together with
$X_a=X_b=0$ (see Fig.~\ref{fig:prune1}a).

In the space of level set trees of tamed continuous functions, the Horton pruning $\cR$ 
corresponds to coarsening the respective function by removing (smoothing) its local maxima, 
as illustrated in Fig.~\ref{fig:prune1}. 
An iterative pruning corresponds to iterative transition to the local minima, as we describe in the next statement.

\begin{prop}[{{\bf Horton pruning of positive excursions, \cite{ZK12}}}]
\label{ts_prune}
The transition from a positive excursion 
$X_t$ to the respective excursion  $X^{(1)}_t$ 
of its local minima corresponds to the Horton pruning of the level set tree $\textsc{level}(X_t)$.
This is illustrated in a diagram of Fig.~\ref{fig:prune2}.
In general,
\[\textsc{level}\left(X^{(m)}_t\right) 
= \cR^m\left(\textsc{level}(X_t)\right), \forall m \ge 1.\]
\end{prop}
\begin{proof}
First, 
\be\label{eq:levelone}
\textsc{level}\left(X^{(1)}_t\right) = \cR\left(\textsc{level}(X_t)\right)
\ee
is established via the following observation. For a pair of consecutive local minima $s_1<s_2$,
the level set tree $\textsc{level}(\widetilde{X}_t)$ of the function
$$\widetilde{X}_t=X_t {\bf 1}_{t \not\in [s_1,s_2]} +\left({s_2 -t \over s_2-s_1}X_{s_1}+ {t-s_1 \over s_2-s_1}X_{s_2}\right){\bf 1}_{t \in [s_1,s_2]}$$
is obtained from $\textsc{level}\left(X_t\right)$ by removing the leaf that corresponds to the unique local maximum
of $X_t$ inside $(s_1,s_2)$ together with its parental edge that connects it to the parental vertex, corresponding to
$\max\{X_{s_1},X_{s_2}\}$. Thus, substituting $X_t$ with linear interpolation of local minima, $X^{(1)}_t$, will result in simultaneous
removal of leaves together with the parental edges. 
The statement of the proposition follows via recursion of \eqref{eq:levelone}.
\end{proof}

It is straightforward to formulate an analog of Prop.~\ref{ts_prune} 
without the excursion assumption, for continuous functions with a
finite number of distinct local minima within $[a,b]$.

\begin{figure}[t] 
\centering\includegraphics[width=0.9\textwidth]{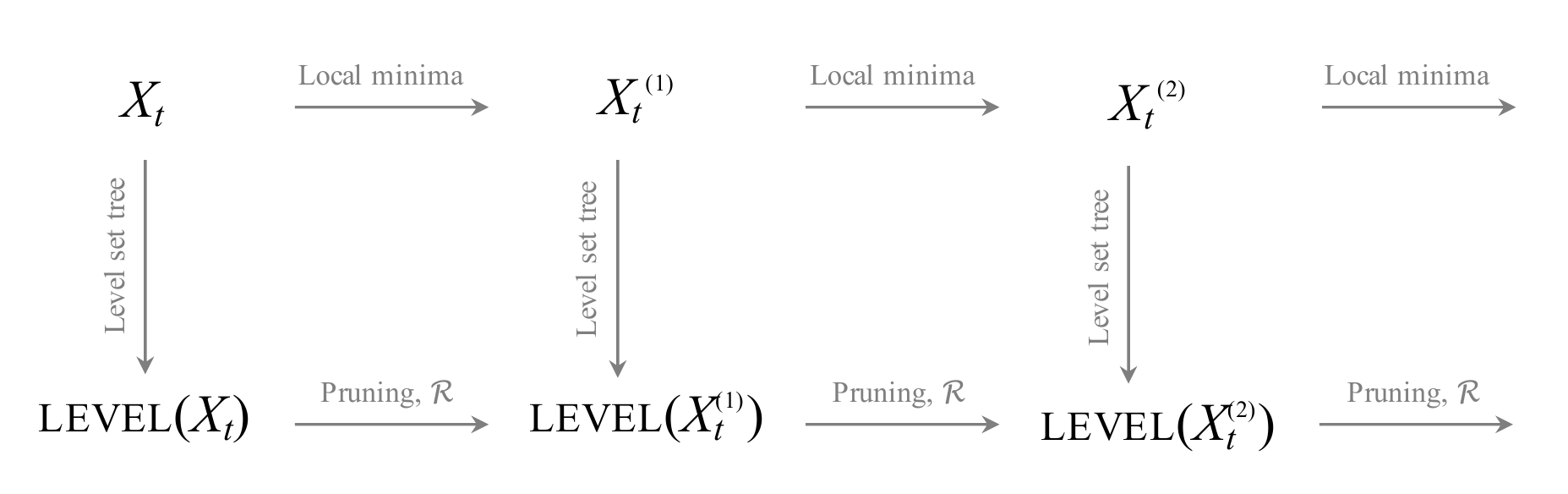}
\caption{Transition to the local minima of a function $X_t$ 
corresponds to the Horton pruning $\cR$ of the corresponding level set tree
$\textsc{level}(X_t)$.}
\label{fig:prune2}
\end{figure}

\subsection{Excursion of a symmetric random walk}
\label{sec:cgw}
We turn now to random walks $\{X_k\}_{k\in\mathbb{Z}}$.
Linear interpolation of their trajectories corresponds to the tamed 
continuous functions.
A random walk $\{X_k\}_{k\in\mathbb{Z}}$ with
a transition kernel $p(x,y)$ is called {\it homogeneous} 
if $p(x,y)\equiv p(y-x)$ for any $x,y\in\mathbb{R}$. 
A homogeneous random walk is {\it symmetric} if $p(x)=p(-x)$ for all $x\in\mathbb{R}$.
The transition kernel of a symmetric random walk can be represented as 
the even part of a p.d.f. $f(x)$ with support ${\sf supp}(f) \subseteq \mathbb{R}_+$: 
\index{random walk!homogeneous}
\index{random walk!homogeneous!symmetric}
\be\label{eqn:semkernel_pff}
p(x)=\frac{f(x)+f(-x)}{2}.
\ee
We assume that $p(x)$, and hence $f(x)$, is an atomless density function.

We write $\{X^{(1)}_k\}_{k\in\mathbb{Z}}$ for the sequence of local minima of  $\{X_k\}_{k\in\mathbb{Z}}$, listed in the order of occurrence, from left to right. 
In particular, we set $X^{(1)}_0$ to be the value of the leftmost local minima of $X_k$ for $k \geq 0$.
Recursively, we let $\{X^{(j+1)}_k\}_{k\in\mathbb{Z}}$ denote the sequence of local minima of $\{X^{(j)}_k\}_{k\in\mathbb{Z}}$.

\begin{lem}[{{\bf Local minima of random walks, \cite{ZK12}}}]
\label{inv_prune}
The following statements hold.
\begin{itemize}
  \item[(i)] The sequence of local minima  $\{X^{(1)}_k\}_{k\in\mathbb{Z}}$ of a homogeneous random walk $\{X_k\}_{k\in\mathbb{Z}}$ is itself a homogeneous random walk.   
  \item[(ii)] The sequence of local minima  $\{X^{(1)}_k\}_{k\in\mathbb{Z}}$ of a symmetric homogeneous random walk $\{X_k\}_{k\in\mathbb{Z}}$ is itself a symmetric homogeneous random walk.
\end{itemize}
\end{lem}
\begin{proof}
Let $d_j = X^{(1)}_{j+1}-X^{(1)}_j$.
We have, for each $j$,
\be
\label{twosum}
d_j = \sum_{i=1}^{\xi_+} Y_i - \sum_{i=1}^{\xi_-} Z_i,
\ee
where the first sum corresponds to $\xi_+$ positive increments of $X_k$
between a local minimum $X^{(1)}_j$ and the subsequent local maximum $m_j$, and
the second sum corresponds to $\xi_-$ negative increments between the local maximum
$m_j$ and the subsequent local minimum $X^{(1)}_{j+1}$.
Accordingly, $\xi_+$ and $\xi_-$ are independent geometric random variables 
\[\xi_{+}\stackrel{d}{\sim}{\sf Geom}_1(p^{+}),\quad
\xi_{-}\stackrel{d}{\sim}{\sf Geom}_1(p^{-})\]
with parameters,
respectively,
\[p^{+}=\int\limits_0^\infty p(x)\,dx\quad\text{ and }\quad 
p^{-}=\int\limits_{-\infty}^0 p(x)\,dx,\]
and
$Y_i$, $Z_i$ are i.i.d. positive continuous random variables
with p.d.f.s, respectively,
\[f_Y(x) = \frac{p(x){\bf 1}_{\{x\ge 0\}}}{p^+}\quad\text{ and }
\quad 
f_Z(x) = \frac{p(-x){\bf 1}_{\{x\le 0\}}}{p^-}.\]

\medskip
\noindent
$(i)$ By independence of increments of a random walk, the random jumps $d_j$ have the same distribution for each $j$. 
This establishes the statement.

\medskip
\noindent
$(ii)$ For the kernel of a symmetric random walk, we have representation \eqref{eqn:semkernel_pff}.
In this case, $\xi_+$ and $\xi_-$ are independent geometric random variables with parameters
$p^+=p^-=1/2$ and
$Y_i$, $Z_i$ are i.i.d. positive continuous random variables with p.d.f. $f(x)$.
Hence, both sums in \eqref{twosum} have the same 
distribution, and their difference has a symmetric distribution.
Thus $\{X^{(1)}_j\}_{j\in\mathbb{Z}}$ is a symmetric homogeneous random walk.
\end{proof}

We notice that the symmetric kernel $p^{(1)}(x)$ for the chain of local minima $\{X^{(1)}_j\}_{j\in\mathbb{Z}}$
is necessarily different from $p(x)$ in both parts of Lemma~\ref{inv_prune}.
Hence, the random walk $\{X^{(1)}_j\}$ of local minima is always different from 
the initial random walk $\{X_k\}$.
In a symmetric case, however, both the processes happen to be closely related in 
terms of the structure of their level set trees. 
Now we explore this relation.

\begin{figure}[t] 
\centering\includegraphics[width=0.7\textwidth]{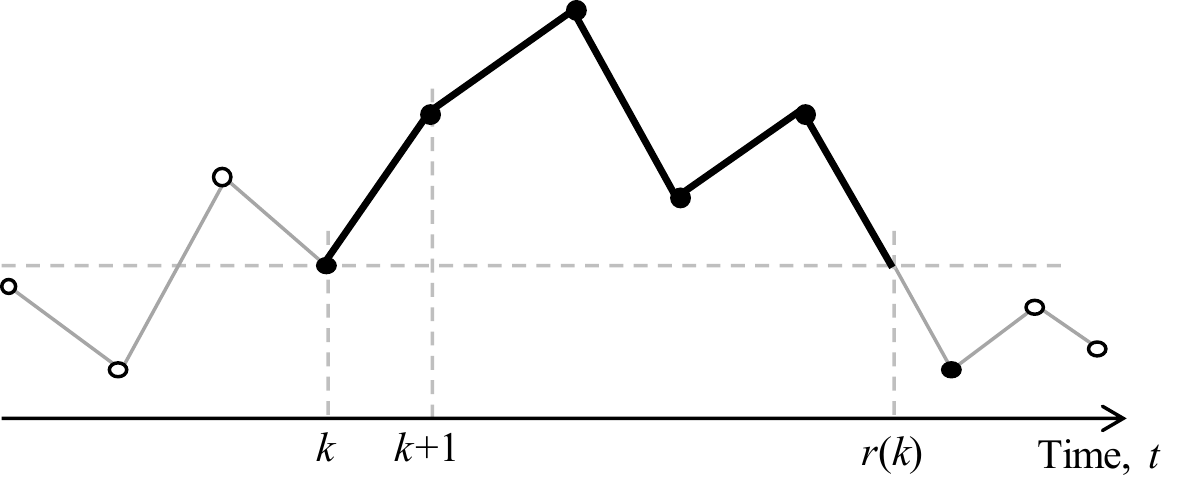}
\caption{Excursion of a symmetric homogeneous random walk: illustration.
The values of time series $X_k$, $k\in\mathbb{Z}$, are shown by circles;
the circles that form the excursion are filled.
The linear interpolation function $X_t$, $t\in\mathbb{R}$, is shown by solid line;
the excursion of $X_t$ on $[k,r(k)]$ is shown in bold. 
This is the first positive excursion of $X_t$ to the right of $k$.
}
\label{fig:ex}
\end{figure}

Consider linear interpolation $\{X_t\}_{t\in\mathbb{R}}$ of a symmetric homogeneous 
random walk $\{X_k\}_{k\in\mathbb{Z}}$ with an atomless transition kernel $p(x)$. 
For any $k\in\mathbb{Z}$, let 
\[T^{\rm ex}=T^{\rm ex}(X_t,k)\in\BL\] 
be the level set tree of the 
first positive excursion of $X_t-X_k$ to the right of $k$, with
convention $T^{\rm ex}=\phi$ if $X_{k+1}-X_k<0$.
Formally, let $r=r(k)\in\mathbb{R}$ be the unique epoch such that (Fig.~\ref{fig:ex}) 
\[r\ge k,\quad X_t>X_k \text{ for all } t\in(k,r),\quad \text{and } X_r=X_k.\] 
The epoch $r(k)$ is almost surely finite, as can be demonstrated by a renewal argument
using the symmetry of the increments of $X_k$.
We define
\[T^{\rm ex}(X_t,k):=\textsc{level}(X_t,t\in[k,r(k)]).\]
It follows from this definition that
\[\{X_{k+1}-X_k>0\}\Leftrightarrow \{T^{\rm ex}(X_t,k)\ne\phi\}.\]
The basic properties of symmetric homogeneous random walks imply that
the distribution of $T^{\rm ex}(X_t,k)$ is the same for all points $k\in\mathbb{Z}$.
This justifies the following definition.
\begin{Def}[{{\bf Positive and nonnegative excursions}}]
\label{def:Tex}
In the above setup, we call process $X^{\rm ex}_t$ a nonnegative 
excursion of the linearly interpolated symmetric homogeneous 
random walk $\{X_t\}_{t\in\mathbb{R}}$ if
\[X^{\rm ex}_t\stackrel{d}{=}\{X_{s-k}-X_k,~s\in[k,r(k)]\}\quad\text{for any }k\in\mathbb{Z}.\]
Furthermore, we call process $X^{\rm ex}_t$ a positive
excursion of the linearly interpolated symmetric homogeneous 
random walk $\{X_t\}_{t\in\mathbb{R}}$ if
\[X^{\rm ex}_t\stackrel{d}{=}\{X_{s-k}-X_k,~s\in[k,r(k)]~\big|X_{k+1}-X_k>0\}\quad\text{for any }k\in\mathbb{Z}.\]
\end{Def}
\index{excursion}
\index{excursion!positive excursion}
\index{excursion!nonnegative excursion}

\noindent 
A positive excursion defined above will also be called a {\it positive right excursion}.
The corresponding positive excursion in the reversed time order, starting from $k$ and going in the negative time direction,
will be called {\it positive left excursion}.
According to Def.~\ref{def:Tex}, a nonnegative excursion
may consist of a single point (if $r(k)=k$), in which case its level set tree
is the empty tree.
A positive excursion necessarily includes at least one positive value, and 
its level set tree is non-empty.
\index{excursion!positive right excursion}
\index{excursion!positive left excursion}

\bigskip
\noindent
Consider a homogeneous random walk $X_k$ with a symmetric atomless transition kernel $p(x)$,
$x\in \mathbb{R}$,
represented as in \eqref{eqn:semkernel_pff}.
Note that $X_k$ is time reversible, with $p(x)$ also being the transition kernel 
of the reversed process. 
The increment between a pair of 
consecutive local extrema (a minimum and a maximum) of $X_k$ is a sum of ${\sf Geom}_1(1/2)$-distributed number of  i.i.d. $f(x)$-distributed random variables, and therefore has density 
\be\label{eqn:geomf_s}
{\sf s}(x):=\sum\limits_{k=1}^\infty 2^{-k}\, \underbrace{f\ast \hdots \ast f}_{k \text{ times}} (x).
\ee
We now examine a positive-time process $\{X_k\}_{k\ge 0}$, conditioned on $X_0=0$.
Consider a sequence of local minima $\big\{X_j^{(1)}\big\}_{j\ge 1}$, where we set $X_0^{(1)}=0$, and $X_1^{(1)}, X_2^{(1)},\hdots$ are the local
minima of the random walk $X_k$, listed from left to right. 
For a positive right excursion $X^{\sf ex}$ originating at $k=0$, the number $N$ of leaves 
in the level set tree $\textsc{level}(X^{\sf ex})$ is determined by the location of the first 
local minimum below zero: 
$$N=\min\{j \geq 1  \,:\, X_j^{(1)} \leq 0\}.$$
The number of edges in the level set tree is $\#\textsc{level}(X^{\sf ex})=2N-1$.
Moreover, let $\kappa>0$ be the time of the first local minimum below zero, $X_\kappa=X_N^{(1)}$. 
Next, we define the quantity by which the first nonpositive local minimum 
of $X_k$ falls below the starting point at zero.
\begin{Def}[{{\bf Extended positive excursion and excess value}}]
\label{def:extendedTex}
\index{excursion!extended positive}
\index{excursion!extended positive!excess value}
In the above setup, the process $\breve{X}^{\sf ex}=\{X_t\}_{t \in [0,\kappa]}$ is called
the {\it extended positive excursion} or {\it extended positive right excursion}. 
That is, $\breve{X}^{\sf ex}$ is obtained by extending the excursion $X^{\sf ex}$ 
until the first local minimum $X_\kappa$ below the starting value.
The quantity $~\Lambda\big(\breve{X}^{\sf ex} \big):=-X_N^{(1)}$ is called 
the {\it excess value} of the extended excursion. 
This definition is illustrated in Fig.~\ref{fig:ex_ex}(a).
\end{Def}
\noindent The notions of the extended positive excursion and the excess value 
$\Lambda\big(\breve{X}^{\sf ex} \big)$ can 
be expanded to the left and right excursions with arbitrary initial values. 

\medskip

\begin{thm}[{{\bf Combinatorial excursion tree is critical Galton-Watson}}]
\label{thm:excursion-tree-main}
Suppose $X^{\sf ex}$ is a positive excursion of a homogeneous random walk on $\mathbb{R}$ with a symmetric atomless transition kernel and $T=\textsc{level}(X^{\sf ex})$. Then, the combinatorial shape of $T$
has the same distribution on $\cBT^|$ as the critical binary Galton-Watson tree:
\[\textsc{shape}(T) \overset{d}{\sim} \mathcal{GW}\left({1\over 2},{1\over 2}\right).\]
\end{thm}
\begin{proof}
Recall that $\textsc{shape}(T)$ is almost surely in $\cBT^|$. 
Without loss of generality we consider a positive right excursion $X^{\sf ex}$ originating at $k=0$, where we set $X_0=0$.
The tree $\textsc{shape}(T)$ has exactly one leaf if and only if 
the first local minimum falls below zero.
That is, if the jump from $X_0=0$ to the first local maximum is smaller than or equal to 
the size of the jump from the first local maximum to the consecutive local minimum.
The probability of this event is: 
\be\label{eqn:orderhalf}
{\sf P}\left({\sf ord}(T)=1\right)=\int\limits_0^\infty \left(\int\limits_x^\infty {\sf s}(y)\,dy\right)\,{\sf s}(x)dx={1 \over 2}.
\ee

According to the characterization of the critical Galton-Watson distribution 
$\mathcal{GW}(1/2,1/2)$ given in Remark \ref{rem:GWchar} of Sect.~\ref{sec:GW},
the proof will be complete if we show that conditioned on ${\sf ord}(T) \geq 2$, the tree $\textsc{shape}(T)$ splits into a pair of
complete subtrees sampled independently from the same distribution as $\textsc{shape}(T)$.
This step is completed as follows.

\begin{figure}[t] 
\centering\includegraphics[width=0.8\textwidth]{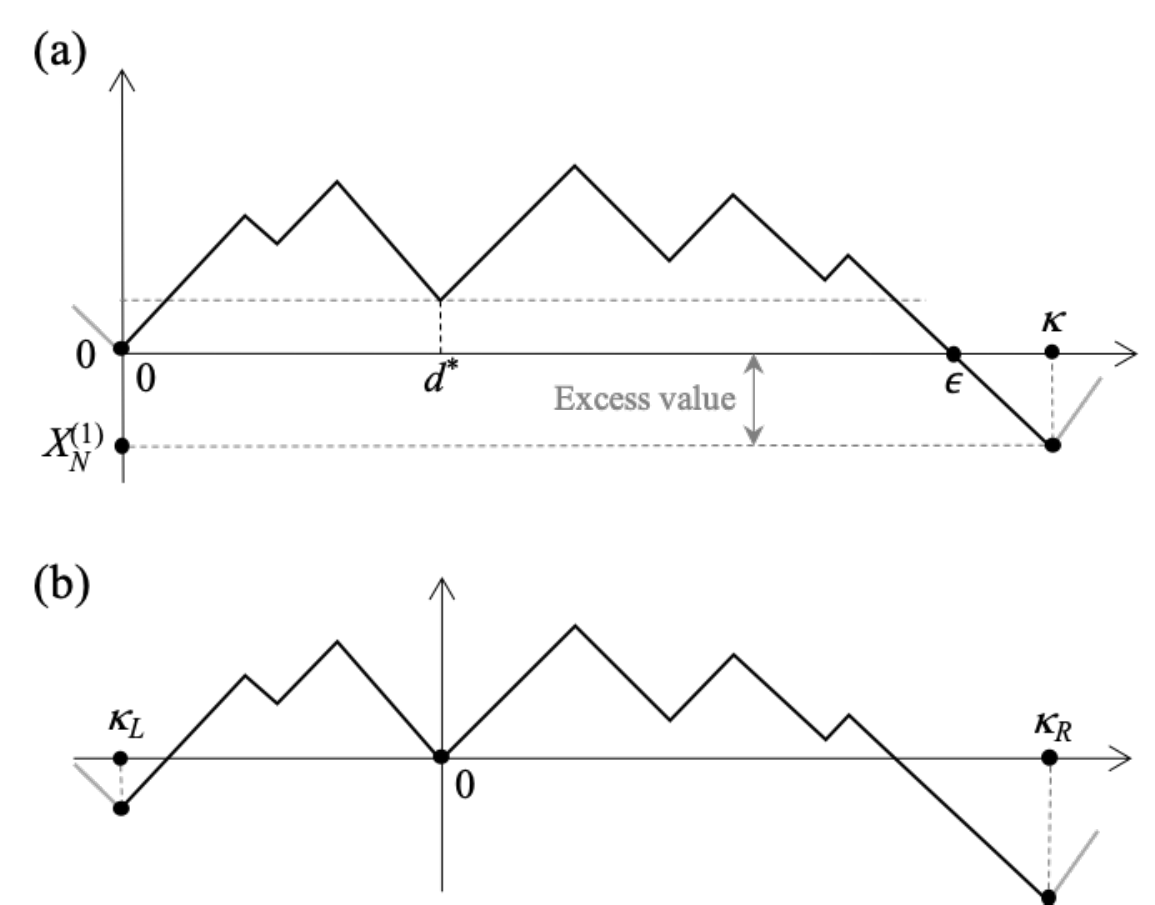}
\caption{Extended excursion: An illustration. 
(a) Extended positive right excursion $\breve{X}^{\sf ex}$ on the interval $[0,\kappa]$.
It is obtained by extending 
the respective positive right excursion $X^{\sf ex}$
on the interval $[0,\varepsilon]$ until the first local minimum $X_{\kappa}$
below zero.
The panel also illustrates the excess value $-X_N^{(1)}$ and the lowest local
minimum of the excursion at epoch $d^*$.
(b) A trajectory from $\mathcal{X}_{LR}$ on the interval $[\kappa_L,\kappa_R]$
consists of a positive left excursion on $[\kappa_L,0]$
and a positive right excursion on $[0,\kappa_R]$.
Observe that the trajectory in panel (b) is obtained by a horizontal
and vertical shift of the trajectory in panel (a).
The proof of Thm.~\ref{thm:excursion-tree-main} uses the one-to-one correspondence between extended (left/right) positive excursions with ${\sf ord}(T) \geq 2$ of  panel (a) and trajectories of panel (b).
 }
\label{fig:ex_ex}
\end{figure}

Consider the space $\mathcal{X}_L$ of all the trajectories of all extended positive left excursions
originating at $X_0=0$ and whose level set trees are of Horton-Strahler order $\geq 2$.
Similarly, consider the space $\mathcal{X}_R$ of all the trajectories of all extended positive right excursions 
originating at $X_0=0$ and whose level set trees are of Horton-Strahler order $\geq 2$.
We know from \eqref{eqn:orderhalf} that the probability measure for each of the sets $\mathcal{X}_L$ and $\mathcal{X}_R$
totals $1/2$. Thus, we may consider the union set of left and right extended positive excursions $\mathcal{X}_L \cup \mathcal{X}_R$
and equip it with a new probability measure obtained by gluing together the two respective restrictions of probability measures for the left and the right positive excursions.
That is the probability measure over the trajectories in $\mathcal{X}_L \cup \mathcal{X}_R$ when restricted to either $\mathcal{X}_L$ or $\mathcal{X}_R$,
will coincide with the respective probability measures for the left and for the right positive excursions, with the total probability adding up to one.
Now, since all the left and the right extended positive excursions in $\mathcal{X}_L \cup \mathcal{X}_R$ have Horton-Strahler order $\geq 2$, 
for each $\breve{X}^{\sf ex} \in \mathcal{X}_L \cup \mathcal{X}_R$ there is almost surely a unique integer $d^*$ such that $\breve{X}^{\sf ex}_{d^*}>0$ 
is the smallest local minimum of the excursion $\breve{X}^{\sf ex}$.

Next, conditioning on $X_0=0$ being a local minimum of $X_t$, we 
consider a space $\mathcal{X}_{LR}$ of all possible trajectories such that 
each trajectory consists of the left and the right extended positive excursions originating from $X_0=0$ (with no restrictions on their orders).
For a trajectory in $\mathcal{X}_{LR}$, let $\kappa_L<0$ and $\kappa_R>0$ denote the (random) endpoints of the left and the right extended positive excursions.
Thus, a trajectory $X_t$, $t \in [\kappa_L, \kappa_R]$, in $\mathcal{X}_{LR}$ consists of a left extended positive excursion 
$X_t~(\kappa_L\leq t \leq 0)$ and a right extended positive excursion $X_t~(0 \leq t \leq \kappa_R)$.
This construction is illustrated in Fig.~\ref{fig:ex_ex}(b).
The probability measure over the space $\mathcal{X}_{LR}$ is a product measure of the left and the right positive excursions.
We claim that there exists a bijective measure preserving shift map  
$$\Psi:\, \mathcal{X}_{LR} \rightarrow \mathcal{X}_L \cup \mathcal{X}_R.$$
Indeed, if the excess value $\Lambda\big(\{X_t\}_{\kappa_L\leq t \leq 0}\big)=-X_{\kappa_L}$ for the left excursion is smaller 
than the excess value $\Lambda\big(\{X_t\}_{\kappa_L \leq t \leq 0}\big)=-X_{\kappa_R}$ for the right excursion, we set 
$$\Psi\big(\{X_t\}_{\kappa_L \leq t \leq \kappa_R}\big)=\{X_{t+\kappa_L}-X_{\kappa_L}\}_{0\leq t \leq -\kappa_L+\kappa_R} ~\in \mathcal{X}_R.$$
Otherwise, we set
$$\Psi\big(\{X_t\}_{\kappa_L \leq t \leq \kappa_R}\big)=\{X_{t+\kappa_R}-X_{\kappa_R}\}_{\kappa_L-\kappa_R \leq t \leq 0} ~\in \mathcal{X}_L.$$
The map $\Psi$ is one-to-one onto as it consists of the vertical and the horizontal shifts. Also observe that under the mapping $\Psi$, the point $(0,0)$ of a trajectory in $\mathcal{X}_{LR}$ is sent to the point $(d^*,\breve{X}^{\sf ex}_{d^*})$ of the image trajectory in $\mathcal{X}_L \cup \mathcal{X}_R$.
We can construct $\Psi^{-1}:\, \mathcal{X}_L \cup \mathcal{X}_R  \rightarrow \mathcal{X}_{LR}$ accordingly as a map that shifts a trajectory $\breve{X}^{\sf ex}$ in $\mathcal{X}_L \cup \mathcal{X}_R$ by  subtracting $(d^*,\breve{X}^{\sf ex}_{d^*})$.
Finally, because we take the same product of the transition kernel values ${\sf s}(x)$ for the increments of a trajectory in $\mathcal{X}_{LR}$
as for its image in $\mathcal{X}_L \cup \mathcal{X}_R$ under the one-to-one mapping $\Psi$, the mapping $\Psi$ is measure preserving.

Thus, since vertical and horizontal shifts of a function preserve its level set tree,
we conclude that the distribution of the level set trees for the trajectories in $\mathcal{X}_{LR}$ and the trajectories in $\mathcal{X}_L \cup \mathcal{X}_R$ coincide.
The level set tree for a trajectory in $\mathcal{X}_{LR}$ consists of a stem that branches into two level set trees of the left and right positive excursions adjacent to $X_0=0$,
sampled independently from the same distribution as $\textsc{shape}(T)$. This is so since for the trajectories in $\mathcal{X}_{LR}$, $~X_0=0$ is the smallest local minimum.
Finally, we observe that the distribution of $\textsc{shape}\big(\textsc{level}(\breve{X}^{\sf ex})\big)$ is the same when $\breve{X}^{\sf ex}$ is sampled from $\mathcal{X}_L$ as when it is sampled from $\mathcal{X}_R$. Thus, for $\breve{X}^{\sf ex}$ sampled from $\mathcal{X}_R$, $\textsc{shape}\big(\textsc{level}(\breve{X}^{\sf ex})\big)$ consists of a stem that branches into two level set trees. If $X^{\sf ex}$ is the right positive excursion corresponding to $\breve{X}^{\sf ex}$ sampled from $\mathcal{X}_R$, then almost surely,
$$\textsc{shape}\big(\textsc{level}(X^{\sf ex})\big)=\textsc{shape}\big(\textsc{level}(\breve{X}^{\sf ex})\big).$$
Thus, conditioned on ${\sf ord}(T) \geq 2$, the tree $\textsc{shape}(T)$ splits into a pair of
complete subtrees sampled independently from the same distribution as $\textsc{shape}(T)$. 
This completes the proof.
\end{proof}

Theorem~\ref{thm:excursion-tree-main} establishes that the level set trees 
of symmetric random walks have the same combinatorial structure (equivalent
to that of a ciritical binary Galton-Watson tree), independently of the
choice of the transition kernel $p(x)$. 
The planar embedding and metric structure of the level set trees, however,
may depend on the kernel, as we illustrate in the following remark.
\begin{Rem}\label{rem:left_right_subtrees}
Consider an extended positive right excursion $\breve{X}^{\rm ex}$ of a symmetric homogeneous random walk and let $T=\textsc{level}(\breve{X}^{\rm ex})$ be its level set tree.
Condition on the event ${\sf ord}(T)\ge 2$, which
ensures that the left and right principal subtrees of $T$, which we denote
$T^\ell$ and $T^r$, respectively, are non-empty.

It follows from the construction in the proof of Thm.~\ref{thm:excursion-tree-main} that  
the subtrees $T^\ell$ and $T^r$ can be sampled as follows.
Consider two independent excursions -- 
an extended positive right excursion $\breve{X}_t^{\sf ex,r},~t\in [0,\kappa_r]$, and an extended positive left excursion $\breve{X}_t^{\sf ex,\ell},~t\in [\kappa_\ell,0]$.
Next, condition on the event that the excess value of the left excursion is less than
that of the right excursion:
$$\Lambda\big(\{\breve{X}_t^{\sf ex,\ell}\}_{t\in [\kappa_\ell,0]} \big) 
< \Lambda\big(\{\breve{X}_t^{\sf ex,r}\}_{t\in [0,\kappa_r]} \big).$$
Denote by $X^{\sf ex,\ell}$ and $X^{\sf ex,r}$ the positive left and right excursions
that correspond to the extended excursions $\breve{X}^{\sf ex,\ell}$ and 
$\breve{X}^{\sf ex,r}$.
Then, 
\be\label{eqn:left-right-subtreesL}
T^\ell\stackrel{d}{=}\textsc{level}({X}_{t}^{\sf ex,\ell}) ~\text{ and }
~T^r\stackrel{d}{=}\textsc{level}({X}_t^{\sf ex,r}).
\ee
Write $X^{\sf ex}$ for the positive right excursion that corresponds 
to the extended excursion $\breve{X}^{\rm ex}$.
Then, the stem of the tree $\textsc{level}(X^{\sf ex}) \in \BL^|$ has length equal to $\Lambda\big(\{\breve{X}_t^{\sf ex,\ell}\}_{t\in [0,\kappa_\ell]} \big)$.
This, in general, may introduce dependence between the planar embedding of $T$ 
and its edge lengths.
Such dependence is absent in the exponential critical binary Galton-Watson tree 
${\sf GW}(\lambda)$.

\end{Rem}

\begin{figure}[t] 
\centering\includegraphics[width=0.8\textwidth]{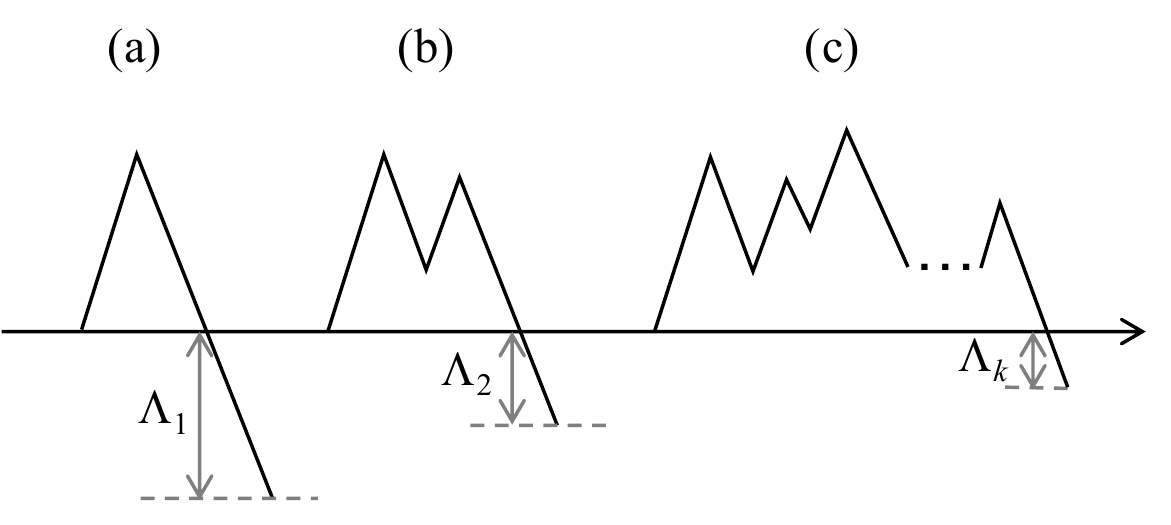}
\caption{Excess value $\Lambda(\breve{X}^{\sf ex})$ may depend on the tree shape 
$\textsc{shape}(\textsc{level}(\breve{X}^{\sf ex}))$.
Equations \eqref{eqn:lambda1s},\eqref{eqn:lambda21s} demonstrate why the excess value 
for a $\Lambda$-shaped excursion of panel (a) may differ from the excess value of an $M$-shaped
excursion of panel (b), and, hence, from the excess value of a general excursion of panel (c). }
\label{fig:excess_vs_tree}
\end{figure}

\bigskip
\noindent
Next, condition on the event that $X^{\sf ex}$ is an $\Lambda$-shaped excursion, which is equivalent to
$$\{\#\textsc{level}(X^{\sf ex})=1\}~= ~\{\#\textsc{level}(\breve{X}^{\sf ex})=1\}.$$
Then, the density function of the  excess value $\Lambda\big(\breve{X}^{\sf ex} \big)$ 
that we denote by
$\lambda_1(x)$ satisfies
\be\label{eqn:lambda1s}
\lambda_1(x)=2\int\limits_0^\infty {\sf s}(x+y) {\sf s}(y) \,dy,
\ee
where ${\sf s}(x)$ was defined in \eqref{eqn:geomf_s}.
This is so because conditioned on $$\{\#\textsc{level}(\breve{X}^{\sf ex})=1\},$$ the extended excursion $\breve{X}^{\sf ex}$ consists of an ${\sf s}$-distributed jump upward, and 
a larger ${\sf s}$-distributed jump downward. The excess value $\Lambda\big(\breve{X}^{\sf ex} \big)$ is the difference between the jumps. The multiple of $2$ in \eqref{eqn:lambda1s}
is due to conditioning upon the event of probability $1/2$ that the jump up is smaller than the jump down. 

\medskip
\noindent
Similarly, one can condition on the event that $X^{\sf ex}$ is an $M$-shaped excursion, which is equivalent to the event that the level set
tree has $2$ leaves and $3$ edges, i.e.,
$$\{\#\textsc{level}(X^{\sf ex})=3\}~= ~\{\#\textsc{level}(\breve{X}^{\sf ex})=3\}.$$
Then, the density function of the  excess value $\Lambda\big(\breve{X}^{\sf ex} \big)$ 
that we denote by
$\lambda_2(x)$ satisfies
\be\label{eqn:lambda21s}
\lambda_2(x)=2\int\limits_0^\infty \lambda_1(x+y) \lambda_1(y) \,dy.
\ee
This is so because conditioned on $$\{\#\textsc{level}(\breve{X}^{\sf ex})=3\},$$ the extended excursion $\breve{X}^{\sf ex}$ consists of two $\Lambda$-shaped (left and right) extended positive excursions originating from the only local minimum within the interior of 
the time domain $[0,\kappa]$ of $\breve{X}^{\sf ex}$. The excess value $\Lambda\big(\breve{X}^{\sf ex} \big)$ is the difference between the two $\lambda_1$-distributed excess values of the two $\Lambda$-shaped extended positive excursions.

\bigskip
\noindent
\begin{lem}\label{lem:shapeVSexcessvalue}
Consider a homogeneous random walk $X_k$ on $\mathbb{R}$ with a symmetric atomless transition kernel $p(x)$, $x\in\mathbb{R}$, i.e.,
there is a p.d.f. $f(x)$ with the support 
${\sf supp}(f) \subseteq \mathbb{R}_+$ such that 
$p(x)={1\over 2}(f(x)+f(-x))$.
Consider an extended positive excursion $\breve{X}^{\sf ex}$ of $X_k$, 
and the corresponding positive excursion $X^{\sf ex}$. 
Let $T=\textsc{level}(X^{\sf ex})$. 
Then, the following statements are equivalent:
\begin{itemize}
\item[(a)] $T$ is independent of the excess value $\Lambda\big(\breve{X}^{\sf ex} \big)$; 
\item[(b)] conditioned on $\textsc{p-shape}(T)$, the edge lengths are identically distributed;
\item[(c)] $f(x)$ is an exponential p.d.f.
\end{itemize}
If any of these statements holds, then the edge lengths are i.i.d. exponential random variables.
\end{lem}

\begin{proof}
$(c)\Rightarrow (a)$. It is easy to show via the characteristic functions that ${\sf s}(x)$ is an exponential p.d.f. if and only if $f(x)$ is an exponential p.d.f..
The memoryless property of the exponential random variables implies that if ${\sf s}(x)$ is an exponential p.d.f., then $T=\textsc{level}(X^{\sf ex})$ is independent of the excess value $\Lambda\big(\breve{X}^{\sf ex} \big)$.

$(a)\Rightarrow (c)$.
The excess value of a $\Lambda$-shaped extended positive excursion has the same 
distribution as the excess value of a $M$-shaped extended  positive excursion if and only if 
 $\lambda_1(x) \equiv \lambda_2(x)$. 
If this equality holds, then by equation \eqref{eqn:lambda21s} 
the p.d.f. $\lambda_1(x)$ satisfies equation \eqref{eqn:intgg} in Lemma \ref{lem:expcharacterization}, which implies that $\lambda_1(x) \equiv \lambda_2(x)$ is an exponential density function.
Hence, from \eqref{eqn:lambda1s} and Lemma \ref{lem:expcharacterization2} we conclude that ${\sf s}(x) \equiv \lambda_1(x)$ is an exponential density function, 
which in turn implies that $f(x)$ is exponential.

$(b)\Rightarrow (c)$.
The distribution of the leaf length is the minimum of two independent ${\sf s}(x)$-distributed random variables. Thus the cumulative distribution function of the leaf length equals 
$$F_1(x) = 1-\left(\int\limits_x^\infty {\sf s}(y) \,dy\right)^2.$$

The cumulative distribution function for the length of the non-leaf edge in a $Y$-shaped branch equals 
$$F_2(x)=1-\left(\int\limits_x^\infty \lambda_1(y) \,dy\right)^2.$$
Here, $F_1(x)\equiv F_2(x)$ if and only if $\lambda_1(x) \equiv {\sf s}(x)$, which by Lemma \ref{lem:expcharacterization} and equation \eqref{eqn:lambda1s}
happens if and only if ${\sf s}(x)$ is exponential. 
This implies that $f(x)$ is an exponential p.d.f..

$(c)\Rightarrow (b)$. 
Suppose $f(x)$ is the exponential density with parameter $\lambda$, i.e., $f(x)=\phi_\lambda(x)$.
According to the construction in the proof of Thm.~\ref{thm:excursion-tree-main}, together with statement $(a)$,
and because any edge in $T$ is a stem of a unique descendant subtree of $T$,
it suffices to prove that conditioned on $\textsc{p-shape}(T)$, the tree stem (root edge) has exponential distribution
with parameter $\lambda$.  

According to \eqref{eqn:geomf_s}, ${\sf s}(x)$ has the exponential density with parameter $\lambda/2$.
Conditioned on ${\sf ord}(T)=1$, the length of the stem (the only edge of the tree) 
equals the minimum of two independent exponentially distributed random variables with
density ${\sf s}(x)$, and hence has the exponential density with parameter $\lambda$.
Conditioned on ${\sf ord}(T)\ge 2$, the length of the stem is the minimum 
of the excess values of two independent extended positive excursions.
By the memoryless property of the exponential distribution,
each of these excess values has the exponential distribution with parameter $\lambda/2$.
Hence, the stem length has the exponential distribution with parameter $\lambda$.
This shows that the edge lengths in $T$ have the same distribution.

\medskip
\noindent
Finally, suppose any and therefore all three of the statements (a)-(c) hold,
then properties (b) and (c) insure that the edge lengths are identically and exponentially distributed,
while property (a) insures the independence of edge lengths. 
This completes the proof.
\end{proof}

\subsection{Exponential random walks}
\label{sec:erw}
Proposition~\ref{ts_prune} (and the subsequent comment) suggests that the problem 
of finding Horton self-similar trees with edge lengths is related to finding 
{\it extreme-invariant processes}
\be
\label{eqn:d_inv}
\big\{X^{(1)}_j-X^{(1)}_0\big\}_{j\in\mathbb{Z}}~\stackrel{d}{=}~\big\{\zeta(X_k-X_0)\big\}_{k\in\mathbb{Z}} \quad {\rm for~some~} \zeta>0,
\ee
where $\{X_k\}_{k\in\mathbb{Z}}$, is a time series with an atomless
distribution at every $k$ and $X^{(1)}_j$ is the corresponding time series of local minima.
The equality \eqref{eqn:d_inv} is understood as the distributional equivalence of two time series.

In this section we establish a sufficient condition for a symmetric homogeneous 
random walk to solve \eqref{eqn:d_inv}, and show that in this case $\zeta=2$.
Moreover, we show that if a symmetric random walk $X_k$ satisfies \eqref{eqn:d_inv},
the level set tree of its finite positive excursion, 
considered as elements of $\L$, is self-similar according to Def.~\ref{def:distss}.
Symmetric random walks with exponential increments is an example of a process
that solves \eqref{eqn:d_inv}.

The following result describes the solution of the problem \eqref{eqn:d_inv}
in terms of the characteristic function of $f(x)$.

\begin{prop}[{{\bf Extreme-invariance of a symmetric homogeneous random walk, \cite{ZK12}}}]
\label{prop:DSS}
Consider a symmetric homogeneous random walk  $\{X_k\}_{k\in\mathbb{Z}}$ 
with a transition kernel $~p(x)=\frac{f(x)+f(-x)}{2}~$, where $f(x)$ is a p.d.f. with
support ${\sf supp}(f) \subseteq \mathbb{R}_+$ and a finite second moment.
Then, the local minima $\{X^{(1)}_j\}_{j\in\mathbb{Z}}$ of $\{X_k\}_{k\in\mathbb{Z}}$ form a symmetric homogeneous random walk 
with transition kernel
\be \label{eq:scaling}
p^{(1)}(x)= \zeta^{-1}\,p(x/\zeta),\quad \zeta >0
\ee
if and only if $\zeta=2$ and
\be
\label{laplace}
\Re\left[\widehat{f}(2s)\right]=\left|\frac{\widehat{f}(s)}{2-\widehat{f}(s)}\right|^2,
\ee
where $\widehat{f}(s)$ is the characteristic function of $f(x)$ and
$\Re[z]$ denotes the real part of $z\in\mathbb{C}$.
\end{prop}
\begin{proof}
Each increment between the consecutive local minima of $X_k$ can be represented as 
$d_j$ of \eqref{twosum}
where $\{Y_i\}$ and $\{Z_i\}$ are i.i.d. with density $f(x)$, 
and $\xi_+$ and $\xi_-$ are independent geometric random variables with parameter $1/2$, i.e., ${\sf Geom}_1(1/2)$. 

\medskip
\noindent
The law of total variance readily implies that $\zeta=2$. Indeed,
\begin{align}\label{eq:ltv}
{\sf Var} \left(\sum_{i=1}^{\xi_+} Y_i \right) &= {\sf E} \left[{\sf Var}\left(\sum_{i=1}^{\xi_+} Y_i ~\Big|~\xi_+\right) \right]+{\sf Var}\left({\sf E}\left[\sum_{i=1}^{\xi_+} Y_i ~\Big|~\xi_+\right] \right) \nonumber \\
&=\sigma^2 {\sf E}[\xi_+]+\mu^2 {\sf Var}(\xi_+)=2(\mu^2+\sigma^2 ),
\end{align}
where $\mu$ and $\mu^2+\sigma^2$ are the first and the second moments of $f(x)$ respectively. Thus, on one hand,
the variance of the increments of $X_k$ is
$${\sf Var}(X_{k+1}-X_k)=\mu^2+\sigma^2$$
since for a symmetric homogeneous random walk,  ${\sf E}[X_{k+1}-X_k]=0$. 
On the other hand, \eqref{twosum} and  \eqref{eq:ltv} imply that the variance of the increments in the sequence of local minima $X^{(1)}_j$ is
$${\sf Var}(X^{(1)}_{j+1}-X^{(1)}_j)={\sf Var}(d_j)={\sf Var}\left(\sum_{i=1}^{\xi_+} Y_i \right)+{\sf Var}\left(\sum_{i=1}^{\xi_-} Z_i \right)=4(\mu^2+\sigma^2).$$
Hence, ${\sf Var}(X^{(1)}_{j+1}-X^{(1)}_j)=4\,{\sf Var}(X_{k+1}-X_k)$, and therefore $\zeta=2$ is the only value of $\zeta$ for which the scaling \eqref{eq:scaling} may hold.

\medskip
\noindent
Taking the characteristic functions in \eqref{eq:scaling}, we obtain
\[\widehat{p}^{(1)}(s)=\widehat{p}(\zeta s)=\Re\left[\widehat{f}(\zeta s)\right].\]
while taking the characteristic function of $d_j$ in \eqref{twosum} we have
\[\widehat{p}^{(1)}(s)={\sf E}[e^{isd_j}]={\sf E}\left[\Big(\widehat{f}(s)\Big)^{\xi_+}\right]\,{\sf E}\left[\Big(\widehat{f}(-s)\Big)^{\xi_-}\right]=\left|\frac{\widehat{f}(s)}{2-\widehat{f}(s)}\right|^2.\]
Thus, \eqref{eq:scaling} is satisfied if and only if
\be \label{eq:zetavs2}
\Re\left[\widehat{f}(\zeta s)\right]=\left|\frac{\widehat{f}(s)}{2-\widehat{f}(s)}\right|^2.
\ee
Substituting $\zeta=2$ into \eqref{eq:zetavs2} completes the proof.
\end{proof}

\begin{ex}\label{ex:ei}
Exponential density $f(x)=\phi_\lambda (x)$ of \eqref{exp} solves \eqref{laplace} with 
any $\lambda>0$; see Thm.~\ref{eH} below for more detail.
\end{ex}

Consider a time series $\{X_k\}_{k \in \mathbb{Z}}$, with an atomless distribution of values at each $k$.
Let  $\{X_t\}_{t \in \mathbb{R}}$, be a continuous function of linearly interpolated values of $X_k$.
We define a {\it positive excursion} of $X_k$ as a fragment
of the time series on an interval $[l,r]$, $l,r\in\mathbb{Z}$,
such that $X_l\ge X_r$ and $X_k>X_l$ for all $l<k<r$ (see Fig.~\ref{fig:ex}).
To each positive excursion of $X_k$ on $[l,r]$ corresponds a
positive excursion of $X_t$ on $[l,\tilde r]$, where $\tilde r\in(r-1,r]$ 
is such that $X_{\tilde r} = X_l$.
The level set tree of a positive excursion of $X_k$ is that
of the corresponding positive excursion of $X_t$.

\medskip
\noindent
Propositions~\ref{prop:DSS} and \ref{ts_prune} imply the following statement.
\begin{cor}\label{cor:DSS}
Consider a symmetric homogeneous random walk  $\{X_k\}_{k\in\mathbb{Z}}$ 
with a transition kernel $~p(x)=\frac{f(x)+f(-x)}{2}~$, where $f(x)$ is a p.d.f. 
with support ${\sf supp}(f) \subseteq \mathbb{R}_+$ and a finite second moment.
Let  
$$T=\textsc{level}\big(\{X_t\}_{t \in [l,r]}\big)$$
be the level set tree for a positive excursion $\{X_t\}_{t \in [l,r]}$ generated by the random walk $X_k$ as defined in Sect. \ref{sec:level}.
Then, the tree $T$ has a Horton self-similar distribution (Def. \ref{def:distss}) over $\BL^|$,
if and only if the condition \eqref{laplace} holds for the characteristic function $\widehat{f}(s)$ of $f(x)$.
\end{cor}
\begin{proof}
The coordination in shapes and lengths follows from the random walk construction. 
Props.~\ref{prop:DSS},\ref{ts_prune} establish the Horton prune-invariance.
\end{proof}

A homogeneous random walk on $\mathbb{R}$ is called {\it exponential random walk} if its transition kernel is 
a mixture of exponential jumps: 
\index{random walk!exponential}
$$p(x)=\rho\,\phi_{\lambda_u}(x)+(1-\rho)\,\phi_{\lambda_d}(-x), \quad 0\le \rho \le 1, \quad \lambda_u,\lambda_d>0,$$
where $\phi_{\lambda}$ is the exponential density with parameter $\lambda>0$ as defined in Eq. (\ref{exp}).
We refer to an exponential random walk by its parameter triplet $\{\rho,\lambda_u,\lambda_d\}$.
Each interpolated exponential random walk with parameters $\{\rho,\lambda_u,\lambda_d\}$ is a
piece-wise linear function whose positive (up) and negative (down) increments are independent exponential 
random variables with respective parameters $\lambda_u$ and $\lambda_d$, and the probabilities of
a positive or negative increment at every integer instant are $\rho$ and $(1-\rho)$, respectively.
After a time change that makes all segments to have slopes $\pm 1$, each interpolated exponential random walk with parameters $\{\rho,\lambda_u,\lambda_d\}$ corresponds to a 
piece-wise linear function with alternating rises and falls that have independent exponential lengths with parameters $(1-\rho)\lambda_u$ and $\rho\lambda_d$, respectively.
An exponential random walk is symmetric if and only if $\rho=1/2$ and $\lambda_u=\lambda_d$.

\begin{thm}[{{\bf Self-similarity of exponential random walks, \cite{ZK12}}}]
\label{eH}
Let $\{X_k\}_{k\in\mathbb{Z}}$ be an exponential random walk with parameters $\{\rho,\lambda_u,\lambda_d\}$.
Then 
\begin{description}
\item[(a)]
The sequence $\{X^{(1)}_j\}_{j\in\mathbb{Z}}$ of the local minima of $X_k$ is an exponential random walk with parameters 
$\{\rho^*,\lambda_u^*,\lambda_d^*\}$ such that
\be 
\label{iteration}
\rho^*=\frac{\rho\,\lambda_d}{\rho\,\lambda_d+(1-\rho)\,\lambda_u}, \quad 
\lambda^*_d=\rho\lambda_d,~~\text{ and }~~
\lambda^*_u=(1-\rho)\lambda_u.
\ee
\item[(b)]
The exponential walk $X_k$ satisfies the self-similarity condition \eqref{eqn:d_inv}
if and only if it is symmetric ($\rho=1/2$ and $\lambda_u=\lambda_d$), i.e., when $p(x)$ is a mean zero Laplace p.d.f. 
\item[(c)]
The self-similarity \eqref{eqn:d_inv} is achieved after the first Horton pruning, for the chain $\{X^{(1)}_j\}_{j\in\mathbb{Z}}$ of the local minima,
if and only if the walk's increments have zero mean, $\rho\,\lambda_d = (1-\rho)\,\lambda_u$. 
\end{description}
\end{thm}
\begin{proof}
(a) By Lemma \ref{inv_prune}(i), the sequence of local minima $X^{(1)}_j$ of $X_k$ is a homogeneous random walk with transition kernel $p^{(1)}(x)$. The latter is the probability distribution of the jumps $d_j$ given by \eqref{twosum} with 
$$\xi_+\overset{d}{\sim}{\sf Geom}_1(1-\rho), \quad \xi_-\overset{d}{\sim}{\sf Geom}_1(\rho), \quad Y_i\overset{d}{\sim}\phi_{\lambda_u}, ~\text{ and }~Z_i\overset{d}{\sim}\phi_{\lambda_d}.$$
The characteristic function $\widehat{p}^{(1)}(s)$ of the transition kernel $p^{(1)}(x)$ is found here as follows
\begin{align*}
\widehat{p}^{(1)}(s)&={\sf E} \left[\exp\Big\{is\left(X^{(1)}_{j+1}-X^{(1)}_j\right)\Big\}\right]
=\frac{\rho(1-\rho)\lambda_d \lambda_u}{\left((1-\rho)\lambda_u-is\right)(\rho\lambda_d+is)} \\
&=\rho^* \, \widehat{\phi}_{\lambda^*_u}(s)+(1-\rho^*) \, \widehat{\phi}_{\lambda^*_d}(s),
\end{align*}
where
\be\label{eqn:phihat}
\widehat{\phi}_\lambda(s)={\lambda \over \lambda -is}
\ee
is the characteristic function of an exponential random variable with parameter $\lambda$, and
$\rho^*, \lambda^*_u, \lambda^*_d$ are given by \eqref{iteration}.
Thus, 
\[p^{(1)}(x)=\rho^* \phi_{\lambda^*_u}(x)+(1-\rho^*)\phi_{\lambda^*_d}(-x).\]
This means that the sequence of local minima $\{X^{(1)}_j\}$ also evolves according to a two-sided 
exponential transition kernel, only with different parameters, $\rho^*$, $\lambda^*_d$, and $\lambda^*_u$. 

\medskip
\noindent
Part (b) of the theorem follows immediately from part (a).
Alternatively, we observe that the exponential density $f(x)=\phi_\lambda (x)$ solves \eqref{laplace} with 
any $\lambda>0$: by \eqref{eqn:phihat} we have
$$\Re\left[\widehat{\phi}_\lambda(2s)\right]=\Re\left[{\lambda \over \lambda -2is}\right]={ \lambda^2 \over \lambda^2+4s^2}$$
and
$$\left|\frac{\widehat{\phi}_\lambda(s)}{2-\widehat{\phi}_\lambda(s)}\right|^2=\left|{\lambda \over \lambda -2is}\right|^2={ \lambda^2 \over \lambda^2+4s^2}.$$
Hence, Part (b) follows from Prop. \ref{prop:DSS}.

\medskip
\noindent
(c) Observe that $\rho^*=1/2$ and $\lambda^*_d=\lambda^*_u$ if and only if $\rho\,\lambda_d = (1-\rho)\,\lambda_u$. 

\end{proof}

We now extend Def.~\ref{def:cbinary} to non-critical Galton-Watson trees.
\begin{Def} [{{\bf Exponential binary Galton-Watson tree, \cite{Pitman}}}]
\label{def:binary}
We say that a random {\it planted embedded binary tree} 
$T\in\BL^{|}$ is an exponential binary Galton-Watson tree  and write $T\stackrel{d}{\sim}{\sf GW}(\lambda',\lambda)$,
for $0\le\lambda'<\lambda$, if  
\begin{itemize}
\item[(i)] \textsc{shape}($T$) is a binary Galton-Watson tree $\mathcal{GW}(q_0,q_2)$ with 
\[q_0=\frac{\lambda+\lambda'}{2\lambda},\quad
q_2 = \frac{\lambda-\lambda'}{2\lambda};\]
\item[(ii)] the orientation for every pair of siblings in $T$ is random and symmetric; and
\item[(iii)] conditioned on a given \textsc{shape}($T$), the edges of $T$ are sampled as independent 
exponential random variables with parameter $\lambda$, i.e., with density~\eqref{exp}.
\end{itemize}
\end{Def}

\noindent In particular, we observe that
${\sf GW}(\lambda)={\sf GW}(0,\lambda).$
A connection between exponential random walks and exponential Galton-Watson trees  
is provided by the following well known result.

\begin{thm}{\rm \cite[Lemma 7.3]{Pitman},\cite{LeGall93,NP}}
\label{Pit7_3}
Consider a random excursion $Y_t$ in $\cE^{\rm ex}$.
The level set tree $\textsc{level}(Y_t)$ 
is an exponential binary Galton-Watson tree ${\sf GW}(\lambda',\lambda)$ 
if and only if the alternating rises and falls of $Y_t$, excluding the last fall, are 
distributed as independent exponential random variables with parameters ${\lambda+\lambda'  \over 2}$
and ${\lambda-\lambda'  \over 2}$, respectively, for some $0\le \lambda' < \lambda$.

Equivalently, for a random excursion $Y_t$ of a homogeneous random walk in $\cE^{\rm ex}$,
the level set tree $\textsc{level}(Y_t)$ is an exponential binary Galton-Watson tree ${\sf GW}(\lambda',\lambda)$ if
and only if $Y_t$, as an element of $\cE^{\rm ex}$,  corresponds to an excursion of an exponential walk with 
parameters $\{\rho,\lambda_u,\lambda_p\}$
satisfying
$(1-\rho)\lambda_u = {\lambda+\lambda' \over 2}$ and
$\rho\lambda_d = {\lambda-\lambda' \over 2}.$
\end{thm}

\noindent We emphasize the following direct consequence of Thms.~\ref{eH}(a) and \ref{Pit7_3}.
\begin{cor}
\label{cor:GW}
Suppose $T \stackrel{d}{\sim} {\sf GW}(\gamma)$ is an exponential critical binary Galton-Watson tree. 
Then, the following statements hold:
\begin{description}
\item[(a)] The pruned exponential critical binary Galton-Watson tree is an exponential critical binary Galton-Watson tree:
$$\Big(\cR^{k}(T) ~\big|~ \cR^{k}(T) \not= \phi \Big) \stackrel{d}{\sim} {\sf GW}\left(2^{-k}\gamma\right) 
\text{ for any }k \in \mathbb{N}.$$
\item[(b)] The lengths of branches of Horton-Strahler order $j\ge 1$ in $T$ (see Def. \ref{def:HS}) has exponential distribution
with parameter $2^{1-j}\,\gamma$. 
The lengths of branches (of all orders) are independent.
\end{description}
\end{cor}

\begin{Rem}
[{{\bf A link between Thm.~\ref{eH} and Thm.~\ref{thm:BWW2}.}}]
Consider an excursion of an exponential random walk $X_t$  with parameters 
$\{\rho,\lambda_u,\lambda_d\}$.
The geometric stability of the exponential distribution implies that the monotone rises and falls of $X_t$
are exponentially distributed with parameters $(1-\rho) \,\lambda_u$ and $\rho \,\lambda_d$, respectively.
Thus, Thm.~\ref{Pit7_3} implies that 
$\textsc{shape}\left(\textsc{level}(X_t)\right)$ is
distributed as a binary Galton-Watson tree $\mathcal{GW}(q_0,q_2)$
with 
\[q_2 = \frac{\rho\,\lambda_d}{(1-\rho) \,\lambda_u+\rho \,\lambda_d}=1-q_0.\]
The first pruning $X_t^{(1)}$ (see Sect. \ref{sec:excursions}), according to \eqref{iteration},
is an exponential random walk with parameters
\[\left\{\rho^*=\frac{\rho \,\lambda_d}{(1-\rho) \, \lambda_u+\rho\,\lambda_d},
~\lambda_u^*=(1-\rho) \, \lambda_u,  ~\lambda_d^*=\rho \,\lambda_d\right\}.\]
Its upward and downward increments are exponentially 
distributed with parameters, respectively,
\[(1-\rho^*)\lambda_u^*=\frac{[(1-\rho) \,\lambda_u]^2}{(1-\rho) \,\lambda_u+\rho\,\lambda_d} \quad {\rm and~} \quad \rho^* \lambda_d^*=\frac{[\rho \,\lambda_d]^2}{(1-\rho) \,\lambda_u+\rho\,\lambda_d}.\]
Accordingly, the level set tree for a positive 
excursion $X_t^{(1)}$ is a binary Galton-Watson tree $\mathcal{GW}(q_0^{(1)},q_2^{(1)})$
with 
\[q_2^{(1)}=\frac{[\rho\,\lambda_d]^2}{[(1-\rho) \,\lambda_u]^2+[\rho \,\lambda_d]^2}=1-q_0^{(1)}.\]

\noindent
Continuing this way, we find that $n$-th pruning $X_t^{(n)}$ of 
$X_t$ is an exponential random walk such that the level set tree of 
its positive excursion has binary Galton-Watson distribution $\mathcal{GW}(q_0^{(n)},q_2^{(n)})$ with
\be\label{eq:iterateBFS}q_2^{(n)}=\frac{\left[q_2^{(n-1)}\right]^2}
{\left[q_0^{(n-1)}\right]^2+\left[q_2^{(n-1)}\right]^2}
=\frac{[\rho\,\lambda_d]^{2^n}}
{[(1-\rho) \,\lambda_u]^{2^n}+[\rho\,\lambda_d]^{2^n}}=1-q_0^{(n)}
\ee
The first equality in \eqref{eq:iterateBFS} defines the same iterative system 
as \eqref{eq:iterateDFS} in Thm.~\ref{thm:BWW2} 
of Burd et al. that describes iterative Horton pruning of Galton-Watson trees.
Another noteworthy relation connecting the exponential random walk $X_t^{(n)}$ with parameters $\{\rho^{(n)},\lambda^{(n)}_u,\lambda^{(n)}_d\}$
and the Galton-Watson tree $\mathcal{GW}(q_0^{(n-1)},q_2^{(n-1)})$ is given by
$$\rho^{(n)} = q_2^{(n-1)} \text{ for any } n\ge 1 ~(\text{where }q_2^{(0)}\equiv q_2).$$
\end{Rem}

\subsection{Geometric random walks and critical non-binary Galton-Watson trees}

A recent study by Barbosa et al. \cite{BCJKZ19} examines the self-similar properties of the level-set trees corresponding to the excursions of the so-called {\it geometric random walk} on $\mathbb{Z}$, defined below (Def. \ref{def:grw}). 
The results in \cite{BCJKZ19} give a discrete-space version of the results discussed in Sect.~\ref{sec:erw}.   

For the given probabilities $\{p_1, p_2, r_1, r_2\}$ such that $p_1 + p_2 \leq 1$, consider a discrete-time random walk on $\mathbb{Z}$, where at each time step, $p_1$ is the probability of an upward jump, $p_2$ is the probability of a downward jump, and $1-p_1-p_2$ is the probability of remaining at the same location. 
Conditioned on jumping upward, the increment size is a ${\sf Geom}_1(r_1)$-distributed random variable, while conditioned on jumping downward, the increment size is a ${\sf Geom}_1(r_2)$-distributed random variable. 
Here is a formal definition. \index{random walk!geometric}

\begin{Def}[{\bf Geometric random walk}]\label{def:grw}
A geometric random walk $X_t$ with probability parameters $$\{p_1, p_2, r_1, r_2\}$$ is a discrete time space-homogeneous random walk on $\mathbb{Z}$ with transition probabilities $p(x,y)=p(y-x)$ such that its jump kernel $p(x)$ is a double-sided geometric probability mass function (discrete Laplace distribution) that can be expressed as 
\begin{equation}\label{eqn:kernel}
p(x) = p_1g_1(x) + (1-p_1-p_2)\delta_0(x) + p_2g_2(-x),
\end{equation}
where $\delta_0(x)$ denotes the Kronecker delta function at $0$, and $g_i(x)$ ($i = 1,2$) is the probability mass function of a ${\sf Geom}_1(r_i)$-distributed random variable.
The distribution for a geometric random walk is denoted by ${\rm GRW}(p_1, p_2, r_1, r_2)$.
\end{Def}
\noindent
\begin{ex}
The most celebrated example of a geometric random walk is the simple random walk on $\mathbb{Z}$ with
distribution ${\rm GRW}\big({1\over 2}, {1\over 2}, 1, 1\big)$.
\end{ex}
By \eqref{eqn:kernel}, the characteristic function for the increments in a geometric walk is given by
\begin{equation}\label{eqn:kernelCh}
\widehat{p}(s) = {p_1r_1e^{is} \over 1-(1-r_1)e^{is}}+{p_2r_2e^{-is} \over 1-(1-r_2)e^{-is}} + (1-p_1-p_2).
\end{equation}
Equation \eqref{eqn:kernelCh} leads to the derivation of the following invariance result,
analogous to Thm. \ref{eH}(a) in a discrete space setting.
\begin{thm}[{{\bf \cite{BCJKZ19}}}]\label{thm:mainGeometricRW}
Suppose $X_t$ is a geometric random walk ${\rm GRW}(p_1, p_2, r_1, r_2)$, then the time series $X^{(1)}_t$ of its consecutive local minima (including flat plateaus)  is also a geometric random walk ${\rm GRW}\big(p^{(1)}_1, p^{(1)}_2, r^{(1)}_1, r^{(1)}_2\big)$ with probability parameters
\begin{align*}
&r_1^{(1)} = \frac{p_2r_1}{p_1 + p_2},
& p_1^{(1)} = \frac{r_2^{(1)}(1-r_1^{(1)})}{1-(1-r_1^{(1)})(1-r_2^{(1)})},\\
& r_2^{(1)} = \frac{p_1r_2}{p_1 + p_2}, 
& p_2^{(1)} = \frac{r_1^{(1)}(1-r_2^{(1)})}{1-(1-r_1^{(1)})(1-r_2^{(1)})}.
\end{align*}
\end{thm}

\noindent
If $r_1=r_2=r$ and $p_1=p_2=p$, the geometric random walk ${\rm SGRW}(p,r)\equiv {\rm GRW}(p, p, r, r)$ is called {\it symmetric geometric random walk} (SGRW).
In this case, Thm.~\ref{thm:mainGeometricRW} can be reinterpreted as the following statement, 
analogous to Thm.~\ref{eH}(b) adapted to the discrete space $\mathbb{Z}$.
\begin{cor}[{{\bf \cite{BCJKZ19}}}]\label{cor:min_symGW}
Suppose $X_t\stackrel{d}{\sim}{\rm SGRW}(p, r)$ is a symmetric geometric random walk on $\mathbb{Z}$. 
Then, the time series $X^{(1)}_t$ of its consecutive local minima is also a symmetric geometric random walk ${\rm SGRW}\big(p^{(1)}, r^{(1)}\big)$ with probability parameters
$$p^{(1)} = \frac{1-r^{(1)}}{2-r^{(1)}} \quad \text{and} \quad r^{(1)} = \frac{r}{2}.$$
\end{cor}

\noindent
Next, consider the case of a geometric random walk $X_t$ with mean zero increments,
$$E[X_{t+1}-X_t]={p_1 \over r_1} -{p_2 \over r_2} = 0.$$
In this case $p_1r_2 = p_2r_1$, and Thm.~\ref{thm:mainGeometricRW} and Cor.~\ref{cor:min_symGW} 
imply the following result.

\begin{cor}[{{\bf \cite{BCJKZ19}}}]\label{cor:mean_zeroGW}
Suppose $X_t\stackrel{d}{\sim}{\rm GRW}(p_1, p_2, r_1, r_2)$ is a mean zero geometric random walk, 
i.e. $p_1r_2 = p_2r_1$. 
Then, the time series $X^{(1)}_t$ of its consecutive local minima is a symmetric geometric random walk ${\rm SGRW}\big(p^{(1)}, r^{(1)}\big)$ with probability parameters
$$p^{(1)} = \frac{1-r^{(1)}}{2-r^{(1)}} \quad \text{and} \quad r^{(1)} = \frac{r}{2},$$
where $r=\frac{2p_1r_2}{p_1 + p_2}=\frac{2p_2r_1}{p_1 + p_2}$.

Furthermore, let $X^{(k+1)}_t$ for $k =1,2,\hdots$ be the time series of the consecutive local minima of $X^{(k)}_t$. Then, $X^{(k)}_t$ is also a symmetric geometric random walk ${\rm SGRW}\big(p^{(k)}, r^{(k)}\big)$ with probability parameters
	\be\label{eq:gen_parGW}
		p^{(k)} = \frac{1-r^{(k)}}{2-r^{(k)}} \quad \text{and} \quad r^{(k)} = \frac{r}{2^k}.
	\ee
\end{cor}
For the remainder of this section, let $\{p^{(k)}, r^{(k)}\}$ denote the parameters of the symmetric 
geometric random walk ${\rm SGRW}\big(p^{(k)}, r^{(k)}\big)$, obtained by taking $k$ iterations of 
local minima of $X_t\stackrel{d}{\sim}{\rm GRW}(p_1, p_2, r_1, r_2)$, as in Cor.~\ref{cor:mean_zeroGW}.

\begin{cor}[{{\bf \cite{BCJKZ19}}}]\label{cor:lim_parGW}
Suppose $X_t\stackrel{d}{\sim}{\rm GRW}(p_1, p_2, r_1, r_2)$ is a mean zero geometric random walk, 
i.e. $p_1r_2 = p_2r_1$. 
Then,
$$\lim_{n \to \infty} r^{(n)} = 0 ~\text{ and }~ \lim_{n \to \infty} p^{(n)} = \frac{1}{2}.$$
\end{cor}

\medskip
\noindent
The following is a discrete analogue of Thm.~\ref{Pit7_3}, stated in Sect. \ref{sec:erw}.
\begin{thm}[{{\bf \cite{BCJKZ19}}}]\label{thm:offspringsGW}
Suppose $X_t\stackrel{d}{\sim}{\rm GRW}(p_1, p_2, r_1, r_2)$ is a geometric random walk 
with a nonnegative drift, i.e., $p_1r_2 \leq p_2r_1$.
Let $T^{\sf ex}$ be the level set tree of a positive excursion of $X_t$.
Then, 
\[\textsc{shape}(T^{\sf ex})\stackrel{d}{\sim}\mathcal{GW}\big(\{q_k\}\big) \text{ on }\cT^|\] 
with
$$q_0 = 1-p_1^{(1)}, \quad \text{and } \quad q_k = p_1^{(1)}\big(r_2^{(1)}\big)^{k-2}(1-r_2^{(1)}) \quad \text{ for }k =2,3 \hdots,$$
where $r_1^{(1)}$, $r_2^{(1)}$, $p_1^{(1)}$, and $p_2^{(1)}$ are as in Theorem \ref{thm:mainGeometricRW}
(recall that $q_1\equiv 0$ since we work with reduced trees).
Moreover, if $X_t$ is a mean zero geometric random walk (i.e., $p_1r_2= p_2r_1$), then 
$$q_0 = \frac{1}{2-r^{(1)}}, \quad \text{and } \quad
	    q_k = \frac{\big(r^{(1)}\big)^{k-2}(1-r^{(1)})^2}{2-r^{(1)}} \quad \text{ for }k =2,3 \hdots,$$
where $r_1^{(1)}$ and $r_2^{(1)}$ are as in Corollary \ref{cor:mean_zeroGW}.
\end{thm}

\medskip
\noindent
Observe that, in the setting of Thm.~\ref{thm:offspringsGW}, if we consider a mean zero GRW ($p_1r_2 = p_2r_1$ and, equivalently, $r_1^{(1)}=r_2^{(1)}$) then,
$$\sum_{k} kq_k = \left(\frac{1-r_1^{(1)}}{1-r_2^{(1)}}\right)\left(\frac{r_2^{(1)} + r_2^{(1)}(1-r_2^{(1)})}{r_2^{(1)} + r_1^{(1)}(1-r_2^{(1)})}\right) = 1.$$	
In other words, the level set tree of its positive excursion is distributed as a {\it critical} Galton-Watson tree $\mathcal{GW}\big(\{q_k\}\big)$.
Combining Prop.~\ref{ts_prune} with Thm.~\ref{thm:offspringsGW} we have the following corollary.
\begin{cor}[{{\bf \cite{BCJKZ19}}}]\label{cor:lim_offGW}
Suppose $X_t\stackrel{d}{\sim}{\rm GRW}(p_1, p_2, r_1, r_2)$ is a mean zero geometric random walk, 
i.e. $p_1r_2 = p_2r_1$.
Let $T^{\sf ex}$ be the level set tree of a positive excursion of $X_t$.  
Then, $\textsc{shape}(T^{\sf ex})\stackrel{d}{\sim} \mathcal{GW}\big(\{q_k\}\big)$, where $\mathcal{GW}\big(\{q_k\}\big)$ is a critical Galton-Watson distribution on $\cT^|$.
Moreover, for any $n \geq 1$, the level set tree of a positive excursion of $X_t^{(n)}$ is distributed as
$$\Big(\cR^n\big(T^{\sf ex}\big)\,\Big|  \cR^n\big(T^{\sf ex}\big) \not= \phi \Big) ~\stackrel{d}{\sim} ~\mathcal{GW}\big(\{q_k^{(n)}\}\big)$$ 
with
    \begin{equation}\label{eqn:pruning_qkn}
    	q_0^{(n)} = \frac{1}{2-r^{(n+1)}} ~~\text{ and }~~   q_k^{(n)} = \frac{(r^{(n+1)})^{k-2}(1-r^{(n+1)})^2}{2-r^{(n+1)}} ~\quad (\forall k\geq 2),
    \end{equation}
where $r^{(n)}$ is given by equation \eqref{eq:gen_parGW} of Corollary \ref{cor:mean_zeroGW}.

\noindent
Finally, letting $n \to \infty$, we have
    \begin{equation}\label{eqn:GRWconfirmsBWW}
    	q_0^{(n)} \to \frac{1}{2}, \quad  q_2^{(n)} \to \frac{1}{2}, ~\text{ and }~q_k^{(n)} \to 0 \quad (\forall k>2).
    \end{equation} 
\end{cor}

\noindent
The convergence in  \eqref{eqn:GRWconfirmsBWW} follows from Cor.~\ref{cor:lim_parGW} as $r^{(n)} \to 0$.
Writing $\nu=\mathcal{GW}\big(\{q_k\}\big)$, we have by Cor.~\ref{cor:lim_offGW} that the pushforward measure satisfies
$$\nu_n:=\cR_*^n(\nu)=\nu\circ\cR^{-n} \stackrel{d}{=}\mathcal{GW}\big(\{q_k^{(n)}\}\big)$$ 
while equation \eqref{eqn:GRWconfirmsBWW} additionally asserts that
  \begin{equation}\label{eqn:GRWconfirmsBWW00}
    \lim\limits_{n \rightarrow \infty}\nu_n(\tau \,|\tau\ne\phi)=\mu^*(\tau),
   \end{equation} 
  where $\mu^*$ denotes the critical binary Galton-Watson measure on $\cT^|$ defined in \eqref{eqn:mustarGWonT}.  
Equation \eqref{eqn:GRWconfirmsBWW00}  provides a specific example of Thm.~\ref{BWW00_1} (Thm.~1.3 in \cite{BWW00}) 
showing that recursive pruning of a critical Galton-Watson tree converges to a binary critical Galton-Watson tree.


\subsection{White noise and Kingman's coalescent}
\label{sec:white_Kingman}

This section establishes an interesting correspondence between the tree representations
of a white noise (sequence of i.i.d. random variables) and celebrated 
Kingman's coalescent process \cite{Kingman82b}.
We begin by an informal review of coalescent processes and their trees.

\subsubsection{Coalescent processes, trees}
\label{sec:coalescent}

{\bf Coalescent processes} \cite{Pitman,Aldous,Bertoin,Berestycki,EP98}.
A general finite coalescent process begins with $N$ singletons. The cluster formation is governed by a symmetric collision rate 
kernel $K(i,j)=K(j,i)>0$. Specifically, a pair of clusters with masses (weights) $i$ and $j$ coalesces at the rate $K(i,j)/N$, independently of the other pairs, to form a new cluster of mass $i+j$. The process continues until there is a single cluster of mass $N$. \index{coalescent processes} 

Formally, for a given $N\ge 1$ consider the space $\mathcal{P}_{[N]}$ of partitions of $[N]=\{1,2,\hdots,N\}$. 
Let $\Pi^{(N)}_0$ be the initial partition in singletons, and $\Pi^{(N)}_t ~~(t \geq 0)$ be a strong Markov process such that 
$\Pi^{(N)}_t$ transitions from partition $\pi \in \mathcal{P}_{[N]}$ to $\pi' \in \mathcal{P}_{[N]}$ with rate $K(i,j)/N$ provided that partition 
$\pi'$ is obtained from partition $\pi$ by merging two clusters of $\pi$ of weights $i$ and $j$.
If $K(i,j) \equiv 1$ for all positive integer masses $i$ and $j$, the process $\Pi^{(N)}_t$ is known as the $N$-particle {\it Kingman's coalescent process}. If $K(i,j)=i+j$ the process is called the $N$-particle {\it additive coalescent}. Finally, if $K(i,j)=ij$ the process is called the $N$-particle {\it multiplicative coalescent}.  
\index{coalescent processes!Kingman's coalescent}
\index{coalescent processes!additive coalescent}
\index{coalescent processes!multiplicative coalescent}

%

{\bf Coalescent tree.}
A merger history of the $N$-particle coalescent process can be naturally described by a time oriented binary tree 
constructed as follows. Start with $N$ leaves that represent the initial $N$ particles and have time mark $t=0$. 
When two clusters coalesce (a transition occurs), merge the corresponding vertices to form an internal vertex with a time mark of the coalescent. 
The final coalescence forms the tree root. The resulting time oriented binary tree represents the history of the process.
We notice that a given unlabeled tree corresponds to multiple coalescent trajectories obtained by relabeling of the initial particles.
\index{tree!coalescent}

\medskip
\noindent
Let $T^{(N)}_{\rm K}$ denote the coalescent tree for the  $N$-particle Kingman's coalescent process.
Let $N_j$ denote the number of branches of Horton-Strahler order $j$ in the tree $T^{(N)}_{\rm K}$. 
In Sect. \ref{sec:Kingman} we will show that for each $j\ge 1$, 
the asymptotic Horton ratios $\cN_j$ are well-defined (Def. \ref{def:HortonWellDef}), that is
\be\label{eq:KingmanLL}
{\frac{N_j}{N}\stackrel{p}{\to}\cN_j\quad\text{ as }\quad N\to\infty}.
\ee
Moreover, the Horton ratios $\cN_j$ are finite and can be expressed as
\[\cN_j = \frac{1}{2}\int_0^\infty g_j^2(x)\,dx,\]
where the sequence $g_j(x)$ solves the following system of
ordinary differential equations (ODEs):
\be\label{eq:ODEggg}
g'_{j+1}(x)-{g^2_j(x) \over 2}+g_j(x) g_{j+1}(x)=0,\quad x\ge 0
\ee
with $g_1(x)=2/(x+2)$, $g_j(0)=0$ for $j \geq 2$.
Equivalently,
\[\cN_j=\int_0^1 \left(1-\left(1-x\right)h_{j-1}(x)\right)^2 dx,\]
where $h_0\equiv 0$ and the sequence $h_j(x)$  satisfies the ODE system
\be\label{eq:ODEhhh}
h'_{j+1}(x)=2h_j(x)h_{j+1}(x)-h_j^2(x),\quad 0\le x \le 1
\ee
with the initial conditions $h_k(0)=1$ for $j\ge 1$.

\medskip
\noindent
The root-Horton law (Def.~\ref{def:HortonList})  for the well-defined Horton ratios $\cN_j$ \eqref{eq:KingmanLL} of the Kingman's coalescent process is stated in Thm.~\ref{Mthm}, 
with the Horton exponent bounded by the interval $2 \le R \le 4$. 
Moreover, the Horton exponent is estimated to be $R=3.0438279\hdots$ via the ODE representation in \eqref{eq:ODEggg} and \eqref{eq:ODEhhh}.
The numerical computation (not shown here) affirms that the ratio-Horton 
and the strong Horton laws of Def.~\ref{def:HortonList} are valid for the Kingman's coalescent as well.

\subsubsection{White noise}\label{sec:finite}
In this section we will show that the combinatorial $\textsc{shape}$ function for the level set tree $T_{\sf wn}$ of white noise is closely connected to the $\textsc{shape}$ function of the Kingman's coalescent tree $T_{\rm K}=T^{(N)}_{\rm K}$ introduced in Sect. \ref{sec:coalescent}.
Specifically, the two are separated by a single Horton pruning $\cR$.
In other words, conditioning on the same number of 
leaves, $\textsc{shape}\big(\cR(T_{\rm K})\big)\stackrel{d}{=}\textsc{shape}\big(T_{\sf wn}\big)$.

\medskip
\noindent
Let $W^{(N)}_j$ with $j=1,\dots,N\!-\!1$ be a {\it discrete white noise} that is 
a discrete time process comprised of $N\!-\!1$ i.i.d. random variables with a common atomless distribution.
Next, we consider an auxiliary process $\tilde W^{(N)}_{i}$ with $i=1,\dots,2N\!-\!1$, such that it has exactly $N$ local maxima and $N\!-\!1$ internal local minima 
$\tilde W^{(N)}_{2j}=W^{(N)}_j$, $j=1,\dots,N\!-\!1$.
We call $\tilde W^{(N)}_{i}$ an {\it extended white noise}. 
It can be constructed as in the following example.
\index{white noise!discrete}
\index{white noise!extended}

\begin{ex}[{{\bf Extended white noise}}] 
\be
\label{wnt}
\tilde W^{(N)}_i=
\left\{
\begin{array}{cc}
W^{(N)}_{i/2},& {\rm for~even~}i,\\
\max\left\{W^{(N)}_{i'},W^{(N)}_{i''}\right\}+1,&{\rm for~odd~}i,
\end{array}
\right.
\ee
where $i'=\max\left(1,\frac{i-1}{2}\right)$ and $i''=\min\left(N-1,\frac{i+1}{2}\right)$.
\end{ex}

\medskip
\noindent
Let $T_{\sf wn}^{(N)}=\textsc{level}\left(\tilde W^{(N)}_i\right)$ be the level set tree of $\tilde W^{(N)}_i$. 
By construction, $T_{\sf wn}^{(N)}$ has exactly $N$ leaves.
Also observe that the level set trees $T_{\sf wn}^{(N)}$ and 
$\textsc{level}\left(W^{(N)}_j\right)$ are separated by a single Horton pruning:
\be\label{eqn:TwShift} 
\cR\left(T_{\sf wn}^{(N)}\right)=\textsc{level}\left(W^{(N)}_j\right).
\ee  

\begin{lem}
\label{any_wh}
The distribution of $\textsc{shape}\left(T_{\sf wn}^{(N)}\right)$ on $\cBT^|$
is the same for any atomless distribution $F$ of the values of the
associated white noise $W^{(N)}_j$.
\end{lem}
\begin{proof}
The condition of atomlessness of $F$ is necessary to ensure that the
level set tree is binary with probability one.
By construction, the combinatorial level set tree is completely determined by the 
ordering of the local minima of the respective trajectory, independently of
the particular values of its local maxima and minima.
We complete the proof by noticing that the distribution for the ordering of $W^{(N)}_j$ is 
the same for any choice of atomless distribution $F$.  
\end{proof}

Let $T^{(N)}_{\rm K}$ be the tree that corresponds to the Kingman's $N$-coalescent, and let
$\textsc{shape}\left(T^{(N)}_{\rm K}\right)$ be its combinatorial version that
drops the time marks of the vertices.
Both the trees $\textsc{shape}\left(T_{\sf wn}^{(N)}\right)$ and 
$\textsc{shape}\left(T^{(N)}_{\rm K}\right)$, belong to the space $\cBT^|$ (or, more
specifically, to $\cBT^|$ conditioned on $N$ leaves).

\begin{thm}\label{main2}
The trees $\textsc{shape}\left(T_{\sf wn}^{(N)}\right)$ and 
$\textsc{shape}\left(T^{(N)}_{\rm K}\right)$ have the same distribution on $\cBT^|$.
\end{thm}
\begin{proof}
The proof uses a construction similar in some respect to the celebrated Kingman paintbox process \cite{Kingman82b,Pitman,Bertoin,Berestycki}.
For the Kingman's $N$-coalescent, let us enumerate the initial singletons from $1$ to $N$. 
We will identify each cluster with a collection of singletons listed from left to right, where the order in which they are listed is important
as it contains a certain amount of information regarding the process's merger history. 
Specifically, consider a pair of clusters ${\bf i}$ and ${\bf j}$, identified with the corresponding collection of singletons as follows
$${\bf i}=\{i_1,\hdots,i_k\} \quad \text{ and }\quad {\bf j}=\{j_1,\hdots,j_m\}.$$
Next, we split the merger rate of ${1 \over N}$ into two. We let the clusters ${\bf i}$ and ${\bf j}$ merge into the new cluster 
$$\{{\bf i,j}\}=\{i_1,\hdots,i_k,j_1,\hdots,j_m\}$$
with rate ${1 \over 2N}$, or into the new cluster
$$\{{\bf j,i}\}=\{j_1,\hdots,j_m,i_1,\hdots,i_k\}$$
also with rate ${1 \over 2N}$. The final merger results in a cluster consisting of all $N$ singletons,
listed as a permutation from $S_N$,
$$\sigma=\{\sigma_1,\hdots,\sigma_N\}.$$
Conditioning on the final permutation $\sigma$, the merger history is described by the random connection times,
$$t_1,t_2,\hdots,t_{N-1},$$
where $t_j$ is the merger time when the singletons $\sigma_j$ and $\sigma_{j+1}$ meet in the same cluster.
The following diagram helps visualize the connection times:
$$\sigma_1 \stackrel{t_1}{\longrightarrow} \sigma_2 \stackrel{t_2}{\longrightarrow} \sigma_3 \stackrel{t_3}{\longrightarrow} \hdots \sigma_{N-1} \stackrel{t_{N-1}}{\longrightarrow}\sigma_N.$$
Since all $(N\!-\!1)!$ orderings of the connection times $t_1,\hdots,t_{N-1}$ are equiprobable, the combinatorial shape of the resulting coalescent tree
is distributed as the combinatorial tree $\textsc{shape}\left(T_{\sf wn}^{(N)}\right)$, where all $(N\!-\!1)!$ orderings of the analogous connection times  
$W^{(N)}_1,W^{(N)}_2,\hdots,W^{(N)}_{N-1}$ are also equiprobable.
\end{proof}

The following result is a consequence of the above Thm.~\ref{main2} and 
Thm.~\ref{Mthm} that we state and prove in Sect. \ref{sec:Kingman} establishing the 
 root-Horton law (Def. \ref{def:HortonList}) for Kingman's coalescent tree $\textsc{shape}\left(T^{(N)}_{\rm K}\right)$.

\begin{cor}
\label{main3}
The combinatorial level set tree of a discrete white noise $W^{(N)}_j$ 
is root-Horton self similar with the same Horton exponent $R$ as 
that for Kingman's $N$-coalescent.
\end{cor}

\begin{proof}
Together, Theorems~\ref{main2} and \ref{Mthm} imply the root-Horton self-similarity for $\textsc{shape}\left(T_{\sf wn}^{(N)}\right)$, with the same Horton exponent $R$.

By definition, Horton pruning corresponds to an index shift in Horton statistics:
$N_j\big[\cR(T)\big]=N_{j+1}[T]$ ($j \ge 1$). 
Thus, the root-Horton self-similarity for $\textsc{shape}\left(T_{\sf wn}^{(N)}\right)$ implies the root-Horton self-similarity 
for $\textsc{shape}\left(\textsc{level}\big(W^{(N)}_j\big)\right)$. 
Finally, the Horton exponent is preserved under the extra Horton pruning as
$$\lim\limits_{j \rightarrow \infty} \Big( \cN_{j+1} \Big)^{-{1 \over j}}=\lim\limits_{j \rightarrow \infty} \Big( \cN_j \Big)^{-{1 \over j}}=R.$$
\end{proof}

\subsection{Level set trees on higher dimensional manifolds and Morse theory}\label{sec:morse}

Consider an $n$-dimensional differentiable manifold $M=M^n$, and a differentiable function 
$f:M \rightarrow \mathbb{R}$. 
A point $p$ is called a {\it critical point} of $f$ if $df(p)=0$, 
in which case, $f(p)$ is said to be a {\it critical value} of $f$. A point $x \in M$ is called a {\it regular point}  of $f$ if it is not a critical point.

\begin{figure}[t] 
\centering\includegraphics[width=0.9\textwidth]{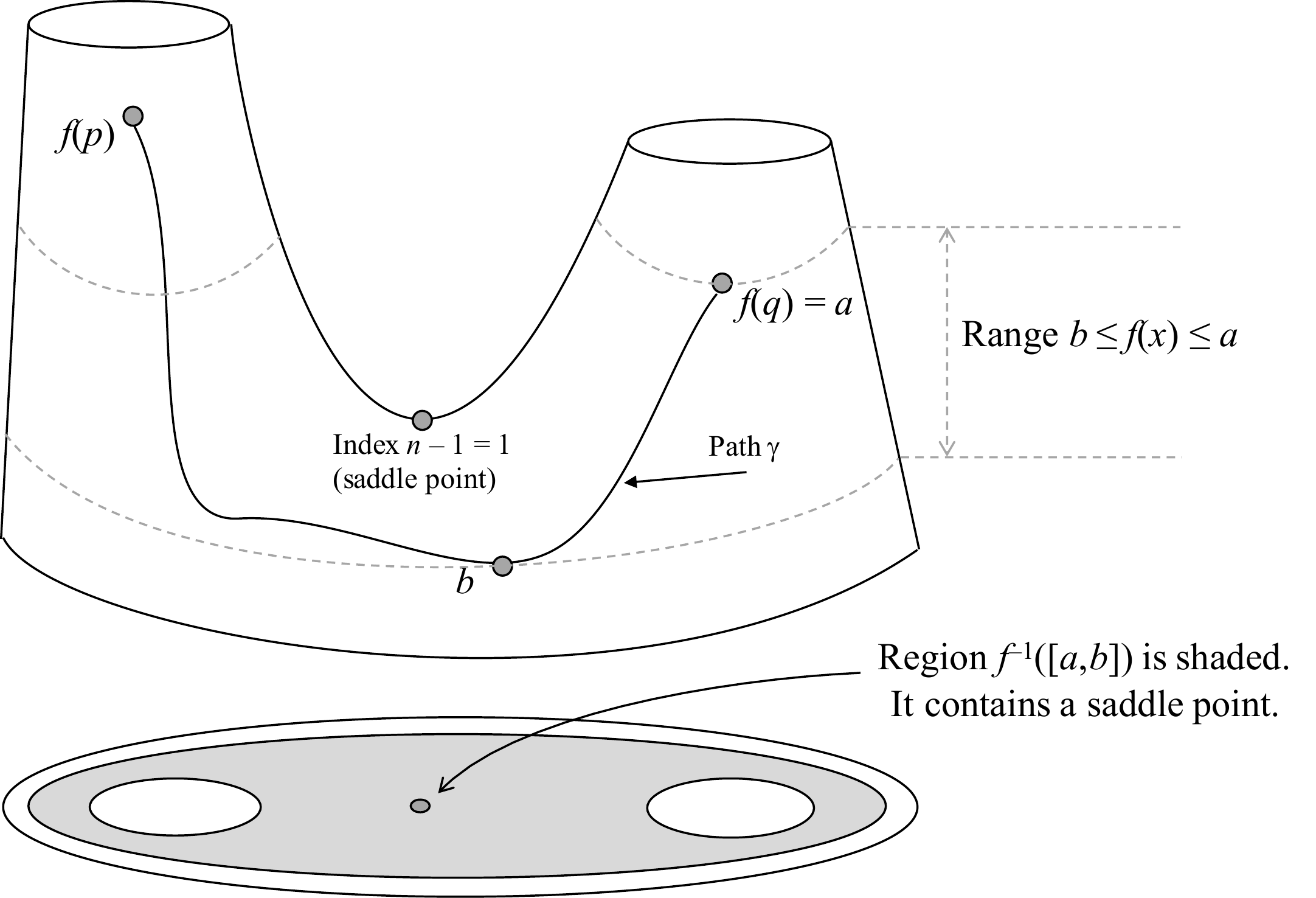}
\caption{Illustration to Lemma~\ref{lem:Milnor} (a counterexample).
Here, a function $f:M\subset\mathbb{R}^2\to\mathbb{R}$ is such that the region
$f^{-1}([a,b])$, which is shaded in the bottom part, contains a saddle  
(critical point of index $n-1=1$); hence the conditions of the Lemma are violated.
Observe, accordingly, that the image of any path $\gamma: p\to q$ must go below the point $a = f(q)$ by
a finite amount, i.e. there exists $\delta>0$ such that $\gamma\nsubset\cL_{a-\delta}$. 
}
\label{fig:Morse_pants}
\end{figure}

\medskip
\noindent
If $p$ is a critical point of $f$, then 
$$f(x)=f(p)+{1 \over 2}H_{f,p}(x,x)+O(3)$$
is the Taylor expansion of $f$ around $p$, where 
$$H_{f,p}(u,v)=\sum\limits_{i,j} {\partial^2 f \over \partial x_i \, \partial x_j}(p) \,u_i v_j \,:\, T_pM \times T_pM \rightarrow \mathbb{R}$$
is a {\it symmetric bilinear form} over the tangent space $T_pM$ generated by the {\it Hessian} matrix ${\partial^2 f \over \partial x_i \, \partial x_j}(p)$, and $O(3)$ denotes the third and higher order terms.

\begin{Def}[{{\bf Nondegenerate points and Morse functions \cite{Nicolaescu}}}]
\label{def:Morse}
Let $M$ and $f$ to be as above. A critical point $p \in M$ of $f$ is said to be nondegenerate if 
the determinant of its Hessian matrix ${\partial^2 f \over \partial x_i \, \partial x_j}(p)$ is not equal to zero. A differentiable function $f:M \rightarrow \mathbb{R}$ is said to be a Morse function if all of its critical points are nondegenerate.
\end{Def}
\index{Morse!function}

\begin{thm}[{{\bf Morse, \cite{Nicolaescu}}}]\label{thm:Morse}
Consider an $n$-dimensional differentiable manifold $M$, and a differentiable function $f:M \rightarrow \mathbb{R}$. 
If $p \in M$ is a nondegenerate critical point of $f$, then there exists an open neighborhood $U$ of $p$ and local coordinates $(x_1,\hdots,x_n)$ on $U$ 
with $$\big(x_1(p),\hdots,x_n(p)\big)=(0,\hdots,0)$$ such that in this coordinates $f(x)$ is a quadratic polynomial represented as
$$f(x)=f(p)+{1 \over 2}H_{f,p}(x,x).$$
\end{thm}

\begin{figure}[t] 
\centering\includegraphics[width=0.8\textwidth]{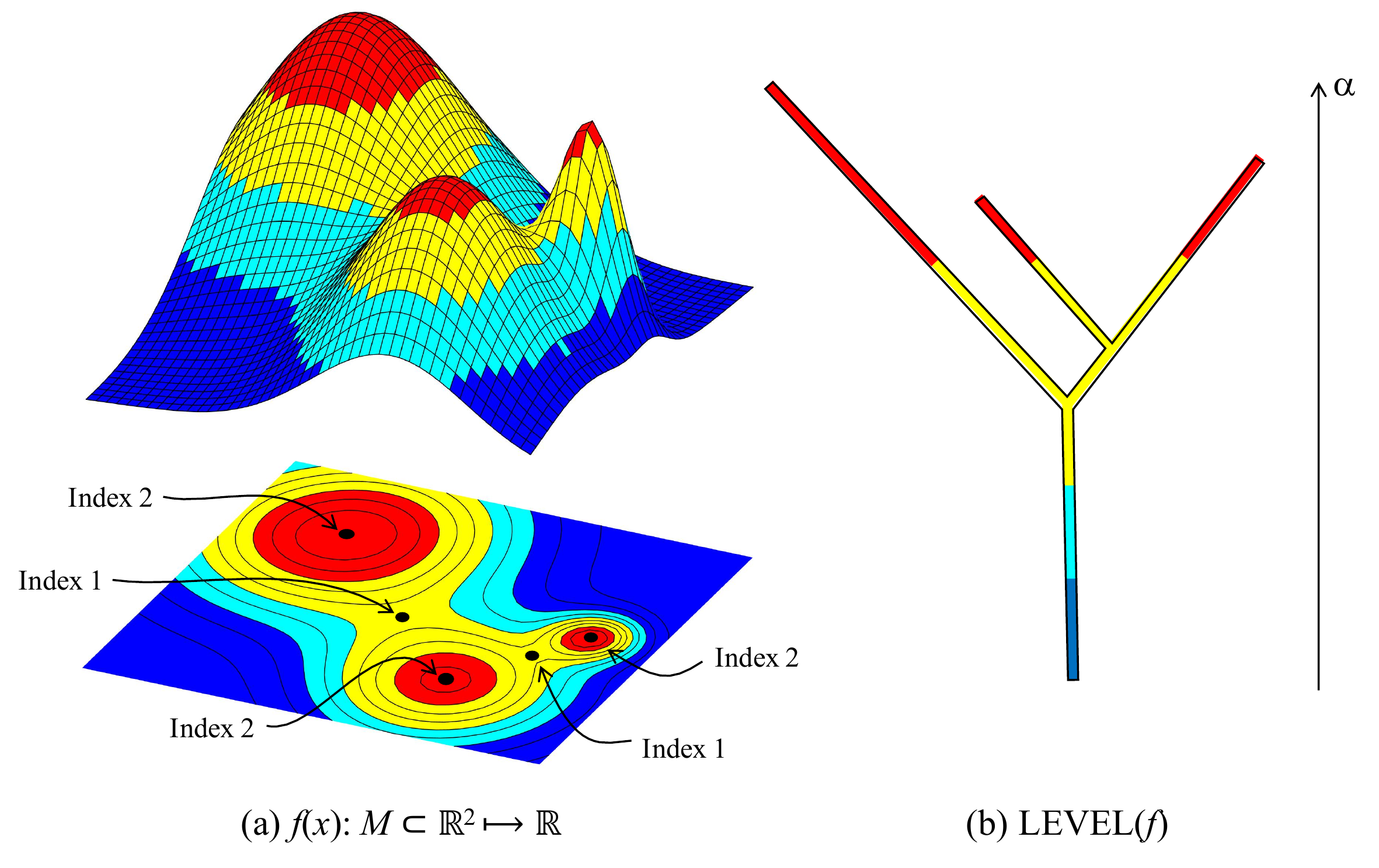}
\caption{Level set tree of a Morse function: An illustration. 
(a) A Morse function $f(x): M\subset\mathbb{R}^2\to \mathbb{R}$ (top) and 
its level sets $\cL_{\alpha}$ (bottom).
(b) The level set tree $\textsc{level}(f)$ shows how distinct 
components of $\cL_{\alpha}$ merge as threshold $\alpha$ decreases.
The color code illustrates the value of $f(x)$ at different level 
lines.
Each critical point of index $2$ (local maximum) corresponds to a leaf.
In this figure, each critical point of index $1$ (saddle) corresponds
to an internal vertex.}
\label{fig:Morse_2D}
\end{figure}

If $B(u,v): V \times V \rightarrow \mathbb{R}$ is a nondegenerate 
(i.e., with non-zero determinant) symmetric bilinear form over an $n$-dimensional vector space $V$,
then there exists a unique nonnegative integer $\lambda \leq n$ and at least one basis $\mathcal{B}$ of $V$ such that, in basis $\mathcal{B}$,
$$B(x,x)=-x_1^2-\hdots-x_\lambda^2+x_{\lambda+1}^2+\hdots+x_n^2.$$
This implies the following corollary to the Morse Theorem (Thm.~\ref{thm:Morse}), known as the Morse Lemma.
\begin{cor}[{{\bf Morse Lemma \cite{Nicolaescu}}}]\label{cor:Morse} \index{Morse!lemma}
Consider an $n$-dimensional differentiable manifold $M$, and a differentiable 
function $f:M \rightarrow \mathbb{R}$. 
If $p \in M$ is a nondegenerate critical point of $f$, then there exists and open neighborhood $U$ of $p$ and local coordinates $(x_1,\hdots,x_n)$ on $U$ 
with $$\big(x_1(p),\hdots,x_n(p)\big)=(0,\hdots,0)$$ such that in this coordinates,
$$f(x)=f(p)-x_1^2-\hdots-x_\lambda^2+x_{\lambda+1}^2+\hdots+x_n^2.$$
\end{cor}
\noindent
The integer $\lambda$ in Cor.~\ref{cor:Morse} is called the {\it index} of the nondegenerate critical point $p \in M$.
The next lemma concerns directly the structure of the level set trees for $f:M \rightarrow \mathbb{R}$. 
Let $M$ and $f$ to be as above. 
Following the one-dimensional setup of Sect.~\ref{sec:level1}, for $\alpha\in\mathbb{R}$ we 
consider the level set
$$\cL_\alpha=\cL_\alpha(f) = \{x \in M \,:\, f(x) \geq \alpha\}.$$
\begin{lem}[{{\bf \cite{MilnorEtAl,Carmo}}}]\label{lem:Milnor}
Consider an $n$-dimensional differentiable manifold $M$, and a Morse function 
$f:M \rightarrow \mathbb{R}$. 
Given points $p,q \in M$ and a differentiable curve $\gamma:\, [0,1] \rightarrow M$ such that $\gamma(0)=p$ and $\gamma(1)=q$.
Let $a=\min\big\{f(p),f(q)\big\}$ be the minimal endpoint value, and let $b=\min\limits_{t \in [0,1]}\big(f \circ \gamma(t)\big)$.

Suppose $f^{-1}\big([a,b]\big)$ is compact and does not contain any critical points of index 
$n$ or $n-1$. 
Then, for any $\delta>0$, there exists a differentiable curve $\widetilde{\gamma}:\, [0,1] \rightarrow M$ homotopic to $\gamma$
 such that $\widetilde{\gamma}(0)=p$ and $\widetilde{\gamma}(1)=q$, and
 $$\widetilde{\gamma}\big([0,1]\big) \subset \cL_{a-\delta}.$$ 
\end{lem}

\begin{figure}[t] 
\centering\includegraphics[width=0.7\textwidth]{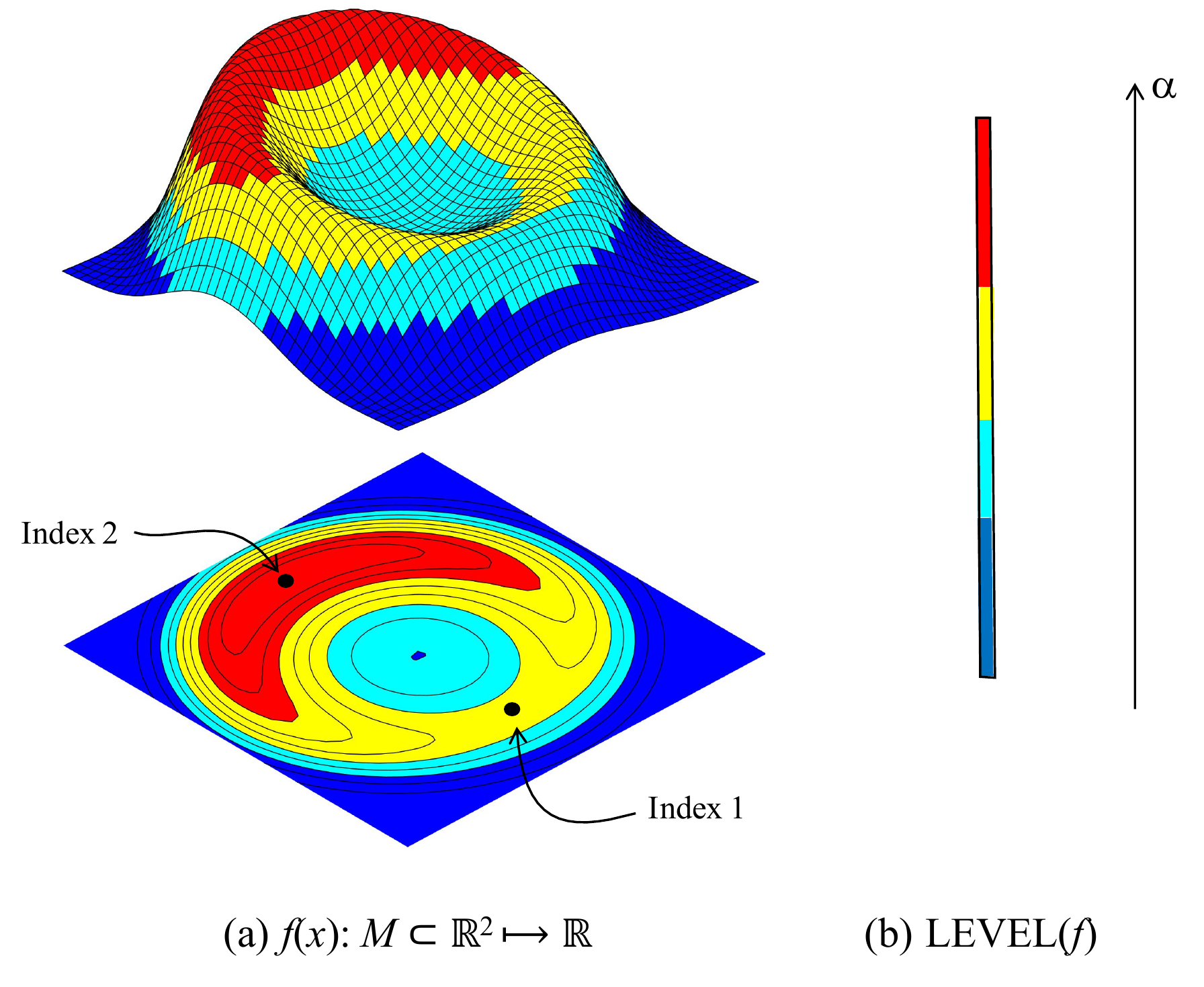}
\caption{Level set tree of a Morse function: An illustration. 
In this figure, the critical point of index $1$ (saddle) does not correspond
to an internal vertex. 
The rest of notations are the same as in Fig.~\ref{fig:Morse_2D}.}
\label{fig:Morse_saddle}
\end{figure}

\medskip
\noindent
Consider an $n$-dimensional compact differentiable manifold $M$, 
and a Morse function $f:M \rightarrow \mathbb{R}$. Recalling the definition of a level set tree in dimension one,
for $p,q \in M$, let $$\underline{f}(p,q):=\sup\limits_{\gamma:p \rightarrow q} \inf\limits_{x\in \gamma} f(x),$$
where the supremum is taken over all continuous curves $\gamma:\, [0,1] \rightarrow M$ such that $\gamma(0)=p$ and $\gamma(1)=q$.
Next, as it was the case when $\dim(M)=1$, we define a {\it pseudo-metric} on $M$ as
\begin{equation}\label{eqn:tree_dist_onM}
d_f(p,q):=\left(f(p)-\underline{f}(p,q)\right)+ \left(f(q)-\underline{f}(p,q)\right),\quad p,q\in\,M.
\end{equation}
We write $p\sim_f q$ if $d_f(p,q)=0$, and observe that $d_f$ is a metric over the quotient space $M/\!\sim_f$. 
Thus, $\left(M/\!\sim_f,d_f\right)$ is a metric space, satisfying Def.~\ref{def:treeL} of a tree. 
This tree will be called the {\it level set tree} of $f$, and  denoted by $\textsc{level}(f)$.
Here, $d_f(p,q)\ge |f(p)-f(q)|$, with $d_f(p,q)= |f(p)-f(q)|$ if and only if 
points $(p/\!\sim_f)$ and $(q/\!\sim_f)$ of $\textsc{level}(f)$ belong to the same lineage.
In particular, if $d_f(p,q)=f(p)-f(q)$, then 
$(p/\!\sim_f)$ is the descendant point to $(q/\!\sim_f)$, and respectively, $(q/\!\sim_f)$ is the ancestral point to $(p/\!\sim_f)$.
Figures~\ref{fig:Morse_2D},\ref{fig:Morse_saddle} show examples of level set trees for 
functions $f$ on $\mathbb{R}^2$.

\begin{ex}[{{\bf Compactness requirement}}]
\label{ex:compact}
The requirement for the manifold $M$ to be compact is necessary to ensure that there 
are no pairs of disjoint closed sets
such that the distance between the two sets equals zero.
As a counterexample, consider a function $f(x,y)=x^2-e^y$ on $M=\mathbb{R}^2$ (Fig.~\ref{fig:Morse_compact}). 
Here, the level set $\cL_0$ consists of two nonintersecting closed regions, 
marked by
gray shading in Fig.~\ref{fig:Morse_compact}(b):
\[A=\{f(x,y)\ge 0, x>0\} = \big\{x \geq e^{y/2}\big\}\]  
and  
\[B=\{f(x,y)\ge 0, x<0\} = \big\{x \leq -e^{y/2}\big\}.\] 
The distance between $A$ and $B$ is zero, as the two sets get arbitrary close
along the line $x=0$ as $y\to-\infty$.
Consider points $p=(e,2) \in A$ and $q=(-e,2) \in B$ marked in
Fig.~\ref{fig:Morse_compact}.
The points $p$ and $q$ are not connected by a continuous path inside $\cL_0$,
since each such a path must intersect the line $x=0$ along which $f<0$. 
Yet, if we were to extend the distance in \eqref{eqn:tree_dist_onM} to $M=\mathbb{R}^2$, 
then $\underline{f}(p,q)=0$ since for any $\delta>0$ there exists a path similar 
to $\gamma$ in Fig.~\ref{fig:Morse_compact}(b), with the tip on the line $x=0$ for large
enough $y$, so that $\gamma\subset \cL_{-\delta}$. 
Consequently, we have $\,d_f(p,q)=0$ implying that the points $p$ and $q$ are equivalent
on the level set tree of $f$, $p=_{\sim_f} q$, albeit they belong to two disconnected
components of $\cL_0$.
\end{ex}

\begin{figure}[t] 
\centering\includegraphics[width=0.9\textwidth]{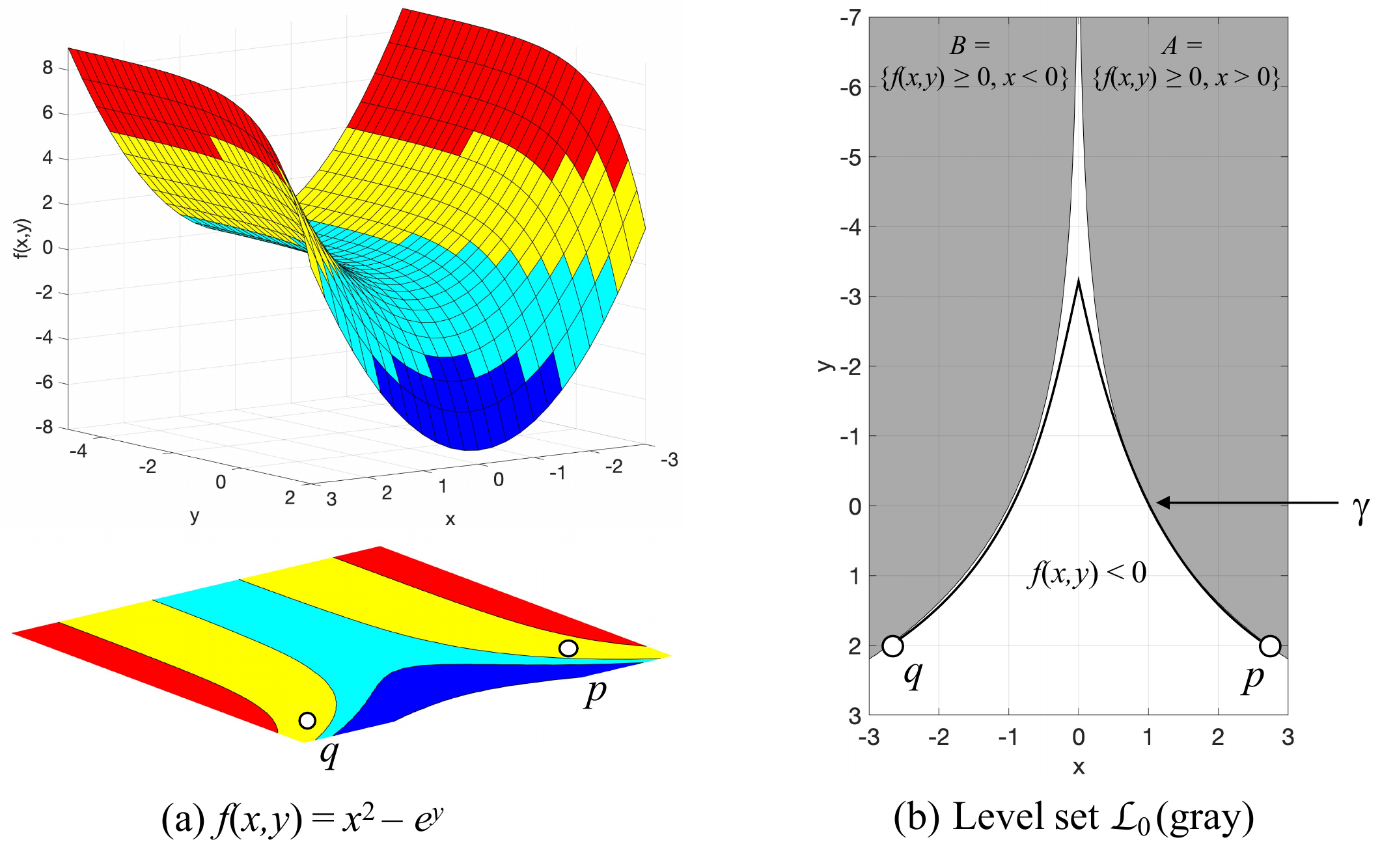}
\caption{Illustration to Example~\ref{ex:compact}. 
The manifold $M$ must be compact to properly define the level set tree of a function 
$f: M\to\mathbb{R}$.
In this example, $M=\mathbb{R}^2$ is not compact. 
This allows for the existence of points $p$ and $q$ such that $d_f(p,q)=0$,
while they belong to disconnected components of $\cL_0$.
}
\label{fig:Morse_compact}
\end{figure}

Naturally, if $f:M \rightarrow \mathbb{R}$ is a Morse function, the critical points of index $n$ 
(local maxima) correspond to the leaves of the level set tree $\textsc{level}(f)$.
As we decrease $\alpha$, new segments of $\cL_\alpha$ appear at the critical 
points of index $n$, and disconnected components of $\cL_\alpha$ merge at 
some critical points of index less than $n$. 
If $M$ is a compact manifold and $f:M \rightarrow \mathbb{R}$ is a Morse function, then by Lem.~\ref{lem:Milnor}
the critical points of index less than $n-1$ cannot be the merger points of separated pieces of $\cL_\alpha$.
Thus, we obtain the following corollary of Lem.~\ref{lem:Milnor}.
\begin{cor}\label{cor:KovMilnor}
Consider an $n$-dimensional compact differentiable manifold $M$, 
and a Morse function $f:M \rightarrow \mathbb{R}$. 
Then, there is a bijection between the leaves of $\textsc{level}(f)$ and the 
critical points of $f$ of index $n$, and a one-to-one (but not necessarily onto) 
correspondence between the internal (non-leaf) vertices of $\textsc{level}(f)$ 
and the critical points of $f$ of index $n-1$.
\end{cor}
\begin{proof}
Suppose $c \in M$ is a critical point of $f$ of index less than $n-1$ such that $(c/\!\sim_f)$ 
is an internal (non-leaf) vertex of $\textsc{level}(f)$.
Then, $(c/\!\sim_f)$ is a parent vertex to at least one pair of points $(p/\!\sim_f)$ and $(q/\!\sim_f)$ of $\textsc{level}(f)$ that do not belong to the same lineage, $\underline{f}(p,q)=f(c)$,
and therefore
\be\label{eqn:pqc_descent}
d_f(p,q)=f(p)+f(q)-2f(c)=|f(p)-f(c)|+2(a-f(c))>|f(p)-f(q)|,
\ee
where $a=\min\big\{f(p),f(q)\big\}$.
Thus, since $M$ is a differentiable manifold, there exists a differentiable curve $\gamma:\, [0,1] \rightarrow M$ such that $\gamma(0)=p$ and $\gamma(1)=q$,
and  $\min\limits_{t \in [0,1]}\big(f \circ \gamma(t)\big)=f(c)$. Then, by Lemma \ref{lem:Milnor}, for any $\delta>0$, there exists a differentiable curve 
$\widetilde{\gamma}:\, [0,1] \rightarrow M$ homotopic to $\gamma$
 such that $\widetilde{\gamma}(0)=p$ and $\widetilde{\gamma}(1)=q$, and
 $$\widetilde{\gamma}\big([0,1]\big) \subset \cL_{a-\delta}.$$
 Hence, $$d_f(p,q)\leq f(p)+f(q)-2(a-\delta)=|f(p)-f(q)|+2\delta$$
 for any $\delta>0$. Therefore, $d_f(p,q)=|f(p)-f(q)|$, contradicting \eqref{eqn:pqc_descent}, i.e., contradicting the assumption
 that $(p/\!\sim_f)$ and $(q/\!\sim_f)$ do not belong to the same lineage in $\textsc{level}(f)$.
\end{proof}

\begin{Rem}
Corollary~\ref{cor:KovMilnor} asserts that while every internal vertex of the level set tree
corresponds to a critical point of index $1$, not every critical point of index $1$ may 
correspond to an internal vertex. 
Figure~\ref{fig:Morse_2D} shows an example of a function where every critical point
of index $1$ (saddle) corresponds to an internal vertex. 
Figure~\ref{fig:Morse_saddle} shows an example of a function where the critical point
of index $1$ (saddle) does not corresponds to an internal vertex. 
\end{Rem}

\medskip
\noindent
Finally, Cor.~\ref{cor:KovMilnor} together with Morse Lemma (Cor.~\ref{cor:Morse}) 
imply the following lemma.
\begin{lem}\label{lem:KovMorse}
Consider an $n$-dimensional compact differentiable manifold $M$, and a 
Morse function $f:M \rightarrow \mathbb{R}$. 
Suppose there is no two distinct critical points $p$ and $q$ of index $n-1$ with the same value $f(p)=f(q)$.
Then, the level set tree $\textsc{level}(f)$ is binary.
\end{lem}
\begin{proof}
Suppose $p$ is a critical point of $f$ corresponding to an internal (non-leaf) 
vertex in $\textsc{level}(f)$. Then, by Corollary \ref{cor:KovMilnor},
$p$ has index $\lambda=n-1$. 
Corollary \ref{cor:Morse} asserts that there exists and open neighborhood $U$ of $p$ and local coordinates $(x_1,\hdots,x_n)$ on $U$ 
with $$\big(x_1(p),\hdots,x_n(p)\big)=(0,\hdots,0)$$ such that in this coordinates,
$$f(x)=f(p)-x_1^2-\hdots-x_{n-1}^2+x_n^2.$$
Hence, as $\alpha$ decreases, the merger of distinct components of $\cL_\alpha$ 
happens along the $x_n$-coordinate axis. 
This allows for the merger of at most two components.
\end{proof}


\medskip
\noindent
Vladimir Arnold studied an alternative (albeit similar in spirit) construction of level set trees 
that he called the {\it graph of Morse function} $f:M \rightarrow \mathbb{R}$, 
concentrating mainly on the spheres $M=S^2$; 
see \cite{Arnold1959,Arnold2006,Arnold2007} and references therein.
Arnold has shown that these graphs are binary trees as well. 
These trees are constructed in such a way that both the local minima (index $0$) 
and the local maxima (index $2$) points of $f$ correspond to the leaves, 
while the saddle points (index $1$) correspond to the internal (non-leaf) vertices. 
The goal of Arnold's study was to shed light on the
problem of classifying all possible configurations of the horizontal lines on the topographical maps formulated by A. Cayley in 1868.
In \cite{Arnold2007}, Arnold quotes a communication with Morse: {\it M. Morse has told me, in 1965, that the problem of the description of the possible combinations of several critical points of a smooth function on a manifold looks hopeless to him. L.\,S.~Pontrjagin and H.~Whitney were of the same opinion.}
Arnold's work of topological classification of level lines for Morse functions on $S^2$ enriched the collection of questions accompanying the Hilbert's sixteenth problem, which promoted the study of the topological structures of the level lines of real polynomials $p(x)$ over $x \in \mathbb{R}^n$, \cite{Hilbert,Arnold2006,Arnold2007}.

\section{Kingman's coalescent process}\label{sec:Kingman}
We refer to a general definition of a coalescent process in Section \ref{sec:coalescent}.
Recall that in an $N$-particle coalescent process, a pair of clusters with masses $i$ 
and $j$ coalesces at the rate $K(i,j)/N$.
The mass-independent rate $K(i,j)=1$ defines the Kingman's coalescent process \cite{Kingman82b}.
The following result establishes a weak form of Horton law for Kingman's coalescent. 
\begin{thm}[{\bf Root-Horton law for Kingman's coalescent, \cite{KZ17ahp}}]\label{Mthm}
 Consider Kingman's $N$-coalescent process and its tree representation
$T^{(N)}_{\rm K}$. Let $N_j=N_j^{(N)}$ denote the number of branches of Horton-Strahler order $j$ in the tree $T^{(N)}_{\rm K}$. 
\begin{description}
  \item[(i)] The asymptotic Horton ratios $\cN_j$ exist  and are finite for all $j \in \mathbb{N}$, as in Def. \ref{def:HortonWellDef}.
That is, for each $j$, the following limit exists and is finite:
\be\label{eq:Nj}
N_j^{(N)}/N\stackrel{p}{\to}\cN_j\quad\text{ as }{N\to\infty}.
\ee
  \item[(ii)] Furthermore, $\cN_j$ satisfy the root-Horton law (Def. \ref{def:HortonList}): 
\[\lim\limits_{j \rightarrow \infty}\left(\cN_j \right)^{-{1 \over j}}=R\]
with Horton exponent $2 \le R \le 4$.
\end{description}
 \end{thm}

\subsection{Smoluchowski-Horton ODEs for Kingman's coalescent}  
\label{informal}

In this section we provide a heuristic derivation of  Smoluchowski-type ODEs for the number of Horton-Strahler branches in the coalescent 
tree $T^{(N)}_{\rm K}$ and consider the asymptotic version of these equations as $N\to\infty$.
Section~\ref{sec:hydro} formally establishes the validity of the hydrodynamic limit.

Recall that $K(i,j) \equiv 1$. Let $|\Pi^{(N)}_t|$ denote the total number of clusters at time $t \geq 0$, and let  $\eta_{(N)}(t):=|\Pi^{(N)}_t|/N$ be the total number of clusters relative to the system size $N$.
Then $\eta_{(N)}(0)=N/N=1$ and $\eta_{(N)}(t)$ decreases by $1/N$ with each 
coalescence of clusters; this happens with the rate
$${1 \over N} \, \binom{N\, \eta_{(N)}(t)}{2}={\eta_{(N)}^2(t) \over 2}\cdot N+o(N),
\quad{\rm as~} N\to\infty,$$
since $1/N$ is the coalescence rate for any pair of clusters regardless of their masses. 
Informally, this implies that the large-system limit relative number of clusters 
$\displaystyle\eta(t)=\lim_{N\to\infty}\eta_{(N)}(t)$ satisfies the following ODE:
\begin{eqnarray} \label{Aeta_t}
{d \over dt} \eta(t)=-\frac{\eta^2(t)}{2}.
\end{eqnarray}
The initial condition $\eta(0)=1$ implies a unique solution $\eta(t)=2/(2+t)$.
The existence of the limit $\eta(t)$ is established in Lem.~\ref{lem3}(a) of Sect. \ref{sec:hydro}.

Next, for any $k \in \mathbb{N}$ we write $\eta_{k,N}(t)$ for the relative number of 
clusters (with respect to the system size $N$) that correspond to branches 
of Horton-Strahler order $k$ in tree $T^{(N)}_{\rm K}$ 
at time $t$. 
Initially, each particle represents a leaf of Horton-Strahler order $1$. 
Accordingly, the initial conditions are set to be, using Kronecker's delta notation, 
$$\eta_{k,N}(0)=\delta_1(k).$$ 
Below we describe the evolution of $\eta_{k,N}(t)$ using the definition of Horton-Strahler orders.

Observe that $~\eta_{k,N}(t)$ increases by $1/N$ with each coalescence of clusters of Horton-Strahler 
order $k-1$ that happens with the rate
$${1 \over N} \, 
\binom{N\, \eta_{k-1,N}(t)}{2}={\eta_{k-1,N}^2(t) \over 2} \cdot N+o(N).$$
Thus ${\eta_{k-1,N}^2(t) \over 2}+o(1)$ is the instantaneous rate of increase of $\eta_{k,N}(t)$.

Similarly,  $~\eta_{k,N}(t)$ decreases by $1/N$ when a cluster of order $k$ coalesces 
with a cluster of order strictly higher than $k$ that happens with the rate
$$\eta_{k,N}(t) \, \left(\eta_{(N)}(t)-\sum\limits_{j=1}^{k} \eta_{j,N}(t) \right)\cdot N,$$
and it decreases by $2/N$ when a cluster 
of order $k$ coalesces with another cluster of order $k$ that happens with the rate
$${1 \over N} \, \binom{N\, \eta_{k,N}(t)}{2} ={\eta_{k,N}^2(t) \over 2}\cdot N+o(N).$$
Thus the instantaneous rate of decrease of $\eta_{k,N}(t)$ is
$$\eta_{k,N}(t) \, \left(\eta_{(N)}(t)-\sum\limits_{j=1}^{k} \eta_{j,N}(t) \right)
+\eta^2_{k,N}(t)+o(1).$$

We can informally write the limit rates-in and the rates-out for the clusters 
of Horton-Strahler order via the following {\it Smoluchowski-Horton system} of ODEs:
\begin{eqnarray} \label{Aeta}
{d \over dt} \eta_k(t)=\frac{\eta^2_{k-1}(t)}{2}-\eta_k(t) 
\, \left(\eta(t)-\sum\limits_{j=1}^{k-1} \eta_j(t) \right),
\end{eqnarray}
with the initial conditions $\eta_k(0)=\delta_1(k)$. 
Here we interpret $\displaystyle\eta_k(t)$ as the hydrodynamic limit of $\eta_{k,N}(t)$ as $N\to\infty$, 
which will be rigorously established in Lem. \ref{lem3}(b) of Sect. \ref{sec:hydro}.  We also let $\eta_0 \equiv 0$.

Since $\eta_k(t)$ has the instantaneous rate of increase $\eta_{k-1}^2(t)/2$, 
the relative total number of clusters corresponding to branches of Horton-Strahler order  $k$ is then
\be
\label{cNj}
\cN_k=\delta_1(k)+\int\limits_0^{\infty} {\eta^2_{k-1}(t) \over 2} dt.
\ee  
This equation has a simple heuristic interpretation.
Specifically, according to the Horton-Strahler rule \eqref{eq:HSorder}, 
a branch of order $k>1$ can only be created by merging two 
branches of order $k-1$.
In Kingman's coalescent process these two branches 
are selected at random from all pairs of branches of order $k-1$ that
exist at instant $t$. 
As $N$ goes to infinity, the asymptotic density of a pair of branches
of order $(k-1)$, 
and hence the instantaneous intensity 
of newly formed branches of order $k$, is $\eta^2_{k-1}(t)/2$.  
The integration over time gives the relative total number of order-$k$ branches.
The validity of equation \eqref{cNj} is established within the proof of Thm.~\ref{Mthm}(i) 
that follows Lem.~\ref{lem3}.

It is not hard to compute the first three terms of the sequence 
$\cN_k$ by solving equations (\ref{Aeta_t}) and (\ref{Aeta}) 
in the first three iterations:
$$\cN_1=1, \quad \cN_2={1 \over 3}, \quad \text{ and } 
\quad \cN_3={e^4 \over 128}-{e^2 \over 8}+{233 \over  384}= 0.109686868100941\hdots$$
Hence, we have 
${\cN_1/ \cN_2}={3}$ and 
${\cN_2/ \cN_3}= 3.038953879388\dots$ 
Our numerical results yield, moreover,
$$\lim\limits_{k \rightarrow \infty}\left(\cN_k \right)^{-{1 \over k}}
=\lim\limits_{k \rightarrow \infty}{\cN_{k} \over \cN_{k+1}}=3.0438279\dots$$

\subsection{Hydrodynamic limit} \label{sec:hydro}
This section establishes the existence of the asymptotic ratios 
$\cN_k$ of \eqref{eq:Nj} as well as the validity of the equations~\eqref{Aeta_t},
\eqref{Aeta} and \eqref{cNj} in a hydrodynamic limit. 
We refer to Darling and Norris \cite{RDJN08} for a survey of techniques 
for establishing convergence of a Markov chain to the solution of a differential equation.

Notice that if the first $k-1$ functions $\eta_1(t),\hdots,\eta_{k-1}(t)$ 
are given, then (\ref{Aeta}) is a linear equation in $\eta_k(t)$. 
This {\it quasilinearity} implies the existence and uniqueness of a solution.

We now proceed with establishing a hydrodynamic limit for the 
Smoluchowski-Horton system of ODEs \eqref{Aeta}.  
Let $$\eta_{k,N}(t):={N_k(t) \over N}~~~\text{ and }~~~g_{k,N}(t):=\eta_{(N)}(t)-\sum\limits_{j:j<k}\eta_{j,N}(t).$$ 

\begin{lem}\label{lem3}
Let $\eta_{(N)}(t)$ be the relative total number of clusters and $\eta(t)$  be the solution to equation (\ref{Aeta_t}) with the initial 
condition $\eta(0)=1$. 
Let $\eta_{k,N}(t)$ denote the relative number of clusters that correspond to 
branches of Horton-Strahler order $k$ and let functions $\eta_k(t)$ solve the 
system of equations \eqref{Aeta} with the initial conditions $\eta_k(0)=\delta_1(k)$. 
Then, as $N\to\infty$,
\begin{description}
  \item[(a)] $~\big\|\eta_{(N)}(t)-\eta (t) \big\|_{L^\infty [0,\infty)} \stackrel{p}{\to} 0$;
  \item[(b)] $~\|\eta_{k,N}(t)-\eta_k(t)  \|_{L^\infty [0,\infty)} \stackrel{p}{\to}  0$,\quad $\forall k\ge1$.
\end{description}
\end{lem}

\begin{proof}
We adopt here the approach of \cite{KOY} that uses the weak limit law established in 
\cite[Theorem 2.1, Chapter 11]{EK86} and \cite[Theorem 8.1]{Kurtz81}; it is 
briefly explained in Appendix~\ref{sec:kurtz} of this manuscript.
This approach is different from the original proof given in \cite{KZ17ahp}, and also from the method developed in Norris \cite{Norris99} for the Smoluchowski equations.

\medskip
\noindent
For a fixed positive integer $K$, let 
$$\hat{X}_N(t)=\Big(N_1(t),N_2(t),\hdots, N_K(t), N(t) \Big) \in \mathbb{Z}_+^{K+1}$$
with $\hat{X}_N(0)=Ne_1$.
The process $\hat{X}_N(t)$ is a finite dimensional Markov process.
Its transition rates can be found using the formalism (\ref{tintK}) 
for density dependent population processes.  
Specifically, let $x=(x_1,x_2,\hdots, x_{K+1})$. 
Then, for any $1\leq k \leq K$, the change vector $\ell=- e_k-e_{K+1}$ corresponding to a merger of a cluster of order $k$ into a cluster of order higher than $k$ has the rate
$$q^{(n)}(x , x +\ell)={1 \over N}x_k \left(x_{K+1}-\sum\limits_{j=1}^k x_j\right) = N \beta_\ell \left({x \over N}\right),$$
where $\beta_\ell (x)=x_k \left(x_{K+1}-\sum\limits_{j=1}^k x_j\right)$. 
For a given $k$ such that $1\leq k \leq K$, the change vector 
\[\ell=-2e_k+e_{k+1}{\bf 1}_{k<K}-e_{K+1}\] 
corresponding to a merger of a pair of clusters of order $k$
is assigned the rate
\be\label{eqn:rateMerge}
q^{(n)}(x , x +\ell)={1 \over N} \left[\frac{x_k^2}{2}-\frac{x_k}{2}\right]= N \left[ \beta_{\ell}\left({x \over N}\right) + O\left( \frac{1}{n} \right) \right],
\ee
where $\beta_{\ell} (x)=\frac{x_k^2}{2}$.
Finally, the change vector $\ell=-e_{K+1}$ corresponding to a merger of two clusters, both of order greater than $K$, is assigned the rate
$$q^{(n)}(x , x +\ell)={1 \over N} \left[\frac{x_{K+1}^2}{2}-\frac{x_{K+1}}{2}\right]= N \left[ \beta_{\ell}\left({x \over N}\right) + O\left( \frac{1}{n} \right) \right],$$
where $\beta_{\ell} (x)=\frac{x_{K+1}^2}{2}$.

\bigskip
\noindent
By Thm.~\ref{kurtzT}, $X_N(t)=N^{-1}\hat{X}_N(t)$ converges to $X(t)$ as in (\ref{KurtzAS}), where $X(t)$ satisfies (\ref{KurtzDE}) with
\begin{align}\label{eqn:Kbl}
F(x) &:= \sum\limits_\ell \ell \beta_\ell (x)= \sum_{k=1}^K  x_k \left(x_{K+1}-\sum\limits_{j=1}^k x_j\right)[-e_k-e_{K+1}] \nonumber\\
&\qquad \qquad \qquad \qquad +{1 \over 2}\sum_{k=1}^{K+1} x_k^2 [-2e_k{\bf 1}_{k\leq K}+e_{k+1}{\bf 1}_{k<K}-e_{K+1}] \nonumber \\
&= \sum_{k=1}^K \left({x_{k-1}^2 \over 2}-x_k\left(x_{K+1}-\sum\limits_{j=1}^{k-1} x_j\right)\right)e_k-\frac{x_{K+1}^2}{2}e_{K+1} , 
\end{align}
where we let $x_{-1}=0$ at all times.
Here, $F(x)$ naturally satisfies the Lipschitz continuity conditions (\ref{KurtzLipschitz}), and the initial conditions $X(0)=X_n(0)=e_1$.

Therefore, for a given integer $K>0$ and a fixed real $T>0$, equation (\ref{KurtzDE}) in 
Thm.~\ref{kurtzT}  with $F(x)$ as in (\ref{eqn:Kbl})  yields
\be\label{eqn:ASall}
\lim\limits_{N \to \infty} \sup\limits_{s \in [0,T]} \left|N^{-1}\eta_{(N)}(s)-\eta(s)\right|=0 \qquad \text{ a.s.}
\ee
and
\be\label{eqn:AS}
\lim\limits_{N \to \infty} \sup\limits_{s \in [0,T]} \left|N^{-1}\eta_{k,N}(s)-\eta_k(s)\right|=0 \qquad \text{ a.s.}
\ee
for all $k=1,2,\hdots,K$, with $\eta_{(N)}$ satisfying \eqref{Aeta_t} and $\eta_{k,N}$ satisfying the system of Smoluckowski-Horton system of ODEs (\ref{Aeta}).

Let $T_m$ be the time when the first $m$ clusters merge. The expectation for the time $T_m$ is
\be\label{eqn:meanTm}
E[T_m]={N \over \binom{N}{2}}+{N \over \binom{N-1}{2}}+\dots+{N \over \binom{N-m+1}{2}}={2m \over N-m}.
\ee
For given $\epsilon \in (0,1)$ and $\gamma>1$ let $m=\lfloor(1-\epsilon)N \rfloor$. Taking $T>{2(1-\epsilon) \over \epsilon}\gamma$, 
we have for all $t \geq T$,
$$0 < \eta(t) \leq \eta(T) < \eta\left({2(1-\epsilon) \over \epsilon}\gamma\right)<\eta\left({2(1-\epsilon) \over \epsilon}\right)=\epsilon.$$ 
Thus $~\big|\eta_{(N)}(t)-\eta (t) \big|>\epsilon~$ would imply $~\eta_{(N)}(t)>\epsilon >\eta (t)>0$,
and by Markov's inequality, we obtain
\begin{eqnarray}\label{MarkovLast}
P\Big(\big\|\eta_{(N)}(t)-\eta (t) \big\|_{L^\infty [T,\infty)}>\epsilon \Big) & \leq & P\Big(\eta_{(N)}(T)>\epsilon \Big) = P\Big(T_m >T \Big)  \nonumber \\
& \leq & {2(1-\epsilon) \over \epsilon T} < 1/\gamma. 
\end{eqnarray}
Together \eqref{eqn:ASall} and the above equation \eqref{MarkovLast} imply
$$\lim\limits_{N \rightarrow \infty}P\Big(\big\|\eta_{(N)}(t)-\eta (t) \big\|_{L^\infty [0,\infty)}<\epsilon \Big)=1.$$
Hence $~\|\eta_{(N)}(t)-\eta(t)  \|_{L^\infty [0,\infty)} \rightarrow 0~$ in probability, establishing Lemma \ref{lem3}(a).

\bigskip
\noindent
Finally, observe that  for any $\epsilon>0$ and for $T>0$ large enough so that $~\eta(T) <\epsilon$,
$$\eta_k (t) \leq \eta(t) \leq \eta(T) <\epsilon \text{ for all } t \geq T.$$
Thus,
\begin{eqnarray}\label{eqn:MarkovT}
P\Big(\big\|\eta_{k,N}(t)-\eta_k (t) \big\|_{L^\infty [T,\infty)}>\epsilon \Big) & \leq & P\Big(\big\|\eta_{k,N}(t) \big\|_{L^\infty [T,\infty)}>\epsilon \Big) \nonumber \\
& \leq & P\Big(\big\|\eta_{(N)}(t) \big\|_{L^\infty [T,\infty)}>\epsilon \Big) \nonumber \\
& = & P\Big(\eta_{(N)}(T)>\epsilon \Big) \nonumber \\
& \leq & {2(1-\epsilon) \over \epsilon T},
\end{eqnarray}
where the last bound is obtained from Markov inequality: for $m=\lfloor(1-\epsilon)N \rfloor$,
$$P\Big(\eta_{(N)}(T)>\epsilon \Big)=P(T_m>T) \leq {E[T_m] \over T}={2m \over (N-m)T} \leq {2(1-\epsilon) \over \epsilon T}$$
by \eqref{eqn:meanTm}. Together, equations \eqref{eqn:AS} and \eqref{eqn:MarkovT} imply 
$$\|\eta_{k,N}-\eta_k  \|_{L^\infty [0,\infty)} \stackrel{p}{\to} 0 \qquad \forall k\geq 1.$$ \end{proof}

\medskip
Consequently, we establish a hydrodynamic limit for the Horton ratios (Thm.~\ref{Mthm}(i)) and validate formula \eqref{cNj}. 
\begin{proof}[Proof of Theorem \ref{Mthm}(i)]
The existence of the limit $\cN_j=\lim_{N\to\infty} N_j/N$ in probability and its expression \eqref{cNj} via the solution $\eta_{(N)}(t)$ of \eqref{Aeta_t}
follows from \eqref{eqn:rateMerge} in the context of Theorem \ref{kurtzT} and the tail bound \eqref{MarkovLast}.
\end{proof}

\subsection{Some properties of the Smoluchowski-Horton system of ODEs}\label{existence}
Here we restate the Smoluchowski-Horton system of ODEs  \eqref{Aeta} as a simpler quasilinear system of ODEs \eqref{eqn:ODEg}, which we later  (Sect. \ref{sec:ODEh}) rescale to the interval $[0,1]$ \eqref{eqn:ODEh}. 
Some of the properties established in Prop. \ref{prop:properties_g} and Lem. \ref{lem:half} of this
section are used in the proof of Thm.~\ref{Mthm}(ii) in Sect. \ref{sec:proofKingman}.

\subsubsection{Simplifying the Smoluchowski-Horton system of ODEs}\label{sec:ODEg}
Let $g_1(t)=\eta(t)$ and $g_k(t)=\eta(t)-\sum\limits_{j:~j<k} \eta_j(t)$ 
be the asymptotic number of clusters of Horton-Strahler order $k$ or higher at time $t$. 
We can rewrite \eqref{Aeta} via $g_k$ using $\eta_k(t)=g_k(t)-g_{k+1}(t)$:
\[{d \over dt}g_k(t)-{d \over dt}g_{k+1}(t)
={\big(g_{k-1}(t)-g_k(t) \big)^2 \over 2}-(g_k(t)-g_{k+1}(t))g_k(t).\]
We now rearrange the terms, obtaining for all $k \geq 2$,
\begin{eqnarray} \label{eqn:odeF1}
 {d \over dt}g_{k+1}(t)-{g^2_k(t) \over 2}+g_k(t) g_{k+1}(t)
 ={d \over dt} g_k(t)-{g^2_{k-1}(t) \over 2}+g_{k-1}(t) g_k(t). 
\end{eqnarray}
One can readily check that ${d \over dt} g_2(t)-{g^2_1(t) \over 2}+g_1(t)\,g_2(t)=0$; 
the above equations hence simplify as follows
\begin{eqnarray} \label{eqn:ODEg}
g'_{k+1}(t)-{g^2_k(t) \over 2}+g_k(t) g_{k+1}(t)=0 \qquad \\ \nonumber \text{ with }~g_1(t)={2 \over t+2}, \text{ and } g_k(0)=0 \text{ for } k \geq 2.
\end{eqnarray}

Observe that the existence and uniqueness of the solution sequence $g_k$ of 
\eqref{eqn:ODEg} follows immediately
from the quasilinear structure of the system \eqref{eqn:ODEg}: 
for a known $g_k(t)$, the next function $g_{k+1}(t)$  
is obtained by solving a first-order linear equation.

From \eqref{eqn:ODEg} one has $g_k(t)>0$ for all $t>0$,
and similarly, from the equation  \eqref{Aeta} one has 
\be\label{eqn:domin_g}
\eta_k(t)=g_k(t)-g_{k+1}(t)>0 ~~\text{ for all }~t>0.
\ee

Next, returning to the asymptotic ratios $\cN_k$, 
we observe that \eqref{eqn:odeF1} implies, for $k\ge 2$,
$$\cN_k=\int\limits_0^{\infty} {\eta^2_{k-1}(t) \over 2} dt
=\int\limits_0^{\infty} {(g_{k-1}(t)-g_k(t))^2 \over 2} dt
=\int\limits_0^{\infty} {g^2_k(t) \over 2} dt$$
since
$$ {(g_{k-1}(t)-g_k(t))^2 \over 2} ={d \over dt}g_k(t)+{g^2_k(t) \over 2},$$ 
where $0 \leq g_k(t) < g_1(t) \rightarrow 0$ as $t \rightarrow \infty$, and $\int\limits_0^{\infty}{d \over dt}g_k(t)dt=g_k(\infty)-g_k(0)=0$ for $k \geq 2$.
Let $n_k$ represent the number of order-$k$ branches relative to the number 
of order-$(k+1)$ branches:
\be\label{eqn:nk_g}
n_k:={\cN_k \over \cN_{k+1}}={\frac{1}{2}\int\limits_0^{\infty} {g^2_k(t)} dt \over \frac{1}{2}\int\limits_0^{\infty} {g^2_{k+1}(t)} dt}={ \|g_k\|^2_{L^2[0, \infty)} \over  \|g_{k+1}\|^2_{L^2[0, \infty)}}.
\ee
Consider the following limits that represent, respectively, the 
root and the ratio asymptotic Horton laws:
$$\lim\limits_{k \rightarrow \infty}\left(\cN_k \right)^{-{1 \over k}}=\lim\limits_{k \rightarrow \infty}\left(\prod\limits_{j=1}^{k} n_j \right)^{1 \over k} \qquad \text{ and } \qquad \lim\limits_{k \rightarrow \infty} n_k=\lim\limits_{k \rightarrow \infty} { \|g_k\|^2_{L^2[0, \infty)} \over  \|g_{k+1}\|^2_{L^2[0, \infty)}}.$$
Theorem~\ref{Mthm}(ii) establishes the existence of the first limit. 
We expect the second, stronger, limit also to exist and both of them to be equal to $3.043827\dots$ 
according to our numerical results.
We now establish some basic facts about $g_k$ and $n_k$.
\begin{prop}  \label{prop:properties_g}
Let $g_k(x)$ solve the ODE system (\ref{eqn:ODEg}). Then\\
\begin{description}
\item[\qquad (a)] 
$~\frac{1}{2}\int\limits_0^{\infty} {g^2_k(t)} dt=\int\limits_0^{\infty} g_k(t)g_{k+1}(t) dt,$\\
\item[\qquad (b)] 
$~\int\limits_0^{\infty} g^2_{k+1}(t) dt=\int\limits_0^{\infty} (g_k(t)-g_{k+1}(t))^2 dt,$\\
\item[\qquad (c)] 
$~\lim\limits_{t \rightarrow \infty} tg_k(t) =2,$\\
\item[\qquad (d)] 
$~n_k={ \|g_k\|^2_{L^2[0, \infty)} \over  \|g_{k+1}\|^2_{L^2[0, \infty)}} \geq {2},$\\
\item[\qquad (e)] 
$~n_k={ \|g_k\|^2_{L^2[0, \infty)} \over  \|g_{k+1}\|^2_{L^2[0, \infty)}} \leq {4}.$\\
\end{description}

\end{prop}

\begin{proof} 

Part (a) follows from integrating (\ref{eqn:ODEg}), and part (b) follows from part (a). 
Part (c) is done by induction, using the L'H\^{o}pital's rule as follows. 
It is obvious that $~\lim\limits_{x \rightarrow \infty} tg_1(t) =2$.  
Hence, for any $k \geq 1$, \eqref{eqn:domin_g} implies
$$tg_k(t) \leq tg_1(t)={2t \over t+2} <2 \quad \forall t \geq 0.$$
Also, 
\begin{align*}
[tg_{k+1}]' &={tg^2_k(t) \over 2}-tg_k(t)g_{k+1}(t)+g_{k+1}(t)\\
&={\big(g_k(t)-g_{k+1}(t)\big)tg_k(t)+\big(2-tg_k(t)\big)g_{k+1}(t) \over 2}
\end{align*}
implying $~[tg_{k+1}]'\geq 0~$ for all $t \geq 0$ as $g_k(t)-g_{k+1}(t) \geq 0$ and $2-tg_k(t)>0$. Hence, $tg_{k+1}(t)$ is bounded and nondecreasing. 
Thus, $~\lim\limits_{t \rightarrow \infty} tg_{k+1}(t)$ exists for all $k \geq 1$.

Next, suppose $~\lim\limits_{t \rightarrow \infty} tg_k(t) =2$. 
Then by the Mean Value Theorem, for any $t>0$ and for all $y>t$,
$${g_{k+1}(t)-g_{k+1}(y) \over t^{-1}-y^{-1}} \leq \sup\limits_{z:~z \geq t}{g'_{k+1}(z) \over -z^{-2}}.$$
Taking $~y \rightarrow \infty$, obtain
$${g_{k+1}(t) \over t^{-1}} \leq \sup\limits_{z:~z \geq t}{g'_{k+1}(z) \over -z^{-2}}.$$
Therefore
$$\lim\limits_{t \rightarrow \infty} tg_{k+1}(t)=\lim\limits_{t \rightarrow \infty} {g_{k+1}(t) \over t^{-1}}=\limsup\limits_{z \rightarrow \infty} {g'_{k+1}(z) \over -z^{-2}}=\limsup\limits_{z \rightarrow \infty} {{g^2_k(z) \over 2}-g_k(z) g_{k+1}(z) \over -z^{-2}}$$
$$=\limsup\limits_{z \rightarrow \infty} \left[z^2g_k(z) g_{k+1}(z)-{z^2 g^2_k(z) \over 2}\right]=2\lim\limits_{t \rightarrow \infty} tg_{k+1}(t) -2$$
implying $~\lim\limits_{t \rightarrow \infty} tg_{k+1}(t)=2$.

Statement (d) follows from \eqref{eqn:nk_g} as we have $\cN_k \geq 2\cN_{k+1}$
from the definition of the Horton-Strahler order. 
An alternative proof of (d) using the system of ODEs \eqref{eqn:ODEh} is given in Sect. \ref{sec:ODEh}. 
 
Part (e) follows from part (a) together with H\"older inequality
$${1 \over 2}\|g_k\|^2_{L^2[0, \infty)}
=\int\limits_0^{\infty} g_k(t)g_{k+1}(t) dt 
\leq \|g_k\|_{L^2[0, \infty)} \cdot \|g_{k+1}\|_{L^2[0, \infty)},$$
which implies $n_k={\|g_{k}\|^2_{L^2[0, \infty)} \over \|g_{k+1}\|^2_{L^2[0, \infty)}} \leq {4}$.
\end{proof}

\begin{Rem}
The statements {\bf (a)} and {\bf (b)} of Proposition~\ref{prop:properties_g} have 
a straightforward heuristic interpretation, similar to that of equation 
\eqref{cNj} above.
Specifically, {\bf (a)} claims that the asymptotic relative total number 
of vertices of order $k+1$ and above in the Kingman's tree (left-hand side) 
equals twice the asymptotic relative total number of vertices of order 
$k+1$ and above except the vertices parental to two vertices of order $k$ 
(right-hand side).
This is nothing but the asymptotic property of a binary tree -- the number 
of leaves equals twice the number of internal nodes.
The item {\bf (a)} hence merely claims that the Kingman's tree formed by clusters
of order above $k$ is binary for any $k\ge 1$. 
Similarly, item {\bf (b)} claims that the asymptotic relative total number
of vertices of order $(k+2)$ and above (left-hand side) equals
the asymptotic relative total number of vertices of order $(k+1)$ (right-hand side).
This is yet another way of saying that the Kingman's tree is binary.
\end{Rem}

\subsubsection{Rescaling to $[0,1]$ interval}\label{sec:ODEh}
Define 
$$h_k(x)=(1-x)^{-1}-(1-x)^{-2} g_{k+1}\left({2x \over 1-x}\right)$$
for $x \in [0,1)$.
Then $h_0 \equiv 0$, $h_1 \equiv 1$, and the system of ODEs (\ref{eqn:ODEg}) rewrites as
\begin{equation} \label{eqn:ODEh}
h'_{k+1}(x)=2h_k(x)h_{k+1}(x)-h_k^2(x)
\end{equation}
with the initial conditions $h_k(0)=1$.

\medskip 
\noindent
Observe that the above quasilinearized system of ODEs (\ref{eqn:ODEh}) has $h_k(x)$ converging to $h(x)={1 \over 1-x}$ as $k \rightarrow \infty$, where $h(x)$ is the solution to Riccati equation $h'(x)=h^2(x)$ over $[0,1)$, with the initial value $h(0)=1$. Specifically, we have proven that $g_k(x) \rightarrow 0$ as $k \rightarrow \infty$. Thus
$$h_k(x)=(1-x)^{-1}-(1-x)^{-2} g_{k+1}\left({2x \over 1-x}\right) ~\longrightarrow ~h(x)={1 \over 1-x}.$$ 
\noindent
Observe that $h_2(x)=(1+e^{2x})/2$, but for $k \geq 3$ finding a closed form expression becomes increasingly hard. 

\medskip 
\noindent
We observe from \eqref{eqn:nk_g} that the quantity $n_k$ rewrites in terms of $h_k$ as follows
\be\label{eqn:nk_h}
n_k={\big\|1-h_k/h\big\|^2_{L^2[0,1]} \over \big\|1-h_{k+1}/h\big\|^2_{L^2[0,1]} }.
\ee
Consequently, equation \eqref{eqn:nk_h} implies
\be\label{eqn:Nk_h1}
\lim\limits_{k \rightarrow \infty}\left(\cN_k \right)^{-{1 \over k}}=\lim\limits_{k \rightarrow \infty}\left(\prod\limits_{j=1}^{k} n_j \right)^{1 \over k}
\!\!\!=\lim\limits_{k \rightarrow \infty}\left(\int_0^1 \left(1-{h_k(x)\over h(x)}\right)^2 dx \right)^{-{1 \over k}}\!\!\!\!\!.
\ee

\medskip 
\noindent
Now, for a known $h_k(x)$, \eqref{eqn:ODEh} is a first-order linear ODE in $h_{k+1}(x)$.
Its solution is given by $h_{k+1}(x)=\mathcal{H}h_k(x)$, where $\mathcal{H}$ is a 
nonlinear operator defined as follows
\be
\label{iter}
\mathcal{H}f(x)=\left[1-\int_0^x f^2(y) e^{-2\int\limits_0^y f(s) ds } dy \right] 
\cdot e^{2\int\limits_0^x f(s) ds }.
\ee
Hence, the problem of establishing the limit \eqref{eqn:Nk_h1}
for the root-Horton law concerns the asymptotic behavior of an iterated nonlinear functional.

\medskip 
\noindent
The following lemma will be used in Sect. \ref{sec:proofKingman}.
\begin{lem}\label{lem:half}
$$\big\|1-h_{k+1}/h\big\|_{L^2[0,1]}=\big\|h_{k+1}/h-h_k/h\big\|_{L^2[0,1]}$$
\end{lem}
\begin{proof} 
Observing $\, h'_{k+1}(x)+(h_{k+1}(x)-h_k(x))^2=h_{k+1}^2(x) \,$,
we use integration by parts to obtain
\begin{align*}
\int\limits_0^1 &{(h_{k+1}(x)-h_k(x))^2 \over h^2(x)}dx=\int\limits_0^1 {h_{k+1}^2(x) \over h^2(x)}dx-\int\limits_0^1{h'_{k+1}(x) \over h^2(x)}dx \\
&=\int\limits_0^1 {h_{k+1}^2(x) \over h^2(x)}dx+1-2\int\limits_0^1 {h_{k+1}(x) \over h(x)}dx
~=\int\limits_0^1 {(1-h_{k+1}(x))^2 \over h^2(x)}dx
\end{align*}
as $1/h(x)=1-x$.
\end{proof}

\medskip
\noindent
Next, we notice that \eqref{eqn:domin_g} implies  
\be\label{eqn:domin_h}
h(x) > h_{k+1}(x) > h_k(x) ~~\text{ for all }~x \in (0,1)
\ee 
for all $k \geq 1$.

\medskip
\noindent
Finally, an alternative proof to Proposition \ref{prop:properties_g}(d) using the system of ODEs \eqref{eqn:ODEh} follows from Lemma \ref{lem:half} and \eqref{eqn:domin_h}. 
\begin{proof} [Alternative proof of Proposition \ref{prop:properties_g}(d)]
Lemma \ref{lem:half} implies
\begin{align*}
\big\|1-h_k/h\big\|_{L^2[0,1]}^2 &=2\int\limits_0^1(1-h_k/h)(1-h_{k+1}/h) dx \\
&=2\big\|1-h_{k+1}/h\big\|_{L^2[0,1]}^2 
+2\int\limits_0^1(h_{k+1}/h-h_k/h)(1-h_{k+1}/h) dx. 
\end{align*}
Hence, equation \eqref{eqn:domin_h}
yields $~n_k= {\big\|1-h_{k}/h\big\|_{L^2[0,1]}^2 \over \big\|1-h_{k+1}/h\big\|_{L^2[0,1]}^2} \geq 2$. 
\end{proof}

\subsection{Proof of the existence of the root-Horton limit}\label{sec:proofKingman}

Here we present a proof of Thm.~\ref{Mthm}(ii). The proof  is based on Lemmas \ref{lem:h1} and \ref{lem:h1exist} stated below
that will be proven in the Sects.~\ref{sec:h1} and \ref{sec:h1exist}. 

\begin{lem} \label{lem:h1}
If the limit $\lim\limits_{k \rightarrow \infty}{h_{k+1}(1) \over h_k(1)}$ exists, then 
$\lim\limits_{k \rightarrow \infty}\left(\cN_k \right)^{-{1 \over k}}=\lim\limits_{k \rightarrow \infty}\left(\prod\limits_{j=1}^{k} n_j \right)^{1 \over k}$ also exists, and
$$\lim\limits_{k \rightarrow \infty}\left(\cN_k \right)^{-{1 \over k}}
=\lim\limits_{k \rightarrow +\infty}\left({1 \over h_k(1)}\right)^{-{1 \over k}}
=\lim\limits_{k \rightarrow \infty}{h_{k+1}(1) \over h_{k}(1)}.$$
\end{lem}

\begin{lem} \label{lem:h1exist}
The limit $\lim\limits_{k \rightarrow \infty}{h_{k+1}(1) \over h_k(1)} \geq 1$ exists, and is finite.
\end{lem}
 
\medskip
\noindent
Once Lemmas \ref{lem:h1} and \ref{lem:h1exist} are established, the validity of root-Horton law Theorem \ref{Mthm}(ii) is proved as follows.
\begin{proof}[Proof of Theorem \ref{Mthm}(ii)]
The existence and finiteness of $\lim\limits_{k \rightarrow \infty}{h_{k+1}(1) \over h_k(1)}$ 
established in Lemma \ref{lem:h1exist} is the precondition for Lemma \ref{lem:h1} that in turn implies 
the existence and finiteness of the limit 
$$\lim\limits_{k \rightarrow \infty}\left(\cN_k \right)^{-{1 \over k}}
=\lim\limits_{k \rightarrow \infty}\left(\prod\limits_{j=1}^{k} n_j \right)^{1 \over k}=R$$ 
as needed for the root-Horton law.
Furthermore, 
\be\label{eqn:Rhk}
R=\lim\limits_{k \rightarrow \infty} {h_{k+1}(1) \over h_{k}(1)},
\ee
and $2 \leq R \leq 4$ by Proposition \ref{prop:properties_g}.
\end{proof}

\subsubsection{Proof of Lemma \ref{lem:h1} and related results}\label{sec:h1}

\begin{prop}\label{prop:one}
\be\label{eqn:propone}
\big\|1-h_{k+1}(x)/h(x)\big\|_{L^2[0,1]}^2 \leq {1 \over h_{k+1}(1)} \leq \big\|1-h_k(x)/h(x)\big\|_{L^2[0,1]}^2.
\ee
\end{prop}
\begin{proof}
Equation \eqref{eqn:ODEh} implies
\be\label{eqn:hk2positive}
{h'_{k+1}(x) \over h_{k+1}^2(x)}=1-{(h_{k+1}(x)-h_k(x))^2 \over h_{k+1}^2(x)} \quad \forall x \in (0,1].
\ee
Integrating both sides of the equation \eqref{eqn:hk2positive} from $0$ to $1$ we obtain 
$${1 \over h_{k+1}(1)}=\int\limits_0^1 {(h_{k+1}(x)-h_k(x))^2 \over h_{k+1}^2(x)} dx=\big\|1-h_k(x)/h_{k+1}(x)\big\|_{L^2[0,1]}^2$$ 
as $h_{k+1}(0)=1$.

\vskip 0.2 in
\noindent
Hence, using Lemma \ref{lem:half}, the first inequality in \eqref{eqn:propone} is proved as follows 
\begin{align*}
{1 \over h_{k+1}(1)} =& \int\limits_0^1 {(h_{k+1}(x)-h_k(x))^2 \over h_{k+1}^2(x)} dx ~\geq \int\limits_0^1 {(h_{k+1}(x)-h_k(x))^2 \over h^2(x)} dx \\
&=\big\|1-h_{k+1}/h\big\|_{L^2[0,1]}^2 ~=\big\|1-h_{k+1}(x)/h(x)\big\|_{L^2[0,1]}^2 .
\end{align*}

\noindent
Finally, equations \eqref{eqn:domin_h} and \eqref{eqn:hk2positive} imply
$${1 \over h_{k+1}(1)}=\big\|1-h_k(x)/h_{k+1}(x)\big\|_{L^2[0,1]}^2 \leq \big\|1-h_k(x)/h(x)\big\|_{L^2[0,1]}^2.$$
This completes the proof.
\end{proof}

\medskip
\begin{proof}[Proof of Lemma \ref{lem:h1}]
If the limit $\lim\limits_{k \rightarrow \infty}{h_{k+1}(1) \over h_k(1)}$ exists and is finite, then so is the limit 
$\lim\limits_{k \rightarrow \infty}\left({1 \over h_k(1)}\right)^{-{1 \over k}}$. 
Then, the existence and the finiteness of the limit
$\lim\limits_{k \rightarrow \infty}\left(\cN_k \right)^{-{1 \over k}}$
follow from equation \eqref{eqn:Nk_h1} and Proposition \ref{prop:one}.
\end{proof}

\subsubsection{Proof of Lemma \ref{lem:h1exist} and related results}\label{sec:h1exist}
In this subsection we use the approach developed by Drmota \cite{MD2009} 
to prove the existence and the finiteness of 
$\lim\limits_{k \rightarrow \infty}{h_{k+1}(1) \over h_k(1)} \geq 1$. 
As we saw earlier, this result was used for proving existence, finiteness, and positivity  of  
$\lim\limits_{k \rightarrow \infty}\left(\cN_k \right)^{-{1 \over k}}=\lim\limits_{k \rightarrow \infty}\left(\prod\limits_{j=1}^{k} n_j \right)^{-{1 \over k}} $, the root-Horton law.

\begin{Def}
Given $\gamma \in  (0,1]$. Let
$$V_{k,\gamma}(x)=\begin{cases}
      {1 \over 1-x} & \text{ for } 0 \leq x \leq 1-\gamma, \\
      \gamma^{-1} h_k\left({x-(1-\gamma) \over \gamma}\right) & \text{ for } 1-\gamma \leq x \leq 1.
\end{cases} $$
\end{Def}
\noindent
Note that the sequences of functions $h_k(x)$ and $V_{k,\gamma}(x)$ can be extended beyond $x=1$.

\medskip
\noindent
Next, we make some observations about the above defined functions. 

\begin{obs}\label{obs:one}
$V_{k,\gamma}(x)$ are positive continuous functions satisfying
$$V'_{k+1,\gamma}(x)=2V_{k+1,\gamma}(x)V_{k,\gamma}(x)-V^2_{k,\gamma}(x)$$
 for all $x \in [0,1] \setminus (1-\gamma)$, with initial conditions $V_{k,\gamma}(0)=1$.
\end{obs}

\begin{obs} \label{obs:two}
Let $\gamma_k={h_k(1) \over h_{k+1}(1)}$.  Then
\begin{equation} \label{gammak1}
V_{k,\gamma_k}(1)=h_{k+1}(1)
\end{equation}
and
\begin{equation} \label{gammak2}
V_{k,\gamma}(1)=\gamma^{-1} h_k(1) \geq h_{k+1}(1) \quad \text{ whenever } \gamma \leq \gamma_k.
\end{equation}
\end{obs}

\begin{obs}\label{obs:three}
$$V_{k,\gamma}(x) \leq V_{k+1,\gamma}(x)$$
for all $x \in [0,1]$ since $h_k(x) \leq h_{k+1}(x)$.
\end{obs}

\begin{obs}\label{obs:four}
Since $h_1(x) \equiv 1$ and $\gamma_1={h_1(1) \over h_2(1)}$, 
$$h_2(x) \leq V_{1,\gamma_1}(x)=\begin{cases}
      {1 \over 1-x} & \text{ for } 0 \leq x \leq 1-\gamma_1, \\
      \gamma_1^{-1}=h_2(1) & \text{ for } 1-\gamma_1 \leq x \leq 1.
\end{cases} $$
\end{obs}

\medskip
\noindent
Observation \ref{obs:four} generalizes as follows.
\begin{prop}\label{drmota}
$$h_{k+1}(x) \leq V_{k,\gamma_k}(x)=\begin{cases}
      {1 \over 1-x} & {\rm~for~} 0 \leq x \leq 1-\gamma_k, \\
     \gamma_k^{-1} h_k\left({x-(1-\gamma_k) \over \gamma_k}\right) 
     & {\rm~ for ~} 1-\gamma_k \leq x \leq 1.
\end{cases}$$
\end{prop}
In order to prove Proposition \ref{drmota} we will need the following lemma.
\begin{lem} 
\label{positivezero}
For any $\gamma \in (0,1)$ and $k \geq 1$, function $V_{k,\gamma}(x)-h_{k+1}(x)$ changes its sign at most once as $x$ increases from $1- \gamma$ to $1$. Moreover, since $V_{k,\gamma}(1-\gamma)=h(1-\gamma) > h_{k+1}(1-\gamma)$,  function $V_{k,\gamma}(x)-h_{k+1}(x)$ can only change sign from nonnegative to negative.
\end{lem}
\begin{proof}
This is a proof by induction with base at $k=1$. Here $V_{1,\gamma}(x) ={1 \over \gamma}$ is constant on $[1-\gamma,1]$, while $h_2(x)=(1+e^{2x})/2$ is an increasing function, and 
$$V_{1,\gamma}(1-\gamma)=h(1-\gamma)>h_2(1-\gamma).$$

For the induction step, we need to show that if $V_{k,\gamma}(x)-h_{k+1}(x)$ changes its sign at most once, 
then so does $V_{k+1,\gamma}(x)-h_{k+2}(x)$. 
Since both sequences of functions satisfy the same ODE relation (see Observation \ref{obs:one}), we have

\noindent
${d \over dx}\left[(V_{k+1,\gamma}(x)-h_{k+2}(x))\cdot e^{-2\int\limits_{1-\gamma}^x h_{k+1}(y)dy} \right]$
$$\qquad \qquad =(2V_{k+1,\gamma}(x)-V_{k,\gamma}(x)-h_{k+1}(x))\cdot (V_{k,\gamma}(x)-h_{k+1}(x))\cdot e^{-2\int\limits_{1-\gamma}^x h_{k+1}(y)dy},$$
where $h_{k+1}(x) \leq V_{k+1,\gamma}(x)$ by definition of $V_{k+1,\gamma}(x)$, and $V_{k,\gamma}(x) \leq V_{k+1,\gamma}(x)$ as in Observation \ref{obs:three}. 
\vskip 0.2 in
\noindent
Now, let
$$I(x):=\int\limits_{1-\gamma}^x (2V_{k+1,\gamma}(s)-V_{k,\gamma}(s)-h_{k+1}(s))\cdot (V_{k,\gamma}(s)-h_{k+1}(s))\cdot e^{-2\int\limits_{1-\gamma}^s h_{k+1}(y)dy} ds.$$
Then 
$$~(V_{k+1,\gamma}(x)-h_{k+2}(x))\cdot e^{-2\int\limits_{1-\gamma}^x h_{k+1}(y)dy}=V_{k+1,\gamma}(1-\gamma)-h_{k+2}(1-\gamma)+I(x).$$
\vskip 0.2 in
\noindent
The function $2V_{k+1,\gamma}(x)-V_{k,\gamma}(x)-h_{k+1}(x) \geq 0$, and since 
$V_{k,\gamma}(x)-h_{k+1}(x)$ changes its sign at most once, then $I(x)$ should change its sign from nonnegative to negative at most once as $x$ increases from $1-\gamma$ to $1$. Hence
$$V_{k+1,\gamma}(x)-h_{k+2}(x)=(V_{k+1,\gamma}(1-\gamma)-h_{k+2}(1-\gamma)+I(x)) \cdot e^{2\int\limits_{1-\gamma}^x h_{k+1}(y)dy}$$
should change its sign from nonnegative to negative at most once as 
$$V_{k+1,\gamma}(1-\gamma)=h(1-\gamma)>h_{k+2}(1-\gamma)$$
by \eqref{eqn:domin_h}.
\end{proof}

\begin{proof}[Proof of Proposition \ref{drmota}]
Take $\gamma=\gamma_k$ in Lemma \ref{positivezero}. Then function $h_{k+1}(x)-V_{k,\gamma_k}(x)$ 
should change its sign from nonnegative to negative at most once within the interval $[1-\gamma_k,1]$. 
Hence, $V_{k,\gamma_k}(1-\gamma_k) > h_{k+1}(1-\gamma_k)$ and $h_{k+1}(1) = V_{k,\gamma_k}(1)$ imply 
$h_{k+1}(x) \leq V_{k,\gamma_k}(x)$ as in the statement of the proposition.
\end{proof}


Now we are ready to prove the monotonicity result.
\begin{lem} \label{gamma}
$$\gamma_k \leq \gamma_{k+1} \qquad \text{ for all } k \in \mathbb{N}^+.$$
\end{lem}
\begin{proof} We prove it by contradiction. 
Suppose $\gamma_k \geq \gamma_{k+1}$ for some $k \in \mathbb{N}^+$. 
Then
$$V_{k,\gamma_k}(x) \leq V_{k,\gamma_{k+1}}(x)=\begin{cases}
      {1 \over 1-x} & \text{ for } 0 \leq x \leq 1-\gamma_{k+1}, \\
     \gamma_{k+1}^{-1} h_k\left({x-(1-\gamma_{k+1}) \over \gamma_{k+1}}\right) 
     & \text{ for } 1-\gamma_{k+1} \leq x \leq 1
\end{cases}$$
and therefore
$$h_{k+1}(x) \leq V_{k,\gamma_k}(x) \leq V_{k,\gamma_{k+1}}(x) \leq V_{k+1,\gamma_{k+1}}(x)$$
as $h_{k+1}(x) \leq V_{k,\gamma_k}(x)$ by Proposition \ref{drmota}.

Recall that for $x \in [1-\gamma_{k+1},1]$, 
$$V'_{k+1,\gamma_{k+1}}(x)=2V_{k,\gamma_{k+1}}(x)V_{k+1,\gamma_{k+1}}(x)-V_{k,\gamma_{k+1}}^2,$$ 
where at $1-\gamma_{k+1}$ we consider only the right-hand derivative.
Thus for $x \in [1-\gamma_{k+1},1]$,
$${d \over dx}\Big(V_{k+1,\gamma_{k+1}}(x)-h_{k+2}(x)\Big)=A(x)+B(x)\Big(V_{k+1,\gamma_{k+1}}(x)-h_{k+2}(x)\Big),$$
where $A(x)=2V_{k+1,\gamma_{k+1}}(x)-V_{k,\gamma_{k+1}}(x)-h_{k+1}(x) \geq 0$, $B(x)=2h_{k+1}(x) >0$, and $V_{k+1,\gamma_{k+1}}(1-\gamma_{k+1})-h_{k+2}(1-\gamma_{k+1})=h(1-\gamma_{k+1})-h_{k+2}(1-\gamma_{k+1})>0$.
Hence $$V_{k+1,\gamma_{k+1}}(1) - h_{k+2}(1) \geq V_{k+1,\gamma_{k+1}}(1-\gamma_{k+1})-h_{k+2}(1-\gamma_{k+1})>0$$
arriving to a contradiction since $V_{k+1,\gamma_{k+1}}(1) = h_{k+2}(1)$.
\end{proof}

\vskip 0.2 in
\noindent
\begin{cor}\label{cor:gamma}
Limit $\lim\limits_{k \rightarrow \infty} \gamma_k$ exists.
\end{cor}
\begin{proof}
Lemma \ref{gamma} implies  $\gamma_k$ is a monotone increasing sequence, bounded by $1$.
\end{proof}

\vskip 0.1 in
\noindent
\begin{proof}[Proof of Lemma \ref{lem:h1exist}]
Lemma \ref{lem:h1exist} follows immediately from Corollary \ref{cor:gamma} and an observation 
that ${h_{k+1}(1) \over h_k(1)}={1 \over \gamma_k}$.
\end{proof}

\section{Generalized dynamical pruning}\label{sec:pruning}

The Horton pruning (Def.~\ref{def:Horton_prune}), which is the key element of the
self-similarity theory developed in previous sections, is a very particular way
of erasing a tree.
Here we suggest a general approach to erasing a finite tree from leaves down 
to the root that include both combinatorial and metric prunings, 
and discuss the respective prune-invariance. 

Given a tree $T \in \cL$ and a point $x \in T$, let $\Delta_{x,T}$ be the {\it descendant tree} 
of $x$: it is comprised of all points of $T$ descendant to $x$, including $x$; see Fig.~\ref{fig:isometry}a. 
Then $\Delta_{x,T}$ is itself a tree in $\cL$ with root at $x$. 
Let $T_1=(M_1,d_1)$ and $T_2=(M_2,d_2)$ be two metric rooted trees (Def.~\ref{def:treeL}), 
and let $\rho_1$ denote the root of $T_1$. 
A function $f: T_1 \rightarrow T_2$ is said to be an {\it isometry} if 
${\sf Image}[f] \subseteq \Delta_{f(\rho_1),T_2}$ and for all pairs $x,y \in T_1$,
$$d_2\big(f(x),f(y)\big)=d_1(x,y).$$
The tree isometry is illustrated in Fig.~\ref{fig:isometry}b. 
We use the isometry to define a {\it partial order} in the space $\cL$ as follows.  
We say that $T_1$ is {\it less than or equal to} $T_2$ and write $T_1 
\preceq T_2$ if 
 there is an isometry $f: T_1 \rightarrow T_2$. 
The relation $\preceq$ is a partial order as it satisfies the 
reflexivity, antisymmetry, and transitivity conditions. 
Moreover, a variety of  other properties of this partial order can be 
observed, including order denseness and semi-continuity. 

\begin{figure}[t] 
\centering\includegraphics[width=0.9\textwidth]{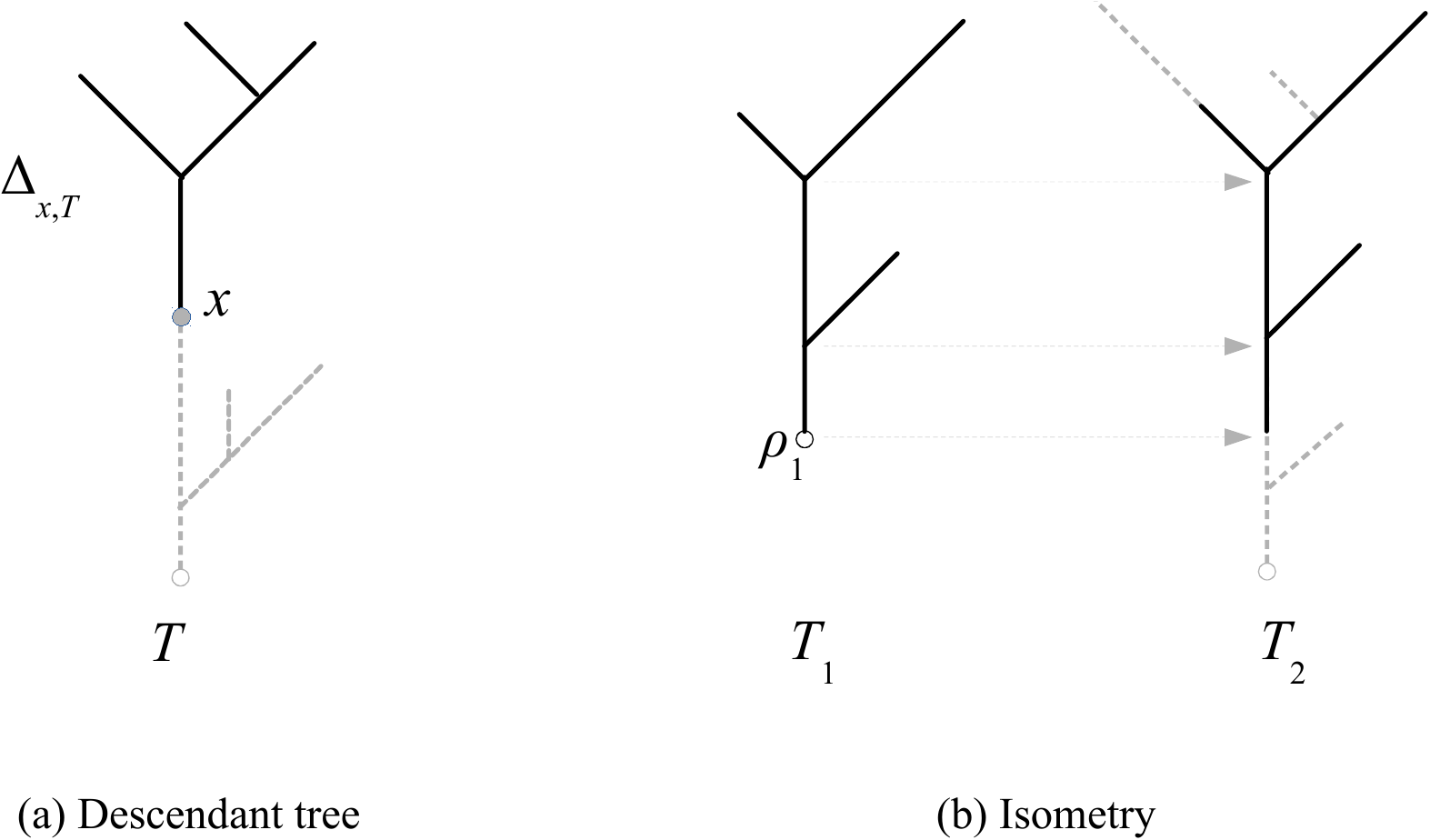}
\caption[Descendant subtree and isometry: an illustration.]
{\small Descendant subtree and tree isometry: an illustration.
(a) Subtree $\Delta_{x,T}$ (solid black lines) descendant to a point $x$ (gray circle)
in a tree $T$ (union of dashed gray and soling black lines).
(b) Isometry of trees. Tree $T_1$ (left) is mapped to tree $T_2$ (right).
The image of $T_1$ within $T_2$ is shown by black lines, the rest of $T_2$
is shown by dashed gray lines. 
Here, tree $T_1$ is less than tree $T_2$, $T_1 \preceq T_2.$}
\label{fig:isometry}
\end{figure}

We say that a function $\varphi:\cL \rightarrow \mathbb{R}$ is {\it monotone nondecreasing} 
with respect to the partial order $\preceq $ if
$\varphi(T_1) \leq \varphi(T_2)$ whenever $T_1 \preceq T_2.$
Consider a monotone nondecreasing function $\varphi:\cL \rightarrow \mathbb{R}_+$. 
We define the {\it generalized dynamical pruning} operator $\S(\varphi,T):\cL\rightarrow\cL$ induced by $\varphi$ 
for any $t\ge 0$ as \index{generalized dynamical pruning}
\be\label{eqn:GenDynamPruning}
\S(\varphi,T):=\rho\cup\Big\{x \in T\setminus\rho ~:~\varphi\big(\Delta_{x,T}\big)\geq t \Big\},
\ee
where $\rho$ denotes the root of tree $T$.
Informally, the operator $\S$ cuts all subtrees $\Delta_{x,T}$ for which the value of $\varphi$
is below threshold $t$, and always keeps the tree root.
Extending the partial order to $\cL$
by assuming $\phi \preceq T$ for all $T \in \cL$, we observe for any $T\in\cL$ that 
$S_s(T) \preceq S_t(T)$ whenever $s \geq t$.

\subsection{Examples of generalized dynamical pruning}\label{sec:GenPruningEx}
The dynamical pruning operator $\S$ encompasses and 
unifies a range of problems, depending on a choice of $\varphi$,
as we illustrate in the following examples.

\subsubsection{Example: pruning via the tree height}\label{ex:height}
Let the function $\varphi(T)$ equal the height of tree $T$:
\begin{equation}\label{phi_hight}
\varphi(T) = \textsc{height}(T).
\end{equation}
In this case the operator $\S$ satisfies the {\bf continuous semigroup property}: \index{continuous semigroup property}
$$\cS_t\circ\cS_s=\cS_{t+s} ~\text{ for any }~t,s\ge 0.$$ 
It coincides with the continuous pruning (a.k.a. tree erasure) studied 
by Jacques Neveu \cite{Neveu86}, who
established invariance of a critical and sub-critical binary
Galton-Watson trees with i.i.d. exponential edge lengths with respect to this operation.

It is readily seen that for a coalescent process (Sect.~\ref{sec:coalescent}), 
the dynamical pruning $\S$ of the corresponding coalescent tree with $\varphi(T)$ as in (\ref{phi_hight}) 
replicates the coalescent process.
More specifically, the timing and order of particle mergers is reproduced by the
dynamics of the leaves of $\S(\varphi,T)$.
See Sect.~\ref{sec:Bdyn}, Thm.~\ref{thm:pruning1} for a concrete version of this statement 
for the coalescent dynamics of shocks in the continuum ballistic annihilation model. 

\subsubsection{Example: pruning via the Horton-Strahler order}\label{ex:H}
Let the function $\varphi(T)$ be one unit less that the Horton-Strahler order 
${\sf ord}(T)$ of a tree $T$:
\begin{equation}\label{phi_Horton}
\varphi(T) = {\sf ord}(T)-1.
\end{equation}
This function is also known as the {\it register number} 
\cite{Ershov1958,FRV79}, as it equals the minimum number of memory registers 
necessary to evaluate an arithmetic expression described by a tree $T$, assuming
that the result is stored in an additional register that also can be used for calculations.

With the choice \eqref{phi_Horton}, the dynamical pruning operator coincides with the
Horton pruning (Def.~\ref{def:Horton_prune}):
$\S=\mathcal{R}^{\lfloor t \rfloor}$,
if we assume that all edge lengths equal to unity. 
It is readily seen that $\S$ satisfies the {\bf discrete semigroup property}: \index{discrete semigroup property}
$$\cS_t\circ\cS_s=\cS_{t+s} ~\text{ for any }~t,s\in \mathbb{N}.$$ 
Most of the present survey is focused on invariance of a tree distribution with respect to 
this operation.

\subsubsection{Example: pruning via the total tree length}\label{ex:L}
Let the function $\varphi(T)$ equal the total lengths of $T$:
\begin{equation}\label{phi_length}
\varphi(T) = \textsc{length}(T).
\end{equation}
The dynamical pruning by the tree length is illustrated in Fig.~\ref{fig:LP} for
a Y-shaped tree that consists of three edges.

Importantly, in this case $\S$ does not satisfy the semigroup property.
To see this, consider an internal vertex point $x \in T$ (see Fig.~\ref{fig:LP},
where the only internal vertex is marked by a gray ball). 
Then $\Delta_{x,T}$ consists of point $x$ as its root, 
the left subtree of length $a$ and the right subtree of length $b$. 
Observe that the whole left subtree is pruned away by time $a$, and the whole right 
subtree is pruned away by time $b$.
However, since $$\varphi(\Delta_{x,T}) = \textsc{length}(\Delta_{x,T})=a+b,$$
the junction point $x$ will not be pruned until time instant $a+b$. 
Thus, $x$ will be a leaf of  $\S(\varphi,T)$ for all $t$ such that
$$\max\{a,b\} \leq t \leq a+b.$$
This situation corresponds to Stage IV in Fig.~\ref{fig:LP}, where
each of the left and right subtrees stemming from point $x$ (marked by
a gray ball) consists of a single root vertex.

\begin{figure}[t] 
\centering\includegraphics[width=0.8\textwidth]{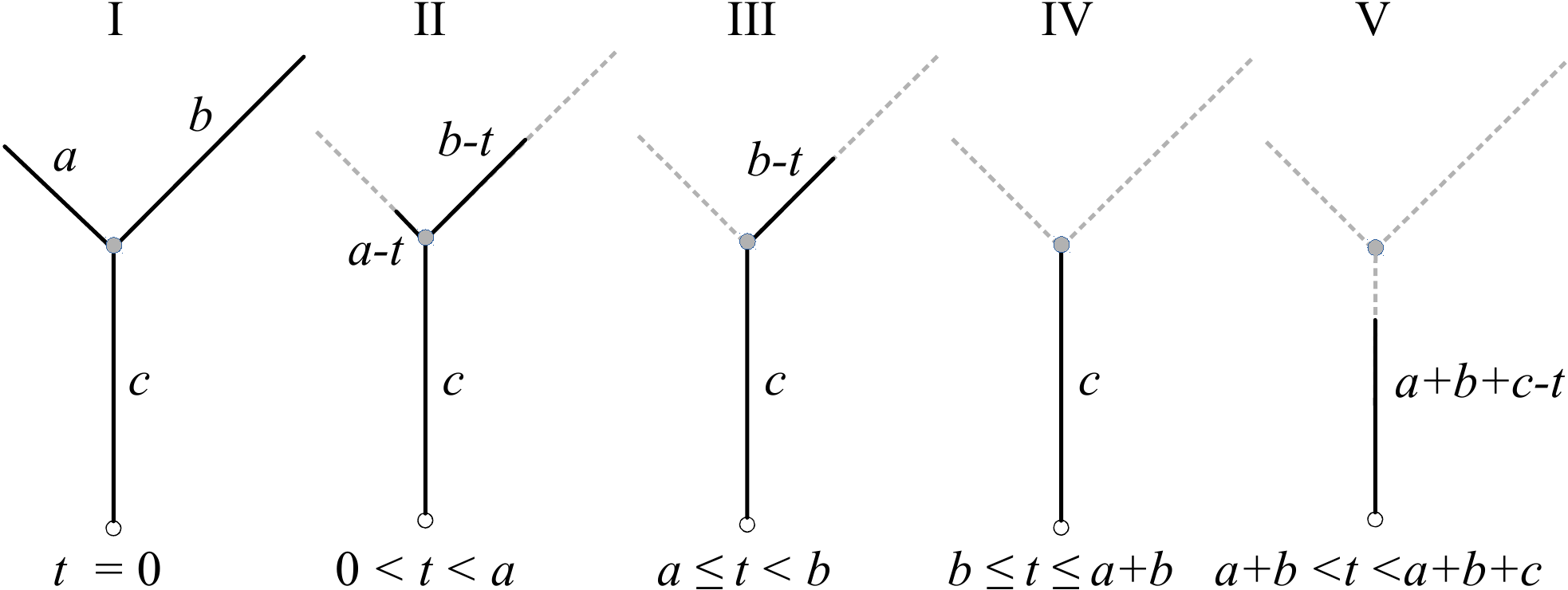}
\caption[Pruning by length: an illustration]
{\small Pruning by tree length: an illustration. 
Figure shows five generic stages in the dynamical pruning of a Y-shaped tree $T$,
with pruning function $\varphi(T) = \textsc{length}(T)$.
The pruned tree $\S$ is shown by solid black lines; the pruned parts of the initial 
tree are shown by dashed gray lines.\\
{\bf Stage I}: Initial tree $T$ consists of three edges, with lengths $a,b,c$
indicated in the panel; without loss of generality we assume $a<b$.\\
{\bf Stage II}: For any $t<a$ the pruned tree $\S$
has a Y-shaped form with leaf edges truncated by $t$. \\
{\bf Stage III}: For any $a\le t < b$ the pruned tree $\S$
consists of a single edge of length $c+b-t$.\\
{\bf Stage IV}: For any $b\le t \le a+b$ the pruned tree $\S$
consists of a single edge of length $c$. Notice that during this stage 
the tree $\S$ does not change with $t$; this loss of memory 
causes the process to violate the semigroup property.\\
{\bf Stage V}: For any $a+b<t<a+b+c$ the pruned tree $\S$
consists of a single edge of length $a+b+c-t$.}
\label{fig:LP}
\end{figure}

The semigroup property in this example can be introduced by considering 
{\bf mass-equipped trees}.
Informally, we replace each pruned subtree $\tau$ of $T$ with a point of mass equal
to the total length of $\tau$.
The massive points contain
some of the information lost during the pruning process, which is enough to 
establish the semigroup property.
Specifically, by time $a$, the pruned away left subtree (Fig.~\ref{fig:LP}, Stage III) 
turns into a massive point of mass $a$ 
attached to $x$ on the left side. 
Similarly, by time $b$, the pruned away right subtree (Fig.~\ref{fig:LP}, Stage IV) 
turns into a massive point of mass $b$ 
attached to $x$ on the right side. 
For $\max\{a,b\} \leq t \leq a+b$, this construction keeps truck of the quantity $a+b-t$ 
associated with point $x$, and when the quantity $a+b-t$ decreases to $0$, the two massive points coalesce into one. 
If at instant $t$ a single massive point seats at a leaf, its mass $m = t$, and the leaf's parental edge is being pruned.  
If at instant $t$ two massive points (left and right) seat at a leaf, 
they total mass $m\ge t$, and further pruning of the leaf's parental edge 
is prevented until the instant $t = m$, when the two massive points coalesce.
Keeping track of all such quantities makes $\S$ satisfy the continuous semigroup property.
This construction is formally introduced in Sect.~\ref{sec:annihilation}, which shows 
that the pruning operator $\S$ with \eqref{phi_length} coincides with 
the potential dynamics of continuum mechanics formulation of the 1-D ballistic annihilation model,
$A+A \rightarrow \zeroslash$.

\subsubsection{Example: pruning via the number of leaves}\label{ex:numL}
Let the function $\varphi(T)$ equal the number of leaves in a tree $T$.
This choice is closely related to the mass-conditioned dynamics of
an aggregation process. 
Specifically, consider $N$ singletons (particles with unit mass)
that appear in a system at instants $t_n\ge 0$, $1\le n\le N$. 
The existing clusters merge into consecutively larger clusters by 
pair-wise mergers. 
The cluster mass is additive: a merger of two clusters of masses
$i$ and $j$ results in a cluster of mass $i+j$.
We consider a time-oriented tree $T$ that describes this process. 
The tree $T$ has $N$ leaves and $(N-1)$ internal vertices.
Each leaf corresponds to an initial particle, each internal 
vertex corresponds to a merger of two clusters, and the edge lengths represent
times between the respective mergers. 
The action of $\S$ on such a tree coincides with a conditional
state of the process that only considers clusters of mass $\ge t$. 
A well-studied special case is a coalescent process with a kernel $K(i,j)$
of Sect.~\ref{sec:coalescent}.

\subsection{Pruning for $\mathbb{R}$-trees}
The generalized dynamical pruning
is readily applied to {\it real trees} (Sect.~\ref{sec:Rsetup}), although this is not
the focus of our work. 
We notice that the total tree length (Example~\ref{ex:L}) and number of leaves 
(Example~\ref{ex:numL}) might be undefined (infinite) for an $\mathbb{R}$-tree. 
We introduce in Sect.~\ref{sec:Rprune} a {\it mass} function that can serve as 
a natural general analog of these and other functions on finite trees. 
We show (Sect.~\ref{sec:Bdyn}, Thm.~\ref{thm:pruning}) that pruning by mass is equivalent to the pruning
by the total tree lengths in a particular situation of ballistic
annihilation model with piece-wise continuous potential with a finite number of segments. 
Accordingly, our results should be straightforwardly
extended to $\mathbb{R}$-trees that appear, for instance, as a description of
the continuum ballistic annihilation dynamics 
for other initial potentials.

\subsection{Relation to other generalizations of pruning}
A pruning operation similar in spirit to the generalized dynamical pruning  
was considered in a work by Duquesne and Winkel \cite{Winkel2012} that extended a formalism by Evans \cite{Evans2005} and Evans et al.~\cite{EPW06}. 
We notice that the two definitions of pruning, the generalized dynamical pruning of Sect.~\ref{sec:pruning}  
and that in \cite{Winkel2012}, are principally different, despite their similar appearance.  
In essence, the work \cite{Winkel2012} assumes the Borel measurability with respect to the Gromov-Hausdorff metric (\cite{Winkel2012}, Section 2), which implies the semigroup property of the respective pruning (\cite{Winkel2012}, Lemma 3.11).
On the contrary, the generalized dynamical pruning defined here may have the semigroup property only 
under very particular choices of $\varphi(T)$ as in  the examples in Sect. \ref{ex:height} and \ref{ex:H}. 
The majority of natural choices of $\varphi(T)$, including the tree length $\varphi(T) = \textsc{length}(T)$
(Sect. \ref{ex:L}) or the number of leaves in a tree (Sect. \ref{ex:numL}), do not satisfy the semigroup property, and hence
are not covered by the pruning of \cite{Winkel2012}.
The main results of our Sect.~\ref{sec:annihilation} refer to the pruning function 
$\varphi(T)=\textsc{length}(T)$ 
that does not satisfy the semigroup property, as shown in Sect. \ref{ex:L}.

Curiously, for the above two examples with no semigroup property, i.e., when $\varphi(T) = \textsc{length}(T)$ and when $\varphi(T)$ equals the number of leaves in $T$, the following discontinuity property holds with respect to the Gromov-Hausdorff metric  $d_{\sf GH}$ defined in \cite{Evans2005,EPW06,Winkel2012}. For any $\epsilon>0$ and any $M>0$, there exist trees $T$ and $T'$ in $\cL$ such that 
$$|\varphi(T) -\varphi(T')|>M ~~\text{ while }~~ d_{\sf GH}(T,T') <\epsilon .$$
Indeed, if $\varphi(T) = \textsc{length}(T)$, we consider a tree $T$ with the number of leaves exceeding $M/\epsilon$, and let $T'$ be the tree obtained from $T$ by elongating each of its leaves by $\epsilon$. Similarly, if  $\varphi(T)$ is the number of leaves in $T$, we construct $T'$ from $T$ by attaching at least $M/\epsilon$ new leaves, each of length $\epsilon$.

\subsection{Invariance with respect to the generalized dynamical pruning}
\label{sec:pi}
Consider a tree $T\in\L$ with edge lengths given
by a vector $l_T=(l_1,\dots,l_{\#T})$.
The vector $l_T$ can be specified by
distribution $\chi(\cdot)$ of a point $x_T=(x_1,\dots,x_{\#T})$ on the standard simplex
\[\Delta^{\#T}=\left\{x_i:\sum_i^{\#T} x_i = 1, 0<x_i\le 1\right\},\] 
and conditional distribution $F(\cdot|x_T)$ of the tree length
$\textsc{length}(T)$, so that
\[l_T = x_T\cdot\textsc{length}(T).\]
Accordingly, a tree $T$ can be completely specified by its planar shape,
a vector of proportional edge lengths, and the total tree length:
\[T=\left\{\textsc{p-shape}(T),x_T,\textsc{length}(T)\right\}.\]
A measure $\eta$ on $\L$ is a joint distribution of 
these three components:
\[\eta(T\in\{\tau,d\bar x,d\ell\}) = 
\mu(\tau)\cdot \chi_{\tau}(d\bar x)\cdot  F_{\tau,\bar x}(d\ell),\]
where the tree planar shape is specified by
\[\mu(\tau)={\sf Law }\left(\textsc{p-shape}(T)=\tau\right),\quad \tau\in\cT_{\rm plane},\]
the relative edge lengths is specified by 
\[\chi_\tau(\bar x) 
={\sf Law }\left(x_T=\bar{x} \,|\,\textsc{p-shape}(T)=\tau \right),\quad \bar x\in \Delta^{\#T},\]
and the total tree length is specified by
\[ F_{\tau, \bar x} (\ell) ={\sf Law }\left(\textsc{length}(T)=\ell \,|\,x_T=\bar{x}, ~ 
\textsc{p-shape}(T)=\tau \right),\quad \ell\ge0.\]
Let us fix $t\ge 0$ and a function $\varphi:\L \rightarrow \mathbb{R}$ that is monotone 
nondecreasing with respect to the partial order $\preceq$.
We denote by $\S^{-1}(\varphi,T)$ the preimage of a tree $T\in\L$ under the generalized dynamical
pruning:
\[\S^{-1}(\varphi,T)=\{\tau\in\L:\S(\varphi,\tau)=T\}.\]

\noindent Consider the distribution of edge lengths induced by the pruning:
$$\Xi_\tau(\bar x) 
={\sf Law }\left(x_{\tilde T}=\bar{x}  \,|\,\textsc{p-shape}\big(\tilde T\big)=\tau \right)$$
and
$$\Phi_{\tau, \bar x} (\ell)
={\sf Law }\left(\textsc{length}\big(\tilde T\big)=\ell \,|\,x_{\tilde T}=\bar{x}, ~\textsc{p-shape}\big(\tilde T\big)=\tau \right),$$
where the notation $\tilde T:=\S(\varphi,T)$ is used for brevity. 

\begin{Def}[{{\bf Generalized prune invariance}}]\label{def:pi2} \index{generalized prune invariance}
Consider a function $\varphi:\L \rightarrow \mathbb{R}_+$ that is 
monotone nondecreasing with respect to the partial order $\preceq$.
A measure $\eta$ on $\L$ is called {\it invariant with respect to the generalized dynamical pruning 
$\S(\cdot)=\S(\varphi,\cdot)$} 
(or simply prune invariant) if the following conditions hold for all $t\ge 0$:
\begin{itemize}
\item[(i)]
The measure is prune-invariant in shapes.
This means that for the pushforward measure 
$\nu=(\S)_*(\mu)=\mu\circ\S^{-1}$ we have
\[\mu(\tau)=\nu(\tau|\tau\ne\phi).\]
\item[(ii)]
The measure is prune-invariant in edge lengths.
This means that for any
combinatorial planar tree $\tau\in\T$
\[\Xi_\tau(\bar x)=\chi_\tau(\bar x)\]
and there exists a {\it scaling exponent} $\zeta\equiv \zeta(\varphi,t)>0$ such that for any 
 relative edge length vector
$\bar x\in\Delta^{\#\tau}$ we have 
\[\Phi_{\tau, \bar x} (\ell)=\zeta^{-1}F_{\tau, \bar x} \left(\frac{\ell}{\zeta}\right).\]
\end{itemize}
\end{Def}

\begin{Rem}[{{\bf Pruning trees with no embedding}}]
The generalized dynamical pruning \eqref{eqn:GenDynamPruning} and the notion of
prune invariance (Def.~\ref{def:pi2}) can be similarly defined   
on the space $\cL$ of metric trees with no planar embedding.
In this work we only apply the concept of prune invariance to planar trees. 
\end{Rem}

\begin{Rem}[{{\bf Relation to Horton prune-invariance}}]
Definition~\ref{def:pi2} is similar to Def.~\ref{def:distss} of prune invariance
with respect to the Horton pruning, with combinatorial Horton pruning $\cR$ being 
replaced with metric generalized dynamical pruning $\S$.
\end{Rem}

The prune invariance of Def.~\ref{def:pi2} unifies multiple invariance properties examined in the literature.
For example, the classical work by Jacques Neveu~\cite{Neveu86} establishes the prune invariance 
of the exponential critical binary Galton-Watson trees ${\sf GW}(\lambda)$ with respect to
the tree erasure from the leaves down to the root at a unit rate, which is equivalent to the
generalized dynamical pruning with function $\varphi(T)=\textsc{height}(T)$ 
(Sect.~\ref{ex:height}).
The prune invariance with respect to the Horton pruning (Sect.~\ref{ex:H}) has been established by Burd et al.~\cite{BWW00}
for the combinatorial critical  binary Galton-Watson $\mathcal{GW}\left({1 \over 2},{1 \over 2}\right)$ trees (Thm. \ref{thm:BWW00} in Sect. \ref{sec:BWW}).
Duquesne and Winkel \cite{Winkel2012} established the prune-invariance of 
the exponential critical binary Galton-Watson ${\sf GW}(\lambda)$ trees with
respect to the so-called {\it hereditary property}, which includes the
tree erasure of Sect.~\ref{ex:height} and Horton pruning of Sect.~\ref{ex:H}.
The {\it critical Tokunaga trees} analyzed in Sect. \ref{sec:Tok}  are prune-invariant with 
respect to the Horton pruning; this model includes ${\sf GW}(\lambda)$ trees as a special case.
Section~\ref{sec:PI} below establishes the prune invariance of the exponential critical binary Galton-Watson ${\sf GW}(\lambda)$ trees 
with respect to the generalized pruning with an arbitrary pruning function $\varphi(T)$.

\subsection{Prune invariance of ${\sf GW}(\lambda)$}
\label{sec:PI}
This section establishes prune invariance of exponential critical binary Galton-Watson trees 
with respect to arbitrary generalized pruning.

\begin{thm}[{{\bf \cite{KZ19}}}] \label{thm:main}
Let $T\stackrel{d}{\sim}{\sf GW}(\lambda)$, $T\in\BL^|$, be an exponential critical binary Galton-Watson tree 
with parameter $\lambda>0$.
Then, for any monotone nondecreasing function $\varphi:\BL^|\rightarrow\mathbb{R}_+$ and
any $\Delta>0$ we have
\[T^\Delta:=\{\cS_\Delta(\varphi,T)|\cS_\Delta(\varphi,T) \not= \phi\} \stackrel{d}{\sim} 
{\sf GW}(\lambda p_{\Delta}(\lambda,\varphi)),\]
where $p_{\Delta}(\lambda,\varphi)={\sf P}(\cS_\Delta(\varphi,T) \not= \phi)$. 
That is, the pruned tree $T^\Delta$ conditioned on surviving is an exponential 
critical binary Galton-Watson tree with parameter
$$\mathcal{E}_\Delta(\lambda,\varphi)=\lambda p_{\Delta}(\lambda,\varphi).$$
\end{thm}
\begin{proof}
Let $X$ denote the length of the stem (edge adjacent to the root) in $T$, 
and $Y$ denote the length of the stem in $T^\Delta$. 
Let $x$ be the nearest descendent vertex (a junction or a leaf) to the root in $T$. 
Then $X$, which is an exponential random variable with parameter $\lambda$, represents the distance from the root of $T$ to $x$. 
Let ${\sf deg}_T(x)$ denote the degree of $x$ in tree $T$ and ${\sf deg}_{T^\Delta}(x)$ denote the degree of $x$ in tree $T^\Delta$.
If $T^\Delta=\phi$, then $Y=0$. 
Let $$F(h)={\sf P}(Y \leq h~|~ \cS_\Delta(\varphi,T) \not= \phi).$$  
The event $\{Y \leq h\}$ is partitioned into the following non-overlapping sub-events S$_1,\dots$ S$_4$ illustrated in Fig.~\ref{fig:thm2}:

\begin{figure}[hbt]
	\centering
	\includegraphics[width=0.9\textwidth]{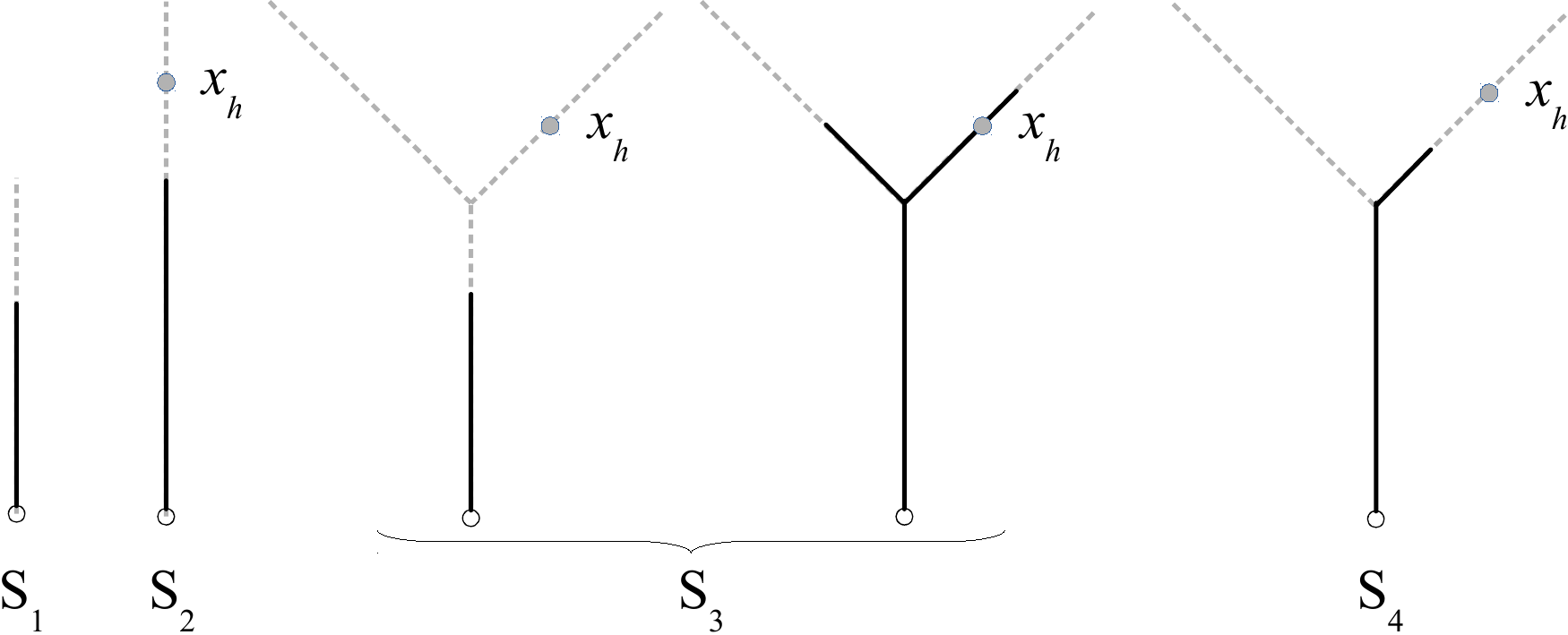}
	\caption{Sub-events used in the proof of Thm.~\ref{thm:main}.
	Gray dashed line shows (a part of) initial tree $T$.
	Solid black line shows (a part of) pruned tree $T^{\Delta}$. 
	We denote by $x_h$ a point in $T$ located at distance $h$ from the root, if it exists.}
	\label{fig:thm2}
\end{figure}
\begin{itemize}
  \item[(S$_1$)] The event 
  $\{{\sf deg}_T(x)=1\text{ and }X \leq h\}$ has probability
  $${1 \over 2}(1-e^{-\lambda h}).$$
  
  \item[(S$_2$)] The event 
  $$\{X>h \text{ and all points of }T\text{ descendant to }x_h\text{ do not belong to }T^\Delta\}$$
  has probability
  $$e^{-\lambda h}(1-p_\Delta).$$
  
  \item[(S$_3$)] The event $\{X \leq h$ and ${\sf deg}_T(x)=3$ and either both subtrees of $T$ descending from $x$ are pruned away
  completely (not intersecting $T^\Delta$) or $\{x \in T^\Delta,~{\sf deg}_{T^\Delta}(x)=3\}\}$
   has probability
  $${1 \over 2}(1-e^{-\lambda h})\big((1-p_\Delta)^2+p_\Delta^2 \big).$$
  
  \item[(S$_4$)] The event 
  $$\{X \leq h, {\sf deg}_T(x)=3\} \cap \{x \in T^\Delta, ~{\sf deg}_{T^\Delta}(x)=2\}  \cap \{Y \leq h\}$$  
  has probability\footnote{Here, ${\sf deg}_{T^\Delta}(x)=2$ means $x$ is neither a junction nor a leaf in $T^\Delta$.}
  $${1 \over 2}\int\limits_0^h \lambda e^{-\lambda t} \cdot 2p_\Delta(1-p_\Delta)\cdot F(h-t)\, dt=p_\Delta(1-p_\Delta) \int\limits_0^\infty \lambda e^{-\lambda t} F(h-t)\, dt.$$
\end{itemize}
Using this we have two representations for the probability ${\sf P}(Y\le h)$:
\begin{align*}
{\sf P}(Y\le h) =& (1-p_\Delta)+p_\Delta F(h)  \\
=& {1 \over 2}(1-e^{-\lambda h})+e^{-\lambda h}(1-p_\Delta)\\
&+{1 \over 2}(1-e^{-\lambda h})\big((1-p_\Delta)^2+p_\Delta^2 \big)\\
&+p_\Delta(1-p_\Delta) \int\limits_0^\infty \lambda e^{-\lambda t} F(h-t)\, dt,
\end{align*}
which simplifies to
$$(1-p_\Delta)+p_\Delta F(h) =(1-p_\Delta+p^2_\Delta )-e^{-\lambda h} p_\Delta +p_\Delta(1-p_\Delta) \int\limits_0^\infty \lambda e^{-\lambda t} F(h-t)\, dt.$$
Differentiating the above equality we obtain the following equation for the p.d.f. $f(y)={d\over dy}F(y)$ of $Y$:  
$$f(h)=p_\Delta\, \phi_\lambda (h) +(1-p_\Delta)\, \phi_\lambda \ast f(h),$$
where as before $\phi_\lambda$ denotes the exponential density with parameter $\lambda$ as in \eqref{exp}.
Applying integral transformation on both sides of the equation, we obtain the characteristic function $\widehat{f}(s)=\E\big[e^{isY}\big]$ of $Y$,
$$\widehat{f}(s)={\lambda p_\Delta \over \lambda p_\Delta -is}=\widehat{\phi}_{\lambda p_\Delta}(s).$$
Thus, we conclude that $Y$ is an exponential random variable with parameter $\lambda p_\Delta$.

\bigskip
Next, let $y$ be the  descendent vertex (a junction or a leaf) to the root in $T^\Delta$. 
If $T^\Delta = \phi$, let $y$ denote the root.
Let $$q={\sf P}({\sf deg}_{T^\Delta}(y)=3 ~|~ S_\Delta(T) \not= \phi).$$
Then,
\begin{align*}
p_\Delta q=& {\sf P}({\sf deg}_{T^\Delta}(y)=3)\\
=& {\sf P}({\sf deg}_T(x)=3) \cdot\Big\{{\sf P}\big({\sf deg}_{T^\Delta}(x)=3~|~{\sf deg}_T(x)=3\big)\\
&+{\sf P}\big({\sf deg}_{T^\Delta}(x)=2~|~{\sf deg}_T(x)=3\big)\cdot q \Big\}\\
=& {1 \over 2}\Big\{p^2_\Delta + 2p_\Delta(1-p_\Delta) q  \Big\}
\end{align*}
implying
$$q={1 \over 2} p_\Delta+(1-p_\Delta) q,$$
which in turn yields $q={1 \over 2}$.

\bigskip
We saw that conditioning on $\cS_\Delta(\varphi,T) \not= \phi$, the pruned tree $T^\Delta$ has the 
stem length distributed exponentially 
with parameter $\lambda p_\Delta$. 
Then, with probability $q={1 \over 2}$, the pruned tree $T^\Delta$ branches at $y$ (the stem end point farthest from the root) into two independent subtrees, each distributed as
$\{T^\Delta ~|~T^\Delta\ne \phi\}$. 
Thus, we recursively obtain that $T^\Delta$ is a critical binary Galton-Watson tree with i.i.d. exponential edge length with parameter $\lambda p_\Delta$.
\end{proof}

Next, we find an exact form of the survival probability $p_\Delta(\lambda,\varphi)$ for three particular choices of $\varphi$, thus obtaining
$\mathcal{E}_\Delta(\lambda,\varphi)$.

\begin{thm}[{{\bf \cite{KZ19}}}]\label{pdelta}
In the settings of Theorem \ref{thm:main}, we have 
\begin{description}
  \item[(a)] If $\varphi(T)$ equals the total length of $T$ $(\varphi = \textsc{length}(T))$, 
  then
  $$\mathcal{E}_\Delta(\lambda,\varphi)=\lambda e^{-\lambda \Delta}\Big[ I_0(\lambda \Delta)+ I_1(\lambda \Delta) \Big].$$
  
  \item[(b)] If $\varphi(T)$ equals the height of $T$ $(\varphi = \textsc{height}(T))$, then
  $$\mathcal{E}_\Delta(\lambda,\varphi)={2\lambda  \over \lambda \Delta +2}.$$
  
  \item[(c)] If $\varphi(T)+1$ equals the Horton-Strahler order of the tree $T$, then 
  $$\mathcal{E}_\Delta(\lambda,\varphi)=\lambda 2^{-\lfloor \Delta \rfloor},$$
  where $\lfloor \Delta \rfloor$ denotes the maximal integer  $\le \Delta$.
\end{description}
\end{thm}

\begin{proof}
{\bf Part (a).} Suppose $T\stackrel{d}{\sim}{\sf GW}(\lambda)$, and let $\ell(x)$ once again denote the p.d.f. of the total length $\textsc{length}(T)$. 
Then, by Lemma \ref{lem:ell}, 
\begin{align}\label{eq:pDelta1}
p_\Delta =& 1-\int\limits_0^{\Delta} \ell(x) \,dx =1-\int\limits_0^{\lambda \Delta}{1 \over x}  e^{-x}  I_1 \big(x\big) \,dx \nonumber \\
&= e^{-\lambda \Delta}\Big[ I_0(\lambda \Delta)+ I_1(\lambda \Delta) \Big],
\end{align}
where for the last equality we used formula 11.3.14 in \cite{AS1964}.

\bigskip
\noindent
{\bf Part (b).} 
Suppose $T\stackrel{d}{\sim}{\sf GW}(\lambda)$. Let ${\sf H}(x)$ once again denote the cumulative distribution function of the height $\textsc{height}(T)$. Then by Lemma \ref{lem:Hx}, for any $\Delta>0$,
$$p_\Delta=1-{\sf H}(\Delta)={2 \over \lambda \Delta +2}.$$

\bigskip
\noindent
{\bf Part (c).} 
Follows from Corollary \ref{cor:GW}(a). 
\end{proof}

\begin{Rem}
Let ${\mathcal E}_\Delta (\lambda,\varphi)={2\lambda \over \lambda \Delta +2}$ as in Theorem \ref{pdelta}(b).  
Here $~{\mathcal E}_0\lambda=\lambda$
and ${\mathcal E}_\Delta (\lambda,\varphi)$ is a linear-fractional transformation 
associated with matrix
$${\mathcal A}_\Delta=\begin{pmatrix}
     1 & 0   \\
     {\Delta \over 2}  &  1
\end{pmatrix}.$$
Since ${\mathcal A}_\Delta$ form a subgroup in $SL_2(\mathbb{R})$, the transformations $\left\{{\mathcal E}_\Delta \right\}_{\Delta \geq 0}$ satisfy the semigroup property
$${\mathcal E}_{\Delta_1} {\mathcal E}_{\Delta_2} ={\mathcal E}_{\Delta_1+\Delta_2}$$
for any pair $\Delta_1, \Delta_2 \geq 0$.

We notice also that the operator ${\mathcal E}_\Delta (\lambda,\varphi)$ in part (c) of Theorem~\ref{pdelta} satisfies only the discrete semigroup property for nonnegative integer times. 
Finally, one can check that ${\mathcal E}_\Delta (\lambda,\varphi)$ in part (a) does not satisfy the semigroup property.
\end{Rem}

\section{Continuum 1-D ballistic annihilation}\label{sec:annihilation}
As an illuminating application of the generalized dynamical pruning (Sect. \ref{sec:pruning}) and its invariance properties (Sect. \ref{sec:pi}), we consider the dynamics of particles governed by $1$-D ballistic annihilation model, traditionally denoted $A+A \rightarrow \zeroslash$ \cite{EF85}.
This model describes the dynamics of particles on a real line: a particle with Lagrangian coordinate $x$ moves with a constant velocity $v(x)$
until it collides with another particle, at which moment both particles annihilate, hence the model notation.
The annihilation dynamics appears in chemical kinetics and bimolecular reactions and has
received attention in physics and probability literature \cite{EF85, BNRL93,Belitsky1995,Piasecki95,Droz95,BNRK96,Ermakov1998,Blythe2000,KRBN2010,Sidoravicius2017}.
\index{ballistic annihilation}
\index{ballistic annihilation!continuum}

In a continuum version of the ballistic annihilation model introduced in \cite{KZ19}, 
the moving {\it shock waves} represent the {\it sinks} that aggregate the annihilated particles and hence accumulate the mass of the media. 
Dynamics of these sinks resembles a coalescent process that generates a tree structure for their trajectories, which explain the term {\it shock wave tree} that we use below.
The dynamics of a ballistic annihilation model with two coalescing sinks is illustrated in Fig.~\ref{fig:bam}.  

\begin{figure}[t] 
\centering\includegraphics[width=0.9\textwidth]{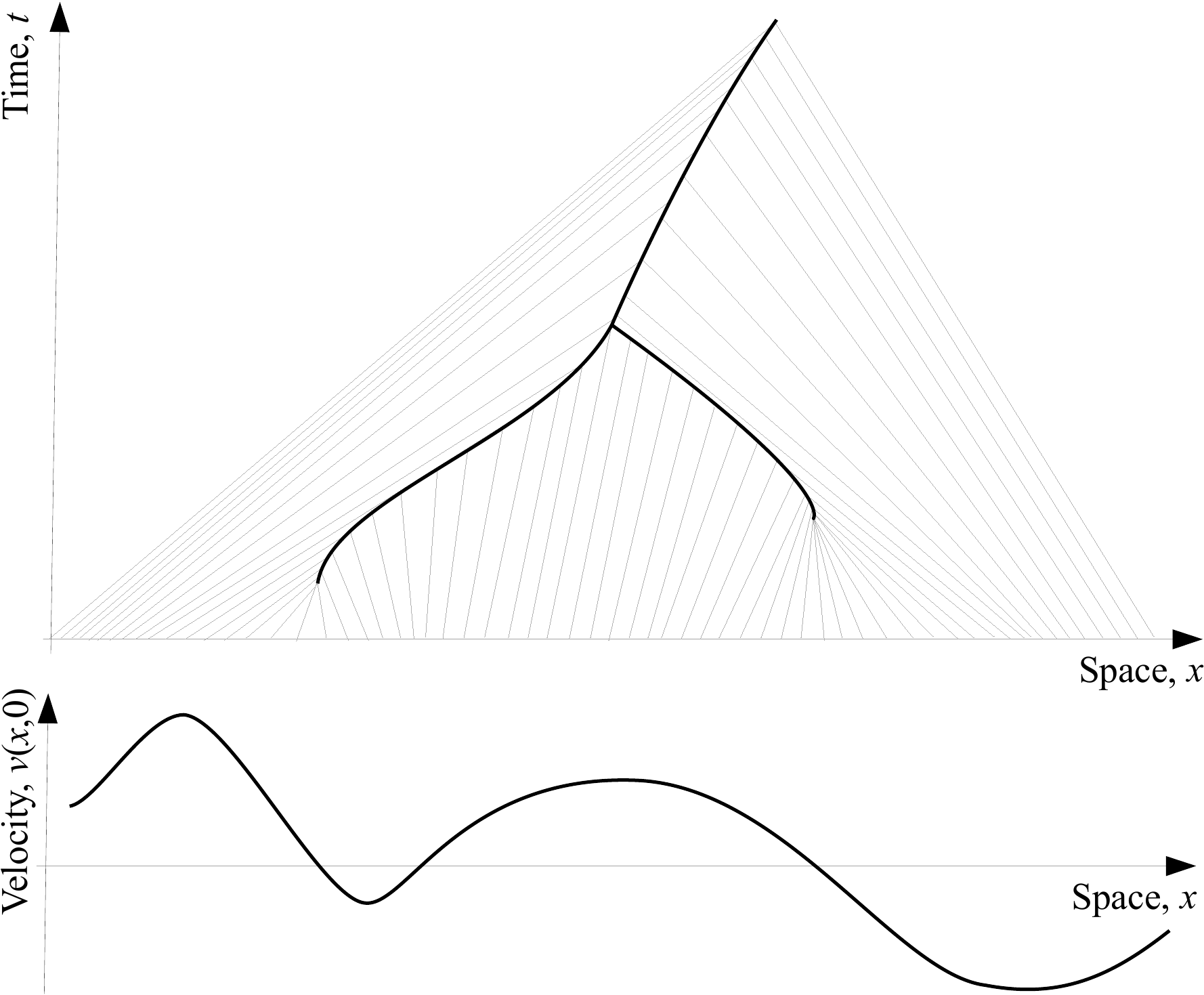}
\caption[Ballistic annihilation model: an illustration]
{Ballistic annihilation model: an illustration.
A particle with Lagrangian coordinate $x$ moves with velocity $v(x,0)$
until it collides with another particle and annihilates.
(Bottom panel): Initial velocity $v(x,0)$.
(Top panel): The space-time portrait of the system. 
The trajectories of selected particles are depicted by gray thin lines.
The shock wave that describes the motion and coalescence of sinks is shown 
by solid black line.
The sink trajectory forms an inverted Y-shaped tree. 
}
\label{fig:bam}
\end{figure} 

Sect.~\ref{sec:shock} introduces the continuum annihilation model and describes 
the natural emergence of sinks (shocks).
The model initial conditions are given by a particle velocity distribution and 
particle density on $\mathbb{R}$.
Subsequently, we only consider a constant density and initial velocity distribution with alternating values $\pm1$, or, equivalently,
initial piece-wise linear potential $\psi(x,0)$ with alternating slopes $\pm1$ (Fig.~\ref{fig:vel}).
Section~\ref{sec:solution} discusses a construction of the graphical 
embedding of the shock wave tree into the phase space $(x,\psi(x,t))$ and space-time domain $(x,t)$.
Theorems~\ref{thm:pruning1}, \ref{thm:pruning} in Sect.~\ref{sec:Bdyn} establish equivalence of 
the ballistic annihilation dynamics to the generalized dynamical pruning of 
a (mass-equipped) shock wave tree. 
Sections~\ref{sec:BDEE},\ref{sec:rand_mass} illustrate how the pruning interpretation
of annihilation dynamics facilitates analytical treatment of the model.
Specifically, we give a complete description of the time-advanced potential 
function $\psi(x,t)$ at any instant $t>0$ for the initial potential in a form of exponential excursion (Thm.~\ref{thm:annihilation}), 
and describe the temporal dynamics of a random sink 
(Thms.~\ref{thm:rand_mass},\ref{thm:mass}).
A real tree representation of ballistic annihilation is discussed in Sect.~\ref{sec:Rtree}.
\index{shock}
\index{shock!wave}
\index{shock!tree}
\index{sink}

\subsection{Continuum model, sinks, and shock trees}
\label{sec:shock}
Consider a Lebesgue measurable initial density $g(x) \geq 0$ of 
particles on an interval $[a,b]\subset \mathbb{R}$. 
The initial particle velocities are given by $v(x,0)=v(x)$.
Prior to collision and subsequent annihilation, a particle located at $x_0$ at time $t=0$ 
moves according to its initial velocity, so its coordinate $x(t)$ changes as
\begin{equation}
\label{eq:tangency_point}
x(t)=x_0+tv(x_0).
\end{equation}
When the particle collides with another particle, it annihilates.
Accordingly, two particles with initial coordinates and velocities $(x_-,v_-)$ and $(x_+,v_+)$
collide and annihilate at time $t$ when they meet at the same new position, 
\[x_-+tv_-=x_++tv_+,\]
given that neither of the particles annihilated prior to $t$.
In this case, the annihilation time is given by
\be
\label{t_sink}
t = -\frac{x_+-x_-}{v_+-v_-}.
\ee

Let $v(x,t)$ be the Eulerian specification of the velocity field at coordinate $x$
and time instant $t$; we define the corresponding 
{\it potential function} 
\[\psi(x,t)=-\int_a^x v(y,t)dy,\quad x\in[a,b], t\ge 0,\] 
so that $v(x,t)=-\partial_x \psi(x,t)$. 
Let $\psi(x,0)=\Psi_0(x)$ be the initial potential. 

We call a point $\sigma(t)$ {\it sink} (or {\it shock}), if there exist two particles that
annihilate at coordinate $\sigma(t)$ at time $t$.
Suppose $v(x) \in C^1(\mathbb{R})$. 
The equation \eqref{t_sink} implies that appearance of a sink is associated with
a negative local minima of $v'(x^*)$; 
we call such points {\it sink sources}.
Specifically, if $x^*$ is a sink source, then a sink will appear at {\it breaking time} 
$t^*=-1/v'(x^*)$ at the location given by 
\[\sigma(t^*)=x^*+t^*v(x^*)=x^*-{v(x^*) \over v'(x^*)},\]
provided there exists a  punctured neighborhood 
$$N_\delta(x^*)=\{x:~0<|x-x^*|<\delta\} \subseteq [a,b]$$ 
such that none of the particles with the initial coordinates in $N_\delta(x^*)$ is annihilated before time $t^*$. 

Sinks, which originate at sink sources, can move and coalesce (see Fig.~\ref{fig:bam}). 
We refer to a sink trajectory as a {\it shock wave}. 
We impose the conservation of mass condition by defining the {\it mass of a sink} 
at time $t$ to be the total mass of particles annihilated in the sink between 
time zero and time $t$. 
When sinks coalesce, their masses add up.
It will be convenient to assume that sinks do not disappear
when they stop accumulating mass.
Informally, we assume that the sinks are being pushed by the system particles.
Formally, there exists three cases depending on the occupancy of a
neighborhood of $\sigma(t)$.
If there exists an empty neighborhood around the sink coordinate $\sigma(t)$, the 
sink is considered 
at rest -- its coordinate does not change. 
If only the left neighborhood of $\sigma(t)$ is empty, and the right adjacent velocity
is negative:
\[v(\sigma_+,t):=\lim_{x\downarrow \sigma(t)} v(x,t)<0,\]
the sink at $\sigma(t)$ moves with velocity $v(\sigma_+,t)$.
A similar rule is applied to the case of right empty neighborhood.
The appearance, motion, and subsequent coalescence of sinks can be described 
by a time oriented {\it shock tree}.
In particular, the coalescence of sinks under initial conditions with a finite 
number of sink sources is described by a finite tree.

The dynamics of ballistic annihilation, either in discrete or 
continuum versions, can be quite intricate and is lacking
a general description.
The existing analyses
focus on the evolution of selected statistics under particular initial conditions. 
In the following sections, we give a complete description of the dynamics in case
of two-valued initial velocity and constant particle density.


\begin{figure}[t] 
\centering\includegraphics[width=0.9\textwidth]{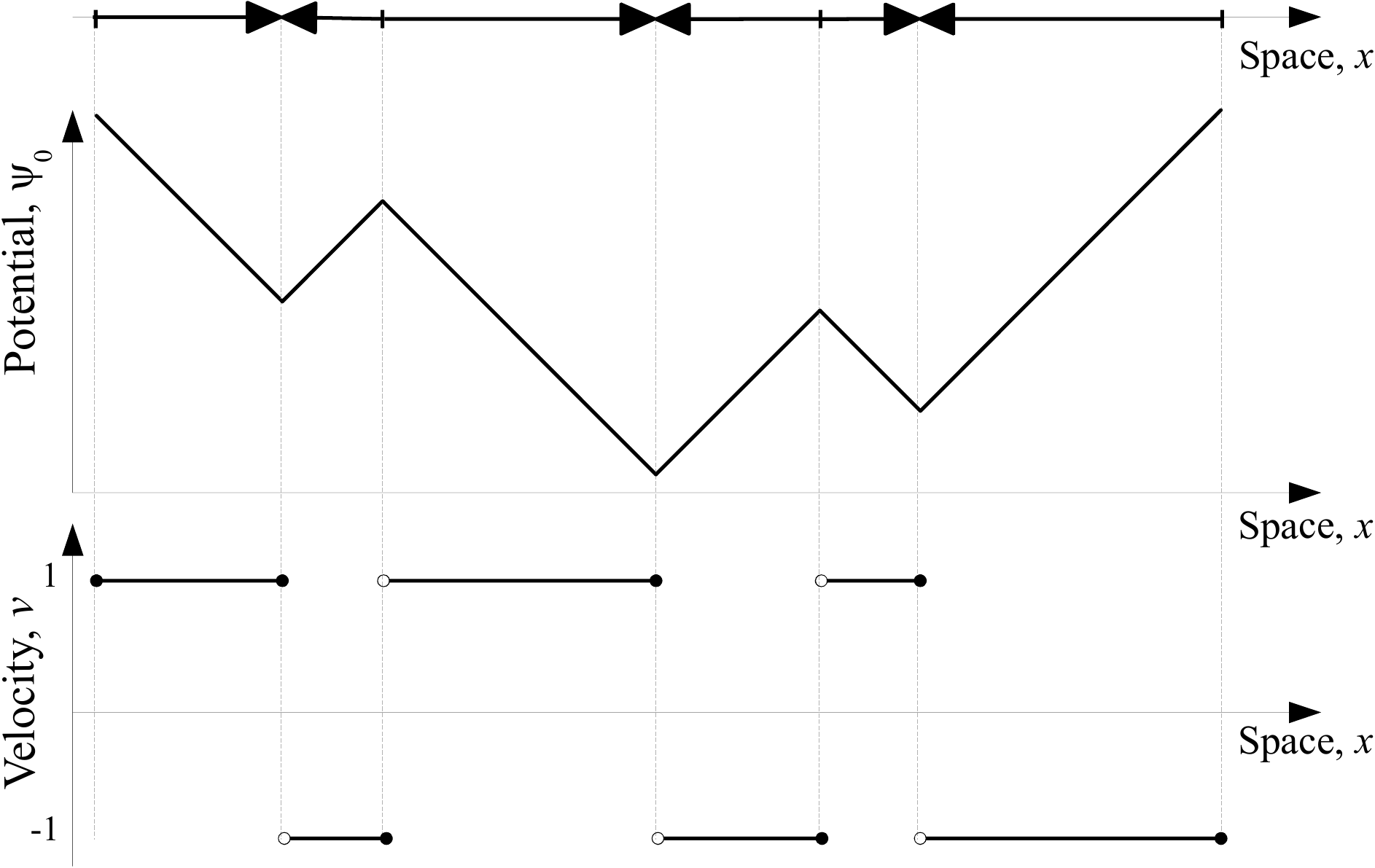}
\caption[Piece-wise linear unit slope potential]
{Piece-wise linear unit slope potential: an illustration.
(Top): Arrows indicate alternating directions of particle movement 
on an interval in $\mathbb{R}$.
(Middle): Potential $\Psi_0(x)$ is a piece-wise linear unit slope function.
(Bottom): Particle velocity alternates between values $\pm1$ within consecutive intervals.}
\label{fig:vel}
\end{figure} 

\subsection{Piece-wise linear potential with unit slopes}
\label{sec:solution}
The discrete 1-D ballistic annihilation model with two possible velocities $\pm v$ was considered in
\cite{EF85,Belitsky1995,BNRK96,Ermakov1998,Blythe2000}; 
the three velocity case ($-1$, $0$, and $+1$) appeared in \cite{Droz95,Sidoravicius2017}. 
Here, we explore a continuum version of the 1-D ballistic annihilation with two possible initial 
velocities and constant initial density, i.e.  $v(x)=\pm v$ and $g(x,0)\equiv g(x)\equiv g_0$ for $x\in[a,b]$.
Since we can scale both space and time, without loss of generality we let $v(x)=\pm 1$ 
and $g(x)\equiv 1$.

Recall (Sect.~\ref{sec:rec}) the space $\cE^{\rm ex}$ of positive piece-wise linear 
continuous excursions with alternating slopes $\pm 1$ and finite number of segments.
We write ${\cE}^{\sf ex}([a,b])$ for the restriction of this space on the real 
interval $[a,b]$.
We consider an initial potential 
$\psi(x,0)=\Psi_0(x)$ such that  $-\psi(x,0)\in{\cE}^{\sf ex}([a,b])$; see Fig.~\ref{fig:vel}.
This space bears a lot of symmetries that facilitate our analysis.

The dynamics of a system with a simple unit slope potential
is illustrated in Fig.~\ref{fig:shock_tree}.
Prior to collision, the particles move at unit speed either to the left or to the right,
so their trajectories in the $(x,t)$ space are given by lines with 
slope $\pm 1$ (Fig.~\ref{fig:shock_tree}, top panel, gray lines).
The local minima of the potential $\Psi_0(x)$ correspond to the points 
whose right neighborhood moves to the left and left neighborhood
moves to the right with unit speed, hence immediately creating a sink.
Accordingly, the sinks appear at $t=0$ at the local minima of the potential; 
and those are the only sinks of the system. 
The sinks move and merge to create a shock wave tree, shown
in blue in Fig.~\ref{fig:shock_tree}. 

Observe that the domain $[a,b]$ is partitioned into non-overlapping subintervals with
boundaries $x_j$ such that the initial particle velocity assumes alternating values of $\pm 1$ within
each interval, with boundary values $v(a,0)=v(a)=1$ and $v(b,0)=v(b)=-1$. 
Because of the choice of potential $\Psi_0(x)$, we have 
$$\int\limits_a^b v(x)\, dx=\Psi_0(b)-\Psi_0(a)=0,$$ 
i.e. the total length of the subintervals with the initial velocity $-1$ equals
the total length of the subintervals with the initial  velocity $1$.
For a finite interval $[a,b]$, there exists a finite time $t_{\rm max} = (b-a)/2$
at which all particles aggregate into a single sink of mass
$m = (b-a)=2\,t_{\rm max}$ \cite{KZ19}.
We only consider the solution on the time interval $[0,t_{\rm max}]$, and
assume that the density of particles vanishes outside of $[a,b]$.

\subsubsection{Graphical representation of the shock wave tree}
For our fixed choice of the initial particle density $g(x)\equiv 1$,
the model dynamics is completely determined by the potential $\Psi_0(x)$. 
We will be particularly interested in the dynamics of sinks (shocks),
which we refer to as {\it shock waves}.
The trajectories of sinks can be described by a set (Fig.~\ref{fig:shock_tree},
top panel)
\[\cG^{(x,t)}(\Psi_0) = \Big\{\big(x,t\big) \in \mathbb{R}^2 \,: \,
\exists \text{ a sink satisfying } \sigma(t)=x \Big\}\]
in the system space-time domain
$(x,t) : x\in[a,b], ~t\in \big[0,(b-a)/2\big].$
These trajectories have a finite binary tree structure: 
the combinatorial planar shape of $\cG^{(x,t)}(\Psi_0)$
is a finite tree in $\BT^|$ \cite{KZ19}.
For any two points $(x_i,t_i)\in\cG^{(x,t)}(\Psi_0) $, $i=1,2$,
connected by a unique self-avoiding path $\gamma$ within $\cG^{(x,t)}(\Psi_0)$,
we define the distance between them as
\[d^{(x,t)}\big((x_1,t_1),(x_2,t_2)\big)=\int\limits_\gamma |dt| = 2t^* - t_1 -t_2,\]
where 
\[t^*:=\max\{t: (x,t)\in\gamma\}.\]
Equivalently, the distance between the points within a single edge is defined as
their nonnegative time increment; this induces the distance $d^{(x,t)}$
on $\cG^{(x,t)}(\Psi_0)$.

Similarly, the trajectories of the sinks can be described by 
a set (Fig.~\ref{fig:shock_tree}, bottom panel)
\[\cG^{(x,\psi)}(\Psi_0) = \Big\{\big(x,\psi(x,t)\big) \in \mathbb{R}^2 \,: \,
\exists \text{ a sink satisfying } \sigma(t)=x \Big\}\]
in the system phase space
$(x,\psi(x,t)) :  ~x\in[a,b], ~t\in \big[0,(b-a)/2\big].$
For any two points $(x_i,\psi_i)\in\cG^{(x,\psi)}(\Psi_0) $, $i=1,2$,
connected by a unique self-avoiding path $\gamma$ within $\cG^{(x,\psi)}(\Psi_0)$,
we define the distance between them as
\[d^{(x,\psi)}\big((x_1,\psi_1),(x_2,\psi_2)\big)
=\int\limits_\gamma \big(|dt|+|dx|\big).\]
Equivalently, one can consider the $L^1$ distance between the points within a single edge;
this induces the distance $d^{(x,\psi)}$ on $\cG^{(x,\psi)}(\Psi_0)$.

\begin{lem}[\cite{KZ19}]
\label{lem:gt}
The metric spaces $\big(\cG^{(x,t)}(\Psi_0),d^{(x,t)}\big)$ and 
$\big(\cG^{(x,\psi)}(\Psi_0),d^{(x,\psi)}\big)$ are trees 
(Def.~\ref{def:treeL}).
Furthermore, they have a finite number of edges and are isomeric to a 
unique binary tree from $\BL^|$ that we denote by $S(\Psi_0)$.  
\end{lem}
\noindent We refer to the trees of Lem.~\ref{lem:gt} as the graphical trees
$\cG^{(x,t)}(\Psi_0)$ and $\cG^{(x,\psi)}(\Psi_0)$ since
they are two alternative graphical representations of 
the {\it shock wave tree} $S(\Psi_0)$.

\subsubsection{Structure of the shock wave tree}
Importantly, for our particular choice of the initial potential, 
the combinatorial structure and the planar embedding of the shock wave tree coincide
with that of the level set tree $T=\textsc{level}\big(-\Psi_0\big)$ of the initial potential,
as we state in the following theorem. 

\begin{thm}[{{\bf Shock wave tree is a level set tree, \cite{KZ19}}}]
\label{thm:SWT}
Suppose $g(x)\equiv1$ and the initial potential $\Psi_0(x)$ is such that $-\Psi_0(x)\in\cE^{\sf ex}$.
Then
\[\textsc{p-shape}\big(\textsc{level}\left(-\Psi_0\right)\big)=
\textsc{p-shape}\big(S(\Psi_0)\big).\]
\end{thm}
\noindent Theorem~\ref{thm:SWT} implies that there is one-to-one correspondence between internal local 
maxima of $\Psi_0(x)$ and internal non-root vertices of $S(\Psi_0)$. 
There is also a one-to-one correspondence between local minima and the leaves.
We label the tree vertices with the indices $j$ that correspond to the 
enumeration of the local extrema $x_j$ of $\Psi_0(x)$; see Fig.~\ref{fig:tree}.
We write ${\sf parent}(i)$ for the index of the parent vertex to vertex $i$;
${\sf right}(i)$ and ${\sf left}(i)$ for the indices of the right and the left
offsprings of an internal vertex $i$;
and ${\sf sibling}(i)$ for the index of the sibling of vertex $i$. 

For a local extremum $x_j$, we define its {\it basin} $\cB_j$ as the shortest interval that contains $x_j$ and supports 
a non-positive excursion of $\Psi_0(x)$.
Formally, $\cB_{j}=[x^{\rm left}_j, x^{\rm right}_j]$, where
\[x^{\rm right}_j = \inf\big\{x: \,x>x_j\text{ and }\Psi_0(x)>\Psi(x_j)\big\},\]
\[x^{\rm left}_j = \sup\big\{x: \,x<x_j\text{ and }\Psi_0(x)>\Psi(x_j)\big\}.\]
We observe that the basin $\cB_j$ for a local minimum $x_j$ coincides with its coordinate: 
$\cB_j = \{x_j = x^{\rm left}_j = x^{\rm right}_j\}$.

The basin's length is $\big|\cB_j\big| = x^{\rm right}_j-x^{\rm left}_j$.
Point $c_j = (x^{\rm right}_j+x^{\rm left}_j)/2$ denotes the center of the basin $\cB_j$.
Additionally, we let
\[\v_j = \Psi_0(x_{{\sf parent}(j)})-\Psi_0(x_j)\quad \text{ and } \quad 
\h_j = \big|\cB_{{\sf sibling}(j)}\big|/2.\]

\begin{figure}[t] 
\centering\includegraphics[width=0.8\textwidth]{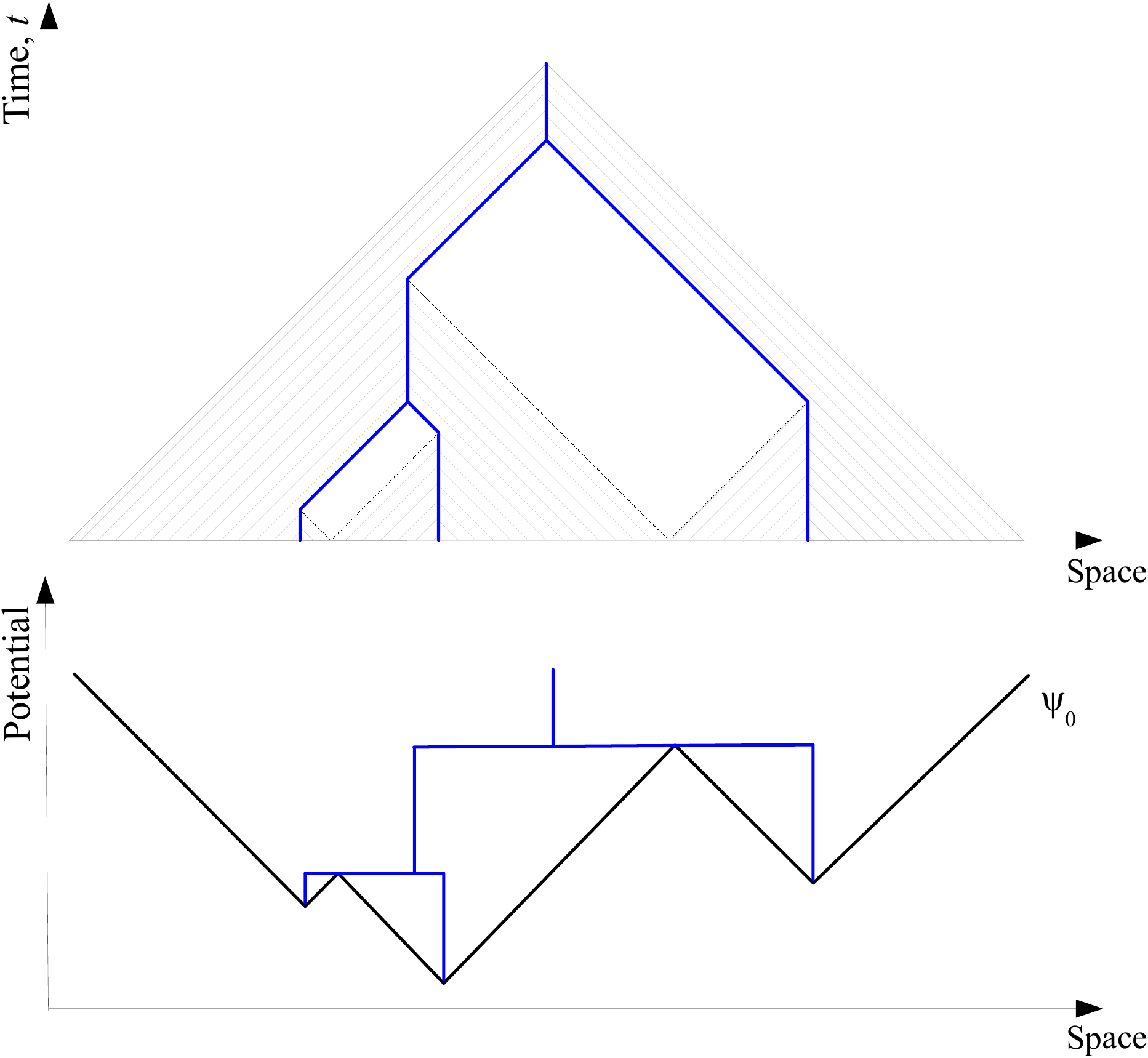}
\caption[Shock tree in a model with a unit slope potential: an illustration]
{Shock wave tree (sink tree) in a model with a unit slope potential: an illustration. 
(Top panel): Space-time dynamics of the system. 
Trajectories of particles are illustrated by gray lines.
The trajectory of coalescing sinks is shown by blue line -- this is 
the graphical representation $\cG^{(x,t)}(\Psi_0)$ of the shock wave tree
$S(\Psi_0)$.
Notice the appearance of empty regions (zero particle density) in the space-time domain.
(Bottom panel): Initial unit slope potential $\Psi_0(x)$ with three local 
minima (black line) and a graphical representation $\cG^{(x,\psi)}(\Psi_0)$ of the shock wave tree 
(blue line) in the phase space $(x,\psi(x,t))$.
}
\label{fig:shock_tree}
\end{figure} 

\noindent
We are now ready to describe the metric structure of the shock tree $S(\Psi_0)$ 
and a constructive embedding $\cG^{(x,\psi)}(\Psi_0)$ of the tree $S(\Psi_0)$ 
into the system's phase space.

{\bf Metric tree structure.}
The length $l_j$ of the parental edge of a non-root vertex $j$ 
within $S(\Psi_0)$ is given by 
$l_j = \v_j + \h_j.$

{\bf Graphical shock tree in the phase space.}
The tree
$\cG^{(x,\psi)}(\Psi_0)$
is the union of the following vertical and horizontal segments:
\begin{itemize}
\item[$(\v)$] For every local extremum $x_j$ of $\Psi_0(x)$ there exists a vertical 
segment from $(c_j,\Psi_0(x_j))$ to $(c_j,\Psi_0(x_j)+\v_j)$.
\item[(h)] For every local maximum $x_j$ of $\Psi_0(x)$ there exists a horizontal
segment of length $\h_{{\sf left}(j)}+\h_{{\sf right}(j)}$ 
from $(c_{{\sf left}(j)},\Psi_0(x_j))$ to $(c_{{\sf right}(j)},\Psi_0(x_j))$.
\end{itemize}

\noindent Figure~\ref{fig:tree} shows the graphical shock trees $\cG^{(x,\psi)}$ 
and $\cG^{(x,t)}$ for an initial potential with two local maxima and
three local minima, and illustrates the labeling of
vertical ($\v_j$) and horizontal ($\h_j$) segments of the tree.
Figure~\ref{fig:bracket} shows an example of the graphical tree $\cG^{(x,\psi)}$
for an initial potential with nine local minima (and, hence, with nine initial sinks).

Consider a tree $\cV(\Psi_0)\in\BL^{|}$ that has the same planar combinatorial 
structure as $S(\Psi_0)$, and the length of the parental edge of vertex $j$ is given by $l_j=\v_j$.
Informally, this is a tree that consists of the vertical segments of 
the graphical tree $\cG^{(x,\psi)}(\Psi_0)$ (Fig.~\ref{fig:shock_tree}, bottom).
We have the following corollary of Thm.~\ref{thm:SWT}.
\begin{cor}[{{\bf \cite{KZ19}}}]
\label{cor:V}
Suppose $g(x)\equiv1$ and potential $\Psi_0(x)$ is such that $-\Psi_0(x)\in\cE^{\sf ex}$.
Then
\[\cV(\Psi_0) = \textsc{level}\left(-\Psi_0\right).
\]
\end{cor}

\begin{figure}[!h]
	\centering
	\includegraphics[width=0.85 \textwidth]{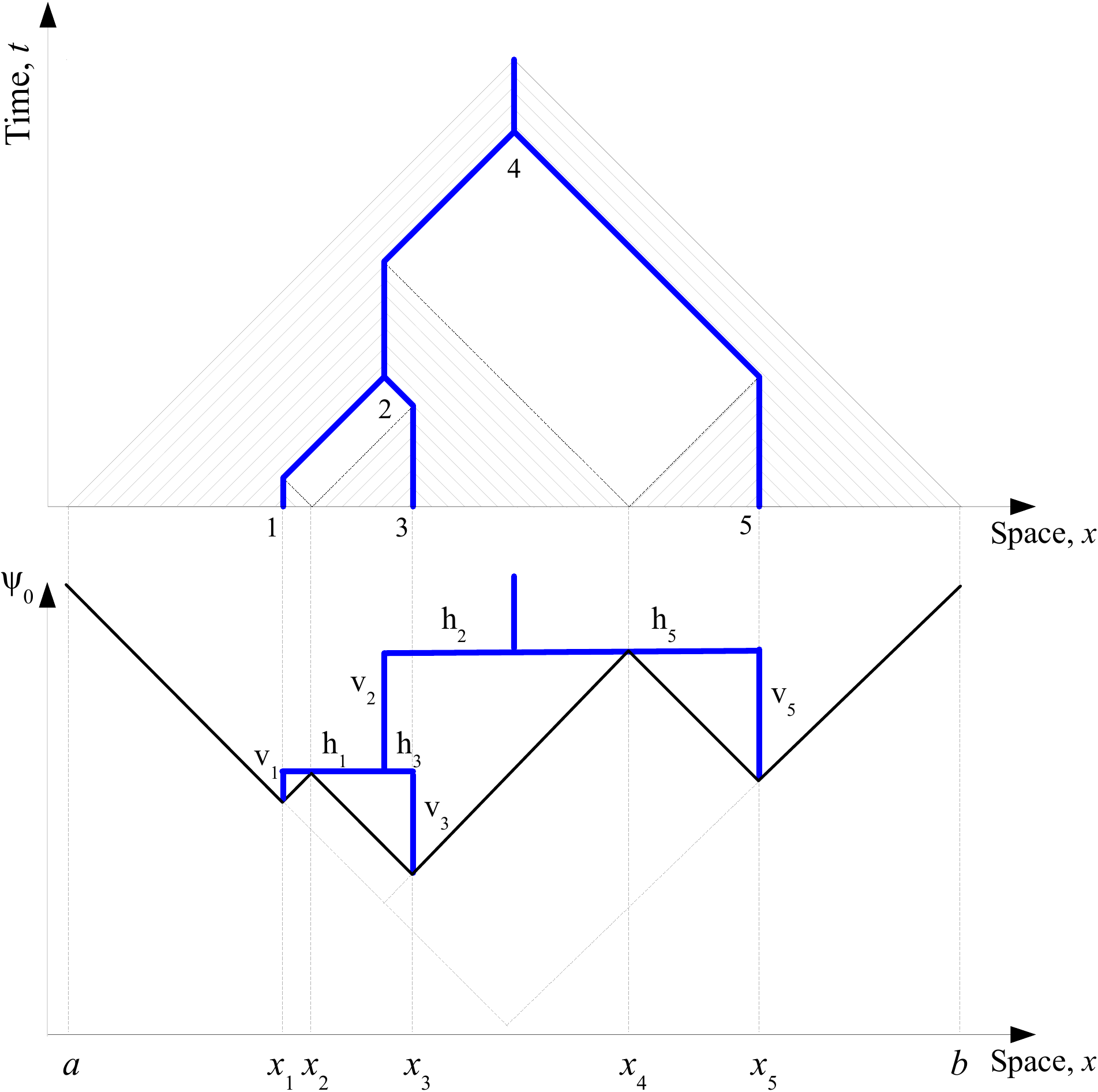}
	\caption{Shock tree for a piece-wise linear potential with two local 
	maxima.
	(Top): The shock tree in space-time domain (blue). Hatching illustrates 
	motion of 
	regular particles. There exist two empty rectangular areas, each corresponding to one
	of the local maxima.
	The panel illustrates indexing of the tree vertices.
	(Bottom): Potential $\Psi_0(x)$ (black) and the shock tree in the phase 
	space (blue).
	The panel illustrates the labeling of vertical ($\v_j$) and horizontal ($\h_j$) 
	segments of the tree.}
	\label{fig:tree}
\end{figure}

\subsubsection{Ballistic annihilation as generazlized pruning}
\label{sec:Bdyn}
This section shows that the dynamics of continuum ballistic annihilation with constant
initial density and unit-slope potential is equivalent to the generalized dynamical
pruning of either the shock wave tree (Thm.~\ref{thm:pruning1}) or 
the level set tree of the potential
(Thm.~\ref{thm:pruning}).

\begin{figure}[!t]
	\centering
	\includegraphics[width=0.95 \textwidth]{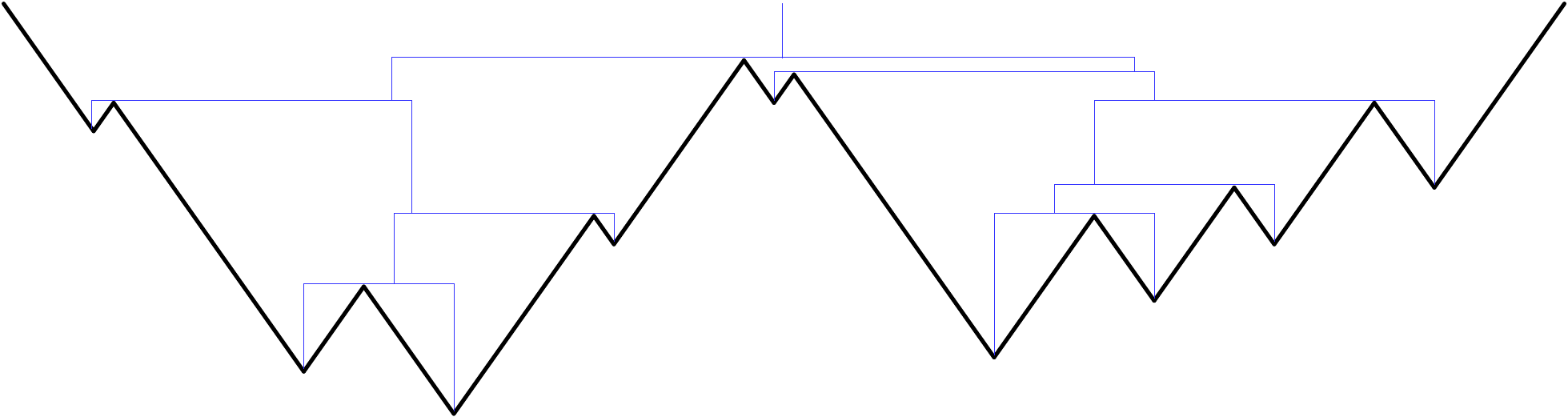}
	\caption{Graphical representation $\cG^{(x,\psi)}(\Psi_0)$ (blue) of the sink tree $S(\Psi_0)$
	for initial potential $\Psi_0(x)$ with nine local minima (black). 
	There are nine sinks that correspond to the leaves of the tree.
	The trajectory of each sink can be traced by going from the corresponding 
	leaf to the root of the tree.}
	\label{fig:bracket}
\end{figure}

Suppose a tree $T\in\BL^|$ has a particular
graphical representation $\cG_T\in\mathbb{R}^2$ 
implemented by a bijective isometry
$f:T\rightarrow\cG_T$ that maps the root of $T$ into the root of  $\cG_T$.
We extend the notion of the generalized dynamical pruning $\S(\varphi,\cG_T)$
for the graphical tree $\cG_T$ by considering the $f$-image of $\S(\varphi,T)$:
\[\S(\varphi,\cG_T) = f\big(\S(\varphi,T)\big).\]
Consider a natural isometry (Lem.~\ref{lem:gt}) between the shock wave tree $S(\Psi_0)$
and either of the graphical shock trees,
$\cG^{(x,t)}(\Psi_0)$ (in the space-time domain) or
$\cG^{(x,\psi)}(\Psi_0)$ (in the phase space). 
The next theorem formalizes an observation that the dynamics
of sinks is described by the continuous pruning (Sect.~\ref{ex:height}) 
of the shock wave tree.
\begin{thm}[{{\bf Annihilation pruning I, \cite{KZ19}}}]
\label{thm:pruning1}
Suppose $g(x)\equiv1$, and the initial potential $\Psi_0(x)$ is such that $-\Psi_0(x)\in\cE^{\sf ex}$.
Then, the dynamics of sinks is described by
the generalized dynamical pruning $\S(\varphi,\cG)$ of either the graphical 
tree $\cG=\cG^{(x,\psi)}(\Psi_0)$ (in the phase space) or $\cG=\cG^{(x,t)}(\Psi_0)$
(in the space-time domain), with the pruning function 
$\varphi(T)=\textsc{height}(T)$.
Specifically, the locations of sinks at any instant $t\in[0,t_{\rm max})$ coincide with the 
location of the leaves of the pruned tree $\S(\varphi,\cG)$.
\end{thm}
\noindent
Theorem~\ref{thm:pruning1} only refers to the dynamics of the sinks; it is, however,
intuitively clear that the entire potential $\psi(x,t)$ at any given $t>0$ 
can be 
uniquely reconstructed from either of the pruned graphical trees, 
$\cG^{(x,t)}(\Psi_0)$ or $\cG^{(x,\psi)}(\Psi_0)$.
Because of the multiple symmetries \cite{KZ19}, 
the graphical trees possess significant redundant information.
It has been shown in \cite{KZ19} that the reduced tree $\cV(\Psi_0)$ (Cor.~\ref{cor:V}) 
equipped with information about the sinks provides a minimal
description sufficient for reconstructing the entire continuum annihilation dynamics.
\begin{lem}[\cite{KZ19}]
\label{lem:Vpshape}
Suppose $g(x)\equiv1$, and the initial potential $\Psi_0(x)$ is such that $-\Psi_0(x)\in\cE^{\sf ex}$.
Then,
\[\textsc{level}(\psi(x,t)) = \S(\textsc{length},\cV(\Psi_0)).\]
\end{lem}
Lemma~\ref{lem:Vpshape} states that the level set tree (i.e., the sequence of
the local extreme values) of $\psi(x,t)$ is uniquely reconstructed from the pruned
tree $\cV(\Psi_0)$.
This, however, is not sufficient to reconstruct the entire time-advanced potential,
which has plateaus corresponding to the intervals of zero density (recall the
empty regions in the top panels of Fig.~\ref{fig:shock_tree}).
The information about such plateaus is lost in the pruned tree.  
It happens that it suffices to remember ``the size'' of the pruned out parts of the tree
in order to completely reconstruct the annihilation dynamics from $\cV(\Psi_0)$.  
Specifically, we store the value $\varphi(\tau)$ 
for each subtree $\tau$ that has been pruned out. 
These values are stored in the {\it cuts} -- the points where the pruned subtrees
were attached to the initial tree; see Fig.~\ref{fig:cuts}(a). 
The cuts is a union of the leaves of the 
pruned tree and the vertices of the initial tree that became edge points in 
the pruned tree. 
A formal definition is given below.
\begin{Def}[{{\bf Cuts}}]
\label{def:cuts}
The set $\D(\varphi,T)$ of cuts in a pruned tree $\S(\varphi,T)$ is defined as
the boundary of the pruned part of the tree
\[\D(\varphi,T)=\partial\{x\in T~:~\varphi(\Delta_{x,T})<t\}.\]
\end{Def} 

We now define an extension $\mS(\varphi,T)$ of the generalized dynamical pruning 
that preserves the sizes of pruned subtrees. 
Such pruning starts with a tree from $\BL^|$ and results in a tree
from the space of mass-equipped trees, denoted $\mBL^|$. 
The pruning $\mS(\varphi,T)$ of a tree $T\in\BL^|$
is a tree from $\mBL^|$, whose projection to $\BL^|$ coincides
with $\S(\varphi,T)$.
In addition, the tree is equipped with massive points placed at the cuts. 
Each massive point corresponds to a pruned out subtree $\tau$ of $T$, 
with mass equal $\varphi(\tau)$.
If a cut is the boundary for two pruned subtrees 
(Fig.~\ref{fig:cuts}(a), cuts {\it a,d}), then it hosts two oriented masses.
Such cuts are typical in prunings that do not have the semigroup property
(see Fig.~\ref{fig:LP}, Stage IV).
Figure~\ref{fig:cuts}(b) illustrates mass-equipped pruning $\mS(\varphi,T)$
with pruning function $\varphi=\textsc{length}$.


\begin{figure}[!t]
	\centering
	\includegraphics[width=0.75 \textwidth]{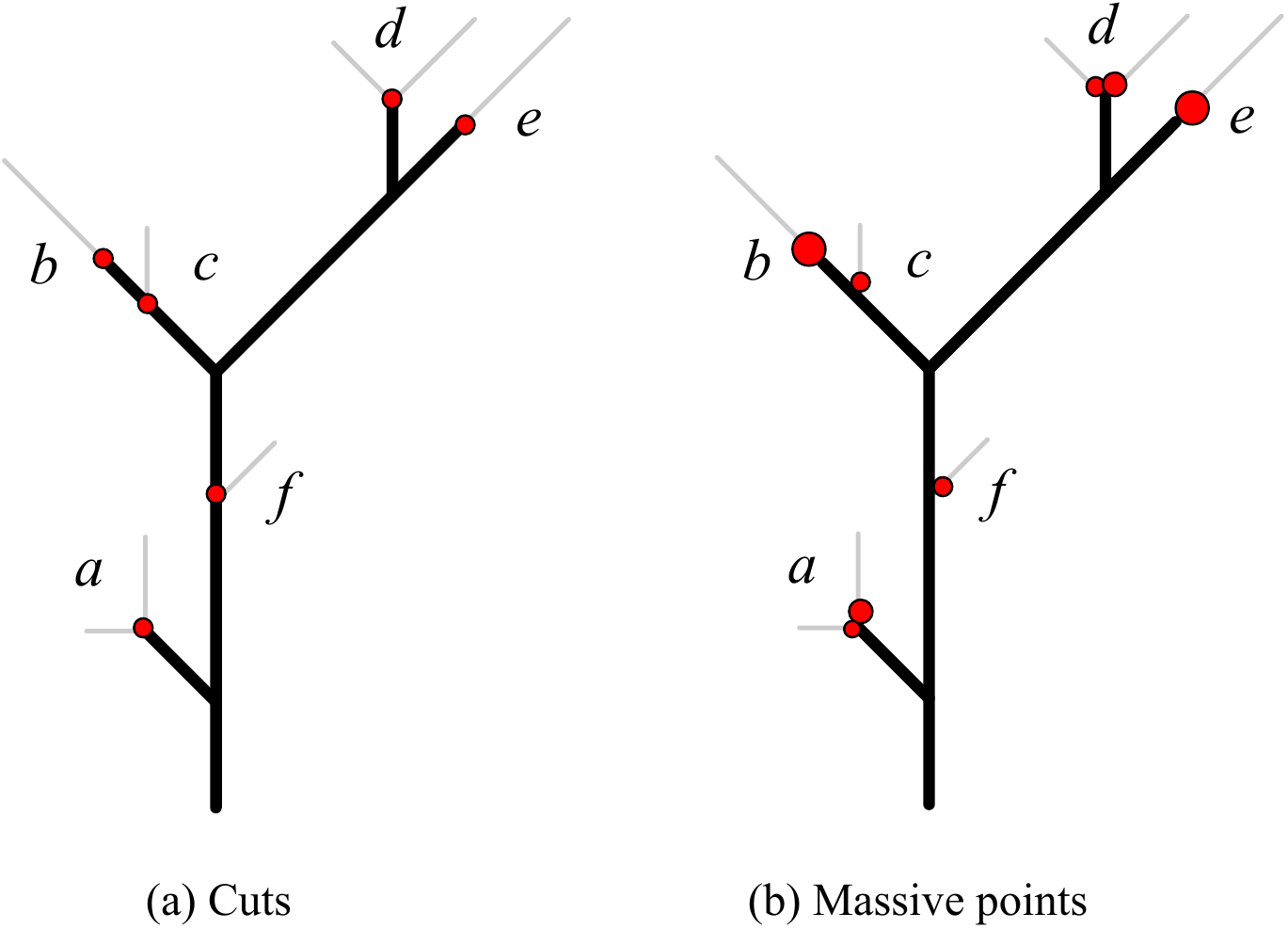}
	\caption{Cuts and massive points: an illustration. 
	(a) Pruned tree $\S(\textsc{length},T)$ (solid black) with
	the set of cuts (red circles).
	The pruned parts of the initial tree $T$ are shown in gray.
	Here, we prune by length;
	the cuts {\it a,d} correspond to Stage IV of Fig.~\ref{fig:LP}.
	The cuts {\it a} and {\it d} are placed at vertices of $T$ that became
	leaves within $\S(\textsc{length},T)$.
	The cuts {\it b} and {\it e} are placed at the leaves of the pruned tree.
	The cuts {\it c} and {\it f} are placed at vertices of $T$ that became
	non-vertex points within $\S(\textsc{length},T)$.
	(b) Massive points (red circles) placed at the cuts. 
	Each of the cuts {\it a} and {\it d} hosts two oriented massive points.
	Each of the cuts {\it b} and {\it e} hosts a single unoriented massive point.
	Each of the cuts {\it c} and {\it f} hosts a single oriented massive point.
	The circle size is proportional to the mass.}
	\label{fig:cuts}
\end{figure}

Next, we describe how to construct a potential $\psi_{T,t}(x)$
for a given $t\in[0,t_{\rm max}]$ and all $x\in[a,b]$
from a pruned mass-equipped tree $T = \mS(\textsc{length},\cV(\Psi_0))$.
Theorem~\ref{thm:pruning} then shows that this reconstructed potential coincides with the
time-advances potential of the annihilation dynamics.

\begin{con}[{{\bf Tree $\rightarrow$ potential}}]
\label{con2}
Suppose  $T = \mS(\textsc{length},\cV(\Psi_0))$.
The corresponding potential $\psi_{T,t}(x)$, with $-\psi_{T,t}(x)\in{\cE}^{\sf ex}$, 
is constructed in the following steps: 
\begin{itemize}
\item[(1)] Construct the Harris path $H_T(x)$ for the projection of $T$ to $\BL^{|}$
(i.e., disregarding masses), and consider the negative excursion $-H_T(x)$.
\item[(2)] At every local minimum of $-H_T(x)$ that corresponds to a double mass
$(m_{\rm L}, m_{\rm R})$, insert a horizontal plateau of length 
\[\varepsilon = 2(m_{\rm L}+m_{\rm R}-t),\]
as illustrated in Fig.~\ref{fig:Vtree}, Stage $2$.
\item[(3)] At every monotone point of $-H_T(x)$ that corresponds to an internal mass $m$, insert
a horizontal plateau of length $2m$ (Fig.~\ref{fig:Vtree}, Stage $3$).
\item[(4)] At every internal local maxima of $-H_T(x)$, insert
a horizontal plateau of length $2t$ (Fig.~\ref{fig:Vtree}, Stage $1$).
\end{itemize}
\end{con}

\begin{figure}[!t]
	\centering
	\includegraphics[width=0.85 \textwidth]{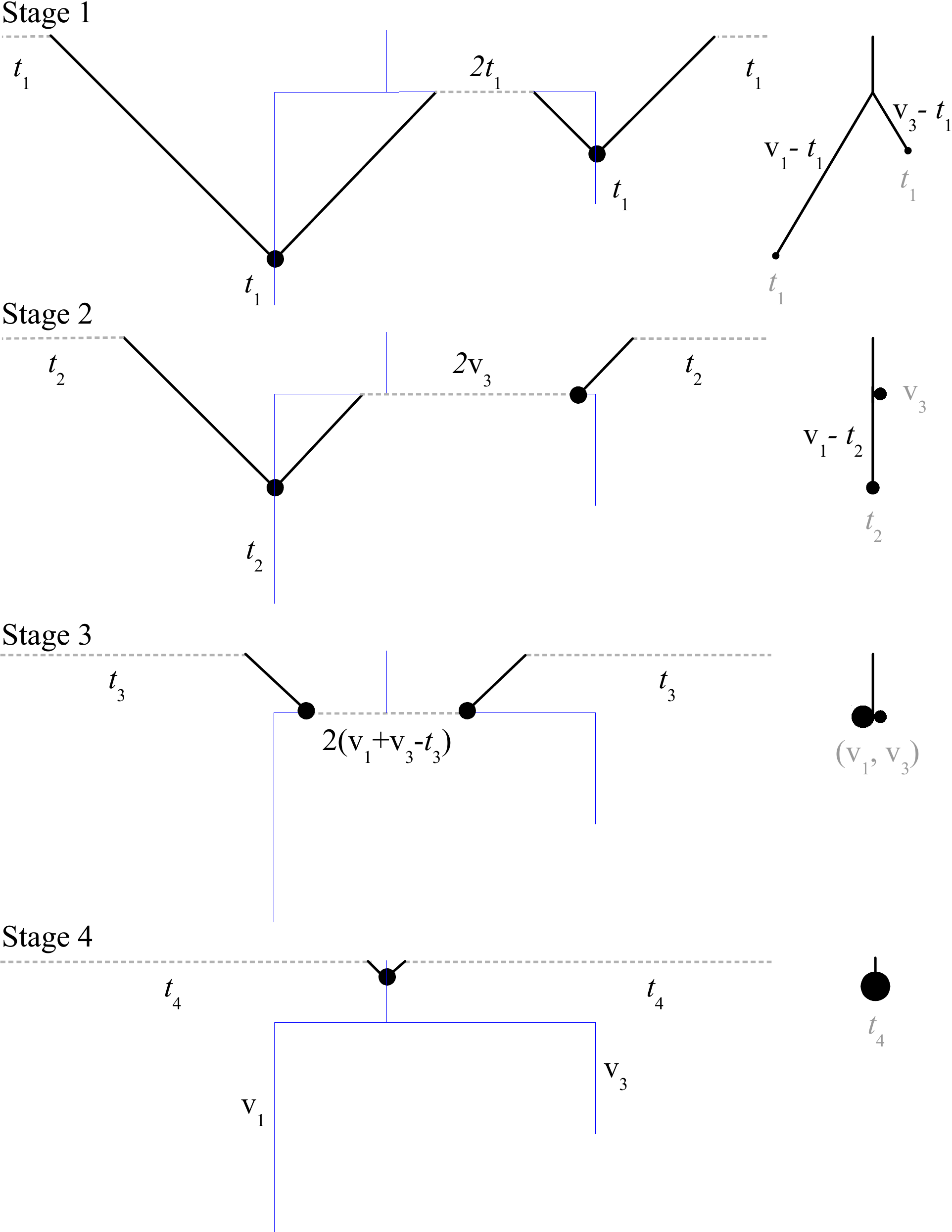}
	\caption{Four generic stages in the ballistic annihilation dynamics of a W-shaped potential (left), 
	and respective mass-equipped trees (right).
	The lengths $\v_1$ and $\v_3$ of the two vertical leaf segments are assigned as
	illustrated in the Stage 4 (see also Fig.~\ref{fig:tree}).
	(Left): Potential $\psi(x,t)$ is shown in solid black.
	Each plateau (dashed gray) corresponds to an interval of zero density.
	The graphical shock tree $\cG^{(x,\psi)}(\Psi_0)$ (blue)
	and sinks (black circles) are shown for visual convenience.
	(Right): Mass-equipped trees. 
	Segment lengths are marked in black, point masses are indicated in gray.
	Notice progressive increase of the point masses from Stage 1 to 4.
	The Stages 1 to 4 refer to time instants $t_1<t_2<t_3<t_4$.
	Here $\v_3<\v_1$, $\v_3>t_1$, $\v_3 < t_2 < \v_1$, $\v_1<t_3$, and $t_3< \v_1+\v_3<t_4$.}
	\label{fig:Vtree}
\end{figure}

\noindent The following theorem establishes the equivalence of the continuum annihilation 
dynamics and mass-equipped generalized dynamical pruning with respect 
to the tree length.
In particular, it includes the statement of Lem.~\ref{lem:Vpshape}.

\begin{thm}[{{\bf Annihilation pruning II, \cite{KZ19}}}]
\label{thm:pruning}
Suppose $g(x)\equiv 1$ and the initial potential $\Psi_0(x)$ is such that 
$-\Psi_0(x)\in{\cE}^{\sf ex}$. 
Then, for any $t\in[0,t_{\rm max}]$, the time-advances potential $\psi(x,t)$
is uniquely reconstructed (by Construction~\ref{con2}) from the
pruned tree 
$T(t) = \mS(\textsc{length},\cV(\Psi_0)).$
That is,
$\psi(x,t)\equiv\psi_{T(t),t}$ for all $x\in[a,b]$.
\end{thm}

It is shown in \cite{KZ19} that, inversely, the mass-equipped 
tree $\mS(\textsc{length},\cV(\Psi_0))$ can be uniquely reconstructed from 
the time-advanced potential $\psi(x,t)$.  
Hence, the continuum ballistic annihilation dynamics is equivalent to
the mass-equipped generalized dynamical pruning of the level set tree
of the initial potential.
The next sections illustrates how this equivalence facilitates the
analytical treatment of the model.

\subsection{Ballistic annihilation of an exponential excursion}
\label{sec:BDEE}

This section examines a special case of piece-wise
linear potential with unit slopes: a negative exponential excursion.
Consider potential $$\psi(x,0)=-H_{{\sf GW}(\lambda)}(x)$$ 
that is the negative Harris path (Sect.~\ref{sec:Harris}) of an exponential critical binary 
Galton-Watson tree with parameter $\lambda$ (Def.~\ref{def:binary}).
In words, the potential is a negative finite excursion with linear segments of
alternating slopes $\pm 1$, such that the lengths of all segments except
the last one are i.i.d. exponential random variables with parameter $\lambda/2$.
Accordingly, the initial particle velocity $v(x,0)$ 
alternates between the values $\pm1$ at epochs of a stationary 
Poisson point process on $\mathbb{R}$ with rate $\lambda/2$, starting with $+1$ and until the 
respective potential crosses the zero level.

\begin{cor}[{{\bf Exponential excursion}}]
\label{cor:GWa}
Suppose $g(x)\equiv 1$ and initial potential 
$\Psi_0(x)=-H_{{\sf GW}(\lambda)}(x)$.
Then the corresponding tree $\cV(\Psi_0)\in\BL^{|}$
is an exponential binary critical Galton-Watson tree 
${\sf GW}(\lambda)$.
\end{cor}
\begin{proof}
By Cor.~\ref{cor:V}, the tree $\cV(\Psi_0)$ is the level 
set tree of the negative potential $-\Psi_0(x)$.
The statement now follows from Thm.~\ref{Pit7_3}.
\end{proof}

To formulate the next result, recall that if $T\stackrel{d}{\sim}{\sf GW}(\lambda)$ and 
$\varphi(T) = \textsc{length}(T)$, then by (\ref{eq:pDelta1}),
$$p_t:={\sf P}(\varphi(T)>t)=e^{-\lambda t}\Big[ I_0(\lambda t)+ I_1(\lambda t) \Big].$$
Also, the p.d.f. of $\textsc{length}(T)$
is given by $\ell(x)$ of \eqref{eq:pdfs}.

\begin{thm}[{\bf Ballistic annihilation dynamics of an exponential excursion, \cite{KZ19}}] 
\label{thm:annihilation}
Suppose the initial particle density is constant, $g(x)\equiv 1$, 
and the initial potential $\psi(x,0)$ is the negative Harris path 
of an exponential critical binary Galton-Watson tree 
with parameter $\lambda$, i.e., $\cV(\Psi_0)\stackrel{d}{\sim}{\sf GW}(\lambda)$.
Then, at any instant $t>0$ the mass-equipped shock tree 
$\cV_t =  \mS(\textsc{length},\cV(\Psi_0))$ conditioned on surviving,
$\cV_t \not= \phi$, is distributed according to the 
following rules.
\begin{enumerate}
  \item[(i)] The planar shape of the tree, as an element of $\BL^|$, 
  is distributed as an exponential binary Galton-Watson tree ${\sf GW}(\lambda_t)$ with
  $\lambda_t:=\lambda p_t$.
  
  \item[(ii)] A single or double mass points are placed independently in each leaf 
  with the probability of a single mass being
  \[{2 \over \lambda}{\ell(t) \over p_t^2}.\]

 \item[(iii)] Each single mass at a leaf has mass $m=t$. 
 For a double mass, the individual masses $(m_{\rm L},m_{\rm R})$ have the following joint p.d.f.
  $$f(a,b)={\ell({a})\ell({b}) \over p_t^2 -{2  \over \lambda}\ell(t)}$$
  for $a,b>0$, $ a\vee b ~\leq t ~< a+b$.
  
  \item[(iv)] The number of mass points placed in the interior of any edge 
  is distributed geometrically with the probability of placing $k$ masses being
  $$p_t\big(1-p_t\big)^k, \qquad k=0,1,2,\hdots.$$
  The locations of $k$ mass points are independent uniform in the interior of the edge. 
  The orientation of each mass is left or right independently with probability $1/2$.
  
  \item[(v)]The edge masses are i.i.d. random variables with the following common p.d.f.
  $${\ell(a) \over 1-p_t}, \qquad a \in (0,t).$$
\end{enumerate}
\end{thm}

\subsection{Random sink in an infinite exponential potential}
\label{sec:rand_mass}
Here we focus on the dynamics of a random sink
in the case of a negative exponential excursion potential.
To avoid subtle conditioning related to a finite potential,
we consider here an infinite exponential potential $\Psi^{\rm exp}_0(x)$,
$x\in\mathbb{R}$, constructed as follows.
Let $x_i$, $i\in\mathbb{Z}$ be the epochs of a Poisson point process on $\mathbb{R}$ with rate $\lambda/2$,
indexed so that $x_0$ is the epoch closest to the origin.
The initial velocity $v(x,0)$ is a piece-wise constant continuous function 
that alternates between values $\pm1$ within the intervals 
$(x_i-1,x_i]$ and with $v(x_0,0)=1$. 
Accordingly, the initial potential $\Psi^{\rm exp}_0(x)$ is a piece-wise 
linear continuous function with a local minimum at $x_0$ and alternating slopes 
$\pm1$ of independent exponential duration.
The results in this section refer to the sink $\cM_0$ with 
initial Lagrangian coordinate $x_0$.
We refer to $\cM_0$ as a {\it random sink}, using
translation invariance of Poisson point process.

\begin{figure}[h]
	\centering
	\includegraphics[width=0.7 \textwidth]{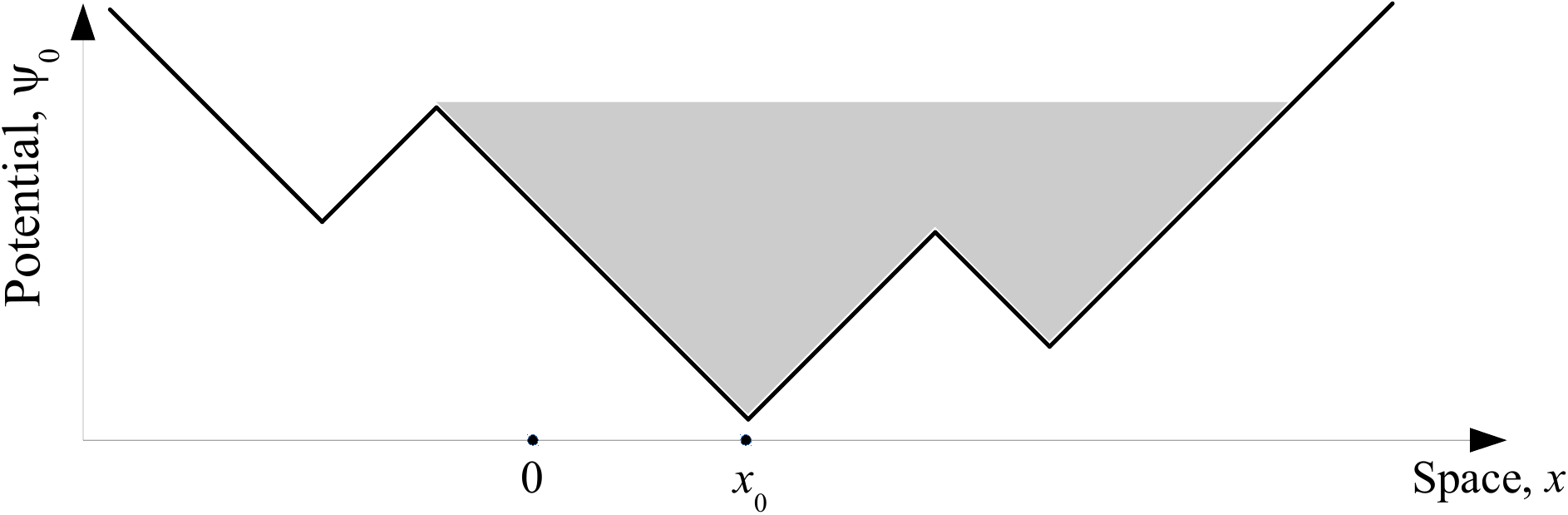}
	\caption{Random sink $\cM_0$ originates at point $x_0$ -- the local minimum 
	closest to the origin.
	Its dynamics during a finite time interval $[0,t]$ is completely specified 
	by a finite negative excursion
	$\cB_0^t$ similar to the one highlighted in the figure.}
	\label{fig:rand_mass}
\end{figure}

Observe that for any fixed $t>0$, the dynamics of $\cM_0$ is completely
specified by a finite excursion within $\Psi^{\rm exp}_0(x)$.
For instance, one can consider the shortest negative excursion of $\Psi^{\rm exp}_0(x)$
within interval $\cB_0^t$ such that
$x_0\in\cB_0^t$, $|\cB_0^t|>2t$, and one end of $\cB_0^t$ is a local
maximum of $\Psi^{\rm exp}_0(x)$ (see Fig.~\ref{fig:rand_mass}).
The respective Harris path is an exponential Galton-Watson tree ${\sf GW}(\lambda)$.  
The dynamics of $\cM_0$ consists of alternating
intervals of mass accumulation (vertical segments of $\cG^{(x,\psi)}$) and
motion (horizontal segments of $\cG^{(x,\psi)}$), starting with a mass accumulation
interval.
Label the lengths $\v_i$ of the vertical segments and the lengths $\h_i$ of the
horizontal segments in the order of appearance in the 
examined trajectory. 
Corollary~\ref{cor:GWa} implies that $\v_i,\h_i$ are independent;
the lengths of $\v_i$ are i.i.d. exponential random variables with parameter $\lambda$;
and the lengths of $\h_i$ equal the total lengths of independent
Galton-Watson trees ${\sf GW}\left(\lambda\right)$.
This description, illustrated in Fig.~\ref{fig:mass}, allows us to find
the mass dynamics of a random sink, which is described in the next two theorems.

\begin{figure}[h]
	\centering
	\includegraphics[width=0.7 \textwidth]{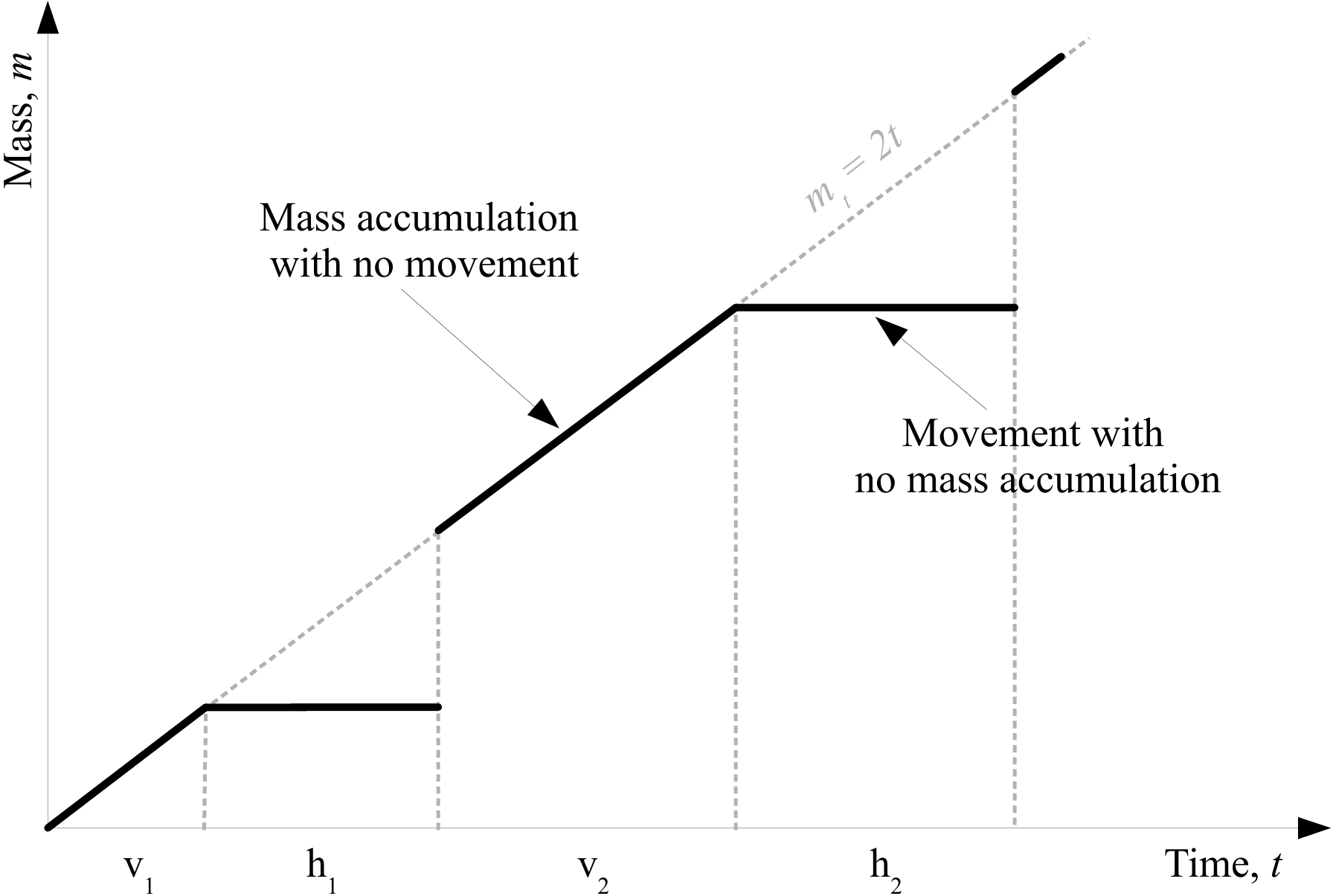}
	\caption{Dynamics of a random sink: an illustration. 
	The trajectory of a sink is partitioned into alternating 
	intervals of mass accumulation of duration $\v_j$ and intervals 
	of movement with no mass accumulation of duration $\h_j$. 
	Each $\v_j$ is an exponential random variable with parameter $\lambda$.
	Each $\h_j$ is distributed as the total length of a critical
	Galton-Watson tree with exponential edge lengths with parameter $\lambda$. }
	\label{fig:mass}
\end{figure}

\begin{thm}[{\bf Growth probability of a random sink, \cite{KZ19}}]
\label{thm:rand_mass}
The probability $\xi(t)$ that a random sink $\cM_0$ is growing at a given 
instant $t>0$ (that is, it is at rest and accumulates mass) is given by 
\begin{equation}\label{eq:growth2}
\xi(t)=e^{-\lambda t}I_0(\lambda t).
\end{equation}
\end{thm}

\begin{thm}[{\bf Mass distribution of a random sink, \cite{KZ19}}]
\label{thm:mass}
The mass of a random sink $\cM_0$ at instant $t>0$ has
probability distribution
\begin{align}
\label{mass}
\mu_t(a)
=  {\lambda \over 2} e^{-\lambda t} &\Big[ I_0\big(\lambda (t-a/2)\big) + I_1\big(\lambda (t-a/2)\big) \Big] \cdot I_0(\lambda a/2) \cdot {\bf 1}_{(0,2t)}(a)\nonumber\\
&+e^{-\lambda t}I_0(\lambda t) \delta_{2t}(a),
\end{align}
where $\delta_{2t}$ denotes Dirac delta function (point mass) at $2t$.
\end{thm}

\begin{Rem}\label{rem:Burgers}
One can notice that the continuum annihilation dynamics of this section, 
with its shock waves, shock wave trees, and sink masses is reminiscent of that in 
the 1-D inviscid Burgers equation that describes the evolution of the velocity field $v(x,t)$:
\be\label{eqn:inviscidBurgers}\index{Burgers equation!inviscid}
\partial_t v(x,t)+v(x,t)\, \partial_x v(x,t)=0,\quad  x\in\mathbb{R}, t\in\mathbb{R}_+.
\ee 
The Burgers dynamics appears in a surprising variety of problems, ranging from cosmology 
to fluid dynamics and vehicle traffic models; see \cite{BK07,FrischBec,Gurbatov} for comprehensive review.
The solution of the Cauchy problem for the Burgers equation 
develops singularities (shocks) that correspond to intersection of individual particles.
The shocks evolve via the shock waves that can be described as massive particles that aggregate the colliding 
regular particles and hence accumulate the mass of the media. 
The dynamics of these massive particles 
generates a tree structure for their world trajectories, the shock wave tree \cite{Bertoin2000,Gurbatov}.

The case of smooth random initial velocity can be treated explicitly via the Hopf-Cole solution.
The case of non-smooth random initial velocities, e.g. a white noise or a (fractional) 
Brownian motion, has been extensively studied, both numerically \cite{She92}
and analytically \cite{Sinai92,Bertoin98,Bertoin2000,Giraud}. 
In this case, tracing the dynamics of the massive particles backward in time (from a 
point within a shock tree to the leaves) corresponds to fragmentation of the mass 
and describes the genealogy of the shocks, i.e., the sets of particles that merge 
with a given massive particle \cite{BernardeauFragmentation,Giraud}.
In particular, it has been established in \cite{Bertoin2000} that 
the shock wave tree for a Brownian motion initial velocity 
becomes the eternal additive coalescent after a proper time change;
similar arguments apply for the L\'evy type initial velocities \cite{Miermont2001}.
However, despite general heuristic understanding of the structure of the Burgers
shock wave tree, a complete analytical description is lacking (e.g., \cite{She92}).
\end{Rem}

\subsection{Real tree description of ballistic annihilation}
\label{sec:Rtree}

Recall that an $\mathbb{R}$-tree is a generalization of 
the concept of a finite tree with edge lengths to infinite spaces;
see Sect.~\ref{sec:Rsetup} for a formal setup.
We construct here (Sect.~\ref{sec:Rfull}) an $\mathbb{R}$-tree $\mathbb{T}=\mathbb{T}(\Psi_0)$ that describes the 
entire model dynamics as coalescence of particles
and sinks; this tree is sketched by gray lines in the top panel of 
Figs.~\ref{fig:shock_tree} and \ref{fig:Rtree}. 
Specifically, the tree consists of points $(x,t)$  such that there exist either 
a particle or a sink with coordinate $x$ at time $t$.
There is one-to-one correspondence between the initial particles $(x,0)$ 
and leaf vertices of $\mathbb{T}$.
Each leaf edge of $\mathbb{T}$ corresponds (one-to-one) to
the free (ballistic) run of a corresponding particle before annihilating in a sink.
Four of such free runs are depicted by green arrows in Fig.~\ref{fig:Rtree}.
The shock wave tree (movement and coalescence of sinks) corresponds to the non-leaf
part of the tree $\mathbb{T}$; it is shown by blue lines in Figs.~\ref{fig:shock_tree},
\ref{fig:Rtree}.
We adopt a convention that the motion of a particle consists of
two parts: an initial ballistic run at unit speed, and subsequent motion within 
a respective sink.
For example, the within-sink motion of particles $x$ and $x'$ is shown by red line in Fig.~\ref{fig:Rtree}.
This interpretation extends motion of all particles to the same 
time interval $[0, t_{\rm max}]$, with $t_{\rm max}$ being the 
time of appearance of the final sink that accumulates the total mass
on the initial interval.
This final sink serves as the tree root.
Section~\ref{sec:Rfull} introduces a proper metric on this space 
so that the model is represented by a time oriented rooted $\mathbb{R}$-tree.
In particular, the metric induced by this tree on 
the initial particles $(x,0)$ becomes an {\it ultrametric},
with the distance between any two particles equal to
the time until their collision (as particles or as respective sinks).

Section~\ref{sec:Rdomain} discusses two non-Lebesgue metrics of the 
system's domain $[a,b]$.
Both describe the ballistic annihilation dynamics and are readily constructed
from the initial potential $\Psi_0(x)$. 
One of these decsriptions is an $\mathbb{R}$-tree and the other is not.
The $\mathbb{R}$-tree description 
establishes an equivalence between the pairs of points that collide
with each other, like the pairs $(x,x')$ and $(y,y')$ in Fig.~\ref{fig:Rtree}.
This tree is isometric to the level set tree $\textsc{level}(-\Psi_0)$ of the 
initial potential that is used in this work to describe the shock wave tree
(Cor.~\ref{cor:V}); it is known in the literature 
as a {\it tree in continuous path} \cite[Def.~7.6]{Pitman},\cite[Ex.~3.14]{Evans2005}. 
In Sect.~\ref{sec:Rprune} we briefly discuss a natural way of introducing 
prunings on $\mathbb{R}$-trees and show that a typical pruning does not have 
the semigroup property.

\begin{figure}[t] 
\centering\includegraphics[width=0.9\textwidth]{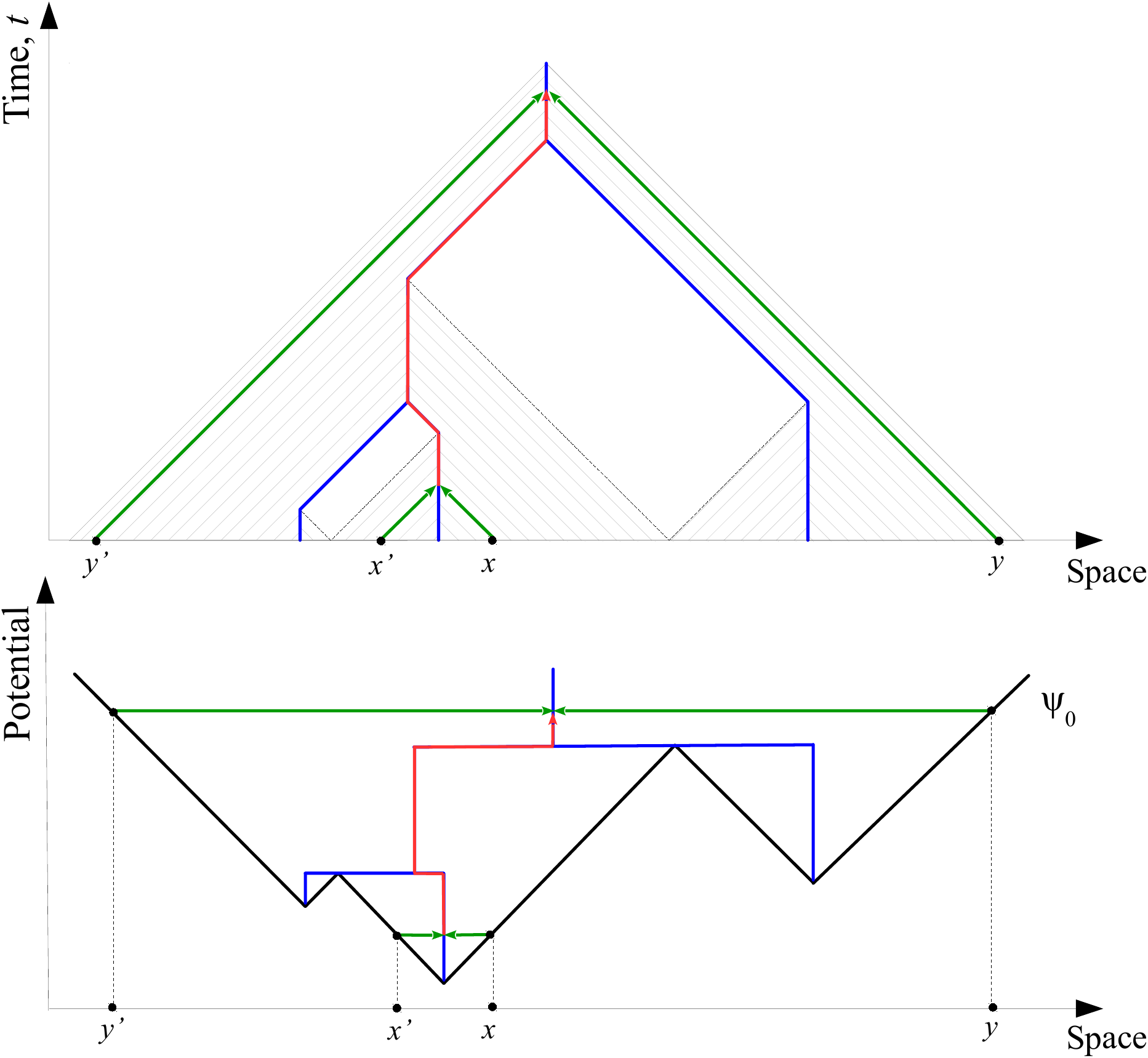}
\caption[Real tree representation of a model with a unit slope potential: an illustration]
{$\mathbb{R}$-tree representation of a ballistic annihilation model with a unit slope 
potential: an illustration. 
Figure illustrates dynamics of four points, $x,x',y,$ and $y'$, marked in the horizontal space axis.
The pairs of points $\{x,x'\}$ and $\{y,y'\}$ collide and annihilate with each 
other.
Green arrows correspond to ballistic runs of points $x,x',y,y'$,
and hence to leaves of tree $\mathbb{T}(\Psi_0)$.
Red line corresponds to the trajectory of points $x,x'$ after their 
collision, within a sink.
The rest of notations are the same as in Fig.~\ref{fig:shock_tree}.
}
\label{fig:Rtree}
\end{figure}

\subsubsection{$\mathbb{R}$-tree representation of ballistic annihilation}
\label{sec:Rfull}
We construct here a real tree 
representation of the continuum ballistic annihilation model of Sect.~\ref{sec:solution}.
Specifically, we assume a unit particle density 
$g(x)\equiv 1$ and initial potential $-\Psi_0(x)\equiv -\psi(x,0)\in{\cE^{\rm ex}}$,
i.e. $\Psi_0(x)$ is a unit slope negative excursion with a finite number of segments 
on a finite interval $[a,b]$ (e.g., bottom panel of Fig.~\ref{fig:shock_tree}). 
Recall that the interval $[a,b]$ completely annihilates by time $t_{\rm max}=(b-a)/2$,
producing a single sink at space-time location $((b+a)/2,t_{\rm max})$.

Consider the model's entire space-time domain
$\mathbb{T}=\mathbb{T}(\Psi_0)$
that consists of all points of the form $(x,t)$, $x\in[a,b]$, $0\le t\le t_{\rm max}$, 
such that there exists either a particle or a sink at location $x$ at time instant $t$. 
The shaded (hatched) regions in the top panels of Figs.~\ref{fig:shock_tree},\ref{fig:tree} 
are examples of such sets of points.
For any pair of points $(x,t)$ and $(y,s)$ in $\mathbb{T}$, we define their unique
{\it earliest common ancestor} as a point 
\[{\sf A}_{\mathbb{T}}((x,t),(y,s)) = (z,w)\in\mathbb{T}\]
such that
$w$ is the infimum over all $w'$ such that 
\[\exists\,z'\,:\, 
\{(x,t),(y,s)\}\in\Delta_{(z',w'),\mathbb{T}}.\]
The length of the unique segment between the points $(x,t)$ and $(y,s)$ is defined as 
\be
\label{T_len}
d\big((x,t),(y,s)\big)=\frac{1}{2}\big((w-t)+(w-s)\big)=\frac{1}{2}(2w-s-t),
\ee
where $w$ is the time component of
$(z,w)={\sf A}_{\mathbb{T}}((x,t),(y,s))$.

The tree $(\mathbb{T},d)$ for a simple initial potential is illustrated in 
the top panel of Fig.~\ref{fig:shock_tree} by gray lines.
The tree has a relatively simple structure.
There is a one-to-one
correspondence between the initial particles $(x,0)$, $x\in[a,b]$, and the leaf vertices
of $\mathbb{T}$.
There is a one-to-one correspondence between the ballistic runs of the initial particles
(runs before collision and annihilation) and the leaf edges of $\mathbb{T}$.
Four of such runs are shown by green arrows in Fig.~\ref{fig:Rtree}.
There is one-to-one correspondence between the sink points $(\sigma(t),t)$ and the non-leaf part
of $\mathbb{T}$.
In particular, the tree root corresponds to the final sink $((a+b)/2,t_{\rm max})$.
The sink points are shown by blue line in Figs.~\ref{fig:shock_tree},\ref{fig:tree}. 
It is now straightforward to check that the tree $(\mathbb{T},d)$ satisfies the
four point condition.

Consider again the sink subspace of $\mathbb{T}$, which consists of the points 
$\{\sigma(t),t)\}$ such that there exists a sink at location $\sigma(t)$ at
time instant $t$, equipped with the distance \eqref{T_len}.
This metric subspace is also a tree, as a connected subspace of 
an $\mathbb{R}$-tree \cite{Evans2005}.
This tree is isometric to 
the shock wave tree $S(\Psi_0)$ and hence to either of its graphical 
representations $\cG^{(x,t)}(\Psi_0)$ or 
$\cG^{(x,\psi)}(\Psi_0)$ that are
illustrated in Figs.~\ref{fig:shock_tree},\ref{fig:tree} (top and bottom panels, respectively).

From the above construction, it follows that all leaves $(x,0)$ are located at the
same depth (distance from the root) $t_{\rm max}$.
To see this, consider the segment that connect a leaf and the root and apply \eqref{T_len}.
Moreover, each {\it time section} at a fixed instant $t_0$, 
${\sf sec}(\mathbb{T},t_0)=\{(x,t_0)\in\mathbb{T}\}$,
is located at the same depth $t_{\max}\!-\!t_0$. 
This implies, in particular, that for any fixed $t_0\ge 0$, the metric induced 
by $\mathbb{T}$ on ${\sf sec}(\mathbb{T},t_0)$ is an {\it ultrametric}, which means 
that $d_1(p,q)\le d_1(p,r)\vee d_1(r,q)$ for 
any triplet of points $p,q,r\in{\sf sec}(\mathbb{T},t_0)$.\index{ultrametric}
Accordingly, each triangle $p,q,r\in{\sf sec}(\mathbb{T},t_0)$
is an {\it isosceles}, meaning that at least two of the three pairwise distances 
between $p,q$ and $r$ are equal and not greater than the third \cite[Def.~3.31]{Evans2005}.\index{isosceles}
The length definition \eqref{T_len} implies that the distance between
any pair of points from any fixed section ${\sf sec}(\mathbb{T},t_0)$ equals
the time until the two points (each of which can be either a particle or a sink) 
collide.

We notice that the collection of leaf vertices $\Delta^{\circ}_{p,\mathbb{T}}$ descendant 
to a point $p\in\mathbb{T}$ can be either a single point $(x_p,0)$, if $p$ is within
a leaf edge and represents the ballistic run of a particle, or 
an interval $\{(x,0): x_{\rm left}(p) \le x \le x_{\rm right}(p)\}$,
if $p$ is a non-leaf point that represents a sink.
We define the {\it mass} $m(p)$ of a point $p\in\mathbb{T}$ as
\[m(p)=\int\limits_{x:(x,0)\in\Delta^{\circ}_{p,\mathbb{T}}}g(z)dz 
= x_{\rm right}(p)-x_{\rm left}(p),\]
where the last equality reflects the assumption $g(z)\equiv 1$.
The mass $m(p)$ generalizes the quantity ``number of descendant leaves'' 
(Sect.~\ref{ex:numL}) to the $\mathbb{R}$-tree situation
with an uncountable set of leaves.
We observe that
(i) a point $p\in\mathbb{T}$ represents a ballistic run if and only if $m(p)=0$;
(ii) a point $p\in\mathbb{T}$ represents a sink if and only if $m(p)>0$. 
This means that the shock wave tree, which is isometric to
the sink part of the tree $(\mathbb{T},d)$, can be extracted from  
$(\mathbb{T},d)$ by the condition $\{p:m(p)>0\}$.

\subsubsection{Metric spaces on the set of initial particles}
\label{sec:Rdomain}
In this section we discuss two metrics on the system's domain $[a,b]$,
which is isometric to the set $\{(x,0): x\in[a,b]\}$ of initial particles.
These spaces contain the key information about 
the system dynamics and, unlike the complete tree $(\mathbb{T},d)$ of Sect.~\ref{sec:Rfull},  
can be readily constructed from the potential $\Psi_0(x)$.
 
Metric $h_1(x,y)$ reproduces the ultrametric induced by 
$(\mathbb{T},d)$ on $[a,b]$.
Below we explicitly connect this metric to $\Psi_0(x)$. 
For any pair of points $x,y\in[a,b]$ 
we define a basin ${\sf B}_{\Psi_0}(x,y)$ as the interval that supports the 
minimal negative excursion within $\Psi_0(x)$ 
that contains the points $x,y$.
Formally, assuming without loss of generality that $x<y$ we find the maximum 
of $\Psi_0$ on $[x,y]$:
\[{\sf m}_{\Psi_0}(x,y) = \sup\limits_{z\in[x,y]}\Psi_0(z)\]
and use it to define the basin
${\sf B}_{\Psi_0}(x,y) = [l,r],$
where
\[l = \sup\{z:z\le x, \Psi_0(z)\ge {\sf m}_{\Psi_0}(x,y)\},\]
\[r = \inf\{z:z\ge y, \Psi_0(z)\ge {\sf m}_{\Psi_0}(x,y)\}.\]
The metric is now defined as
\[h_1(x,y)=\frac{1}{2}|{\sf B}_{\Psi_0(x,y)}|.\]
It is straightforward to check that 
\[h_1(x,y) = \text{the time until collision of the particles } (x,0) \text{ and } (y,0),\]
where the collision is understood as either collision of particles, collision of
sinks that annihilated the particles, or collision between a sink that annihilated
one of the particles 
and the other particle.
For instance, the claim is readily verified, by examining the bottom panel of Fig.~\ref{fig:Rtree},
for any pair of points from the set $\{x,x',y,y'\}$.
The metric space $([a,b],h_1)$ is not a tree.
Moreover, this space is {\it totally disconnected}, since there only exists a finite number
of points (local minima of $\Psi_0(x)$) that have a neighborhood of arbitrarily small size.
Any other point at the Euclidean distance $\epsilon$ from the nearest local minimum is separated
from other points by at least $\epsilon/2$.

Metric $h_2(x,y)$ describes the mass accumulation by sinks during the annihilation process.
Specifically, we introduce an equivalence relation among the annihilating
particles, by writing $x\sim_{\Psi_0}y$ 
if the particles with initial coordinates $x$ and $y$ collide and annihilate with each other.
For example, in Fig.~\ref{fig:Rtree} we have $x \sim_{\Psi_0} x'$ 
and $y \sim_{\Psi_0}y'$.
The following metric is now defined on the quotient space $[a,b]|_{\sim_{\Psi_0}}$:
\[h_2(x,y) = 2\sup\limits_{z\in[x,y]}\left[\Psi_0(z)\right]-\Psi_0(x)-\Psi_0(y).\]
In words, the distance $h_2(x,y)$ between particles $x$ and $y$ equals the 
total mass accumulated by the sinks to which the particles belong during the
time intervals between the instants when the particles joined the respective
sinks and the instant of particle (or respective sink) collision.
Another interpretation is that $h_2(x,y)$ equals to the minimal Euclidean 
distance between points $x,y\in[a,b]|_{\sim_{\Psi_0}}$ in the quotient space; one can travel in this quotient space
as along a regular real interval, with a possibility to jump (with no distance accumulation)
between equivalent points.  
This $\mathbb{R}$-tree construction is know as the {\it tree in continuous path}
\cite[Def.~7.6]{Pitman},\cite[Ex.~3.14]{Evans2005}.

The metric space $([a,b]|_{\sim_{\Psi_0}},h_2)$ is a tree that is 
isometric to the level set tree of the potential $\Psi_0(x)$ on $[a,b]$
and hence to the (finite) shock wave tree $\cV(\Psi_0)$ (by Cor.~\ref{cor:V}), with 
the convention that the root is placed 
in $a\sim_{\Psi_0} b$.
This means, in particular, that prunings of these two trees,  with the same pruning function
and pruning time, coincide.

\subsubsection{Other prunings on $\mathbb{T}$}
\label{sec:Rprune}
One can introduce a large class of prunings on an $\mathbb{R}$-tree $(\mathbb{T},d)$ 
following the approach used above to define the point mass $m(p)$.
Specifically, consider a measure $\eta(\cdot)$ on $[a,b]$
and define $m_{\eta}(p) = \eta(\Delta^{\circ}_{p,\mathbb{T}})$.
The function $m_{\eta}(p)$ is nondecreasing along each segment
that connect a leaf and the root $\rho_{\mathbb{T}}$ of $\mathbb{T}$.
Hence, one can define a pruning with respect to $m_{\eta}$ on $\mathbb{T}$
by cutting all points $p$ with $m_{\eta}(p)<t$
for a given $t\ge 0$.
It is readily seen that the function $m_{\eta}(p)$ typically has discontinuities along a
path between a leaf and the root of $\mathbb{T}$.
This means that pruning with respect to $m_{\eta}$ typically does not
have the semigroup property.

\section{Infinite trees built from leaves down}\label{sec:infiniteT}
Examples of {\it infinite trees built from the root up} are plentiful; 
they include the infinite trees induced by the Yule processes or any other birth processes; 
infinite trees generated by a supercritical branching process; 
the trees that represent depth-first search and breadth-first search algorithms on infinite networks.
In this section we explore the {\it infinite trees built from leaves down} that arise naturally in the context of infinitely many coalescing particles or the level set trees of continuous functions. Interestingly, many of the results about finite trees can be obtained from the 
characterizations of the corresponding infinite trees built from leaves down.

\subsection{Infinite plane trees built from the leaves down}

In the context of Sect.~\ref{sec:level}, set $I=\mathbb{R}$ and 
consider a function $f(x) \in C(\mathbb{R})$.
Let $\mathcal{X}$ and $\mathcal{Y}$ be the sets containing all locations of local 
minima and local maxima of $f(x)$, respectively. 
Formally, $x_0 \in \mathcal{X}$ if $\exists \delta>0$ s.t. $f(x)\geq f(x_0)$ 
$\forall x \in (x_0-\delta,x_0+\delta)$, and $\mathcal{Y}$ is defined analogously.
Hence, the local extrema may include plateaus of constant values.
We assume that $f(x)$ satisfies the following conditions:
\begin{itemize}
  \item[{\bf (a)}] The set $\mathcal{X}$ of the locations of local minima has infinite image, i.e.,
  $$|f\big(\mathcal{X}\big)|=\infty.$$
  This condition guarantees that the level set tree of $f(x)$ that we construct below has 
  an infinite number of vertices.
  \item[{\bf (b)}] The intersection of $\mathcal{X}$ with any finite interval $[a,b]$ 
  is either empty or consists of a finite number of closed intervals 
  (possibly including separate points).
  This condition guarantees that every descendant subtree of the infinite level set tree of 
  $f(x)$ is finite.
  The conditions {\bf (a),(b)} guarantee that the level set tree has countably many 
  vertices.
  \item[{\bf (c)}] $\forall a \in \mathbb{R}$, the sets
  $$(a,\infty) \cap  f^{-1}\Big(\inf_{(a,\infty)}f(x)\Big) \quad \text{ and } \quad (-\infty,a) \cap f^{-1}\Big(\inf_{(-\infty,a)}f(x)\Big)$$
  are empty or consist of finitely many closed intervals (including separate points). 
  Here, $f^{-1}(-\infty)$ is an empty set.
  This condition, or equivalent, guarantees that the level set tree has finite branching
  (no vertices of infinite degree). 
   \end{itemize}

\medskip
\noindent
Recalling the construction in Sect.~\ref{sec:level2},  the level set tree $T_\infty=\textsc{level}\big(f(x) \big)$ has infinitely many leaves.
There, $T_\infty=\Big(\mathbb{R}/\!\sim_f, \,d_f \Big)$  is a metric quotient space obtained with respect to identification (denoted by  $a_\ell \sim_f a_r$) of pairs of points $a_\ell$ and $a_r$ in $\mathbb{R}$ as one point.  
Recall that we have $a_\ell \sim_f a_r$ whenever the following conditions are satisfied
\begin{enumerate}
  \item $a_\ell < a_r$ and $f(a_\ell)=f(a_r)$;
  \item $\forall x \in (a_\ell, a_r)$ we have  $f(x) \geq f(a_\ell)=f(a_r)$.
\end{enumerate}
The local maxima $\mathcal{Y}$ (including plateaus) constitute the leaves in $T_\infty$, and the local minima $\mathcal{X}$ (including plateaus) constitute the internal vertices (junctions) in $T_\infty$.   Such $T_\infty$ is also called an {\it infinite plane tree built from the leaves down} induced by function $f(x)$. The reason for the name being that as we study $f(x)$ over larger and larger intervals (e.g. $[-a,a]$ as $a \rightarrow \infty$) we discover more and more leaves of $T_\infty$ (local maxima) and their merger history (local minima) from leaves down, but never reaching the root. 

\begin{figure}[!t]
	\centering
	\includegraphics[width=\textwidth]{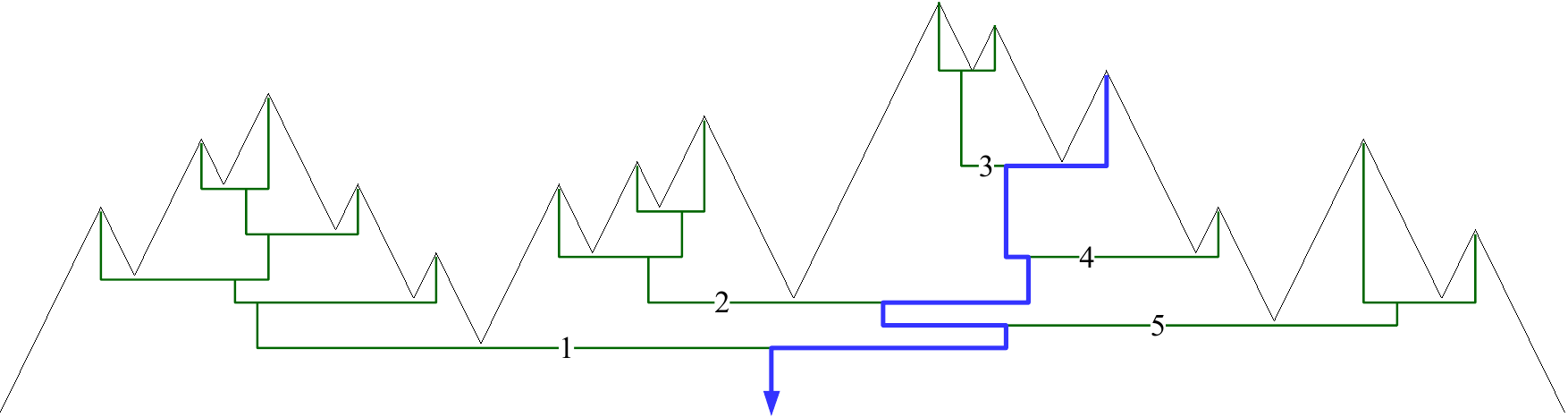}
	\caption{Golden lineage representation of a level set tree: an illustration.
	The figure shows a part of the level set tree for a piece-wise linear function on $\mathbb{R}$. 
	The initial part of the golden lineage is shown in blue.
	There exist five finite binary trees attached to the golden lineage -- they are shown in green.
	The tree index is shown along the stem of each tree.
	The trees indexed $1,2,3$ have left orientation, and the trees indexed $4,5$ have right orientation
	with respect to the golden lineage.  
}
	\label{fig:golden}
\end{figure}

\medskip
\noindent
To give a convenient description of an infinite tree $T_\infty$ built from leaves down, we designate one leaf as the {\it golden leaf}, and its ancestral lineage is called the 
{\it golden lineage} (Fig.~\ref{fig:golden}).  
In the above construction, we let the leaf that corresponds to the first 
local maximum in the nonnegative half-line,
$$\min\{x \in \mathcal{Y} \,:x \geq 0\},$$ 
 to be designated as the golden leaf. 
Let $\L^\infty$ denote the space of infinite plane trees built from the leaves down, with edge lengths and designated golden leaf.
For a tree $T_\infty\in\L^\infty$ with a designated golden leaf $\gamma^*$, we let 
$\ell=[\gamma^*,\phi]$ denote the unique ancestral path from the golden leaf $\gamma^*$ to its parent, grandparent, great-grandparent and on towards the tree root $\phi$, where $\phi$ is a {\it point at infinity}. 
Here, the ancestral path $\ell$ will be called the {\it golden lineage}. 
The golden lineage $\ell=\left\{\ell(i),e(i)\right\}$ consists of infinitely many vertices $\ell(i)$  that we enumerate by the index $i\ge 0$ along the path, starting from the golden leaf $\ell(0)=\gamma^*$ and increasing as we go down the golden lineage $\ell$, and infinitely many edges $e(i)=[\ell(i),\ell(i+1)]$.

Each tree $T_\infty \in\L^\infty$ can be represented as a forest of finite trees 
attached to the golden lineage $\ell$ as follows
\be\label{eqn:forest_infinite}
T_\infty=\left(\ell, \big\{\cD_i,\sigma_i \big\}_{i\ge 1} \right),
\ee 
where for each $i \geq 1$, $\cD_i=\Delta_{\ell(i)} \in \L^|$ denotes the complete 
subtree of $T_\infty$ rooted at $\ell(i)$ that does not include the golden leaf, 
and $\sigma_i \in \{-1,+1\}$ denotes the left-right orientation 
of $\cD_i$ with respect to the golden lineage $\ell$.
Figure~\ref{fig:golden} illustrates this construction. 

\medskip
\noindent
The representation \eqref{eqn:forest_infinite} of a tree $T_\infty \in \L^\infty$
allows one to relate the space $\L^\infty$ of infinite planar trees built from the leaves 
down with edge lengths and a designated golden leaf
to the notion of a {\it forest of trees attached to the floor line} described in Sect.~7.4 of \cite{Pitman}.
In addition, the golden lineage construct helps at meterizing the space $\L^\infty$.

\medskip
\noindent
Importantly,  for any point $x \in T_\infty$,  the descendant tree $\Delta_{x,T}$ is a finite tree in $\L$.
Therefore, the definition of generalized dynamical pruning \eqref{eqn:GenDynamPruning} 
extends naturally to the space $\L^\infty$ of infinite plane trees built from the leaves down. 
Applying the generalized dynamical pruning $\S$ to an infinite tree built from the leaves down, the uppermost point of the golden lineage within $\S(\varphi,T)$ will become the golden leaf for the pruned tree $\S(\varphi,T)$.

\medskip
\noindent
Next, we extend the notion of prune-invariance in planar shapes from Def. \ref{def:pi2}(i) to a subspace $S^\infty$ of the space $\L^\infty$.
Consider a subspace $S^\infty$ of $\L^\infty$. For a given monotone nondecreasing function $\varphi:\L\rightarrow\mathbb{R}^+$,
consider generalized pruning dynamics $\cS_t(\varphi,T_\infty)$ ($T_\infty \in S^\infty$).
 We say that a probability measure $\mu$ on $S_\infty$ is {\it prune-invariant in planar shapes} if
\be\label{eqn:pushforwardInf}
\mu(A)=\mu_t(A) \quad \forall t\geq 0, A \in\Sigma,
\ee
where $\mu_t=(\S)_*(\mu)=\mu\circ\S^{-1}$ is the pushforward measure, and $\Sigma$ is the induced $\sigma$-algebra.

The above definition of prune invariance \eqref{eqn:pushforwardInf} is significantly different from the original Def. \ref{def:pi2}(i) for finite trees as 
$\phi \not\in S^\infty$ and we do not need to condition on the event $\S(\varphi,T)\not= \phi$ in the pushforward measure. 
Importantly, the prune-invariance in \eqref{eqn:pushforwardInf} coincides with the John Von Neumann \cite{vonNeumann} definition of the {\it invariant measure},
fundamental for ergodic theory and dynamical systems. At the same time, the definition of {\it prune-invariance in edge lengths} Def. \ref{def:pi2}(ii) does not need to be reformulated any differently for the infinite trees built from leaves down.

\medskip
\noindent
The renown Krylov-Bogolyubov theorem \cite{KrylovBogolyubov} states that for a compact metrizable topological space $\Omega$ with the induced Borel $\sigma$-algebra $\Sigma$, and a continuous function $\mathcal{S}:\Omega \rightarrow \Omega$, there exists an {\it invariant} probability measure $\mu$ on $(\Omega,\Sigma)$ satisfying
\be\label{eqn:pushforwardKB}
\mu(A)=\mu_*(A) \quad \forall A \in\Sigma,
\ee
where $\mu_*=(\mathcal{S})_*(\mu)=\mu\circ\mathcal{S}^{-1}$ is the pushforward measure.

Here we will not concentrate on constructing a suitable metric for the space $\L^\infty$. 
However, in the spirit of the Krylov-Bogolyubov theorem, we will
show in Thm.~\ref{thm:infmain} that the infinite critical planar binary Galton-Watson tree ${\sf GW}_\infty (\lambda)$ built from the leaves down 
that we construct in Sect. \ref{sec:infGWtree} is prune invariant under generalized dynamical pruning $\cS_t$
induced by a monotone nondecreasing function $\varphi:\L\rightarrow\mathbb{R}^+$.
Additionally, it will be observed that Thm.~\ref{thm:infmain} is a generalization of
Thm.~\ref{thm:main}.

\subsection{Infinite exponential critical binary Galton-Watson tree built from the leaves down}\label{sec:infGWtree}

Consider a Poisson point process $\{T_k\}_{k \in \mathbb{Z}}$ on $\mathbb{R}$ with parameter 
$\lambda/2$, enumerated from left to right (where $T_0$ is the epoch closest to zero).
Let
$$X_t=\begin{cases}
  \sum\limits_{j=1}^k (-1)^{j+1} \big(T_j -T_{j-1}\big) +(-1)^k (T_k-t)    & \text{ if } t \in [T_{k-1},T_k),~ k \geq 1, \\
   \sum\limits_{j=k+1}^0 (-1)^{j+1} \big(T_j -T_{j-1}\big) +(-1)^k (T_k-t)    & \text{ if } t \in [T_k,T_{k+1}),~ k \leq -1.
\end{cases}$$
In other words, $X_t$ is a continuous piecewise linear function with 
slopes alternating between $\pm 1$ as it crosses the Poisson epochs 
$\{T_k\}_{k \in \mathbb{Z}}$, i.e., the slope
$${d \over dt}X_t=\begin{cases}
  -1    & \text{ if } t \in (T_{k-1},T_k),~ k \text{ even}, \\
  +1    & \text{ if } t \in (T_{k-1},T_k),~ k \text{ odd}.
\end{cases}$$
The level set tree $T_\infty=\textsc{level}\big(X_t\big)$ is invariant under shifting $X_t$ vertically, or shifting and scaling $X_t$ horizontally. 

\medskip
\noindent
Fix a point $t^* \in \mathbb{R}$ and generate $X_t$ with a Poisson point process $\{T_k\}_{k \in \mathbb{Z}}$. 
Then, with probability one, there will be a positive excursion of $X_t-X_{t^*}$ over an interval that begins or ends at $t^*$.
By Thm.~\ref{Pit7_3}, the level set tree of this adjacent positive excursion is distributed as ${\sf GW}(\lambda)$. 
Therefore, the infinite binary level set tree $T_\infty =\textsc{level}\big(X_t\big)$ for $X_t$ will be referred to as the {\it infinite planar exponential critical binary Galton-Watson tree built from the leaves down} with parameter $\lambda$, and denoted by ${\sf GW}_\infty (\lambda)$. 
We also refer to this tree as the infinite exponential  critical binary Galton-Watson tree.

\medskip
\noindent
In the representation \eqref{eqn:forest_infinite} of a tree $T_\infty \stackrel{d}{\sim}{\sf GW}_\infty(\lambda)$, the golden lineage $\ell$ is distributed 
as a one-dimensional Poisson process with parameter $\lambda$, the orientation variables $\sigma_i$ are i.i.d. Bernoulli with parameter $1/2$, and 
the complete subtrees $\cD_i$ are i.i.d. ${\sf GW}(\lambda)$ trees. Finally,  the golden lineage $\lambda$ and the sequences, $\sigma_i$ and $\cD_i$, 
are all sampled independently of each other.

\begin{figure}[!t]
	\centering
	\includegraphics[width=0.8\textwidth]{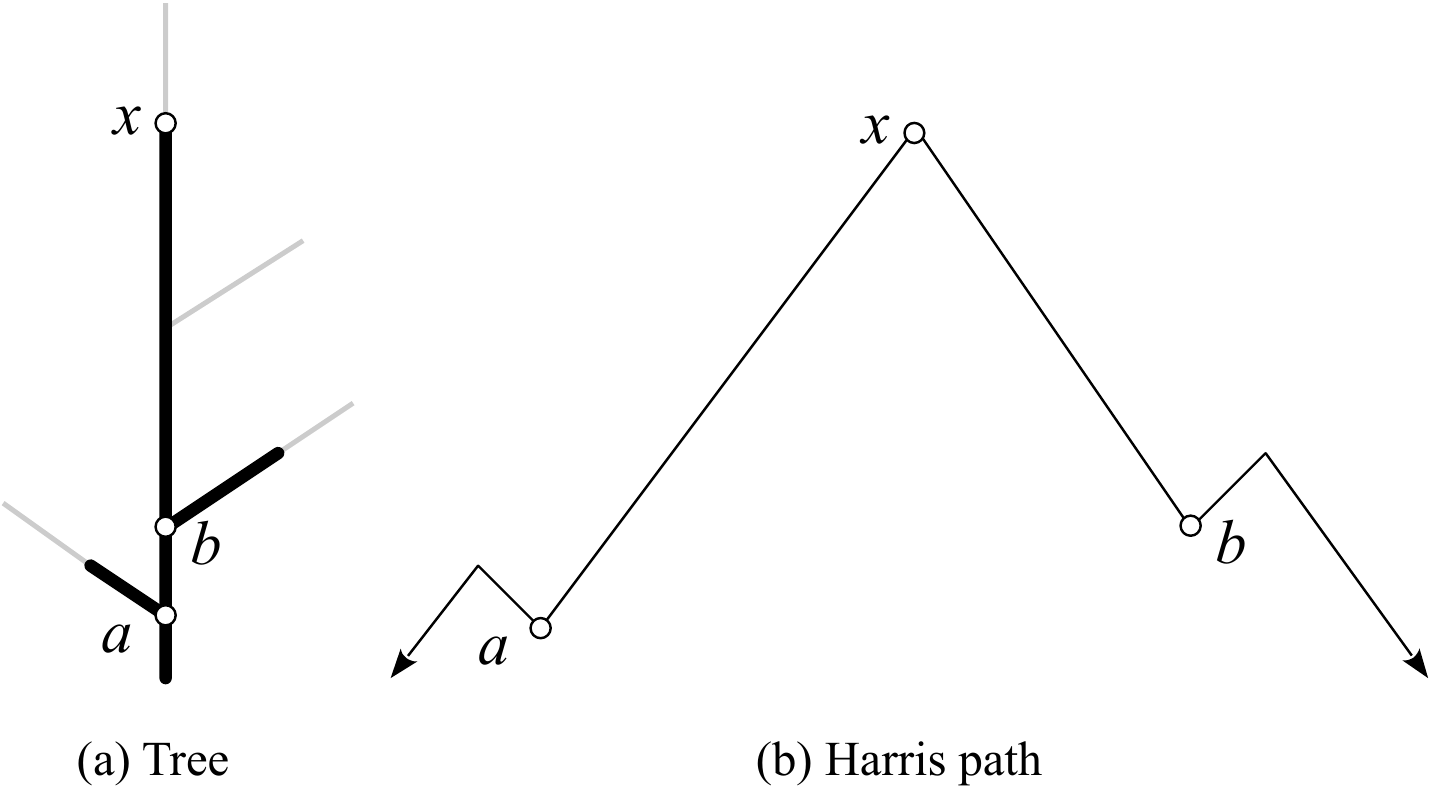}
	\caption{Illustration to the proof of Thm.~\ref{thm:infmain}. 
	(a) A fragment of the pruned tree $T_\infty^\Delta$ (black). 
	Parts of the pruned out subtrees of the initial tree $T_\infty$ are shown in gray.
	The point $x$ is a leaf of $T_\infty^\Delta$; it has the left parent $a$ and 
	right parent $b$ in $T_\infty^\Delta$.
	The same $x$, as an edge point in $T_\infty$, has the same left parent $a$ in $T_\infty$,
	but its right parent in $T_\infty$ is different from $b$. 
	(b) A part of the Harris path for the pruned tree $T_\infty^\Delta$.
	The proof of Thm.~\ref{thm:infmain} finds that the increments of the segments $|xa|$
	and $|xb|$ are i.i.d. exponential r.v.s. with parameter $\lambda p_\Delta/2$.
	}
	\label{fig:inf}
\end{figure}

\medskip
\noindent
The following is a variation of Thm.~\ref{thm:main} for the infinite critical exponential binary Galton-Watson tree.
\begin{thm} \label{thm:infmain}
Let $T_\infty \stackrel{d}{\sim}{\sf GW}_\infty(\lambda)$ 
with $\lambda>0$.
Then, for any monotone nondecreasing function $\varphi:\BL^|\rightarrow\mathbb{R}_+$ and
any $\Delta>0$ we have
\[T_\infty^\Delta:=\cS_\Delta(\varphi,T_\infty) \stackrel{d}{\sim} {\sf GW}_\infty\big(\lambda p_{\Delta}(\lambda,\varphi)\big),\]
where 
$$p_\Delta:=p_\Delta(\lambda,\varphi)={\sf P}(\cS_\Delta(\varphi,T) \not= \phi) 
\quad\text{for}\quad T\stackrel{d}{\sim}{\sf GW}(\lambda).$$ 
That is, the pruned tree $T_\infty^\Delta$ is also an infinite exponential critical binary Galton-Watson tree with the scaled parameter
$$\mathcal{E}_\Delta(\lambda,\varphi)=\lambda p_{\Delta}(\lambda,\varphi).$$
\end{thm}
\noindent Notice that since we are dealing with an infinite tree $T_{\infty}$, we do not need to be concerned about it surviving under the pruning operation $\cS_\Delta$.
The survival probability $p_{\Delta}$ used in the statement of Thm.~\ref{thm:infmain} is computed for finite trees,
so the values of scaled parameter $\mathcal{E}_\Delta(\lambda,\varphi)$
for selected pruning functions are given by Thm.~\ref{pdelta}.
\begin{proof}
Let ${\sf par}(x)$ denote the {\it right parent} to a point $x$ in $T_\infty$.
This means that the vertex ${\sf par}(x)$ is the parent of the first right subtree 
that one meets when travels the tree $T_{\infty}$ from $x$ down to the root.
In the Harris path of $T_{\infty}$, there exist two points that correspond to $x$
(they merge into a single point when $x$ is a leaf).
Consider the rightmost of these points, $r_x$, which belongs to a downward 
increment of the Harris path.
The vertex ${\sf par}(x)$ corresponds to the 
nearest right local minima of $r_x$.
Similarly, we let ${\sf par}_\Delta(\cdot)$ denote the right parent on $T_\infty^\Delta$.

Consider a leaf $x \in T_\infty^\Delta$, which is also a point in $T_\infty$; see Fig.~\ref{fig:inf}(a). 
We now find the distribution of the distance from $x$ to ${\sf par}_\Delta(x)$,
i.e., the length of the respective downward segment of the Harris path; see Fig.~\ref{fig:inf}(b).
Consider the descendant lineage of $x$ in $T_\infty$, which consists of vertices
$$x_1={\sf par}(x), ~~x_2={\sf par}(x_1),~~x_3={\sf par}(x_2),~\hdots~.$$
Due to the memorylessness property of exponential distribution, and the symmetry of left-right
orientation of subtrees in $T_{\infty}$, the distance from $x$ down to ${\sf par}(x)$ has
exponential distribution with rate $\lambda/2$.
The point $x$ belongs to one (left) of the two complete subtrees rooted 
at ${\sf par}(x)$ in $T_{\infty}$. 
Observe that ${\sf par}_\Delta(x)={\sf par}(x)$ if and only if the subtree that does not
contain $x$ (we call it sibling subtree) has not been pruned out completely, i.e., 
the intersection of the sibling subtree with $T_\infty^\Delta$ is not empty. 
(In the example of Fig.~\ref{fig:inf}(a), we have ${\sf par}_{\Delta}(x) = x_2\equiv b$.)
The sibling subtree is known to be distributed as ${\sf GW}(\lambda)$.
Therefore,
$${\sf P}\big({\sf par}_\Delta(x)=x_1\big)=p_\Delta.$$
Iterating this argument, we have for $k \geq 1$,
$${\sf P}\big({\sf par}_\Delta(x)=x_k ~\big|~ {\sf par}_\Delta(x)\not=x_1,\hdots, {\sf par}_\Delta(x)\not=x_{k-1}\big)=p_\Delta (1-p_\Delta)^{k-1}.$$
Therefore, the distance from a vertex $x$ down to ${\sf par}_\Delta(x)$ 
is a geometric ${\sf Geom}_1(p_\Delta)$
sum of independent exponential random variables with parameter $\lambda/2$.
Hence, it is itself an exponential random variable with parameter $\lambda p_\Delta/2$.
In other words, the downward segment of the Harris path 
of the pruned tree $T_\infty^\Delta$ adjacent to the local maximum that corresponds
to the leaf $x$ has exponential lengths with parameter $\lambda p_\Delta/2$; see Fig.~\ref{fig:inf}(b).

The same argument (using left parents) shows that the upward segment of the Harris path 
of the pruned tree $T_\infty^\Delta$ adjacent to the local maximum that corresponds
to the leaf $x$ has exponential lengths with parameter $\lambda p_\Delta/2$.
The lengths of the upward and downward segments are independent; see Fig.~\ref{fig:inf}(b).

Applying the above argument to all leaves in $T_\infty^\Delta$, we conclude that the 
Harris path of $T_\infty^\Delta$ consists of alternating up/down increments with 
independent lengths, distributed exponentially with the parameter $\lambda p_\Delta/2$.
Theorem~\ref{Pit7_3} states that in this case $T_\infty^\Delta$ is an exponential critical binary
Galton-Watson tree with parameter $\lambda p_\Delta$.
This completes the proof.
\end{proof}
Observe that Thm.~\ref{thm:main} can be obtained from Thm.~\ref{thm:infmain} by 
considering finite excursions of $X_t$.
Also notice that for the particular case of Horton pruning (Sect.~\ref{ex:H}), 
the statement of Thm.~\ref{thm:infmain} follows from Thm.~\ref{eH}.

\subsection{Continuum annihilation}
One can observe that the continuum annihilation dynamics 
that begins with an infinite exponential potential $\Psi^{\rm exp}_0(x)$,
$x\in\mathbb{R}$ (see Sect.~\ref{sec:rand_mass}), is nothing but the generalized dynamical pruning 
$\cS_t(\varphi,T_\infty)$ of the infinite planar critical exponential binary Galton-Watson 
tree built from the leaves down 
$$T_\infty:=\textsc{level}\big(-\Psi^{\rm exp}_0(x)\big)  \stackrel{d}{\sim}{\sf GW}_\infty(\lambda),$$ 
where $\varphi(T)=\textsc{length}(T)$ for $T \in \BL^|$.
Moreover, the key results of Sect. \ref{sec:rand_mass}, Thms.~\ref{thm:rand_mass} and \ref{thm:mass},
that describe the growth dynamics of a sink in the continuum annihilation model are in fact describing the length distributions of pruned out sections 
of $T_\infty \stackrel{d}{\sim}{\sf GW}_\infty(\lambda)$ under the generalized dynamical pruning 
$\cS_t(\varphi,T_\infty)$. 
The proofs of these results can be rewritten in the infinite tree style of Thm.~\ref{thm:infmain}.

\begin{figure}[!t]
	\centering
	\includegraphics[width=0.9\textwidth]{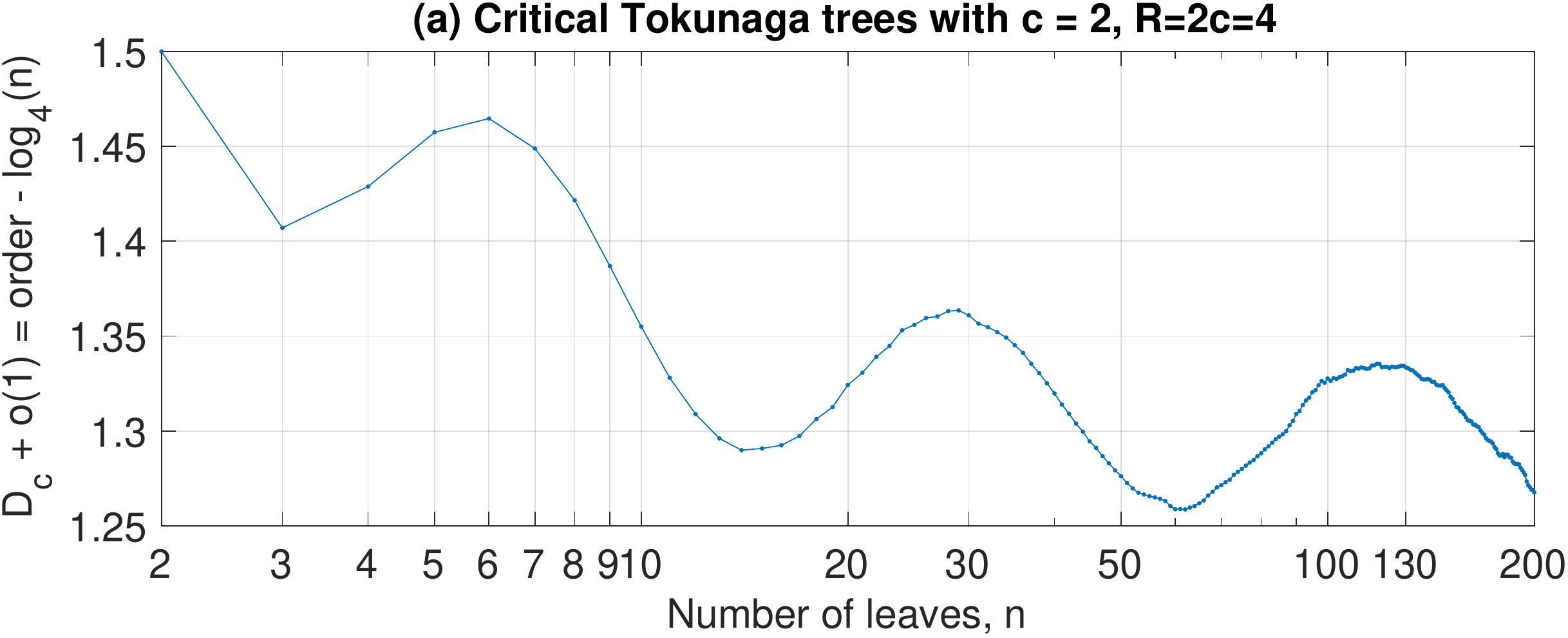}
	\includegraphics[width=0.9\textwidth]{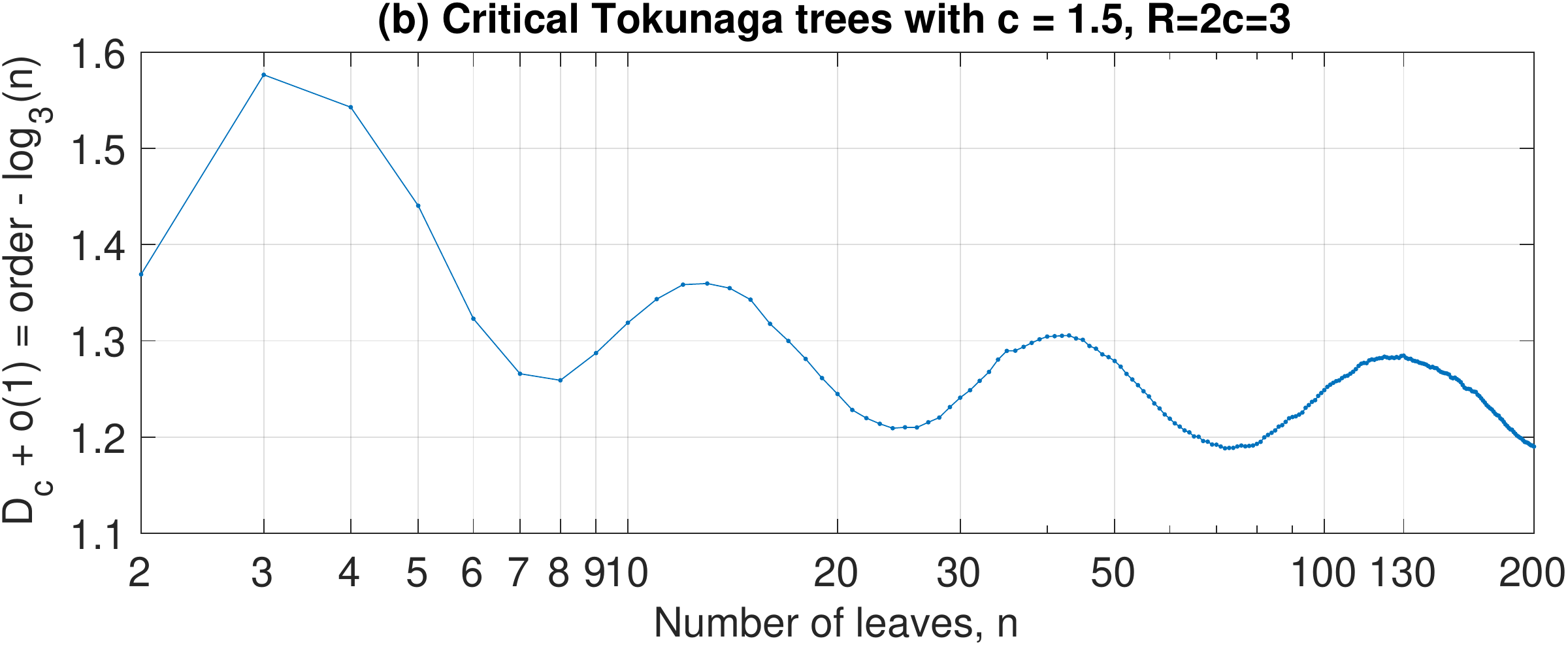}
	\caption{Periodic fluctuations of the average order of a 
	critical Tokunaga tree $S^{\rm Tok}(t;c,\gamma)$
	as a function of the tree size (number $n$ of leaves): numerical simulations.
	(a) Critical Tokunaga tree with $c=2, R=2c=4$ (critical binary Galton-Watson tree); 
	see \eqref{eqn:Flajolet}.
	(b) Critical Tokunaga tree with $c = 1.5, R=2c=3$; see \eqref{eqn:FlajoletTok}.
	The jitter at hight values of $n$ is due to stochastic variability in our numerical simulations. 
	}
	\label{fig:D}
\end{figure}

\section{Some open problems}\label{open}
\begin{enumerate}

\item
Consider the cumulative distribution function ${\sf H}_n(x)$ for the height of an exponential critical binary Galton-Watson tree ${\sf GW}(\lambda)$  
(Def. \ref{def:cbinary}) conditioned on having $n$ leaves; see \eqref{eqn:HnDef} of Sect. \ref{sec:heightGW}.
Can one derive the limit \eqref{eqn:WaymireKolchin} from the equation \eqref{eqn:HnDEsolh}?  

\medskip
\item 
For a given sequence $\{T_k\}_{k \in \mathbb{Z_+}}$
of positive real numbers, construct a coalescent process whose symmetric kernel is a function of the clusters' Horton-Strahler orders,
in such a way that the combinatorial part of the coalescent tree is mean self-similar with respect to Horton pruning
(Defs.~\ref{ss1} and \ref{def:ss2}), with Tokunaga coefficients $\{T_k\}$.
This would complement an analogous branching process construction of Sect.~\ref{HBP}.

\medskip
\item
Generalize equation \eqref{eqn:Flajolet} of Flajolet et al. \cite{FRV79} for the critical Tokunaga processes (Sect. \ref{sec:Tok}).
Formally, consider a tree $T$ that corresponds to a critical Tokunaga process $S^{\rm Tok}(t;c,\gamma)$ (Def. \ref{def:TokProcess}).
Establish the following generalization of \eqref{eqn:Flajolet}:
for any given $c>1$, there exists a periodic function $D_c(\cdot )$ of period one such that
\be\label{eqn:FlajoletTok}
{\sf E}\big[{\sf ord}(T)\,\big| \#T=2n-1\big]=\log_R{n}+D_c\big(\log_R{n} \big)+o(1)
\ee
as $n \rightarrow \infty$, where $R=2c$.
We confirmed the validity of \eqref{eqn:FlajoletTok} numerically; see Fig.~\ref{fig:D}.

\medskip
\item
For a hierarchical branching process $S(t)$ (Def. \ref{def:HBP}, Sect. \ref{HBP:def}),
describe the correlation structure of its Harris path.
A special case is given by Thm.~\ref{Pit7_3}; it shows
that the Harris path of the exponential critical binary Galton-Watson 
tree ${\sf GW}(\lambda)$,
which corresponds to the hierarchical branching process $S(t)\stackrel{d}{\sim}S^{\rm Tok}(t;c,\gamma)$
(Sect.~\ref{sec:Tok}), 
is 
an excursion of the exponential random walk
(Sect. \ref{sec:erw}), 
with parameters $\left\{{1 \over 2}, \lambda, \lambda \right\}$.

\medskip
\item
Recall that a rescaled Harris path of an exponential critical binary Galton-Watson tree  ${\sf GW}(\lambda)$ 
converges to the excursion of a standard Brownian motion \cite{LeGall93,NP}.
For a hierarchical branching process $S(t)$ (Def. \ref{def:HBP}, Sect. \ref{HBP:def}), 
explore the existence of a proper infinite-tree limit and the respective limiting 
excursion process.

\medskip
\item
Prove the following extension of Lem.~\ref{lem:shapeVSexcessvalue}.
In the setup of the Lemma, 
suppose that for any tree $T$, conditioned on $\textsc{p-shape}(T)$, 
the edge lengths in $T$ are independent.
Show that $f(x)$ is an exponential p.d.f.

\medskip
\item
Can the finite second moment assumption in Prop. \ref{prop:DSS} be removed?
Also, does \eqref{laplace} characterize the exponential distribution 
(like the characterizations in Appendix \ref{sec:char_exp})?

\medskip
\item 
In the context of Sect. \ref{sec:morse}, extend the one-dimensional result of
Prop.~\ref{ts_prune} to higher dimensions.
Specifically,
consider an $n$-dimensional compact differentiable manifold $M=M^n$, 
and a Morse function $f:M \rightarrow \mathbb{R}$. 
Construct a natural Morse function $f^{(1)}:M \rightarrow \mathbb{R}$ such that 
$$\textsc{level}\big(f^{(1)}\big)=\cR\big(\textsc{level}(f)\big).$$

\medskip
\item 
In the setting of Thm.~\ref{Mthm} from Sect. \ref{sec:Kingman}, establish the asymptotic ratio-Horton law (Def.~\ref{def:HortonList}) for the Kingman's coalescent tree, and, if possible, prove the asymptotic strong Horton law (Def.~\ref{def:HortonList}).
Specifically, prove $\lim\limits_{j \rightarrow \infty} {\cN_j \over \cN_{j+1}}=R$, and if possible, $\lim\limits_{j \rightarrow \infty} \big(\cN_j R^j\big)=const$. 
Is it possible to derive a closed form expression for the Horton exponent $R$?

\medskip
\item 
Find a suitable ramification
of the generalized dynamical pruning sufficient for describing the evolution of the shock tree 
in the one-dimensional inviscid Burgers equation \eqref{eqn:inviscidBurgers} and 
its multidimensional modification known as the {\it adhesion model} \cite{BK07,FrischBec,Gurbatov}.
Use this to complement the framework developed in \cite{Sinai92,Bertoin98,Bertoin2000,Giraud}.

\end{enumerate}

\newpage
\appendix

\section{Weak convergence results of Kurtz for density dependent population processes}\label{sec:kurtz}
We first formulate the framework for the  convergence result of Kurtz as stated in Theorem 2.1 in Chapter 11 of \cite{EK86} (Theorem 8.1 in \cite{Kurtz81}).
There, the {\it density dependent population processes} are defined as continuous time Markov processes with state spaces in $\mathbb{Z}^d$, and transition intensities represented as follows
\begin{equation}\label{tintK}
q^{(n)}(k, k+\ell) = n \left[ \beta_{\ell} \left( \frac{k}{n} \right) + O\left( \frac{1}{n} \right) \right],
\end{equation}
where $\ell, k \in \mathbb{Z}^d$, and $\beta_\ell$ is a given collection of rate functions. 
\index{density dependent population process}

\medskip
\noindent
In Section 5.1 of \cite{Aldous}, Aldous observes that the results from Chapter 11 of Ethier and Kurtz \cite{EK86} can be used to prove the weak convergence of  a Marcus-Lushnikov  process to the solutions of Smoluchowski system of equations in the case when the Marcus-Lushnikov process can be formulated as a finite dimensional density dependent population process. 
Specifically,  the  Marcus-Lushnikov processes  corresponding to the multiplicative and Kingman's coalescent with the monodisperse initial conditions ($n$ singletons) can be represented as finite dimensional density dependent population processes defined above.

\medskip
\noindent
Define $F(x)=\sum\limits_\ell \ell \beta_\ell (x)$. Then, Theorem 2.1 in Chapter 11 of \cite{EK86} (Theorem 8.1 in \cite{Kurtz81}) states the following law of large numbers. 
Let $\hat{X}_n(t)$ be the Markov process with the intensities $q^{(n)}(k, k+\ell)$ given in (\ref{tintK}), and let $X_n(t)=n^{-1}\hat{X}_n(t)$. Finally, let $|x|=\sqrt{\sum x_i^2 }$ denote the Euclidean norm in $\mathbb{R}^d$.

\begin{thm} \label{kurtzT}
Suppose for all compact $\mathcal{K} \subset \mathbb{R}^d$,
$$\sum_{\ell} |\ell| \sup_{x \in \mathcal{K}} \beta_\ell (\bar x) < \infty,$$
and there exists $M_\mathcal{K}>0$ such that
\begin{equation}\label{KurtzLipschitz}
|F(x)-F(y)| \leq M_\mathcal{K}|x-y|, \qquad \text{ for all }x,y \in \mathcal{K}.
\end{equation}
Suppose $\lim\limits_{n \to \infty}X_n(0)=x_0$, and $X(t)$ satisfies
\begin{equation}\label{KurtzDE}
X(t) = X(0) + \int_0^t F(X(s)) ds,
\end{equation}
for all $T \geq 0$. Then
\begin{equation}\label{KurtzAS}
\lim\limits_{n \to \infty} \sup\limits_{s \in [0,T]} |X_n(s)-X(s)|=0 \qquad \text{ a.s. }
\end{equation}
\end{thm}


\section{Characterization of exponential random variables}\label{sec:char_exp}

This section contains a number of characterization results for exponential random variables that we use in this manuscript.
We refer the reader to \cite{AV86,Balakrishnan} for more on characterization of exponential random variables. 

\medskip
\noindent
The following result of K.\,S.~Lau and C.\,R.~Rao \cite{LauRao} that implies a characterization of exponential random variables
is used by us for establishing Lemma \ref{lem:shapeVSexcessvalue}. See \cite{BL14book} for more on Integrated Cauchy Functional Equations.
\index{integrated Cauchy functional equation}
\begin{lem}[{{\bf \cite{LauRao}}}]\label{thm:LauRao}
Consider an Integrated Cauchy Functional Equation
\be\label{eqn:ICFE}
\int\limits_0^\infty {G(x+y) \over G(y)} \,d\mu(y) \,=G(x) \quad \forall x \geq 0,
\ee
where $\mu(\cdot)$ is a p.d.f. on $[0,\infty)$ and $G(x)>0$ for $x$ in the support of $\mu$.
Then, $G(x)=e^{-\lambda x}$ for some $\lambda \geq 0$.
\end{lem}

\noindent
The following characterization of exponential random variables follows immediately from Lemma \ref{thm:LauRao}.
\begin{lem}\label{lem:expcharacterization}
Consider a p.d.f. $g(x)$ defined on $[0,\infty)$, and satisfying
\be\label{eqn:intgg}
g(x)=2\int\limits_0^\infty g(x+y) g(y) \,dy \quad \forall x \geq 0.
\ee
Then, $g(x)$ is an exponential density function.
\end{lem}
\begin{proof}
Let $G(a)=\int\limits_a^\infty g(x)\,dx$. Then, integrating \eqref{eqn:intgg}, we have for all $a \geq 0$,
\be\label{eqn:intintgg}
G(a)=\int\limits_a^\infty g(x)\,dx=2\int\limits_0^\infty G(a+y) g(y) \,dy =\int\limits_0^\infty {G(a+y) \over G(y)} \,d\mu(y),
\ee
where $\mu(y)=1-G^2(y)$  is a p.d.f. on $[0,\infty)$. 
We notice that \eqref{eqn:intintgg} produces equation \eqref{eqn:ICFE}. 
Hence, by Lem.~\ref{thm:LauRao}, $G(x)=e^{-\lambda x}$, where $\lambda>0$ as $g(x)$ is p.d.f.
\end{proof}

\medskip
\noindent
Next, we recall the Parseval's identity, which we will use in the proof of characterization Lemma \ref{lem:expcharacterization2}.
\begin{thm}[{\bf Parseval's identity,} \cite{Ushakov}]\label{thm:Parselval}
For a pair of cumulative distribution functions $F(x)$ and $G(x)$ and their respective characteristic functions $\widehat{f}(s)$ and $\widehat{g}(s)$
the following identity holds for all $s \in \mathbb{R}$
$$\int\limits_{-\infty}^\infty e^{isx} \,\widehat{g}(x) \,dF(x) \,=\, \int\limits_{-\infty}^\infty  \widehat{f}(x+s) \,dG(x).$$
\end{thm}

\noindent
We give yet another characterization of the exponential p.d.f. $\phi_\lambda(x)=\lambda e^{-\lambda x}{\bf 1}_{\{x \geq 0\}}$ as defined in \eqref{exp}.
\begin{lem}\label{lem:expcharacterization2}
Consider a p.d.f. $g(x)$ defined on $[0,\infty)$, and satisfying
\be\label{eqn:intgg_phi}
\phi_\lambda(x)=2\int\limits_0^\infty g(x+y) g(y) \,dy \quad \forall x \geq 0.
\ee
Then, $g(x)=\phi_\lambda(x)$.
\end{lem}
\begin{proof}
Observe that $\phi_\lambda(x)$ satisfies
\be\label{eqn:intphiphi_phi}
\phi_\lambda(x)=2\int\limits_0^\infty \phi_\lambda(x+y) \phi_\lambda(y) \,dy \quad \forall x \geq 0.
\ee
Thus,
\be\label{eqn:intphiphi_gg}
\int\limits_0^\infty \phi_\lambda(x+y) \phi_\lambda(y) \,dy = \int\limits_0^\infty g(x+y) g(y) \,dy \quad \forall x \geq 0.
\ee
Hence, for the two pairs of independent random variables 
$$X_1,X_2  \overset{d}{\sim} \text{ p.d.f. } \phi_\lambda(x) ~~\text{ and }~~Y_1,Y_2  \overset{d}{\sim} \text{ p.d.f. }g(x),$$
we have
$$X_1-X_2 ~\overset{d}{=}~Y_1-Y_2.$$
Therefore, for the characteristic functions $\widehat{\phi}_\lambda$ and $\widehat{g}$, we have
\be\label{eqn:L2_phi_g}
\big|\widehat{\phi}_\lambda(s)\big|^2 \,= \, {\sf E} \left[e^{is(X_1-X_2)}\right] \,= \, {\sf E} \left[e^{is(Y_1-Y_2)}\right] \,=\, \big|\widehat{g}(s)\big|^2 .
\ee
Observe that \eqref{eqn:L2_phi_g}  can be also obtained from \eqref{eqn:intphiphi_gg} via multiplying both sides by $e^{isx}$ and integrating.

\medskip
\noindent
Next, from the Parseval's identity Theorem \ref{thm:Parselval} and \eqref{eqn:intphiphi_gg}, we have $\forall s \geq 0$,
\begin{align}\label{eqn:hatphiphi_hatgg}
\int\limits_0^\infty e^{isy} \widehat{g}(y)\, g(y) \,dy &= \int\limits_0^\infty g(s+y) g(y) \,dy  \nonumber \\
&= \int\limits_0^\infty \phi_\lambda(s+y) \phi_\lambda(y) \,dy \nonumber \\
&= \int\limits_0^\infty e^{isy} \widehat{\phi}_\lambda(y)\, \phi_\lambda(y) \,dy. 
\end{align}
Therefore,
$$\widehat{g}(x) \, g(x) \equiv \widehat{\phi}_\lambda(x) \, \phi_\lambda(x),$$
and \eqref{eqn:L2_phi_g} implies for any $x>0$, 
$$g(x)={\big|\widehat{g}(x)\, g(x)\big| \over \big|\widehat{g}(x)\big|}={\big|\widehat{\phi}_\lambda(x) \, \phi_\lambda(x)\big| \over \big|\widehat{\phi}_\lambda(x)\big|}=\phi_\lambda(x).$$
\end{proof}

\section{Notations}\label{sec:notations}

\begin{tabular}{l p{10cm} }
 $\rho$ &  {\it root} vertex; \\
 $\phi$ &  the {\it empty tree} comprised of a root vertex and no edges; \\
 $\cT$ &   the space of finite unlabeled rooted reduced trees with no planar embedding; \\
 $\cL$ &   the space of trees from $\cT$ with edge lengths; \\
 $\T$ &   the space of finite unlabeled rooted reduced trees with planar embedding;\\
 $\L$ &   the space of trees from $\T$ with edge lengths;\\
 $\cB\cS$ & the subspace of {\it binary} trees in a given space of rooted trees $\cS$, e.g., $S=\cT,\T,\cL,\L$;\\
 $\cS^{|}$ &    the subspace of {\it planted} trees in a given space of rooted trees $\cS$;\\
 $\cS^{\vee}$ &    the subspace of {\it stemless} trees in a given space of rooted trees $\cS$;\\
 $\mBL$ &    the subspace of {\it mass-equipped} trees in $\BL$;\\
 $\mathcal{GW}(\{q_k\})$ & the probability distribution of (combinatorial) Galton-Watson trees 
 on $\cT^|$  with offspring p.m.f. $\{q_k\}$;\\
 $\mathcal{GW}(q_0,q_2)$ &    the probability distribution of (combinatorial) binary Galton-Watson trees on $\cBT^|$ with termination 
  probability $q_0$ and split probability $q_2$;\\
  $\mathcal{GW}_{\rm plane}(q_0,q_2)$ &  the planar embedding of trees in $\mathcal{GW}(q_0,q_2)$  that assigns the left-right orientation to each pair of offsprings uniformly and independently;\\
 ${\sf GW}(\lambda',\lambda)$  &    the probability distribution of exponential binary Galton-Watson trees (see Def. \ref{def:binary});\\
 ${\sf GW}(\lambda)$  &    the probability distribution of exponential critical binary Galton-Watson trees (Def. \ref{def:cbinary});\\
 ${\sf GW}_\infty(\lambda)$ &    the probability distribution of infinite exponential critical binary Galton-Watson tree built from the leaves down, 
with parameter $\lambda>0$;\\
 $X \stackrel{d}{\sim} D$ &  random element $X$ has distribution $D$; \\
 $X \stackrel{d}{=} Y$ &  random elements $X$ and $Y$ are equidistributed; \\
 $\hat{t}(z)$ &  the generating function ($z$-transform) of a sequence $\{t(j)\}_{j=0,1,\hdots}$; \\
 $\widehat{f}(s)$ &  the characteristic function of a random variable with p.d.f. $f(x)$; \\
 $\mathcal{L}f(s)$ &  the Laplace transform of $f(x)$; \\
 $\stackrel{a.s.}{\to}$ & almost sure convergence;\\
 $\overset{d}{\rightarrow}$ & convergence in distribution; \\
 $\stackrel{p}{\to}$ & convergence in probability;\\
 $x_n \sim y_n$ & asymptotic equivalence: $\displaystyle\lim_{n\to\infty} \frac{x_n}{y_n} = 1$;\\
 $\cE^{\rm ex}$ &    the space of all positive piece-wise linear continuous finite excursions with alternating slopes $\pm 1$;\\
 $\mathbb{N}$ & the set of natural numbers $\{1,2,\dots\}$;\\ 
 $\mathbb{Z}_+$ & the set of nonnegative integer numbers $\{0,1,2,\dots\}$. 
 
\end{tabular}

\section{Standard distributions}\label{sec:distributions}
\begin{tabular}{l p{10cm} }
 ${\sf Exp}(\lambda)$ & the exponential distribution with rate $\lambda$; the respective
 p.d.f. is $\phi_{\lambda}(x) = \lambda e^{-\lambda x}$, $x\ge 0$; \\
 ${\sf Gamma}(\alpha,\beta)$ & the gamma distribution with shape parameter $\alpha>0$ and rate 
 parameter $\beta>0$; the respective p.d.f. is $f(x) = \beta^{\alpha}x^{\alpha-1}e^{-\beta x}/\Gamma(\alpha)$ for $x\ge 0$;\\
 ${\sf Geom}_0(p)$ & the geometric distribution with p.m.f. $p(m)=p\,(1-p)^m$ for $~m=0,1,2,\hdots ~$;\\
 ${\sf Geom}_1(p)$ &    the geometric distribution with p.m.f. $p(m)=p\,(1-p)^{m-1}$ for $~m=1,2,3,\hdots ~$;\\
 $ {\sf Poi}(\lambda)$ &the Poisson distribution with rate $\lambda>0$; the respective p.m.f. is
 $p(m) = \lambda^m e^{-\lambda}/m!$ for $m = 0,1,2,\dots$;\\
 ${\sf Unif}(A)$ &    the uniform distribution over a set $A$.\\

\end{tabular}

\section{Tree functions and mappings}\label{sec:treefunctions}
\begin{tabular}{l p{10cm} }
 $\textsc{length}(T)$ & the length of a tree $T\in \cL$ (or $\L$) defined as the sum of the lengths of its edges; \\
 $\textsc{height}(T)$ &  the height of a tree $T\in \cL$ (or $\L$) defined as the maximal distance between the root and a vertex;\\
 $\textsc{shape}(T)$ & the combinatorial shape of a tree $T\in \cL$ (or $\L$); it is a mapping from $\cL$ (or $\L$) to $\cT$;\\
 ${\textsc{p-shape}(T)}$ & the combinatorial shape of a tree $T\in \L$ together with the tree's planar embedding; it is a mapping from $\L$ to $\T$;\\
 $\textsc{level}(f)$ & the level set tree of a continuous function $f(x)$.\\

\end{tabular}

\section*{Acknowledgements}
First and foremost, we are grateful to Ed Waymire for his continuing advice, encouragement,
and support on more levels than one. 
We would like to thank Amir Dembo for providing valuable feedback, including 
the very idea of writing this survey;
Jim Pitman for his comments and suggesting relevant publications; and
Tom Kurtz for his insight regarding infinite dimensional population processes. 

We would like to express our appreciation to the colleagues with whom we discussed 
this work at different stages of its preparation:
Maxim Arnold, Krishna Athreya, Bruno Barbosa, Vladimir Belitsky, Yehuda Ben-Zion, Robert M. Burton, Mickael Checkroun, Evgenia Chunikhina, 
Steve Evans, Efi Foufoula-Georgiou, Andrei Gabrielov, Michael Ghil, Mark Meerschaert, 
George Molchan, Peter T. Otto, Scott Peckham, Victor P\'{e}rez-Abreu, Jorge Ramirez, 
Andrey Sarantsev, Sunder Sethuraman, Alejandro Tejedor, Enrique Thomann, Donald L. Turcotte, Guochen Xu, Anatoly Yambartsev,  
and many others.
Finally, we thank the participants of the workshop {\it Random Trees: Structure, Self-similarity, and Dynamics} that took place during April 23-27, 2018, at the
Centro de Investigaci\'on en Matem\'aticas (CIMAT), Guanajuato, M\'exico, for sharing their knowledge and research results.
 
YK would like to express his gratitude to  IME - University of S\~{a}o Paulo (USP), S\~{a}o Paulo, Brazil, for hosting him during his 2018-2019 sabbatical.


\printindex


\end{document}